\documentclass[oneside,a4paper,reqno,psamsfonts]{amsbook}
\usepackage[latin1]{inputenc} 
\usepackage{amscd}
\usepackage{amsopn}
\usepackage{amstext}
\usepackage{amsxtra}
\usepackage{amssymb}
\usepackage{stmaryrd}
\usepackage{upref}
\usepackage{textcase} 
\usepackage{clrscode}
\renewcommand{\Comment}{$\hspace*{-0.075em} //$ } 
\usepackage{bm} 
\usepackage{graphicx}
\usepackage{setspace}
\numberwithin{section}{chapter}
\numberwithin{equation}{chapter}
\numberwithin{figure}{chapter}

\newcounter{algorithm}
\numberwithin{algorithm}{chapter}

\theoremstyle{plain}
\newtheorem{thm}{Theorem}[chapter]
\newtheorem{lem}[thm]{Lemma}
\newtheorem{cl}[thm]{Corollary}
\newtheorem{pr}[thm]{Proposition}

\newtheorem{conj}[thm]{Conjecture}

\theoremstyle{definition}
\newtheorem{deff}[thm]{Definition}

\theoremstyle{remark}

\newtheorem{rem}[thm]{Remark}

\providecommand{\abs}[1]{\left\lvert #1 \right\rvert}

\providecommand{\ceil}[1]{\left\lceil #1 \right\rceil}

\providecommand{\set}[1]{\left\lbrace #1 \right\rbrace}
\providecommand{\gen}[1]{\left\langle #1 \right\rangle}
\providecommand{\Sym}[1]{\operatorname{Sym}( #1 )}
\renewcommand{\Pr}[1]{\operatorname{Pr}[ #1 ]}

\newcommand{\field}[1]{\mathbb{#1}}

\newcommand{\N}{\field{N}}
\newcommand{\Z}{\field{Z}}

\newcommand{\OO}{\field{O}}

\newcommand{\F}{\field{F}}
\newcommand{\PS}{\field{P}}
\newcommand{\GAP}{\textsf{GAP}}
\newcommand{\MAGMA}{\textsc{Magma}}

\newcommand{\OV}{\mathcal{O}}

\DeclareMathOperator{\GL}{GL}
\DeclareMathOperator{\PGL}{PGL}

\DeclareMathOperator{\SL}{SL}

\DeclareMathOperator{\SO}{SO}
\DeclareMathOperator{\Sz}{Sz}
\DeclareMathOperator{\Sp}{Sp}
\DeclareMathOperator{\chr}{char}
\DeclareMathOperator{\Aut}{Aut}

\DeclareMathOperator{\PSL}{PSL}
\DeclareMathOperator{\PPSL}{(P)SL}
\DeclareMathOperator{\diag}{diag}
\DeclareMathOperator{\Tr}{Tr}
\DeclareMathOperator{\SLP}{\mathtt{SLP}}
\DeclareMathOperator{\G2}{{^2}G_2}
\DeclareMathOperator{\LargeRee}{{^2}F_4}
\DeclareMathOperator{\Ree}{Ree}

\DeclareMathOperator{\Norm}{N}
\DeclareMathOperator{\Cent}{C}
\DeclareMathOperator{\Zent}{Z}
\DeclareMathOperator{\EndR}{End}
\DeclareMathOperator{\Hom}{Hom}
\DeclareMathOperator{\Ker}{Ker}

\DeclareMathOperator{\O2}{O_2}

\DeclareMathOperator{\RP}{\bf{RP}}
\DeclareMathOperator{\NP}{\bf{NP}}
\DeclareMathOperator{\PP}{\bf{P}}
\DeclareMathOperator{\coRP}{\bf{co-RP}}
\DeclareMathOperator{\ZPP}{\bf{ZPP}}
\DeclareMathOperator{\Imm}{Im}
\DeclareMathOperator{\Mat}{Mat}
\DeclareMathOperator{\Dih}{D}

\newcommand{\OR}[1]{\operatorname{O} \bigl( #1 \bigr)}

\title{Algorithmic problems in twisted groups of Lie type}
\author{Jonas Henrik Ambj\"orn B\"a\"arnhielm}
\address{A thesis submitted for the degree of Doctor of Philosophy at}
\author{Queen Mary, University of London}
\date{August 2007}

\begin{document}

\mainmatter

\begin{titlepage}
\refstepcounter{page}


\begin{center}
\begin{figure}[b]
\includegraphics[scale=1.5]{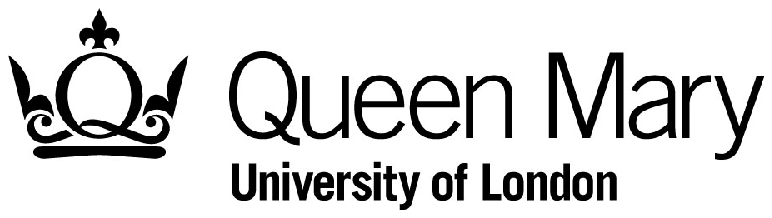}
\end{figure}

\mbox{}

\Huge
\textbf{Algorithmic problems in \\ twisted groups of Lie type}

\mbox{}

\LARGE
Jonas Henrik Ambj\"orn B\"a\"arnhielm

\mbox{}

\large
A thesis submitted for the degree of \\
\emph{Doctor of Philosophy} \\
at Queen Mary, University of London

\mbox{}

\end{center}

\end{titlepage}


\onehalfspacing

\chapter*{Declaration}
I hereby declare that, to the best of my knowledge, the material
contained in this thesis is original and my own work, except where
otherwise indicated, cited, or commonly-known.

I have not submitted any of this material in partial of complete
fulfilment of requirements for another degree at this or any other
university.

\chapter*{Abstract}
  This thesis contains a collection of algorithms for working with the
  twisted groups of Lie type known as \emph{Suzuki} groups, and small and
  large \emph{Ree} groups.
  
  The two main problems under consideration are \emph{constructive
    recognition} and \emph{constructive membership testing}. We also
  consider problems of generating and conjugating Sylow and maximal
  subgroups.

  The algorithms are motivated by, and form a part of, the Matrix Group
  Recognition Project. Obtaining both theoretically and practically
  efficient algorithms has been a central goal. The algorithms have
  been developed with, and implemented in, the computer algebra system
  $\MAGMA$.

\tableofcontents

\listoffigures

\chapter*{Acknowledgements}

I would first of all like to thank my supervisor, Charles Leedham-Green, for
endless help and encouragement. I would also like to thank the
following people, since they have helped me to various extent during
the work presented in this thesis: John Bray, Peter Brooksbank, John
Cannon, Sergei Haller, Derek Holt, Alexander Hulpke, Bill Kantor, Ross Lawther,
Martin Liebeck, Klaus Lux, Frank L\"ubeck, Scott Murray, Eamonn
O'Brien, Geoffrey Robinson, Colva Roney-Dougal, Alexander Ryba, Ákos
Seress, Leonard Soicher, Mark Stather, Bill Unger, Maud de Visscher,
Robert Wilson.

\chapter*{Notation}
\begin{itemize}
\item[$\bar{F}$] Algebraic closure of the field $F$
\item[$g^h$] conjugate of $g$ by $h$, \emph{i.e.} $g^h = h^{-1} g h$
\item[${[g, h]}$] commutator of $g$ and $h$, \emph{i.e.} $[g, h] = g^{-1} h^{-1} g h$
\item[$\Tr(g)$] trace of the matrix $g$
\item[$\Cent_n$] cyclic group of order $n$, \emph{i.e.} $\Cent_n \cong \Z / n\Z$
\item[$\F_q$] finite field of size $q$ (or its additive group)
\item[$\F_q^{\times}$] multiplicative group of $\F_q$
\item[$\xi(d)$] The number of field operations required by a random element oracle for $\GL(d, q)$.
\item[$\xi$] The number of field operations required by a random element oracle for $\GL(k, q)$ with $k$ constant
\item[$\chi_D(q)$] number of field operations required by a discrete logarithm oracle in $\F_q$
\item[$\chi_F(d, q)$] number of field operations required by an integer factorisation oracle (which factorises $q^i - 1$, for $1 \leqslant i \leqslant d$)
\item[$\Mat_n(R)$] matrix algebra of $n \times n$ matrices over ring $R$
\item[$I_n$] identity $n \times n$ matrix
\item[$E_{i,j}$] matrix with $1$ in position $(i,j)$ and $0$ elsewhere
\item[$\abs{g}$] order of a group element $g$
\item[$\phi$] Euler totient function
\item[$\sigma_0(n)$] number of positive divisors of $n \in \N$ (where $\sigma_0(n) < 2^{(1 + \varepsilon) \log_{\mathrm{e}}(n) / \log\log_{\mathrm{e}}(n)}$)
\item[$\psi$] Frobenius automorphism in a field $F$ (\emph{i.e.} $\psi : x \mapsto x^p$ where $\chr(F) = p$)
\item[$\Phi(G)$] Frattini subgroup of $G$ (intersection of all maximal subgroups of $G$)
\item[$\mathrm{O}_p(G)$] largest normal $p$-subgroup of $G$
\item[$G_P$] stabiliser in $G$ of the point $P$
\item[$G^{\prime}$] derived subgroup (commutator subgroup) of $G$
\item[$\Norm_G(H)$] normaliser of $H$ in $G$
\item[$\Cent_G(g)$] centraliser of $g$ in $G$
\item[$\Zent(G)$] centre of $G$
\item[$\EndR_H(M)$] algebra of $H$-endomorphisms of $G$-module $M$, where $H \leqslant G$
\item[$\Aut(M)$] automorphism group of $G$-module $M$ (if $G \leqslant \GL(d, q) = H$ then $\Aut(M) = \Cent_H(G)$)
\item[$\mathcal{S}^2(M)$] symmetric square of module $M$ over a field $F$ (where $M \otimes M = \mathcal{S}^2(M) \oplus \wedge^2(M)$ if $\chr(F) > 2$)
\item[$\wedge^2(M)$] exterior square of module $M$
\item[$\Sym{O}$] symmetric group on the set $O$
\item[$\Sym{n}$] symmetric group on $n$ points
\item[$\Dih_n$] dihedral group of order $n$
\item[$\PS(V)$] projective space corresponding to vector space $V$
\item[$\PS^n(F)$] $n$-dimensional projective space over the field $F$
\item[${[G : H]}$] index of $H$ in $G$
\item[$G {.} H$] extension of $G$ by $H$ (a group $E$ such that $G
  \trianglelefteqslant E$ and $E / G \cong H$)
\item[$G {:} H$] split extension of $G$ by $H$ (as with $G {.} H$ but
  $E$ also has a subgroup $H_0 \cong H$ such that $E = G H_0$ and $G \cap H_0 = \gen{1}$)
\item[$\OR{\cdot}$] standard time complexity notation
\item[$\Psi$] Automorphisms defining $\Sz(q)$ or $\Ree(q)$
\item[$\varphi, \theta, \rho$] group homomorphisms
\end{itemize}

\chapter{Introduction and preliminaries}
\label{chapter:intro}

\section{Introduction}

This thesis contains algorithms for some computational problems
involving a few classes of the finite simple groups. The main focus is
on providing efficient algorithms for constructive recognition and
constructive membership testing, but we also consider the conjugacy problem for Sylow and maximal subgroups.

The work is in the area of \emph{computational group theory}
  (abbreviated CGT), where
  one studies algorithmic aspects of groups, or solves group theoretic
  problems using computers. A group can be represented as a data
  structure in several ways, and perhaps the most important ones are
  \emph{permutation groups}, \emph{matrix groups} and \emph{finitely presented groups}.

The permutation group setting has been studied
since the early 1970's, and the basic technique which underlies most
algorithms is the construction of a \emph{base} and a \emph{strong
generating set}. If $G$ is a permutation
group of degree $n$, this involves constructing a descending chain of
subgroups of $G$, where each subgroup in the chain has index at most
$n$ in its predecessor. The permutation group algorithm machinery is
summarised in \cite{seress03}.

For matrix groups, the classical method is also to construct a base and strong generating set. However, in
general a matrix group has no faithful permutation representation whose degree
is polynomial in the size of the input. Hence the indices in
the subgroup chain will be too large and the permutation group
algorithms will not be efficient. For example, $\SL(d, q)$ has no
proper subgroup of index less than $(q^d - 1) / (q - 1)$, and $q$ is
exponential in the size of the input, since a matrix has size $\OR{d^2
\log(q)}$.

Historically there were two schools within CGT. One consists of people
with a more computational complexity background, whose primary goal
was to find theoretically good (polynomial time) algorithms.  The
implementation and practical performance of the algorithms were less
important, since the computational complexity view, based on many real
examples, is that if the \lq\lq polynomial barrier'' is broken, then
further research will surely lead also to good practical algorithms.
The other school consists of people with a more group theoretic
background, whose primary goal was to solve computational problems in
group theory (historically often one-time problems with the sporadic
groups), and hence to develop algorithms that can be easily
implemented and that run fast on the current hardware and on the
specific input in question. The asymptotic complexity of the
algorithms was less important, and perhaps did not even make sense in
the case of sporadic groups.

The distinction between these schools has become less noticeable during
the last $15$ years, but during that time there has also been much
work on algorithms for matrix groups, and there are two main
approaches that roughly correspond to these two schools. The first
approach is the \lq\lq black box'' approach, that considers the matrix
groups as \emph{black box groups} (see \cite[pp. 17]{seress03}). This
was initiated by \cite{luks92computing} (but goes back to
\cite{BabaiSzemeredi84}) and much of it is summarised in
\cite{babaibeals01}. The other approach is the \lq\lq geometric''
approach, also known as \emph{The Matrix Group Recognition Project}
(abbreviated MGRP), whose underlying idea is to use a famous theorem
of Aschbacher\cite{aschbacher84} which roughly says that a matrix
group either preserves some geometric structure, or is a simple group.

Although the author has a background perhaps more in the
computational complexity school, the work in this thesis forms a part
of MGRP. The first specific goal in that approach is to obtain an efficient
algorithm that finds a composition series of a matrix group. There
exists a recursive algorithm\cite{crlg01} for this, which relies on a
number of other algorithms that determine which kind of geometric
structure is preserved, and the base cases in the recursion are the
classes of finite simple groups.

Hence this algorithm reduces the problem of finding
a composition series to various problems concerning the composition
factors of the matrix group, which are simple groups. The work
presented here is about computing with some of these simple groups.

For each simple group, a number of problems arise. The simple group is
given as $G = \gen{X} \leqslant \PGL(d, q)$ for some $d, q$
and we need to consider the following problems:
\begin{enumerate}
\item The problem of \emph{recognition} or \emph{naming} of $G$, \emph{i.e.} decide the name of $G$, as in the classification of the finite simple groups.
\item The \emph{constructive membership} problem. Given $g \in
  \PGL(d, q)$, decide whether or not $g \in G$, and if so express
  $g$ as a straight line program in $X$.
\item The problem of \emph{constructive recognition}. Construct an
  isomorphism $\varphi$ from $G$ to a \emph{standard copy} $H$ of $G$
  such that $\varphi(g)$ can be computed efficiently for every $g \in
  G$. Such an isomorphism is called \emph{effective}. Also of interest
  is to construct an effective inverse of $\varphi^{-1}$, which
  essentially is constructive membership testing.
\end{enumerate}

To find a composition series using \cite{crlg01}, the problems involving the composition factors that we need to solve are naming
and constructive membership. However, the effective isomorphisms of these composition factors to standard copies can also be very useful. Given such isomorphisms, many problems,
sometimes including constructive membership, can be reduced to the standard
copies. Hence these isomorphisms play a central role in computing with
a matrix group once a composition series has been constructed.

In general, the constructive recognition problem is computationally
harder than the constructive membership problem, which in turn is
harder than the naming problem. Most of our efforts will therefore go
towards solving constructive recognition. In fact, the naming problem is not
considered in every case here, since the algorithm of
\cite{general_recognition} solves this problem. However, that
algorithm is a Monte Carlo algorithm, and in some cases we can improve
on that and provide Las Vegas algorithms (see Section \ref{section:prob_algorithms}).

The algorithms presented here
are not \emph{black box} algorithms, but
rely heavily on the fact that the group elements are matrices, and use
the representation theory of the groups in question. However, the
algorithms will work with all possible representations of the groups,
so a user of the algorithms can consider them black box in that sense.

\section{Preliminaries}

We now give preliminary discussions and results that will be necessary later on.

\subsection{Complexity}

We shall be concerned with the time complexity of the algorithms
involved, where the basic operations are the field operations, and not
the bit operations. All simple arithmetic with matrices can be done
using $\OR{d^3}$ field operations, and raising a matrix to the
$\OR{q}$ power can be done using $\OR{d^3}$ field operations using
\cite{crlg95}. These complexity bounds arise when using the naive
matrix multiplication algorithm, which uses $\OR{d^3}$ field
operations to multiply two $d \times d$ matrices. More efficient algorithms for matrix multiplication do exist. Some are also fast in practice,
like the famous algorithm of Strassen\cite{MR0248973}, which uses $\OR{d^{\log_2 7}}$
field operations. Currently, the most efficient matrix multiplication algorithm is the Coppersmith-Winograd algorithm of \cite{MR1056627, 1097473}, which uses
$\OR{d^{2.376}}$ field operations, but it is not practical. The
improvements made by these algorithms over the naive matrix
multiplication algorithm are not noticeable in practice for the matrix
dimensions that are currently within range in the MGRP. Therefore we will only use the naive algorithm, which also
simplifies the complexity statements.

When we are given a group $G \leqslant \GL(d, q)$ defined by
a set $X$ of generators, the size of the input is $\OR{\abs{X}
  d^2 \log(q)}$. A field element takes up $\OR{\log(q)}$ space, and a
matrix has $d^2$ entries.

We shall often assume an oracle for the discrete logarithm problem in
$\F_q$ (see \cite[Section $20.3$]{VonzurGathen03} and \cite[Chapter
$3$]{MR1745660}). In the general discrete logarithm problem, we
consider a cyclic group $G$ of order $n$. The input is a generator
$\alpha$ of $G$, and some $x \in G$. The task is to find $1 \leqslant
k < n$ such that $\alpha^k = x$. In $\F_q$ the multiplicative group
$\F_q^{\times}$ is cyclic, and the discrete logarithm problem turns
up.  It is a famous and well-studied problem in theoretical computer
science and computational number theory, and it is unknown if it is $\NP$-complete or if it is in
$\PP$, although the latter would be very surprising. Currently the
most efficient algorithm has sub-exponential complexity. There are
also algorithms for special cases, and an important case for us is
when $q = 2^n$. Then we can use Coppersmith's algorithm of
\cite{cop84, 705538}, which is much faster in practice than the
general algorithms. It is not polynomial time, but has time complexity
$\OR{\exp (c n^{1/3} \log(n)^{2/3})}$, where $c > 0$ is a small
constant. We shall assume that the discrete logarithm oracle in $\F_q$
uses $\OR{\chi_D(q)}$ field operations.

Similarly we will sometimes assume an oracle for the integer
factorisation problem (see \cite[Chapter $19$]{VonzurGathen03}), whose
status in the complexity hierarchy is similar to the discrete
logarithm problem. More precisely, we shall assume we have an oracle that, given $d
\geqslant 1$ and $\F_q$, factorises all the integers $q^i - 1$ for $1
\leqslant i \leqslant d$, using $\OR{\chi_F(d, q)}$ field operations. By \cite[Theorem $8.2$]{babaibeals01},
assuming the Extended Riemann Hypothesis this is equivalent to the
standard integer factorisation problem. The reason for having this slightly
different factorisation oracle will become clear in Section
\ref{section:matrix_orders}.

Except for these oracles, our algorithms will be polynomial time, so
from a computational complexity perspective our results will imply
that the problems we study can be reduced to the discrete logarithm
problem or the integer factorisation problem. This is in line with
MGRP, whose goal from a complexity point of view is to prove that
computations with matrix groups are not harder than (and hence equally
hard as) these two well-known problems.

\subsection{Probabilistic algorithms}
\label{section:prob_algorithms}
The algorithms we consider are probabilistic of the types known as
\emph{Monte Carlo} or \emph{Las Vegas} algorithms. These types of
algorithms are discussed in \cite[Section 1.3]{seress03} and
\cite[Section 3.2.1]{hcgt}. In short, a Monte Carlo algorithm for a
language $X$ is a
probabilistic algorithm with an input parameter $\varepsilon \in (0,
1)$ such that on input $x$
\begin{itemize}
\item if $x \in X$ then the algorithm returns \texttt{true} with
  probability at least $1 - \varepsilon$ (otherwise it returns \texttt{false}),
\item if $x \notin X$ then the algorithm returns \texttt{false}.
\end{itemize}
The parameter $\varepsilon$ is therefore the maximum error
  probability. This type of algorithm is also called \emph{one-sided Monte Carlo
  algorithm} with no \emph{false negatives}. The languages with such
  algorithms form the complexity class $\RP$. In the same way one can
  define algorithms with no false positives, and the corresponding
  languages form the class $\coRP$. The class $\ZPP = \RP \cap \coRP$
  consists of the languages that have Las Vegas algorithms, and these
  are the type of algorithms that we will be most concerned with. A Las Vegas algorithm either returns \texttt{failure}, with
probability at most $\varepsilon$, or otherwise returns a correct
result. Such an algorithm is easily contructed given a Monte Carlo
  algorithm of each type. The time complexity of a Las Vegas algorithm
  naturally depends on $\varepsilon$.

Las Vegas algorithms can be presented concisely as probabilistic
algorithms that either return a correct result, with probability
bounded below by $1/p(n)$ for some polynomial $p(n)$ in the size $n$
of the input, or otherwise return \texttt{failure}. By enclosing such
an algorithm in a loop that iterates $\ceil{\log \varepsilon / \log{(1
    - 1/p(n))}}$ times, we obtain an algorithm that returns
\texttt{failure} with probability at most $\varepsilon$, and hence is
a Las Vegas algorithm in the above sense. Clearly if the enclosed
algorithm is polynomial time, the Las Vegas algorithm is polynomial
time.

One can also enclose the algorithm in a loop that iterates until the
algorithm returns a correct result, thus obtaining a probabilistic
time complexity, and the expected number of iterations is then
$\OR{p(n)}$. This is the way we present Las Vegas algorithms since it
is the one that is closest to how the algorithm is used in practice.

\subsection{Straight line programs}

For constructive membership testing, we want to express an element of a group $G = \gen{X}$ as
a word in $X$. Actually, it should be
a \emph{straight line program}, abbreviated to $\SLP$. If we express the elements as
words, the length of the words might be too large, requiring
exponential space complexity. 

An $\SLP$ is a data structure for words, which ensures that during
evaluation, subwords occurring multiple times are not computed more
often than during construction. Often we want to express an element as
an $\SLP$ in order to obtain its homomorphic image in another group $H = \gen{Y}$ where $\abs{X} = \abs{Y}$. The
evaluation time for the $\SLP$ is then bounded by the time to
construct it times the ratio of the time required for a group
operation in $H$ and in $G$.

Formally, given a set of generators $X$,
an $\SLP$ is a sequence $(s_1, s_2, \dotsc, s_n)$ where
each $s_i$ represents one of the following
\begin{itemize}
\item an $x \in X$
\item a product $s_j s_k$, where  $j, k < i$
\item a power $s_j^n$ where $j < i$ and $n \in \Z$
\item a conjugate $s_j^{s_k}$ where $j, k < i$
\end{itemize}
so $s_i$ is either a pointer into $X$, a pair of pointers to earlier elements
of the sequence, or a pointer to an earlier element and an integer.

To construct an $\SLP$ for a word, one starts by listing
pointers to the generators of $X$, and then builds up the word. To
evaluate the $\SLP$, go through the sequence and perform the
specified operations. Since we use pointers to the elements of $X$, we
can immediately evaluate the $\SLP$ on $Y$, by just changing the pointers so that they point to elements of $Y$.

\subsection{Solving polynomial equations}
One of the main themes in this work is that we reduce search problems
in computational group theory to the problem of solving polynomial
equations over finite fields. The method we use to find the
solutions of a system of polynomial equations is the classical
resultant technique, described in \cite[Section
$6.8$]{VonzurGathen03}. For completeness, we also state the corresponding result for univariate polynomials.

\begin{thm} \label{thm_solve_univariate_polys}
Let $f \in \F_q[x]$ have degree $d$. There exists a Las Vegas algorithm that finds all the roots of $f$ that lie in $\F_q$. The expected time complexity is $\OR{d (\log{d})^2 \log\log(d) \log(dq)}$ field operations.
\end{thm}
\begin{proof}
Immediate from \cite[Corollary 14.16]{VonzurGathen03}.
\end{proof}

\begin{thm} \label{thm_poly_eqns_2vars} 
Let $f_1, \dotsc, f_k \in
  \F_q[x, y] = R$ be such that the ideal $I = \gen{f_1, \dotsc, f_k}
  \trianglelefteqslant R$ is zero-dimensional. Let $n_x =
  \max_i{\deg_x f_i} \geqslant \max_i{\deg_y f_i} = n_y$. There exists a
  Las Vegas algorithm that finds the corresponding affine variety
  $V(I) \subset \F_q^2$. The expected time complexity is $\OR{k n_x^3 (\log{n_x})^2 \log\log(n_x) \log(n_x q)}$ field operations.
\end{thm}
\begin{proof}
Following \cite[Section $6.8$]{VonzurGathen03}, we compute $k - 1$
pairwise resultants of the $f_i$ with respect to $y$ to obtain $k - 1$
univariate polynomials in $x$. By \cite[Theorem
$6.37$]{VonzurGathen03}, the expected time complexity will be $\OR{k
(n_x n_y^2 + n_x^2 n_y)}$ field operations and the resultants will be non-zero
since the ideal is zero-dimensional. 

We can find the set $X_1$ of roots of the first polynomial, then find
the roots $X_2$ of the second polynomial and simultaneously find
$X_1 \cap X_2$. By continuing in the same way we can find the set $X$ of common roots of the resultants. Since their degrees will be
$\OR{n_x^2}$, by Theorem \ref{thm_solve_univariate_polys} we can find
$X$ using $\OR{k n_x^2 (\log{(n_x^2)})^2
  \log\log(n_x^2) \log(n_x^2 q)}$ field operations. Clearly $\abs{X}
\in \OR{n_x^2}$.

We then substitute each $a \in X$ into the $k$ polynomials and
obtain univariate polynomials $f_1(a, y), \dotsc, f_k(a, y)$. These will have degrees $\OR{n_y}$ and as above
we find the set $Y_a$ of their common roots using $\OR{k n_y
(\log{(n_y)})^2 \log\log(n_y) \log(n_y q)}$ field operations. Clearly
$V(I) = \set{(a, b) \mid a \in X, b \in Y_a}$ and hence we can find $V(I)$ using $\OR{\abs{X} k n_y (\log{(n_y)})^2 \log\log(n_y) \log(n_y q)}$ field operations. Thus the time complexity is as stated.
\end{proof}

The following result is a generalisation of the previous result, and we omit the proof, since it is complicated and outside the scope of this thesis.

\begin{thm} \label{thm_poly_eqns_many_vars}
Let $f_1, \dotsc, f_k \in
  \F_q[x_1, \dotsc, x_k] = F$ be such that the ideal $I = \gen{f_1, \dotsc, f_k}
  \trianglelefteqslant F$ is zero-dimensional. Let $n =
  \max_{i,j}{\deg_{x_j} f_i}$. There exists a Las Vegas algorithm that
  finds the corresponding affine variety $V(I) \subset \F_q^k$. The
  expected time complexity is $\OR{(k n^3 + n^k k^2 (\log{n})^2
  \log\log(n^k) \log(n^kq))^k}$ field operations.
\end{thm}
%

As can be seen, if the number of equations, the number of variables
and the degree of the equations are constant, then the variety of a
zero-dimensional ideal can be found in polynomial time. That is the
situation we will be most concerned with. As an alternative to the
resultant technique, one can compute a Gr\"obner basis and then find the
variety. By \cite{MR1106424}, if the ideal is zero-dimensional the time
complexity is $\OR{n^{\OR{k}}}$ where $n$ is the maximum degree of the
polynomials and $k$ is the number of variables. Hence in the situation above
this will be polynomial time.

\subsection{Orders of invertible matrices}
\label{section:matrix_orders}

The \emph{order} of a group element $g$ is the smallest $k \in \N$
such that $g^k = 1$. We denote the order of $g$ by $\abs{g}$. For
elements $g$ of a matrix group $G \leqslant \GL(d, q)$, an algorithm
for finding $\abs{g}$ is presented in \cite{crlg95}. In general, to
obtain the \emph{precise} order, this algorithm requires a
factorisation of $q^i - 1$ for $1 \leqslant i \leqslant d$, otherwise
it might return a multiple of the correct order. Therefore it depends
on the presumably difficult problem of integer factorisation, see
\cite[Chapter $19$]{VonzurGathen03}.

However, in most of the cases we will consider, it will turn out that
a multiple of the correct order will be sufficient. For example, at
certain points in our algorithms we shall be concerned with finding
elements of order \emph{dividing} $q - 1$. Hence if we use the above
algorithm to find the order of $g \in G$ and it reports that
$\abs{g} = q - 1$, this is sufficient for us to use $g$, even though
we might have $\abs{g}$ properly dividing $q - 1$. Hence
integer factorisation is avoided in this case.

In \cite[Section $8$]{babaibeals01} the concept of \emph{pseudo-order}
is defined. A pseudo-order of an element $g$ is a product of primes
and pseudo-primes. A pseudo-prime is a composite factor of $q^i - 1$
for some $i \leqslant d$ that cannot conveniently be factorised. The
order of $g$ is a factor of the pseudo-order in which each
pseudo-prime is replaced by a non-identity factor of that
pseudo-prime. Hence a pseudo-prime is not a multiple of a \lq\lq
known'' prime, and any two pseudo-primes are relatively prime.

The algorithm of \cite{crlg95} can also be used to obtain a
pseudo-order, and for this it has time complexity $\OR{d^3
  \log{(q)}\log{\log{(q^d)}}}$ field operations. In fact, the
algorithm computes the order factorised into primes and pseudo-primes.
However, even if we have just the pseudo-order, we can still determine
if a given prime divides the order, without integer factorisation. 

\begin{pr} \label{pr_element_powering} Let $G \leqslant \GL(d, q)$.
  There exists a Las Vegas algorithm that, given $g \in G$, a prime $p \in \N$ and $e \in \N$, determines if $p^e \mid
  \abs{g}$ and if so finds the power of $g$ of order $p^e$, using
  $\OR{d^3 \log(q) \log\log(q^d)}$ field operations.
\end{pr}
\begin{proof}
  Use \cite{crlg95} to find a pseudo-order $n$ of $g$. Assume that $n
  = p^k s$ where $p \nmid s$, and $\abs{g} = p^l r$ where $p \nmid r$,
  $l \leqslant k$ and $r \mid s$. Since $p$ is given, \cite{crlg95}
  will make sure that $k = l$, so we assume that this is the case.
  Moreover, from \cite{crlg95} we see $\gcd(r, s/r) = 1$. If a prime $p_1$ divides $s/r$ we must have $p_1 \neq p$. So if $p_1 \mid \abs{g}$ we must have $p_1 \mid r$ and hence $p_1 = 1$. Thus $\gcd(s/r, \abs{g}) = 1$. 

  Hence $\abs{g^{s/r}} = \abs{g}$, $\abs{g^r} = p^l$ and $\abs{g^s} =
  \abs{(g^{s/r})^r} = p^l$. Therefore $g^{s p^{l - e}} \neq 1$ if and
  only if $p^e \mid \abs{g}$. Then $g^{s p^{l - e}}$ has order $p^e$.
\end{proof}

\subsection{Random group elements}
\label{section:random_elements}

Our analysis assumes that we can construct uniformly (or nearly
uniformly) distributed random elements of a group $G = \gen{X}
\leqslant \GL(d, q)$. The algorithm of \cite{babai91} produces
independent nearly uniformly distributed random elements, but
it is not a practical algorithm. It has a preprocessing step with time
complexity $\OR{\log{\abs{G}}^5}$ group operations, and each random
element is then found using $\OR{\log{\abs{G}}}$ group operations.

A more commonly used algorithm is the \emph{product replacement}
algorithm of \cite{lg95}. It also consists of a preprocessing step,
which is polynomial time by \cite{Pak00}, and each random element is
then found using a constant number of group operations (usually $2$).
This algorithm is practical and included in $\MAGMA$ and $\GAP$. Most
of the theory about it is summarised in \cite{MR1829489}. For a
discussion of both these algorithms, see \cite[pp. 26-30]{seress03}.

We shall assume that we have a random element oracle, which produces a
uniformly random element of $\gen{X}$ using $\OR{\xi(d)}$ field
operations, and returns it as an $\SLP$ in $X$.

An important issue is the length of the $\SLP$s that are computed. The
length of the $\SLP$s must be polynomial, otherwise it would not be
polynomial time to evaluate them. We assume that $\SLP$s of random
elements have length $\OR{n}$ where $n$ is the number of random
elements that have been selected so far during the execution of the
algorithm. 

In \cite{lgsm02}, a variant of the product replacement algorithm is
presented that finds random elements of the normal closure of a
subgroup. This will be used here to find random elements of the
derived subgroup of a group $\gen{X}$, using the fact that this is
precisely the normal closure of $\gen{[x, y] : x, y \in X}$.

\subsection{Constructive recognition overview} 
\label{section:algorithm_overview}

If $V$ is an $FG$-module for some group $G$ and field $F$, with action
$f : V \times FG \to V$, and if $\Phi$ is an automorphism of $G$,
denote by $V^{\Phi}$ the $FG$-module which has the same elements as
$V$ and where the action is given by $(v, g) \mapsto f(v, \Phi(g))$
for $g \in G$ and $v \in V^{\Phi}$, extended to $FG$ by linearity. We
call $V^{\Phi}$ a \emph{twisted version} of $V$, or $V$ \emph{twisted
  by} $\Phi$. If $G$ is a matrix group and the automorphism $\Phi$ is
a field automorphism, we call it a \emph{Galois twist}.

From \cite[Section $7.5.4$]{hcgt} we know that $G$ preserves a
classical (non-unitary) form if and only if $V$ is isomorphic to its
dual. We shall use this fact occasionally.

When we say that an algorithm is \lq\lq given a group $\gen{X}$'',
then the generating set $X$ is fixed and known to the algorithm. In
other words, the algorithm is given the generating set $X$ and
will operate in $\gen{X}$.

\begin{deff}
Let $G, H$ be matrix groups. An isomorphism $\varphi : G \to H$ is \emph{effective} if there exists a polynomial time Las Vegas algorithm that computes $\varphi(g)$ for any given $g \in G$.
\end{deff}

Of course, an effective isomorphism might be deterministic, since $\PP \subseteq \ZPP$.

\begin{deff}
The problem of \emph{constructive recognition} is:
\begin{itemize}
\item[\textbf{Input}:] A matrix group $G = \gen{X}$ with standard copy $H \cong G$. 
\item[\textbf{Ouput}:] An effective isomorphism $\varphi : G \to H$, such that $\varphi^{-1}$ is also effective.
\end{itemize}
\end{deff}

Now consider an exceptional group with standard copy $H \leqslant
\GL(d, \F_q)$, where $\F_q$ has characteristic $p$. The standard copies of the exceptional groups under consideration will be defined in Chapter \ref{chapter:group_def_theory}. Our algorithms
should be able to constructively recognise any input group
$G \leqslant \GL(d^{\prime}, q^{\prime})$ that is isomorphic to $H$. 

The assumptions that we make on the input group $G$ are:
\begin{enumerate}
\item $G$ acts absolutely irreducibly on $\F_{q^{\prime}}^{d^{\prime}}$,
\item $G$ is written over the minimal field modulo scalars,
\item $G$ is known to be isomorphic to $H$, and hence $d$ and $q$ are known.
\end{enumerate}

A user of our algorithms can easily first apply the algorithms of
\cite{meataxe} and \cite{smallerfield}, described in Sections
\ref{section:meataxe} and \ref{section:smallerfield}, to make the
group satisfy the first two assumptions. The last two assumptions
remove much of the need for input verification using non-explicit
recognition. They are motivated by the context in which our algorithms
are supposed to be used. The idea is that our algorithms will serve as
a base case for the algorithm of \cite{crlg01} or a similar algorithm.
In the base case it will be known that the group under consideration
is almost simple modulo scalars. We can then assume that the algorithm
can decide if it is dealing with a group of Lie type. Then it can use
the Monte Carlo algorithm of \cite{blackbox_char} to determine the
defining characteristic of the group, and next use the Monte Carlo
algorithm of \cite{general_recognition} to determine the name of the
group, as well as the defining field size $q$. This standard machinery
motivates our assumptions. Because the group has only been identified
by a Monte Carlo algorithm, there is a small non-zero probability that
our algorithms might be executed on the wrong group. This has to be
kept in mind when implementing the algorithms.

We do not assume that the input is tensor
indecomposable, since the tensor decomposition algorithm described in
Section \ref{section:tensor_decomposition} is not polynomial time.

A number of different cases arise:
\begin{enumerate}
\item $G \leqslant \GL(d^{\prime}, \F_{q^{\prime}})$ where
  $\F_{q^{\prime}}$ has characteristic $p^{\prime} \neq p$. This is
  called the \emph{cross characteristic} case. Then \cite{min_deg_1}
  and \cite{min_deg_2} tells us that $q \in \OR{f(d^{\prime})}$ for
  some polynomial $f$. This means that $q$ is polynomial in the size
  of the input, which is not the case in general. In this case we can
  therefore use algorithms which normally are exponential time.  In
  particular, by \cite[Theorem 8.6]{babaibeals01} we can use the
  classical permutation group methods.  Therefore we will only
  consider the case when we are given a group in \emph{defining
    characteristic}, so that $p = p^{\prime}$.

\item $G \leqslant \GL(d^{\prime}, \F_s)$ where $d^{\prime} > d$ and $\F_s \leqslant \F_q$. Let
  $W$ be the module of $G$. If $W$ is isomorphic to a tensor product
  of two modules which both have dimension less than $\dim W$, then we
  say that $W$ is \emph{tensor decomposable}. Otherwise $W$ is
  \emph{tensor indecomposable}. 

  Every possible $W$ is isomorphic to a tensor product of twisted
  versions of tensor indecomposable modules of $G$ (and hence of
  $H$). By the Steinberg tensor product theorem of \cite{steinberg63},
  in our cases the twists are Galois twists and the number of tensor
  indecomposable modules is independent of the field size, up to twists.

  If $W$ is tensor decomposable, we want to construct a tensor
  indecomposable representation $V$ of $G$. In general, this is done
  using the tensor decomposition algorithm described in Section
  \ref{section:tensor_decomposition} on $W$, which also provides an
  effective isomorphism from $W$ to $V$ (\emph{i.e.} between their
  acting groups). But since the algorithm is not polynomial time, a
  special version of one its subroutines has to be provided for each
  exceptional group.

  If $W$ is tensor indecomposable, we want to construct a
  representation $V$ of $G$ of dimension $d$, and we want do it in a
  way that also constructs an effective isomorphism. In principle this
  is always possible by computing tensor products of $W$ and chopping
  them with the MeatAxe, because a composition factor of dimension $d$
  will always turn up. However, this is not always a
  practical algorithm, and the time complexity is not very good.
  
  Note that if the minimal field $\F_s$ is a proper subfield of
  $\F_q$, then the tensor decomposition will not succeed. Since we
  assume that we know $q$, we can embed $W$ canonically into an $\F_q
  G$-module. In this case we shall therefore always assume that $s =
  q$, contrary to our second assumption above.

\item $G \leqslant \GL(d, \F_q)$, so that, by \cite{steinberg63}, $G$
  is conjugate to $H$ in $\GL(d, \F_q)$. This is the most interesting
  case since there are no standard methods, and we shall devote much
  effort to this case for the exceptional groups that we consider. A
  central issue will be to find elements of order a multiple of $p$. This is a 
  serious obstacle since by \cite{MR1345299, MR1829478}, the proportion $\rho(G)$ of these elements in $G$ satisfies
\begin{equation}
\frac{2}{5q} < \rho(G) < \frac{5}{q}.
\end{equation}
Hence we cannot find elements of order a multiple of $p$ by random search in polynomial time, so there is no straightforward way to find them. 
\end{enumerate}

To be able to deal with these various cases, we need to know all the
absolutely irreducible tensor indecomposable representations of $H$ in
defining characteristic. We also need to know how they arise from the
\emph{natural representation}, which is the representation of
dimension $d$ over $\F_q$. In our cases, this information is provided by \cite{MR1901354}.

\subsection{Constructive membership testing overview}
The other computational problem that we shall consider is the
following.

\begin{deff}
The problem of \emph{constructive membership testing} is:
\begin{itemize}
\item[\textbf{Input}:] A matrix group $G = \gen{X}$, an element $g \in
  U \geqslant G$. 
\item[\textbf{Output}:] If $g \in G$, then \texttt{true} and an $\SLP$
  for $g$ in $X$, \texttt{false} otherwise.
\end{itemize}
\end{deff}

In our cases, $U$ is always taken to be the general linear group. One
can take two slightly different approaches to the problem of
expressing an element as an $\SLP$ in the given generators, depending
on whether one wants to find an effective isomorphism or find standard
generators.

\begin{enumerate}
\item The approach using an effective isomorphism.
\begin{enumerate} 
\item Given $G = \gen{X}$ with standard copy $H$, first solve
  constructive recognition and obtain an effective isomorphism
  $\varphi : G \to H$. Hence obtain a generating set $\varphi(X)$ of
  $H$.
\item Given $g \in G$, express $\varphi(g)$ as an $\SLP$ in
  $\varphi(X)$, hence also expressing $g$ in $X$.
\end{enumerate}
\item The approach using standard generators.
\begin{enumerate}
\item Given $G = \gen{X}$ with standard copy $H = \gen{y_1, \dotsc, y_k}$, find
  $g_1, \dotsc, g_k \in G$ as $\SLP$s in
  $X$, such that the mapping $g_i \mapsto y_i$ is an isomorphism.
\item Given $g \in G$, express $g$ as an $\SLP$ in $\set{g_1, \dotsc,
  g_k}$, hence also expressing it in $X$.
\end{enumerate}
\end{enumerate}

In the first case, the constructive membership testing takes place in
$H$, which is probably faster than in $G$, so in this case we use the
standard copy in computations. In the second case, the standard
copy is only used as a theoretical tool. As it stands, the first
approach is stronger, since it provides the effective isomorphism, and
the standard generators in $G$ can be obtained in the first
approach, if necessary. However, if the representation theory of $G$ is
known, so that we can construct a module isomorphic to the module of
$G$ from the module of $H$, then the standard generators can be used,
together with the MeatAxe, to solve constructive recognition. Hence
the two approaches are not very different. One can also mix them in
various ways, for example in the first case by finding standard
generators in $H$ expressed in $\varphi(X)$, and then only express
each element in the standard generators, which might be easier than to
express the elements directly in $\varphi(X)$.

\subsection{CGT methods}
Here we describe some algorithmic methods that we will use. Like many
methods in CGT they are not really algorithms (\emph{i.e.} they may
not terminate on all inputs), or if they are they
have very bad (worst-case) time complexity. Nevertheless, they can be
useful for particular groups, as in our cases.

\subsubsection{The dihedral trick}

This trick is a method for conjugating involutions (\emph{i.e.}
elements of order $2$) to each
other in a black-box group, defined by a set of generators. The nice
feature is that if the involutions are given as straight line programs
in the generators, the conjugating element will be found as a straight
line program. The dihedral trick is based on the following
observation.

\begin{pr} \label{dihedral_trick_basis}
Let $G$ be a group and let $a, b \in G$ be involutions such that $\abs{ba} = 2k + 1$ for some $k \in \Z$. Then $(ba)^k$ conjugates $a$ to $b$.
\end{pr}
\begin{proof}
Observe that
\[ (ba)^{-k} a (ba)^k = (ba)^{k + 1} a (ba)^k = (ba)^k b (ba)^k = (ba)^k a (ba)^{k - 1} = \dotsb = (ba) a = b\]
since $a$ and $b$ are involutions.
\end{proof}

\begin{thm}[The dihedral trick] \label{dihedral_trick} Let $G =
  \gen{X} \leqslant \GL(d, q)$. Assume that the probability of the
  product of two
  random conjugate involutions in $G$ having odd order is at least
  $1/c$. There exists a Las Vegas algorithm that, given conjugate
  involutions $a, b \in G$, finds $g \in G$ such that $a^g = b$. If
  $a, b$ are given as $\SLP$s of lengths $l_a, l_b$, then $g$ will be found as an $\SLP$ of
  length $\OR{c (l_a + l_b)}$. The algorithm has expected time complexity
  $\OR{c(\xi(d) + d^3 \log(q) \log\log(q^d))}$ field operations.
\end{thm}
\begin{proof}
The algorithm proceeds as follows:

\begin{enumerate}
\item Find random $h \in G$ and let $a_1 = a^h$. 
\item Let $b_1 = b a_1$. Use Proposition \ref{pr_element_powering} to determine if $b_1$ has even order, and if so, return to the first step.
\item Let $n = (\abs{b_1} - 1) / 2$ and let $g = h b_1^n = h (a^h b)^n$.
\end{enumerate}

By Proposition \ref{dihedral_trick_basis}, this is a Las Vegas algorithm. The probability that $b_1$ has
odd order is $1/c$ and hence the expected time complexity is as
stated. Note that if $a$ and $b$ are given as $\SLP$s in $X$, then we obtain $g$
as an $\SLP$ in $X$.
\end{proof}

\subsubsection{Involution centralisers} \label{section:inv_centraliser}
In \cite{ryba_trick} an algorithm is described that reduces the
constructive membership problem in a group $G$ to the same problem in
three involution centralisers in $G$. The reduction algorithm is known
as the \emph{Ryba algorithm} and can be a convenient method to solve
the constructive membership problem. However, there are obstacles:
\begin{enumerate}
\item We have to solve the constructive membership problem in the
  involution centralisers of $G$. In principle this can be done using
  the Ryba algorithm recursively, but such a blind descent might not
  be very satisfactory. For instance, it might not be easy to
  determine the time complexity of such a procedure. Another approach
  is to provide a special algorithm for the involution centraliser. This
  assumes that the structure of $G$ and its involution centralisers
  are known, which it will be in the cases we consider.
\item We have to find involutions in $G$. As described in Section \ref{section:algorithm_overview}, this is a serious obstacle if the defining field $\F_q$ of $G$ has characteristic $2$. In odd characteristic the
  situation is better, and in \cite{ryba_trick} it is proved that the
  Ryba algorithm is polynomial time in this case. Another approach is
  to provide a special algorithm that finds involutions.

\item We have to find generators $Y$ for $\Cent_G(j)$ of a
  given involution $j \in G = \gen{X}$. This is possible using the \emph{Bray
    algorithm} of \cite{bray00}. It works by computing random elements of $\Cent_G(j)$ until the whole
centraliser is generated. This automatically gives the elements of $Y$
as $\SLP$s in $X$, which is a central feature needed by the Ryba algorithm.

There are two issues involved when using this algorithm. First, the
generators that are computed may not be uniformly random in $\Cent_G(j)$, so that we
might have trouble generating the whole centraliser. In
\cite{ryba_trick} it is shown that this is not a problem with the
exceptional groups. Second, we need to provide an algorithm that
determines if the whole centraliser has been generated. In the cases
that we will consider, this will be possible. It should be noted that
the Bray algorithm works for any black-box group and not just for
matrix groups.
\end{enumerate}

Given these obstacles, we will still use the Ryba algorithm for
constructive membership testing in some cases. We will also use the Bray algorithm independently, since it is a powerful tool.

\subsubsection{The Formula}

Like the dihedral trick, this is a method for conjugating elements to
each other. For a group $G$, denote by $\Phi(G)$ the \emph{Frattini
  subgroup} of $G$, which is the intersection of the maximal subgroups
of $G$. 

\begin{lem}[The Formula] \label{lem_the_formula} Let $G \cong H {:}
  \Cent_n$, where $H$ is a $2$-group and $n$ is odd. If $a, b \in G$ have
  order $n = 2k + 1$ and $a \equiv b \mod H$, then $b \equiv a^g \mod
  \Phi(H)$ where $g = (b a)^k$.
\end{lem}
\begin{proof}
  The orders of $a,b$ are their orders in $\Cent_n$. Hence we can replace
  $H$ with $H / \Phi(H)$ without affecting the rest of the
  assumptions. We can therefore reduce to the case when $\Phi(H) = \gen{1}$,
  in other words when $H$ is elementary abelian.

  Then $a = a_1 g$, $b = b_1 g$, with $\abs{g} = n$ and $a_1, b_1
  \in H$. We want to prove that $a^h = b$ where $h = (ba)^k$ or
  equivalently that $(ba)^k b = a (ba)^k$.

  Now $(ba)^k = (b_1 g a_1 g)^k = (b_1 a_1^{g^{-1}} g^2)^k$. We can
  move all occurrences of $g$ to the right, so that
\[ (ba)^k = b_1^{1 + g^{-2} + g^{-4} + \dotsb + g^{-2(k-1)}} a_1^{g^{-1} + g^{-3} + \dotsb + g^{-2k - 1}} g^{2k}\]
from which we see that
\begin{align*}
(ba)^k b_1 g &= b_1^{1 + g^{-2} + g^{-4} + \dotsb + g^{-2k}} a_1^{g^{-1} + g^{-3} + \dotsb + g^{-2k - 1}} g^{2k + 1} \\
a_1 g (ba)^k &= b_1^{g^{-1} + g^{-3} + \dotsb + g^{-2k - 1}} a_1^{1 + g^{-2} + g^{-4} + \dotsb + g^{-2k}} g^{2k + 1} 
\end{align*}
and we want these to be equal. Since $g^{2k + 1} = 1$, we see that
$(ba)^k b, a(ba)^k \in H$, and because $H$ is elementary abelian, they
are equal if and only if their product is the identity. But clearly $(ba)^k b a (ba)^k = b_1^s a_1^s$, where $s = \sum_{i = 0}^k g^{-i}$, and finally $b_1^s a_1^s = (b_1 g)^{2k + 1} (a_1 g)^{2k + 1} = 1$.
\end{proof}

\begin{cl} \label{cl_the_formula} Let $H \cong P {:} \Cent_n$, where $P$
  is a $2$-group and $n$ is odd, and let $G \cong H {:} S$ for some
  group $S$. If $a \in H$ has order $n = 2k + 1$ then for $h \in G$, such that $a \equiv a^h \mod P$, we have
  $a^{g^{-1} h} \equiv a \mod \Phi(P)$, where $g = (a^h a)^k$.
\end{cl}
\begin{proof}
Observe that both $a$ and $a^h$ have order $n$ and lie in $H \trianglelefteq G$. Now apply Lemma \ref{lem_the_formula}, conclude that $a^g \equiv a^h \mod \Phi(P)$, and the result follows.
\end{proof}

\subsubsection{Recognition of $\PSL(2, q)$}
In \cite{psl_recognition}, an algorithm for constructive recognition
and constructive membership testing of $\PSL(2, q)$ is presented. This
algorithm is in several aspects the original which our algorithms are
modelled after, and it is in itself an extension of \cite{MR1829474},
which handles the natural representation.

We will use \cite{psl_recognition} since $\PSL(2, q)$ arise as
subgroups of some of the exceptional groups that we consider. Because
of this, we state the main results here. Let $\sigma_0(d)$ be the
number of divisors of $d \in \N$. From \cite[pp. $64, 359,
262$]{MR568909}, we know that for every $\varepsilon > 0$, if $d$ is
sufficiently large then $\sigma_0(d) < 2^{(1 + \varepsilon)
  \log_{\mathrm{e}}(d) / \log\log_{\mathrm{e}}(d)}$.

Here, $\PSL(2,q)$ is viewed as a quotient of $\SL(2, q)$. Hence the
elements are cosets of matrices.

\begin{thm} \label{thm_psl_recognition} Assume an oracle for the
  discrete logarithm problem in $\F_q$. There exists a Las Vegas
  algorithm that, given $\gen{X} \leqslant \GL(d, q)$ satisfying the
  assumptions in Section \ref{section:algorithm_overview}, with
  $\gen{X} \cong \PPSL(2, q)$ and $q = p^e$, finds an effective isomorphism $\varphi
  : \gen{X} \to \PPSL(2, q)$ and performs preprocessing for constructive membership testing. The algorithm has expected time
  complexity
\[ \OR{(\xi(d) + d^3 \log(q) \log \log(q^d)) \log \log(q) + d^5 \sigma_0(d) \abs{X} + d \chi_D(q) + \xi(d) d} \]
 field operations.

 The inverse of $\varphi$ is also effective. Each image of $\varphi$ can be computed using $\OR{d^3}$ field
 operations, and each pre-image using $\OR{d^3 \log(q)\log\log(q) + e^3}$ field operations.
 After the algorithm has run, constructive membership testing of $g
 \in \GL(d, q)$ uses $\OR{d^3 \log(q)\log\log(q) + e^3}$ field operations, and the
 resulting $\SLP$ has length $\OR{\log (q) \log\log(q)}$.
\end{thm}

The existence of this constructive recognition algorithm has led to
several other constructive recognition algorithms for the classical
groups\cite{MR1829473, MR1958954, MR2051584, MR2228648}, which are
polynomial time assuming an oracle for constructive recognition of
$\PSL(2, q)$.

In some situations we shall also need a fast non-constructive
recognition algorithm of $\PSL(2, q)$. It will be used to test if a given
subgroup of $\PSL(2, q)$ is in fact the whole of $\PSL(2, q)$, so it
is enough to have a Monte Carlo algorithm with no false positives. We
need to use it in any representation, so the correct context is a
black-box group.

\begin{thm} \label{thm_psl_naming}
  Let $G = \gen{X} \leqslant \GL(d, q^{\prime})$. Assume that $G$ is
  isomorphic to a subgroup of $\PSL(2, q)$ and that $q$ is known.
  There exists a one-sided Monte Carlo algorithm with no false
  positives that determines if $G \cong \PSL(2, q)$. Given maximum
  error probability $\varepsilon > 0$, the time complexity is
  \[ \OR{(\xi(d) + d^3 \log(q) \log\log(q^d)) \ceil{\log(\varepsilon) / \log(\delta)}} \] field
  operations, where
\[\delta = (1 - \phi(q - 1)/(q - 1)) (1 - \phi(q + 1)/(q + 1)). \]
\end{thm}
\begin{proof}
  It is well known that $\PSL(2, \F_q)$ is generated by two elements
  having order dividing $(q \pm 1) / 2$, unless one of them lie in
  $\PSL(2, \F_s)$ for some $\F_s < \F_q$. We are thus led to an
  algorithm that performs at most $n$ steps. In each step it finds two
  random elements $g$ and $h$ of $G$ and computes their pseudo-orders.
  If $g$ has order dividing $(q - 1) / 2$, it then determines if $g$ lies
  in some $\PSL(2, \F_s)$ by testing if $g^{(s - 1) / 2} = 1$. If $q =
  p^a$ then the possible values of $s$ are $p^b$ where $b \mid a$, so
  there are $\sigma_0(a)$ subfields. If $g$ does not lie in any
  $\PSL(2, \F_s)$ then $g$ is remembered. 

Now do similarly for $h$. Then, if it
  has found the two elements, it returns
  \texttt{true}. On the other hand, if it completes the $n$ steps
  without finding these elements, it returns \texttt{false}.

  The proportion in $\PSL(2, q)$ of elements of order $(q \pm 1) / 2$ is
  $\phi(q \pm 1) / (q \pm 1)$, so the probability that $G \cong
  \PSL(2, q)$ but we fail to find the elements is at most $\delta^n$. We
  require $\delta^n \leqslant \varepsilon$ and hence $n$ can be chosen as
  $\ceil{\log(\varepsilon) / \log(\delta)}$.
\end{proof}

\subsection{Aschbacher classes}
\label{section:aschbacher_algorithms}
The Aschbacher classification of \cite{aschbacher84} classifies matrix
groups into a number of classes, and a major part of the MGRP has been
to develop algorithms that determine constructively if a given matrix
group belongs to a certain class. Some of these algorithms are used
here.

\subsubsection{The MeatAxe}
\label{section:meataxe}

Let $G = \gen{X} \leqslant \GL(d, q)$ acting on a
module $V \cong \F_q^d$. The algorithm known as the \emph{MeatAxe}
determines if $V$ is irreducible. If not, it finds a proper
non-trivial submodule $W$ of $V$, and a change of basis matrix $c \in
\GL(d, q)$ that exhibits the action of $G$ on $W$ and on $V / W$. In
other words, the first $\dim W$ rows of $c$ form a basis of $W$, and
$g^c$ is block lower triangular for every $g \in G$.

Applying the MeatAxe recursively, one finds a composition series of
$V$, and a change of basis that exhibits the action of $G$ on the
composition factors of $V$. Hence we also obtain effective
isomorphisms from $G$ to the groups acting on the composition factors of $V$.

The MeatAxe was originally developed by Parker in \cite{Parker84} and
later extended and formalised into a Las Vegas algorithm by Holt and
Rees in \cite{meataxe}. They also explain how very similar algorithms
can be used to test if $V$ is absolutely irreducible, and if two
modules are isomorphic. These are also Las Vegas algorithms with the
same time complexity as the MeatAxe. The worst case of the MeatAxe is
treated in \cite{better_meataxe}, where it is proved that the expected
time complexity is $\OR{\abs{X} d^4}$ field operations. Unless the module is reducible and all composition factors of the module are isomorphic, the
expected time complexity is $\OR{\abs{X} d^3}$ field
operations. 

The MeatAxe is also fast in practice and is implemented both in
$\MAGMA$ and $\GAP$. This important feature is the reason that it is
used rather than the first known polynomial time algorithm for the same
problem, which was given in \cite{ronyai90} (and is, at best, $\OR{\abs{X} d^6}$).

Some related problems are the following:
\begin{itemize}
\item Determine if $G$ acts absolutely irreducibly on $V$.
\item Given two irreducible $G$-modules $U, W$ of dimension $d$, determine if they are isomorphic, and if so find a change of basis matrix $c \in \GL(d, q)$ that conjugates $U$ to $W$.
\item Given that $G$ acts absolutely irreducibly on $V$, determine if $G$ preserves a classical form and if so find a matrix for the form.
\item Find a basis for $\EndR_G(V)$ as matrices of degree $d$.
\end{itemize}

Algorithms for these problems are described in \cite[Section
$7.5$]{hcgt} and they all have expected time complexity $\OR{\abs{X} d^3}$ field operations. We will refer to the algorithms for all these problems as \lq\lq the MeatAxe''.

\subsubsection{Writing matrix groups over subfields}
\label{section:smallerfield}

If $G = \gen{X} \leqslant \GL(d, \F_q)$, then $G$ might be conjugate
in $\GL(d, \F_q)$ to a subgroup of $\GL(d, \F_s)$ where $\F_s < \F_q$,
so that $q$ is a proper power of $s$. An algorithm for deciding if
this is case is given in \cite{smallerfield}. It is a Las Vegas
algorithm with expected time complexity $\OR{\sigma_0(\log(q))
  (\abs{X} d^3 + d^2 \log(q))}$ field operations. In case $G$ can be
written over a subfield, then the algorithm also returns a conjugating
matrix $c$ that exhibits this fact, \emph{i.e.} so that $G^c$ can be
immediately embedded in $\GL(d, \F_s)$. The algorithm can also write a
group over a subfield modulo scalars.

\subsubsection{Tensor decomposition}
\label{section:tensor_decomposition}

Now let $G = \gen{X} \leqslant \GL(d, q)$ acting on a module $V \cong
\F_q^d$. The module might have the structure of a tensor product $V
\cong U_1 \otimes U_2$, so that $G \leqslant G_1 \circ G_2$ where $G_1
\leqslant \GL(U_1)$ and $G_2 \leqslant \GL(U_2)$.

The Las Vegas algorithm of \cite{tensorprodalg} determines if $V$ has the
structure of a tensor product, and if so it also returns a change of
basis $c \in \GL(d, q)$ which exhibits the tensor decomposition. In
other words, $g^c$ is an explicit Kronecker product for each $g \in
G$. The images of $g$ in $G_1$ and $G_2$ can therefore immediately be
extracted from $g^c$, and hence we obtain an effective embedding of $G$ into $G_1 \circ G_2$.

By \cite{tensorprodtheory}, for tensor decomposition it is sufficient
to find a \emph{flat} in a projective geometry corresponding to the
decomposition. A flat is a subspace of $V$ of the form $A \otimes U_2$ or $U_1 \otimes B$ where $A$ and $B$ are proper subspaces of $U_1$ and $U_2$ respectively. This flat contains a
\emph{point}, which is a flat with $\dim A = 1$ or $\dim B = 1$. If we can provide a
proposed flat to the algorithm of \cite{tensorprodalg}, then it will verify that it is a flat, and if so find a
tensor decomposition, using expected $\OR{\abs{X} d^3
  \log(q)}$ field operations.

However, in general there is no efficient algorithm for finding a flat
of $V$. If we want a polynomial time algorithm for decomposing a
specific tensor product, we therefore have to provide an efficient
algorithm that finds a flat.

\subsection{Conjectures}
\label{section:intro_conj}

Most of the results presented will depend on a few conjectures. This
might be considered awkward and somewhat non-mathematical, but it is a
result of how the work in this thesis was produced. In almost every
case with the algorithms that are presented, the implementation of the
algorithm did exist before the proof of correctness of the algorithm.
In fact, the algorithms have been developed using a rather empirical
method, an interplay between theory (mathematical thought) and
practice (programming). We consider this to be an essential
feature of the work, and it has proven to be an effective way to
develop algorithms that are good in both theory and practice.

However, it has lead to the fact that there are certain results that
have been left unproven, either because they have been too hard to
prove or have been from an area of mathematics outside the scope of
this thesis (usually both). But because of the way the algorithms have
been developed, there should be no doubt that every one of the
conjectures are true. The implementations of the algorithms have been
tested on a vast number of inputs, and therefore the conjectures have also
been tested equally many times. There has been no case of a conjecture failing.

More detailed information about the implementations can be found in Chapter \ref{chapter:implementation}.

\chapter{Twisted exceptional groups} \label{chapter:group_def_theory}

Here we will present the necessary theory about the twisted groups
under consideration. 

\section{Suzuki groups} \label{section:suzuki_theory} 

The family of exceptional groups now known as the \emph{Suzuki groups}
were first found by Suzuki in \cite{MR0120283, suzuki62, MR0162840},
and also described in \cite[Chapter $11$]{huppertIII} which is the
exposition that we follow. They should not be confused with the Suzuki
$2$-groups or the sporadic Suzuki group.

\subsection{Definition and properties}

We begin by defining our standard copy of the Suzuki group. Following
\cite[Chapter $11$]{huppertIII}, let $q = 2^{2m + 1}$ for some $m > 0$
and let $\pi$ be the unique automorphism of $\F_q$ such that $\pi^2(x)
= x^2$ for every $x \in \F_q$, \emph{i.e.} $\pi(x) = x^t$ where $t =
2^{m + 1} = \sqrt{2q}$. For $a, b \in \F_q$ and $c \in \F_q^{\times}$, define the
following matrices:
\begin{align}
S(a, b) &= \begin{bmatrix}
1 & 0 & 0 & 0 \\
a & 1 & 0 & 0 \\
b & \pi(a) & 1 & 0 \\
a^2 \pi(a) + ab + \pi(b) & a \pi(a) + b & a & 1 
\end{bmatrix}, \\
M(c) &= \begin{bmatrix}
c^{1 + 2^m} & 0 & 0 & 0 \\
0 & c^{2^m} & 0 & 0 \\
0 & 0 & c^{-2^m} & 0 \\
0 & 0 & 0 & c^{-1 - 2^m}
\end{bmatrix}, \\ 
T &= \begin{bmatrix}
0 & 0 & 0 & 1 \\
0 & 0 & 1 & 0 \\
0 & 1 & 0 & 0 \\
1 & 0 & 0 & 0 
\end{bmatrix}.
\end{align}
By definition,
\begin{equation} \label{suzuki_def}
\Sz(q) = \gen{S(a, b), M(c), T \mid a, b \in  \F_q, c \in \F_q^{\times}}.
\end{equation}
If we define 
\begin{align}
\mathcal{F} &= \set{S(a, b) \mid a, b \in \F_q} \\
\mathcal{H} &= \set{M(c) \mid c \in \F_q^{\times}} 
\end{align}
then $\mathcal{F} \leqslant \Sz(q)$ with 
$\abs{\mathcal{F}} = q^2$ and $\mathcal{H} \cong \F_q^{\times}$ so
that $\mathcal{H}$ is cyclic of order $q - 1$. Moreover, we can write $M(c)$ as
\begin{equation}
M(c) = M^{\prime}(\lambda) = \begin{bmatrix}
\lambda^{t + 1} & 0 & 0 & 0 \\
0 & \lambda & 0 & 0 \\
0 & 0 & \lambda^{-1} & 0 \\
0 & 0 & 0 & \lambda^{-t-1}
\end{bmatrix}
\end{equation}
where $\lambda = c^{2^m}$. Hence $M(\lambda)^t = M^{\prime}(\lambda)$.

The following result follows from \cite[Chapter $11$]{huppertIII}.
\begin{thm} \label{thm_suzuki_props}
\begin{enumerate}
\item The order of the Suzuki group is
\begin{equation} \label{suzuki_order}
\abs{\Sz(q)} = q^2 (q^2 + 1) (q - 1)
\end{equation}
and $q^2 + 1 = (q + t + 1)(q - t + 1)$.
\item $\gcd(q - 1, q^2 + 1) = 1$ and hence the three factors in \eqref{suzuki_order} are pairwise relatively prime.
\item For all $a_1, b_1, a_2, b_2 \in \F_q$ and $\lambda \in \F_q^{\times}$:
\begin{align}
S(a_1, b_1) S(a_2, b_2) &= S(a_1 + a_2, b_1 + b_2 + a_1^t a_2) \label{sz_matrix_id1} \\
S(a, b)^{-1} &= S(a, b + a^{t + 1}) \\
S(a_1, b_1)^{S(a_2, b_2)} &= S(a_1, b_1 + a_1^t a_2 + a_1 a_2^t) \\
S(a, b)^{M^{\prime}(\lambda)} &= S(\lambda^t a, \lambda^{t + 2} b). \label{sz_matrix_id2}
\end{align}
\item There exists $\OV \subseteq \PS^3(\F_q)$ on which $\Sz(q)$ acts faithfully and doubly transitively, such that no nontrivial element of $\Sz(q)$ fixes
  more than $2$ points. This set is
\begin{equation} \label{sz_ovoid_def}
\OV = \set{(1 : 0 : 0 : 0)} \cup \set{(a b + \pi(a) a^2 + \pi(b) : b : a : 1) \mid a,b \in \F_q}.
\end{equation}
\item The stabiliser of $P_{\infty} = (1 : 0 : 0 : 0) \in \OV$ is $\mathcal{F} \mathcal{H}$
and if $P_0 = (0 : 0 : 0 :1)$ then the stabiliser of $(P_{\infty},
P_0)$ is $\mathcal{H}$. 
\item $\Zent(\mathcal{F}) = \set{S(0, b) \mid b \in \F_q}$.
\item $\mathcal{FH}$ is a Frobenius group with Frobenius kernel $\mathcal{F}$.
\item $\Sz(q)$ has cyclic Hall subgroups $U_1$ and $U_2$ of orders $q
\pm t + 1$. These act fixed point freely on $\OV$ and irreducibly on
$\F_q^4$. For each non-trivial $g \in U_i$, we have $\Cent_G(g) = U_i$.
\item The conjugates of $\mathcal{F}$, $\mathcal{H}$, $U_1$ and $U_2$ partition $\Sz(q)$.
\item The proportion of elements of order $q-1$ in $\mathcal{FH}$ is $\phi(q - 1) / (q - 1)$, where $\phi$ is the Euler totient function.
\end{enumerate}
\end{thm}

\begin{rem}[Standard generators of $\Sz(q)$]
  As standard generators for $\Sz(q)$ we will use 
\[\set{S(1, 0), M^{\prime}(\lambda), T},\] 
where $\lambda$ is a primitive element of
  $\F_q$, whose minimal polynomial is the defining polynomial of
  $\F_q$. Other sets are possible: in \cite{bray_sz_presentation}, the
  standard generators are 
\[ \set{S(1, 0)^{-1}, M^{\prime}(\lambda)^{2^m}, T},\]
 and $\MAGMA$ uses 
\[ \set{S(1, 0)^{M^{\prime}(\lambda)^{q / 2}}, M^{\prime}(\lambda)^{1 - 2^m}, T}. \]
\end{rem}

From \cite[Chapter $11$, Remark $3.12$]{huppertIII} we also immediately obtain the following result.
\begin{thm} \label{sz_maximal_subgroups}
A maximal subgroup of $G = \Sz(q)$ is conjugate to one of the following subgroups.
\begin{enumerate}
\item The point stabiliser $\mathcal{F} \mathcal{H}$.
\item The normaliser $\Norm_G(\mathcal{H}) \cong \Dih_{2(q - 1)}$.
\item The normalisers $\mathcal{B}_i = \Norm_G(U_i)$ for $i = 1,2$. These satisfy $\mathcal{B}_i = \gen{U_i, t_i}$ where $u^{t_i} = u^q$ for every $u \in U_i$ and $[\mathcal{B}_i : U_i] = 4$.
\item $\Sz(s)$ where $q$ is a proper power of $s$.
\end{enumerate}
\end{thm}



\begin{pr} \label{pr_sz_subgroups_conj}
Let $G = \Sz(q)$.
\begin{enumerate}
\item Distinct conjugates of $\mathcal{F}$, $\mathcal{H}$, $U_1$ or $U_2$ intersect trivially.
\item The subgroups of $G$ of order $q^2$ are conjugate, and there are $q^2 + 1$ distinct conjugates.
\item The cyclic subgroups of $G$ of order $q - 1$ are conjugate, and there are $q^2 (q^2 + 1) / 2$ distinct conjugates.
\item The cyclic subgroups of $G$ of order $q \pm t + 1$ are conjugate, and there are $q^2 (q - 1) (q \mp t + 1) / 4$ distinct conjugates.
\end{enumerate}
\end{pr}
\begin{proof}
\begin{enumerate}
\item By Theorem \ref{thm_suzuki_props}, each conjugate of
$\mathcal{F}$ fixes exactly one point of $\OV$. If an element $g$ lies
in two distinct conjugates it must fix two distinct points and hence
lie in a conjugate of $\mathcal{H}$. But, by the partitioning, the
conjugates of $\mathcal{H}$ and $\mathcal{F}$ intersect trivially, so
$g = 1$.

If $\mathcal{H} \neq \mathcal{H}^x$ for some $x \in G$ and $g \in
\mathcal{H} \cap \mathcal{H}^x$, then $g$ fixes more than $2$ points
of $\OV$, so that $g = 1$. If $U_i \neq U_i^x$ for some $x \in G$ and
$g \in U_i \cap U_i^x$, then $\Cent_G(g) = \gen{U_i \cup U_i^x} = G$, so that $g = 1$.
\item This is clear since these subgroups are Sylow $2$-subgroups, and hence conjugate to $\mathcal{F}$. Each subgroup fixes a point of $\OV$ and hence there are $\abs{\OV}$ distinct conjugates.
\item Because of the partitioning, an element of order $q - 1$ must lie in a conjugate of
$\mathcal{H}$, which must be the cyclic subgroup that it generates. By
Theorem \ref{sz_maximal_subgroups} there are $[G :
\Norm_G(\mathcal{H})] = q^2 (q^2 + 1) / 2$ distinct conjugates.
\item Analogous to the previous case.
\end{enumerate}
\end{proof}

\begin{pr} \label{element_proportions}
Let $G = \Sz(q)$ and let $\phi$ be the Euler totient function.
\begin{enumerate}
\item The number of elements in $G$ that fix at least one point of $\OV$
is $q^2 (q - 1) (q^2 + q + 2) / 2$.
\item The number of elements in $G$ of order $q - 1$ is $\phi(q - 1) q^2 (q^2 + 1) / 2$.
\item The number of elements in $G$ of order $q \pm t + 1$ is $\phi(q \pm t + 1) (q \mp t + 1) q^2 (q - 1) / 4$.
\end{enumerate}
\end{pr}
\begin{proof}
\begin{enumerate}
\item  By Theorem \ref{thm_suzuki_props}, if $g \in G$ fixes exactly one point,
  then $g$ is in a conjugate of $\mathcal{F}$, and if $g$ fixes two
  points, then $g$ is in a conjugate of $\mathcal{H}$. Hence by Proposition \ref{pr_sz_subgroups_conj}, there are $(\abs{\mathcal{F}} - 1)(q^2 + 1)$ elements that fix
  exactly one point. Similarly, there are $q^2 (q^2 + 1) (\abs{\mathcal{H}} - 1) / 2$
  elements that fix exactly two points.
  
  Thus the number of elements that fix at least one point is
\begin{equation}
1 + (\abs{\mathcal{F}} - 1)(q^2 + 1) + q^2 (q^2 + 1)(\abs{\mathcal{H}} - 1) / 2 = \frac{q^2 ( q - 1) (q^2 + q + 2)}{2}.
\end{equation}
\item By Proposition \ref{pr_sz_subgroups_conj}, an element of order
$q - 1$ must lie in a conjugate of $\mathcal{H}$. Since distinct
conjugates intersect trivially, the number of such elements is the
number of generators of all cyclic subgroups of order $q - 1$.
\item Analogous to the previous case.
\end{enumerate}
\end{proof}

\begin{pr} \label{sz_totient_prop}
If $g \in G = \Sz(q)$ is uniformly random, then
\begin{align}
\Pr{\abs{g} = q - 1} = \frac{\phi(q - 1)}{2(q - 1)} &> \frac{1}{12\log{\log(q)}} \\
\Pr{\abs{g} = q \pm t + 1} = \frac{\phi(q \pm t + 1)}{4(q \pm t + 1)} &> \frac{1}{24\log{\log(q)}} \\
\Pr{g \ \text{fixes a point of $\OV$}} = \frac{q^2 + q + 2}{2(q^2 + 1)} &> \frac{1}{2}
\end{align}
and hence the expected number of random selections required to obtain an element of order $q - 1$ or $q \pm t + 1$ is
$\OR{\log\log{q}}$, and $\OR{1}$ to obtain an element that fixes a point.
\end{pr}
\begin{proof}
The first equality follows immediately from Theorem \ref{thm_suzuki_props} and Proposition \ref{element_proportions}. The
inequalities follow from \cite[Section II.8]{totient_prop}.

Clearly the number of selections required is geometrically distributed, where
the success probabilities for each selection are given by the
inequalities. Hence the expectations are as stated.
\end{proof}

\begin{pr} \label{lemma_double_coset} 
Let $G = \Sz(q)$.
\begin{enumerate}
\item For every $g \in G$, distinct conjugates of $\Cent_G(g)$ intersect trivially.
\item If $H \leqslant G$ is cyclic of order $q - 1$ and $g \in G \setminus \Norm_G(H)$ then $\abs{HgH} = (q - 1)^2$.
\end{enumerate}
\end{pr}
\begin{proof}
\begin{enumerate}
\item By Theorem \ref{thm_suzuki_props}, we consider three cases. If
$g$ lies in a conjugate $F$ of $\mathcal{F}$, then $\Cent_G(g)
\leqslant F$. If $g \neq 1$ lies in a conjugate $H$ of $\mathcal{H}$, then
$\Cent_G(g) = H$ and if $g$ lies in a conjugate $U$ of $U_1$ or $U_2$,
then $\Cent_G(g) = U$. In each case the result follows from
Proposition \ref{pr_sz_subgroups_conj}.
\item Since $\abs{H} = q - 1$ it is enough to show that $H \cap H^g =
  \gen{1}$. This follows immediately from Proposition
  \ref{pr_sz_subgroups_conj}.
\end{enumerate}
\end{proof}

\begin{pr} \label{sz_conjugacy_classes}
Elements of odd order in $\Sz(q)$  that have the same trace are conjugate.
\end{pr}
\begin{proof}
  From \cite[§17]{suzuki62}, the number of conjugacy classes of
  non-identity elements of odd order is $q - 1$, and all elements of even
  order have trace $0$.  Observe that
\begin{equation}
S(0, b) T = \begin{bmatrix}
0 & 0 & 0 & 1 \\
0 & 0 & 1 & 0 \\
0 & 1 & 0 & b \\
1 & 0 & b & b^t 
\end{bmatrix}.
\end{equation}
Since $b$ can be any element of $\F_q$, so can $\Tr{(S(0, b) T)}$,
and this also implies that $S(0, b) T$ has odd order when $b \neq 0$.
Therefore there are $q - 1$ possible traces for
non-identity elements of odd order, and elements with different trace
must be non-conjugate, so all conjugacy classes must have different
traces.
\end{proof}

\begin{pr} \label{sz_trace0_order4}
The proportion of elements of order $4$ among the elements of trace $0$ is $1 - 1/q + 1/q^2 - 1/q^3$.
\end{pr}
\begin{proof}
  The elements of trace $0$ are those with orders $1,2,4$, and apart
  from the identity these are the elements that fix precisely one
  point of $\OV$. From the proof of Proposition
  \ref{element_proportions}, there are $q^4$ elements of trace $0$.

  The elements of order $4$ lie in a conjugate of $\mathcal{F}$, and
  there are $q^2 - q$ elements in each conjugate. Hence from
  Proposition \ref{pr_sz_subgroups_conj} there are $q (q - 1)(q^2 +
  1)$ elements of order $4$.
\end{proof}

\begin{pr} \label{prop_trick_general_case}
Let $P = (p_1 : p_2 : p_3 : p_4) \in \OV^g$ be uniformly random, where $\OV^g = \set{Rg \mid R \in \OV}$ for some $g \in \GL(4, q)$. Then
\begin{enumerate}
\item \begin{equation}
\Pr{p_i \neq 0 \mid i = 1,\dotsc,4} \geqslant \left(1 - \frac{q + 2}{q^2 + 1}\right)^4.
\end{equation}
\item If $Q = (q_1 : q_2 : q_3 : q_4) \in \OV^g$ is fixed, then
\begin{equation}
\Pr{p_2^t q_3^{t+2} \neq q_2^t p_3^{t+2}} \geqslant 1 - \frac{1 + (t + 2)q}{q^2 + 1}.
\end{equation}
\end{enumerate}
\end{pr}
\begin{proof}
\begin{enumerate}
\item By \cite[Chapter $11$, Lemma $3.4$]{huppertIII}, $\OV$ is an
  ovoid, so it intersects any (projective) line in $\PS^3(\F_q)$ in at most $2$
  points. The condition $p_i = 0$ defines a projective plane
  $\mathcal{P} \subseteq \PS^3(\F_q)$. If $\OV^g \cap \mathcal{P} \neq
  \emptyset$ then it contains a point $A$, and there are $q+1$ lines
  in $\mathcal{P}$ that passes through $A$. Each one of these lines passes
  through at most one other point of $\OV^g$, but each line contains
  $q+1$ points of $\mathcal{P}$, and hence at least $q$ of those points are not
  in $\OV^g$. Moreover, each pair of lines has only the point $A$ in common.
  
  Now we have considered $1 + q(q + 1)$ distinct points of $\mathcal{P}$, which
  are all points of $\mathcal{P}$, and we have proved that at most $q + 2$ of
  those lie in $\OV^g$.

\item Clearly, $p_2^t q_3^{t+2} = q_2^t p_3^{t+2}$ if and only if $p_2
  q_3^{t+1} = q_2 p_3^{t+1}$. If $P = P^{\prime} g$, where $P_{\infty}
  \neq P^{\prime} \in \OV$ and $g = [g_{i,j}]$, then $p_i = g_{1, i}
  (a^{t + 2} + b^t + ab) + g_{2,i} b + g_{3,i} a + g_{4,i}$ for some
  $a, b \in \F_q$.
  
  Introducing indeterminates $x$ and $y$ in place $a$ and $b$, it
  follows that the expression $p_2 q_3^{t+1} - q_2 p_3^{t+1}$ is a
  polynomial $f \in \F_q[x, y]$ with $\deg_x(f) \leqslant 3t+4$ and
  $\deg_y(f) \leqslant t + 2$. For each $a \in \F_q$, the number of
  roots of $f(a, y)$ is therefore at most $t + 2$, so the number of
  roots of $f$ is at most $q(t + 2)$.



\end{enumerate}

\end{proof}

\begin{pr} \label{sz_prop_frobenius}
If $g_1, g_2 \in \mathcal{FH}$ are uniformly random, then
\begin{equation}
\Pr{\abs{[g_1, g_2]} = 4} = 1 - \frac{1}{q - 1}.
\end{equation}
\end{pr}
\begin{proof}
  Let $A = \mathcal{FH} / \Zent(\mathcal{F})$. By Theorem
  \ref{thm_suzuki_props}, $[g_1, g_2] \in \mathcal{F}$ and has order
  $4$ if and only if $[g_1, g_2] \notin \Zent(\mathcal{F})
  \triangleleft \mathcal{FH}$. It therefore suffices to find the
  proportion of pairs $k_1, k_2 \in A$ such that $[k_1, k_2] = 1$.

  If $k_1 = 1$ then $k_2$ can be any element of $A$, which contributes $q(q
  - 1)$ pairs.  If $1 \neq k_1 \in \mathcal{F} / \Zent(\mathcal{F})
  \cong \F_q$ then $\Cent_A(k_1) = \mathcal{F} / \Zent(\mathcal{F})$,
  so we again obtain $q(q - 1)$ pairs. Finally, if $k_1 \notin
  \mathcal{F} / \Zent(\mathcal{F})$ then $\abs{\Cent_A(k_1)} = q - 1$,
  so we obtain $q(q - 2)(q - 1)$ pairs. Thus we obtain $q^2 (q - 1)$ pairs
  from a total of $\abs{A \times A} =q^2 (q - 1)^2$ pairs, and the
  result follows.
\end{proof}

\begin{pr} \label{sz_2_generation}
Let $G = \Sz(q)$. If $x, y \in G$ are uniformly random, then 
\begin{equation}
\Pr{\gen{x, y} = G} = 1 - \OR{\sigma_0(\log(q))/q^2}
\end{equation}
\end{pr}
\begin{proof}
By \eqref{suzuki_order} and Theorem \ref{sz_maximal_subgroups}, the maximal subgroup $M \leqslant G$ with smallest index is $M = \mathcal{FH}$. Then $[G : M] = q^2 + 1$ and since $M = \Norm_G(M)$, there are $q^2 + 1$ conjugates of $M$. 
\begin{equation}
\Pr{\gen{x, y} \leqslant M^g \; \text{some} \; g \in G} \leqslant \sum_{i = 1}^{q^2 + 1} \Pr{\gen{x, y} \leqslant M} = \frac{1}{q^2 + 1}
\end{equation}
The probability that $\gen{x, y}$ lies in any maximal not conjugate to
$M$ must be less than $1/(q^2 + 1)$ because the other maximals have larger indices. There are $\OR{\sigma_0(\log(q))}$ number of conjugacy classes of
maximal subgroups, and hence the probability that $\gen{x, y}$ lies in a maximal subgroup is $\OR{\sigma_0(\log(q))/q^2}$.
\end{proof}

\subsection{Alternative definition}
\label{section:sz_alt_def}
The way we have defined the Suzuki groups resembles the original
definition, but it is not clear that the groups are exceptional groups
of Lie type. This was first proved in \cite{ono62, MR0144967}. A more common
way to define the groups are as the fixed points of a certain
automorphism of $\Sp(4, q)$. This approach is followed in
\cite[Chapter $4.10$]{raw04}, and it provides a more straightforward method
to deal with non-constructive recognition of $\Sz(q)$.

Let $\Sp(4, q)$ denote the standard copy of the symplectic group, preserving the following symplectic form:
\begin{equation} \label{standard_symplectic_form}
J = \begin{bmatrix}
0 & 0 & 0 & 1 \\
0 & 0 & 1 & 0 \\
0 & 1 & 0 & 0 \\
1 & 0 & 0 & 0
\end{bmatrix}.
\end{equation}

From \cite[Chapter $4.10$]{raw04}, we know that the elements of $\Sz(q)$
are precisely the fixed points of an automorphism $\Psi$ of $\Sp(4,
q)$. Computing $\Psi(g)$ for some $g \in \Sp(4,q)$ amounts to taking a
submatrix of the exterior square of $g$ and then replacing each matrix
entry $x$ by $x^{2^m}$. Moreover, $\Psi$ is defined on $\Sp(4, F)$ for
$F \geqslant \F_q$. A more detailed description of how to compute
$\Psi(g)$ can be found in \cite[Chapter $4.10$]{raw04}.

\begin{lem} \label{lem_steinberg_lang}
Let $G \leqslant \Sp(4, q)$ have natural module $V$ and assume that $V$ is absolutely irreducible. Then $G^h \leqslant \Sz(q)$ for some $h \in \GL(4, q)$ if and only if $V \cong V^{\Psi}$.
\end{lem}
\begin{proof}
  Assume $G^h \leqslant \Sz(q)$. Both $G$ and $\Sz(q)$ preserve the
  form \eqref{standard_symplectic_form}, and this form is unique up to a
  scalar multiple, since $V$ is absolutely irreducible. Therefore $h J
  h^T = \lambda J$ for some $\lambda \in \F_q^{\times}$. But if $\mu =
  \sqrt{\lambda^{-1}}$ then $(\mu h) J (\mu h)^T = J$, so that $\mu h \in
  \Sp(4, q)$. Moreover, $G^h = G^{\mu h}$, and hence we may assume
  that $h \in \Sp(4, q)$. Let $x = h \Psi(h^{-1})$ and observe that
  for each $g \in G$, $\Psi(g^h) = g^h$. It follows that
\begin{equation}
g^x = \Psi(h) g^h \Psi(h^{-1}) = \Psi(h g^h h^{-1}) = \Psi(g)
\end{equation}
so $V \cong V^{\Psi}$.

Conversely, assume that $V \cong V^{\Psi}$. Then there is some $h \in
\GL(4, q)$ such that for each $g \in G$ we have $g^h = \Psi(g)$. As above, since
both $G$ and $\Psi(G)$ preserve the form
\eqref{standard_symplectic_form}, we may assume that $h \in \Sp(4,
q)$. 

Let $K$ be the algebraic closure of $\F_q$. The Steinberg-Lang
Theorem (see \cite{steinberg_lang}) asserts that there exists $x \in
\Sp(4, K)$ such that $h = x^{-1} \Psi(x)$. It follows that
\begin{equation}
\Psi(g^{x^{-1}}) = \Psi(g)^{h^{-1} x^{-1}} = g^{x^{-1}}
\end{equation}
so that $G^{x^{-1}} \leqslant \Sz(q)$. Thus $G$ is conjugate in
$\GL(4, K)$ to a subgroup $S$ of $\Sz(q)$, and it follows from
\cite[Theorem $29.7$]{curtis_reiner}, that $G$ is conjugate to $S$ in
$\GL(4, q)$.
\end{proof}

\subsection{Tensor indecomposable representations}
\label{section:sz_indecomposables}

It follows from \cite{MR1901354} that over an algebraically closed
field in defining characteristic, up to Galois twists, there is only
one absolutely irreducible tensor indecomposable representation of
$\Sz(q)$ : the natural representation.

\section{Small Ree groups}  \label{section:ree_theory}

The small Ree groups were first described in \cite{MR0125154, small_ree}. Their
structure has been investigated in \cite{MR0141707, ward66, ree85, kleidman88}. A short survey is also given in \cite[Chapter
$11$]{huppertIII}. They should not be confused with the Big Ree groups, which are described in Section \ref{section:big_ree_theory}.

\subsection{Definition and properties}
 We now define our standard copy of the Ree groups. The generators that we use are those described in \cite{ree_gens}. Let $q = 3^{2m +
   1}$ for some $m > 0$ and let $t = 3^m$. For $x \in \F_q$ and
 $\lambda \in \F_q^{\times}$, define the matrices
\begin{equation}
\alpha(x) = \begin{bmatrix}
1 & x^t & 0 & 0 & -x^{3 t + 1}  & -x^{3t + 2}  & x^{4t + 2}  \\
0 & 1 & x &  x^{t + 1} & -x^{2t + 1}  & 0 & -x^{3t + 2} \\
0 & 0 & 1 & x^t & -x^{2 t} & 0 & x^{3t + 1}  \\
0 & 0 & 0 & 1 & x^t & 0 & 0 \\
0 & 0 & 0 & 0 &	1 & -x & x^{t + 1} \\
0 & 0 & 0 & 0 & 0 & 1 & -x^t \\
0 & 0 & 0 & 0 & 0 & 0 &	1
\end{bmatrix} 
\end{equation}
\begin{equation}
\beta(x) = \begin{bmatrix}
1 & 0 & -x^t & 0 & -x & 0 & -x^{t + 1} \\
0 & 1 & 0 & x^t & 0 & -x^{2t} & 0 \\
0 & 0 & 1 & 0 & 0 & 0 & x \\
0 & 0 & 0 & 1 & 0 & x^t & 0 \\
0 & 0 & 0 & 0 &	1 & 0 & x^t \\
0 & 0 & 0 & 0 & 0 & 1 & 0 \\
0 & 0 & 0 & 0 & 0 & 0 &	1
\end{bmatrix} 
\end{equation}
\begin{equation}
\gamma(x) = \begin{bmatrix}
1 & 0 & 0 &  -x^t & 0 & -x &  -x^{2t} \\
0 & 1 & 0 & 0 &  -x^t & 0 & x  \\
0 & 0 & 1 & 0 & 0 & x^t & 0 \\
0 & 0 & 0 & 1 & 0 & 0 & -x^t \\
0 & 0 & 0 & 0 &	1 & 0 & 0 \\
0 & 0 & 0 & 0 & 0 & 1 & 0 \\
0 & 0 & 0 & 0 & 0 & 0 &	1
\end{bmatrix} 
\end{equation}
\begin{equation}
h(\lambda) = \begin{bmatrix}
\lambda^t & 0 & 0 & 0 & 0 & 0 & 0 \\
0 & \lambda^{1 - t} & 0 & 0 & 0 & 0 & 0 \\
0 & 0 & \lambda^{2t - 1} & 0 & 0 & 0  & 0 \\
0 & 0 & 0 & 1 & 0 & 0 & 0 \\
0 & 0 & 0 & 0 & \lambda^{1 - 2t} & 0 & 0 \\
0 & 0 & 0 & 0 & 0 & \lambda^{t - 1} & 0 \\
0 & 0 & 0 & 0 & 0 & 0 & \lambda^{-t} 
\end{bmatrix} 
\end{equation}
\begin{equation}
\Upsilon = \begin{bmatrix}
0 & 0 & 0 & 0 & 0 & 0 & -1 \\
0 & 0 & 0 & 0 & 0 & -1 & 0 \\
0 & 0 & 0 & 0 & -1 & 0 & 0 \\
0 & 0 & 0 & -1 & 0 & 0 & 0 \\
0 & 0 & -1 & 0 & 0 & 0 & 0 \\
0 & -1 & 0 & 0 & 0 & 0 & 0 \\
-1 & 0 & 0 & 0 & 0 & 0 & 0
\end{bmatrix}
\end{equation}

and define the Ree group as
\begin{equation}
\Ree(q) = \gen{ \alpha(x), \beta(x), \gamma(x), h(\lambda), \Upsilon \mid x \in \F_q, \lambda \in \F_q^{\times}}.
\end{equation}
Also, define the subgroups of upper triangular and diagonal matrices:
\begin{align}
U(q) &= \gen{\alpha(x), \beta(x), \gamma(x) \mid x \in \F_q} \\
H(q) &= \set{h(\lambda) \mid \lambda \in \F_q^{\times}} \cong \F_q^{\times}.
\end{align}

From \cite{ree85} we then know that each element of $U(q)$ can be
expressed in a unique way as 
\begin{equation}
S(a, b, c) = \alpha(a) \beta(b) \gamma(c)
\end{equation}
so that $U(q) = \set{S(a, b, c) \mid a, b, c \in \F_q}$, and it
follows that $\abs{U(q)} = q^3$. We also know that $U(q)$ is a Sylow
$3$-subgroup of $\Ree(q)$, and direct calculations show that 
\begin{align}
\label{ree_matrix_id1} \begin{split}
& S(a_1, b_1, c_1) S(a_2, b_2, c_2) = \\
& = S(a_1 + a_2, b_1 + b_2 - a_1 a_2^{3t}, c_1 + c_2 - a_2 b_1 + a_1 a_2^{3t + 1} - a_1^2 a_2^{3t}), 
\end{split} \\
& S(a, b, c)^{-1} = S(-a, -(b + a^{3t + 1}), -(c + ab - a^{3t + 2})), \\
\begin{split}
& S(a_1, b_1, c_1)^{S(a_2, b_2, c_2)} = \\
& = S(a_1, b_1 - a_1 a_2^{3t} + a_2 a_1^{3t}, c_1 + a_1 b_2 - a_2 b_1 + a_1 a_2^{3t + 1} - a_2 a_1^{3t + 1} - a_1^2 a_2^{3t} + a_2^2 a_1^{3t})
\end{split}
\end{align}
and
\begin{align} 
\label{ree_matrix_id2} S(a, b, c)^{h(\lambda)} &= S(\lambda^{3t - 2} a, \lambda^{1 - 3t} b, \lambda^{-1} c).
\end{align}

\begin{rem}[Standard generators of $\Ree(q)$]
  As standard generators for $\Ree(q)$ we will use 
\[\set{S(1, 0, 0), h(\lambda), \Upsilon},\] 
where $\lambda$ is a primitive element of
  $\F_q$, whose minimal polynomial is the defining polynomial of
  $\F_q$.
\end{rem}

The Ree groups preserve a symmetric bilinear form on $\F_q^7$, represented by the matrix
\begin{equation} \label{standard_bilinear_form}
J = \begin{bmatrix}
0 & 0 & 0 & 0 & 0 & 0 & 1 \\
0 & 0 & 0 & 0 & 0 & 1 & 0 \\
0 & 0 & 0 & 0 & 1 & 0 & 0 \\
0 & 0 & 0 & -1 & 0 & 0 & 0 \\
0 & 0 & 1 & 0 & 0 & 0 & 0 \\
0 & 1 & 0 & 0 & 0 & 0 & 0 \\
1 & 0 & 0 & 0 & 0 & 0 & 0
\end{bmatrix}
\end{equation}

From \cite{ward66} and \cite[Chapter $11$]{huppertIII} we immediately obtain
\begin{pr} \label{sylow3_props}
Let $G = \Ree(q)$.
\begin{enumerate}
\item $\abs{G} = q^3 (q^3 + 1) (q - 1)$ where $\gcd(q^3 + 1, q - 1) = 2$.
\item Conjugates of $U(q)$ intersect trivially.
\item The centre $\Zent(U(q)) = \set{S(0, 0, c) \mid c \in F_q}$.
\item The derived group $U(q)^{\prime} = \set{S(0, b, c) \mid b, c \in F_q}$, and its elements have order $3$.
\item The elements in $U(q) \setminus U(q)^{\prime} = \set{S(a, b, c) \mid a \neq 0}$ have order $9$ and their cubes form $\Zent(U(q)) \setminus \gen{1}$. 
\item $\Norm_G(U(q)) = U(q) H(q)$ and $G$ acts doubly transitively on the right cosets of $\Norm_G(U(q))$, \emph{i.e.} on a set of size $q^3 + 1$.
\item $U(q) H(q)$ is a Frobenius group with Frobenius kernel $U(q)$.
\item The proportion of elements of order $q-1$ in $U(q) H(q)$ is $\phi(q - 1) / (q - 1)$, where $\phi$ is the Euler totient function.
\end{enumerate}
\end{pr}

For our purposes, we want another set to act (equivalently) on.
\begin{pr} \label{ree_doubly_transitive_action}
There exists $\OV \subseteq \PS^6(\F_q)$ on which $G = \Ree(q)$ acts faithfully and doubly transitively. This set is
\begin{equation} \label{ree_ovoid_def}
\begin{split}
\OV = &\set{(0 : 0 : 0 : 0 : 0 : 0 : 1)} \cup \\
 &\lbrace (1 : a^t : -b^t : (ab)^t - c^t : -b - a^{3t + 1} - (ac)^t :  -c - (bc)^t - a^{3t + 2} - a^t b^{2t} : \\ 
& a^t c - b^{t + 1} + a^{4t + 2} - c^{2t} - a^{3t + 1} b^t - (abc)^t) \rbrace 
\end{split}
\end{equation}
Moreover, the stabiliser of $P_{\infty} = (0 : 0 : 0 : 0 : 0 : 0 : 1)$ is $U(q) H(q)$, the stabiliser of $P_0 = (1 : 0 : 0 : 0 : 0 : 0 : 0)$ is $(U(q) H(q))^{\Upsilon}$ and the stabiliser of $(P_{\infty}, P_0)$ is $H(q)$.
\end{pr}
\begin{proof} Notice that $\OV \setminus
  \set{P_{\infty}}$ consists of the first rows of the elements of $U(q)H(q)$.
  From Proposition \ref{sylow3_props} it follows that $G$ is the
  disjoint union of $U(q)H(q)$ and $U(q)H(q) \Upsilon U(q)H(q)$.
  Define a map between the $G$-sets as $(U(q)H(q))g \mapsto P_{\infty}
  g$.

  If $g \in U(q) H(q)$ then $P_{\infty} g = P_{\infty}$ and hence the
  stabiliser of $P_{\infty}$ is $U(q) H(q)$. If $g \notin U(q) H(q)$
  then $g = x \Upsilon y$ where $x,y \in U(q) H(q)$. Hence $P_{\infty}
  g = P_0 y \in \OV$ since $P_0 y$ is the first row of $y$. It follows
  that the map defines an equivalence between the $G$-sets.
\end{proof}



\begin{pr} \label{pr_involution_props}
Let $G = \Ree(q)$.
\begin{enumerate}
\item The stabiliser in $G$ of any two distinct points of $\OV$ is conjugate to $H(q)$, and the stabiliser of any triple of points has order $2$.
\item The number of elements in $G$ that fix exactly one point is $q^6 - 1$.
\item All involutions in $G$ are conjugate in $G$.
\item An involution fixes $q + 1$ points.
\end{enumerate}
\end{pr}
\begin{proof}
\begin{enumerate}
\item Immediate from \cite[Chapter $11$, Theorem $13.2$(d)]{huppertIII}.
\item A stabiliser of a point is conjugate to $U(q) H(q)$, and there
  are $\abs{\OV}$ conjugates. The elements fixing exactly one point are the non-trivial elements of $U(q)$. Therefore the number of such elements is 
$\abs{\OV} (\abs{U(q)} - 1) = (q^3 + 1)(q^3  - 1) = (q^6 - 1)$.
\item Immediate from \cite[Chapter $11$, Theorem $13.2$(e)]{huppertIII}.
\item Each involution is conjugate to \[h(-1) = \diag(-1, 1,
  -1, 1, -1, 1, -1).\] Evidently, $h(-1)$ fixes $P_{\infty}$ since
  $h(-1) \in H(q)$ and if $P = (p_1 : \dotsm : p_7) \in \OV$ with $p_1
  \neq 0$, then $P$ is fixed by $h(-1)$ if and only if $p_2 = p_4 = p_6 = 0$.
  But then $P$ is uniquely determined by $p_3$, so there are $q$
  possible choices for $P$. Thus the number of points fixed by $h(-1)$ is $q
  + 1$.
\end{enumerate}
\end{proof}

\begin{pr} \label{pr_inv_centraliser_split} Let $G = \Ree(q)$ with natural module $V$ and let $j \in G$ be an involution. Then $V\vert_{\Cent_G(j)} \cong S_j \oplus T_j$ where $\dim S_j =
  3$ and $\dim T_j = 4$. Moreover, $S_j$ is irreducible and $j$ acts
  trivially on $S_j$.
\end{pr}
\begin{proof}
By Proposition \ref{pr_involution_props}, $j$ is conjugate to $h(-1) =
\diag{(-1, 1, -1, 1, -1, 1, -1)}$ so it has two eigenspaces $S_j$ and
$T_j$ for $1$ and $-1$ respectively. Clearly $\dim S_j = 3$ and $\dim
T_j = 4$, and it is sufficient to show that these are preserved by
$\PSL(2, q)$, so that they are in fact submodules of $V_j$.

Let $v \in S_j$ and $g \in \PSL(2, q)$. Then $(vg) j = (vj)g = vg$
since $g$ centralises $j$ and $j$ fixes $v$, which shows that $vg \in
S_j$, so this subspace is fixed by $\PSL(2, q)$. Similarly, $T_j$ is
also fixed.

Let $\gamma : V \times V \to V$ be the bilinear form preserved by $G$.
Observe that if $x \in S_j$, $y \in V_j$ then $\gamma(x, y) = \gamma
(x j, yj) = \gamma(x, -y) = -\gamma(x,y) = 0$ and hence $V_j \subseteq
S_j^{\perp}$. If $\gamma\vert_{S_j}$ is degenerate then also $S_j \subseteq
S_j^{\perp}$ so that $S_j \subseteq V^{\perp}$ which is impossible
since $\gamma$ is non-degenerate. Hence $\gamma\vert_{S_j}$ is
non-degenerate and $S_j$ is isomorphic to its dual.

Now if $S_j$ is reducible, it must split as a direct sum of two
submodules of dimension $1$ and $2$. Since $j$ acts trivially on
$S_j$, it is in fact a module for $\PSL(2, q)$, but $\PSL(2, q)$ have
no irreducible modules of dimension $2$. Therefore $S_j$ must be
irreducible.
\end{proof}

\begin{lem} \label{lem_1eigenvalue}
  Let $g \in G \leqslant \GL(d, F)$ with $d$ odd and $F$ any finite
  field, and assume that $G$ preserves a non-degenerate bilinear form
  and that $\det(g) = 1$. Then $g$ has $1$ as an eigenvalue.
\end{lem}
\begin{proof}
  Let $V = F^d$ be the natural module of $G$. Let $f : V \times
  V \to F$ be a non-degenerate bilinear form preserved by $G$. Over
  $\bar{F}$, the multiset of eigenvalues of $g$ is $\lambda_1,
  \dotsc, \lambda_d$. Let $E_{\lambda_1}, \dotsc, E_{\lambda_d}$ be
  the corresponding eigenspaces (some of them might be equal).
  
  If $v \in E_{\lambda_i}$ and $w \in E_{\lambda_j}$ then
  \[ f(v, w) = f(vg, wg) = f(v \lambda_i, w \lambda_j) = \lambda_i
  \lambda_j f(v, w) \] so either $\lambda_i \lambda_j = 1$ or $f(v, w)
  = 0$. However, for a given $i$ there must be some $j$ such that
  $\lambda_i \lambda_j = 1$, otherwise $f(v, w) = 0$ for every $v \in
  E_{\lambda_i}$ and $w \in V$, which is impossible since $f$ is
  non-degenerate.
  
  Hence the eigenvalues can be arranged into pairs of inverse values.
  Since $d$ is odd, there must be a $k$ such that $\lambda_k$ is left
  over. The above argument then implies that $\lambda_k^2 = 1$, and
  finally $1 = \det(g) = \lambda_k$.
\end{proof}

From \cite{ree85} and \cite{kleidman88} we obtain
\begin{pr} \label{ree_maximal_subgroup_list}
A maximal subgroup of $G = \Ree(q)$ is conjugate to one of the following subgroups
\begin{itemize}
\item $\Norm_G(U(q)) = U(q) H(q)$, the point stabiliser .
\item $\Cent_G(j) \cong \gen{j} \times \PSL(2, q)$, the centraliser of an involution $j$.
\item $\Norm_G(A_0) \cong (\Cent_2 \times \Cent_2 \times A_0) {:} \Cent_6$, where $A_0 \leqslant \Ree(q)$ is cyclic of order $(q + 1) / 4$.
\item $\Norm_G(A_1) \cong A_1 {:} \Cent_6$, where $A_1 \leqslant \Ree(q)$ is cyclic of order $q + 1 - 3t$.
\item $\Norm_G(A_2) \cong A_2 {:} \Cent_6$, where $A_2 \leqslant \Ree(q)$ is cyclic of order $q + 1 + 3t$.
\item $\Ree(s)$ where $q$ is a proper power of $s$.
\end{itemize}
Moreover, all maximal subgroups except the last are reducible. 
\end{pr}
\begin{proof}
  It is sufficient to prove the final statement. 
  
  Clearly the point stabiliser is reducible, and the involution
  centraliser is reducible by Proposition \ref{pr_inv_centraliser_split}.

  Let $H$ be a normaliser of a cyclic subgroup and let $x$ be a
  generator of the cyclic subgroup that is normalised. Since $G \leqslant \SO(7, q)$, by Lemma
  \ref{lem_1eigenvalue}, $x$ has an eigenspace $E$ for the
  eigenvalue $1$, where $V \neq E \neq \set{0}$. Given $v \in E$ and
  $h \in H$, we see that $(vh) x^h = vh$ so that $vh$ is fixed by
  $\gen{x^h} = \gen{x}$. This implies that $vh \in E$ and thus $E$ is
  a proper non-trivial $H$-invariant subspace, so $H$ is reducible.
\end{proof}

\begin{pr} \label{cyclic_subgroups_conjugate}
Let $G = \Ree(q)$.
\begin{enumerate}
\item All cyclic subgroups of $G = \Ree(q)$ of order $q - 1$ are conjugate to $H(q)$ and hence each one is a stabiliser of two points of $\OV$.
\item All cyclic subgroups of order $(q + 1) / 2$ or $q \pm 3t + 1$ are conjugate.
\item If $C$ is a cyclic subgroup of order $q \pm 3t + 1$, then distinct conjugates of $C$ intersect trivially.
\item If $C$ is a cyclic subgroup of order $(q + 1) / 2$, $C > C^{\prime} \cong A_0$ and $x \in G \setminus \Norm_G(C^{\prime})$, then $C \cap C^x = \gen{1}$.
\end{enumerate}
\end{pr}
\begin{proof}
\begin{enumerate}
\item Let $C = \gen{g} \leqslant G$ be cyclic of order $q - 1$ and let $p$
  be an odd prime such that $p \mid q - 1$. Then there exists $k \in \Z$
  such that $\abs{g^k} = p$. Since $q^3 + 1 \equiv 2 \pmod p$, the
  cycle structure of $g^k$ on $\OV$ must be a number of $p$-cycles and
  $2$ fixed points $P$ and $Q$. Since $G$ is doubly transitive there
  exists $x \in G$ such that $Px = P_{\infty}$ and $Qx = P_0$.
  
  Now either $g$ fixes $P$ and $Q$ or interchanges them, so $g^x \in
  \Norm_G(H(q)) = \gen{H(q), \Upsilon} \cong \Dih_{2(q-1)}$. Hence
  $\gen{g^x} = H(q)$ since that is the unique cyclic subgroup of order
  $q - 1$ in $\gen{H(q), \Upsilon}$.
\item This follows immediately from \cite[Lemma $2$]{ree85}.
\item Let $C$ be such a cyclic subgroup. If $C \neq C^x$ for some $x \in G$ and $g \in C \cap C^x$, then $\Cent_G(g) = \gen{C \cup C^x}$. But $\gen{C \cup C^x} = G$, so that $g = 1$.
\item Since $\gen{C \cup C^x} = G$, this is analogous to the previous case.
\end{enumerate}
\end{proof}

\begin{pr} \label{pr_element_props}
Let $G = \Ree(q)$ and let $\phi$ be the Euler totient function.
\begin{enumerate}
\item The centraliser of an involution $j \in G$ is isomorphic to $\gen{j} \times \PSL(2, q)$ and hence has order $q (q^2 - 1)$.
\item The number of involutions in $G$ is $q^2 (q^2 - q + 1)$.
\item The number of elements in $G$ of order $q - 1$ is $\phi(q - 1) q^3 (q^3 + 1) / 2$.
\item The number of elements in $G$ of order $(q + 1) / 2$ is $\phi((q + 1) / 2) q^3 (q - 1) (q^2 - q + 1) / 6$.
\item The number of elements in $G$ of order $(q \pm 3t + 1)$ is $\phi(q \pm 3t + 1) q^3 (q^2 - 1) (q \mp 3t + 1) / 6$.
\item The number of elements in $G$ of even order is $q^2(7 q^5 - 23 q^4 + 8 q^3 + 23 q^2 - 39 q + 24) / 24$.
\item The number of elements in $G$ that fix at least one point is $q^2 (q^5 - q^4 + 3 q^2 - 5q + 2) / 2$
\end{enumerate}
\end{pr}
\begin{proof} 
\begin{enumerate}
\item Immediate from \cite[Chapter $11$]{huppertIII}.
\item All involutions are conjugate, and the index in $G$ of the involution centraliser is
\begin{equation}
\frac{q^3 (q^3 + 1) (q - 1)}{q (q^2 - 1)} = q^2 (q^2 - q + 1)
\end{equation}
where we have used the fact that $q^3 + 1 = (q + 1) (q^2 - q + 1)$.
\item By Proposition \ref{cyclic_subgroups_conjugate}, each cyclic subgroup of order $q - 1$ is a stabiliser of two
  points and is uniquely determined by the pair of points that it
  fixes. Hence the number of cyclic subgroups of order $q - 1$ is
\begin{equation}
\abs{\binom{\OV}{2}} = \frac{q^3 (q^3 + 1)}{2}.
\end{equation}
By Proposition \ref{pr_involution_props}, the intersection of two distinct subgroups has order $2$, so the number of elements of order $q - 1$ is the number of generators of all these subgroups. 
\item By Proposition \ref{cyclic_subgroups_conjugate}, the number of
  cyclic subgroups of order $(q + 1) / 2$ is $[G : \Norm_G(A_0)] = q^3
  (q - 1) (q^2 - q + 1) / 6$. Since distinct conjugates intersect
  trivially, the number of elements of order $(q + 1) / 2$ is the
  number of generators of these subgroups.
\item Analogous to the previous case.
\item By \cite[Lemma $2$]{ree85}, every element of even order lies in
  a cyclic subgroup of order $q - 1$ or $(q + 1) / 2$. In each cyclic subgroup of order
  $q - 1$ there is a unique involution and hence $(q - 3) / 2$
  non-involutions of even order, and similarly $(q - 3) / 4$ in a cyclic subgroup of order $(q + 1) / 2$. By Proposition \ref{cyclic_subgroups_conjugate} the total number of elements of even
  order is therefore
\begin{multline}
(q - 3)(q^3 + 1) q^3 / 4 + (q - 3) (q - 1) (q^2 - q + 1) q^3 / 24 + q^2 (q^2 - q + 1) \\
= q^2(7 q^5 - 23 q^4 + 8 q^3 + 23 q^2 - 39 q + 24) / 24
\end{multline}
\item The only non-trivial elements of $G$ that fix more than $2$
  points are involutions. Hence in each cyclic subgroup of order $q -
  1$ there are $q - 3$ elements that fix exactly $2$ points, so by Proposition \ref{pr_involution_props}, the
  number of elements that fix at least one point is
\begin{equation}
\begin{split}
& q^6 + \frac{(q - 3)(q^3 + 1) q^3}{2} + q^2 (q^2 - q + 1) = \\
& =  \frac{q^2 (q^5 - q^4 + 3 q^2 - 5q + 2)}{2}
\end{split}
\end{equation}
\end{enumerate}
\end{proof}

\begin{lem} \label{ree_totient_prop}
If $g \in G = \Ree(q)$ is uniformly random, then
\begin{align}
\Pr{\abs{g} = q - 1} = \frac{\phi(q - 1)}{2(q - 1)} &> \frac{1}{12\log{\log(q)}} \\
\Pr{\abs{g} = q \pm 3t + 1} = \frac{\phi(q \pm 3t + 1)}{6(q \pm 3t +1)} &> \frac{1}{36\log{\log(q)}}\\
\Pr{\abs{g} = (q + 1) / 2} = \frac{\phi((q + 1) / 2)}{6 (q + 1)} &> \frac{1}{36\log{\log(q)}} \\
\Pr{\abs{g} \text{even}} = \frac{7q^2 - 9q - 24}{24 q (q + 1)} &> 1/4  
\end{align}
\begin{equation}
\Pr{g \ \text{fixes a point}} = \frac{-2 + 3 q + q^4}{2(q + q^4)} \geqslant 1/2
\end{equation}
\end{lem}
\begin{proof}
In each case, the first equality follows from Proposition \ref{pr_element_props} and Proposition \ref{sylow3_props}. In the first case, the
inequality follows from \cite[Section II.8]{totient_prop}, and in the other cases the inequalities are clear since $m > 0$.
\end{proof}

\begin{cl} \label{cl_random_selections} In $G = \Ree(q)$, the expected
  number of random selections required to obtain an element of order
  $q - 1$, $q \pm 3t + 1$ or $(q + 1) / 2$ is $\OR{\log \log q}$, and
  $\OR{1}$ to obtain an element that fixes a point, or an element of
  even order.
\end{cl}
\begin{proof}
  Clearly the number of selections is geometrically distributed, where
  the success probabilities for each selection are given by Lemma
  \ref{ree_totient_prop}. Hence the expectations are as stated.
\end{proof}

\begin{pr} \label{ree_conjugacy_classes}
Elements in $\Ree(q)$ of order prime to $3$ with the same trace are conjugate.
\end{pr}
\begin{proof}
  From \cite{ward66}, the number of conjugacy classes of non-identity
  elements of order prime to $3$ is $q - 1$. Observe that for $\lambda
  \in \F_q^{\times}$, $\Tr (S(0, 0, 1) \Upsilon h(\lambda)) = \lambda^t -
  1$ and $\abs{S(0, 0, 1) \Upsilon h(\lambda)}$ is prime to $3$ if also
  $\lambda \neq -1$.
  
  Moreover, $h(-1)$ has order $2$ and trace $-1$ so there are $q - 1$
  possible traces for non-identity elements of order prime to $3$, and
  elements with different trace must be non-conjugate. Thus all
  conjugacy classes must have different traces.
\end{proof}

\begin{pr} \label{ree_dihedral_trick}
If $i, j \in Ree(q)$ are uniformly random involutions, then $\Pr{\abs{ij} \, \text{odd}} > c$ for some constant $c > 0$.
\end{pr}
\begin{proof}
Follows immediately from \cite[Theorem $13$]{parkerwilson06}.
\end{proof}

\begin{pr} \label{psl_generation}
Let $G = \PSL(2, q)$. If $x, y \in G$ are uniformly random, then 
\begin{equation}
\Pr{\gen{x, y} = G} = 1 - \OR{\sigma_0(\log(q)) / q}
\end{equation}
\end{pr}
\begin{proof}
The maximal subgroup $M \leqslant G$ consisting of the upper triangular matrices modulo scalars has index $q + 1$, and all subgroups isomorphic to $M$ are conjugate. Since $M = \Norm_G(M)$, there are $q + 1$ conjugates of $M$. 
\begin{equation}
\Pr{\gen{x, y} \leqslant M^g \; \text{some} \; g \in G} \leqslant \sum_{i = 1}^{q + 1} \Pr{\gen{x, y} \leqslant M} = \frac{1}{q + 1}
\end{equation}
The other maximal subgroups have index strictly greater than $q + 1$, so
the probability that $\gen{x, y}$ lies in any maximal not conjugate to
$M$ must be less than $1/(q + 1)$. The number of conjugacy classes of
maximal subgroups is $\OR{\sigma_0(\log(q))}$, and hence the probability that $\gen{x,
  y}$ lies in a maximal subgroup is $\OR{\sigma_0(\log(q)) / q}$.
\end{proof}

\begin{pr} \label{ree_ovoid_points}
Let $P = (p_1 : \dotsb : p_7) \in \OV^g$ be uniformly random, where $\OV^g = \set{Rg \mid R \in \OV}$ for some $g \in \GL(7, q)$. Then
\begin{enumerate}
\item \begin{equation}
\Pr{p_3 \neq 0} \geqslant 1 - \frac{t q^2 + 1}{q^3 + 1}.
\end{equation}
\item If $Q = (q_1 : \dotsb : q_7) \in \OV^g$ is given, then
\begin{equation}
\Pr{p_3 q_2^{3t + 1} \neq q_3 p_2^{3t + 1}} \geqslant 1 - \frac{1 + q^2(3t + 1)}{q^3 + 1}.
\end{equation}
\end{enumerate}
\end{pr}
\begin{proof}
If $P = P^{\prime} g$, where $P_{\infty} \neq P^{\prime} \in
  \OV$ and $g = [g_{i,j}]$, then $p_i = g_{1, i} + a^t g_{2, i} -
  b^t g_{3, i} + ((ab)^t - c^t) g_{4, i} + (-b - a^{3t + 1} - (ac)^t)
  g_{5, i} + (-c - (bc)^t - a^{3t + 2} - a^t b^{2t}) g_{6, i} + (a^t c
  - b^{t + 1} + a^{4t + 2} - c^{2t} - a^{3t + 1} b^t - (abc)^t) g_{7,
    i}$ for some $a, b, c \in \F_q$.

\begin{enumerate}

\item By introducing indeterminates $x$, $y$ and $z$ in place of $a$, $b$ and $c$, it
  follows that $p_3$ is a
  polynomial $f \in \F_q[x, y, z]$ with $\deg_x(f) \leqslant 4t + 2$,
  $\deg_y(f) \leqslant 2t$ and $\deg_z(f) \leqslant t$. For each $(a, b) \in \F_q^2$, the number of
  roots of $f(a, b, z)$ is therefore at most $t$, so the number of
  roots of $f$ is at most $q^2(3t + 1)$.

\item Similarly, by introducing indeterminates $x$, $y$ and $z$ in place of $a$, $b$ and $c$, it
  follows that the expression $p_3 q_2^{3t + 1} - q_3 p_2^{3t + 1}$ is a
  polynomial $f \in \F_q[x, y, z]$ with $\deg_x(f) \leqslant 10t+6$,
  $\deg_y(f) \leqslant 5 t + 1$ and $\deg_z(f) \leqslant 3 t + 1$. For each $(a, b) \in \F_q^2$, the number of
  roots of $f(a, b, z)$ is therefore at most $3 t + 1$, so the number of
  roots of $f$ is at most $q^2(3t + 1)$.
\end{enumerate}
\end{proof}

\begin{pr} \label{ree_prop_frobenius}
If $g_1, g_2 \in U(q)H(q)$ are uniformly random and independent, then
\begin{equation}
\Pr{\abs{[g_1, g_2]} = 9} = 1 - \frac{1}{q - 1}
\end{equation}
\end{pr}
\begin{proof}
  By Proposition \ref{sylow3_props}, $[g_1, g_2] \in U(q)$ and
  has order $9$ if and only if $[g_1, g_2] \notin U(q)^{\prime}
  \triangleleft U(q)H(q)$. It is therefore sufficient to find the
  proportion of (unordered) pairs $k_1, k_2 \in U(q) H(q) / U(q)^{\prime} = A$ such that $[k_1,
  k_2] = 1$. 

  If $k_1 = 1$ then $k_2$ can be any element of $A$, which gives $q(q
  - 1)$ pairs.  If $1 \neq k_1 \in U(q) / U(q)^{\prime} \cong
  \F_q$ then $\Cent_A(k_1) = U(q) / U(q)^{\prime}$, so we again obtain
  $q(q - 1)$ pairs. Finally, if $k_1 \notin U(q)$ then
  $\abs{\Cent_A(k_1)} = q - 1$ so we obtain $q(q - 2)(q - 1)$ pairs.
  Thus we obtain $q^2 (q - 1)$ pairs from a total of $\abs{A \times A}
  =q^2 (q - 1)^2$ pairs, and the result follows.
\end{proof}

\subsection{Alternative definition} \label{section:alt_definition} The
definition of $\Ree(q)$ that we have given is the one that best suits
most our purposes. However, to deal with non-constructive recognition, we need to
mention the more common definition of $\Ree(q)$.

Following \cite[Chapter $4$]{raw04} and \cite{elementary_ree}, the exceptional group $G_2(q)$ is
constructed by considering the Cayley algebra $\OO$ (the octonion
algebra), which has dimension $8$, and defining $G_2(q)$ to be the
automorphism group of $\OO$. Thus each element of $G_2(q)$ fixes the
identity and preserves the algebra multiplication, and it follows that
it is isomorphic to a subgroup of $\SO(7, q)$.

Furthermore, when $q$ is an odd power of $3$, the group $G_2(q)$ has a
certain automorphism $\Psi$, sometimes called the \emph{exceptional outer
  automorphism}, whose set of fixed points form a group, and this is
defined to be the Ree group $\Ree(q) = \G2(q)$. A more detailed description of how to compute $\Psi(g)$ can be found in \cite{elementary_ree}.

\subsection{Tensor indecomposable representations}
\label{section:ree_indecomposables}

It follows from \cite{MR1901354} that over an algebraically closed
field in defining characteristic, up to Galois twists, there are
precisely two absolutely irreducible tensor indecomposable
representations of $\Ree(q)$: the natural representation and a
representation of dimension $27$. Let $V$ be the natural module of
$\Ree(q)$, of dimension $7$. The symmetric square $\mathcal{S}^2(V)$
has dimension $28$, and is a direct sum of two submodules of
dimensions $1$ and $27$. The $1$-dimensional submodule arises because
$\Ree(q)$ preserves a quadratic form.

\section{Big Ree groups} \label{section:big_ree_theory}

The Big Ree groups were first described in \cite{MR0125155, large_ree}, and are
covered in \cite[Chapter $4$]{raw04}. The maximal subgroups are given in
\cite{malle_bigree}, and representatives of the conjugacy classes
are given in \cite{large_ree_conj_1} and \cite{large_ree_conj_2}. An
elementary construction, suitable for our purposes, is described in
\cite{elementary_ree}.

\subsection{Definition and properties}

We take the definition of the standard copy of $\LargeRee(q)$ from
\cite{elementary_ree}. 

The exceptional group $\mathrm{F}_4(q)$ is constructed by considering
the exceptional Jordan algebra (the Albert algebra), which has
dimension $27$, and defining $\mathrm{F}_4(q)$ to be its automorphism
group. Thus each element of $\mathrm{F}_4(q)$ fixes the identity and
preserves the algebra multiplication, and one can show that it is a
subgroup of $\mathrm{O}^{-}(26, q)$.

Furthermore, when $q$ is an odd power of $2$, the group
$\mathrm{F}_4(q)$ has a certain automorphism $\Psi$ whose set of fixed
points form a group, and this is defined to be the Big Ree group
$\LargeRee(q)$. A more detailed description of how to compute
$\Psi(g)$ can be found in \cite{elementary_ree}.

Let $t = 2^{m + 1} = \sqrt{2q}$ and let $V = \F_q^{26}$. From \cite{elementary_ree} we
immediately obtain:

\begin{pr} \label{pr_torus_elt}
Let $G = \LargeRee(q)$ and $g \in G$ with $\abs{g} = q - 1$. Then $g$ is conjugate in $G$ to an element of the form
\begin{equation}
\begin{split}
\varsigma(a, b) = \diag(& a, b, a^{t - 1} b^{t - 1}, a b^{1 - t}, a^t b^{-1}, b a^{1 - t}, 
 b^{t - 1}, b^{1 - t} a^{2 - t}, a b^{-1}, a^{t - 1}, 1, a^{-1} b^t, \\
& b^{2 - t} a^{1 - t}, a^{t - 1} b^{t - 2}, a b^{-t}, 1, 
 a^{1 - t}, b a^{-1}, a^{t - 2} b^{t - 1}, b^{1 - t}, a^{t - 1} b^{-1}, \\
& b a^{-t}, a^{-1} b^{t - 1}, b^{1 - t} a^{1 - t}, b^{-1}, a^{-1})
\end{split}
\end{equation}
for some $a, b \in \F_q^{\times}$.
\end{pr}

Let $\set{e_1, \dotsc, e_{26}}$ be the standard basis of $V$. Following \cite{elementary_ree}, we then define the following matrices as permutations on this basis:
\begin{align}
\varrho =& (15, 12)(14, 13)(2, 5)(7, 8)(19, 20)(22, 25)(3, 4)(6, 9)(18, 21)(24, 23) \\
\kappa = \begin{split}
& (11, 16)(1, 2)(8, 13)(14, 19)(25, 26)(7, 10) \\
& (5, 12)(15, 22)(17, 20)(4, 6)(9, 18)(21, 23)
\end{split}
\end{align}

We also define linear transformations $z$ and $\nu$, where $z$ fixes $e_i$ for $i = 1, \dotsc, 15$ and otherwise acts as follows:
\begin{align}
e_{16} &\mapsto e_1 + e_{16} & e_{17} &\mapsto e_1 + e_{17} &  e_{18} &\mapsto e_2 + e_{18}  \\
e_{19} &\mapsto e_3 + e_{19} &  e_{20} &\mapsto e_4 + e_{20} & e_{21} & \mapsto e_5 + e_{21} \\
e_{22} &\mapsto e_2 + e_6 + e_{22} &  e_{23} &\mapsto e_3 + e_7 + e_{23} & e_{24} & \mapsto e_4 + e_8 + e_{24} \\
e_{25} &\mapsto e_5 + e_9 + e_{25} & e_{26} & \mapsto e_1 + e_{10} + e_{11} + e_{26} 
\end{align}

Furthermore, we define a block-diagonal matrix $\zeta$ as follows:
\begin{equation}
\zeta_1 = \begin{bmatrix}
1 & 0 & 0 & 0 \\
1 & 1 & 0 & 0 \\
1 & 1 & 1 & 0 \\
1 & 0 & 1 & 1 
\end{bmatrix} 
\end{equation}
\begin{equation}
\zeta_2 = \begin{bmatrix}
1 & 0 & 0 & 0 & 0 & 0 \\
0 & 1 & 0 & 0 & 0 & 0 \\
0 & 1 & 1 & 0 & 0 & 0 \\
1 & 1 & 1 & 1 & 0 & 0 \\
1 & 1 & 0 & 1 & 1 & 0 \\
0 & 0 & 1 & 0 & 0 & 1 
\end{bmatrix} 
\end{equation}
\begin{equation}
\zeta_3 = \begin{bmatrix}
1
\end{bmatrix}
\end{equation}
and then $\zeta$ has diagonal blocks $\zeta_3, \zeta_1, \zeta_1,
\zeta_3, \zeta_2, \zeta_3, \zeta_1, \zeta_1, \zeta_3$. 

Finally, $\nu$ fixes $e_i$ for $i \in \set{1, 3, 4, 5, 8, 9, 14, 15, 21, 24, 25}$ and otherwise acts as follows:
\begin{align}
e_2 &\mapsto e_1 + e_2 & e_6 &\mapsto e_4 + e_6 & e_7 &\mapsto e_5 + e_7 \\
e_{10} &\mapsto e_5 + e_{10} & e_{11} &\mapsto e_9 + e_{11} & e_{12} &\mapsto e_5 + e_7 + e_{10} + e_{12} 
\end{align}
\begin{align}
e_{13} &\mapsto e_8 + e_{13} & e_{16} &\mapsto e_9 + e_{16} & e_{17} &\mapsto e_{15} + e_{17} \\
e_{18} &\mapsto e_9 + e_{11} + e_{16} + e_{18} & e_{19} &\mapsto e_{14} + e_{19} & e_{20} &\mapsto e_{15} + e_{20} \\
e_{22} &\mapsto e_{15} + e_{17} + e_{20} + e_{22} & e_{23} &\mapsto e_{21} + e_{23} & e_{26} &\mapsto e_{25} + e_{26}
\end{align}

From \cite{elementary_ree} we then immediately obtain:
\begin{thm}
Let $\lambda$ be a primitive element of $\F_q$. The elements $\varrho$, $\kappa$, $z$, $\nu$ and $\zeta$ lie in $\LargeRee(q)$, and if $G = \gen{\varsigma(1, \lambda), \varrho \nu \zeta}$, then $G \cong \LargeRee(q)$.
\end{thm}

\begin{pr} \label{pr_diagonal_elt}
Let $G = \LargeRee(q)$ and $g \in G$ with $\abs{g} = (q - 1) (q + t +
1)$. Then $g^{q + t + 1}$ is conjugate in $G$ to $\varsigma(1, b)$, which we
write as
\begin{equation}
\begin{split}
h(\lambda, \mu) = \diag( & 1, \lambda, \mu \lambda^{-1}, \lambda \mu^{-1}, \lambda^{-1}, 
 \lambda, \mu \lambda^{-1}, \lambda \mu^{-1}, \lambda^{-1}, 1, 1, \mu, \lambda^2 \mu^{-1}, 
 \mu \lambda^{-2}, \\ 
& \mu^{-1}, 1, 1, \lambda, \mu \lambda^{-1}, \lambda \mu^{-1}, \lambda^{-1}, \lambda, \mu \lambda^{-1}, \lambda \mu^{-1}, \lambda^{-1}, 1)
\end{split}
\end{equation}
where $\lambda = b \in \F_q^{\times}$ and $\mu = \lambda^t$.
\end{pr}
 \begin{proof}
Using the notation of \cite[Page $10$]{large_ree_conj_2}, we see that
with respect to a suitable basis, $g$ lies in $T(4) \cong \Cent_{q - 1}
\times \Cent_{q +t + 1}$ and that $g^{q + t + 1}$ will have the form
$(\epsilon, \epsilon^{2 \theta - 1}, 1, 1)$ for some $\epsilon \in
\F_q^{\times}$. Hence it will lie in one of the factors of $T(1) \cong
\Cent_{q - 1}^2$.
\end{proof}

\begin{pr} \label{pr_random_torus_elt}
If $g \in G = \LargeRee(q)$ is uniformly random, then
\begin{equation}
\Pr{\abs{g} = (q - 1)(q + t + 1)} \geqslant \frac{5349}{54080 \log
  \log(q)} \approx \frac{1}{10 \log \log (q)}
\end{equation}
and hence the expected number of random selections required to obtain an element of order $(q - 1)(q + t + 1)$ is
$\OR{\log\log{q}}$.
\end{pr}
\begin{proof}
Using the notation of \cite[Page $10$]{large_ree_conj_2}, we see that
such elements lie in a conjugate of $T(4) \cong \Cent_{q - 1} \times
\Cent_{q +t + 1}$. The proportion of the elements in $T(4)$ is therefore
$\phi(q - 1) \phi(q + t + 1) / ((q - 1)(q + t + 1))$.

From \cite[Table $\mathrm{IV}$]{large_ree_conj_2} we see that the
elements in $T(4)$ are of types $t_9$ and $t_{10}$ and that the total
number of such elements in $G$ is 
\begin{equation*}
\frac{(q + t) \abs{G}}{4 q^2 (q + t + 1)(q - 1)(q^2 + 1)} + \frac{(q -
  2) (q + t) \abs{G}}{8 (q - 1) (q + t + 1)}.
\end{equation*}
The proportion in $G$ of the elements of the required order is
therefore at least
\begin{equation*}
\frac{\phi(q - 1) \phi(q + t + 1)}{(q - 1)(q + t + 1)} \left(
\frac{(q + t)}{4 q^2 (q + t + 1)(q - 1)(q^2 + 1)} + \frac{(q -
  2) (q + t)}{8 (q - 1) (q + t + 1)}\right)
\end{equation*}
and the expression in parentheses is minimised when $m = 1$. The result now follows from \cite[Section II.8]{totient_prop}.
\end{proof}

By definition we have the inclusions $\LargeRee(q) < \mathrm{F}_4(q) <
\mathrm{O}^{-}(26, q) < \Sp(26, q) < \SL(26, q) < \GL(26, q)$, so
$\LargeRee(q)$ preserves a quadratic form $Q^{*}$ with associated
bilinear form $\beta^{*}$. It follows from
\cite{elementary_ree} that
\begin{align}
Q^{*}(e_i) &= \begin{cases}
1 & i \in \set{11, 16} \\
0 & i \notin \set{11, 16}
\end{cases} \\
\beta^{*}(e_i, e_j) &= \delta_{i, 27 - j}
\end{align}

\begin{pr} \label{pr_invol_facts}
In $G = \LargeRee(q)$, there are two conjugacy classes of involutions. The \emph{rank} of an involution is the number of $2$-blocks in the Jordan form.

\begin{tabular}{l|l|l|l|l}
Name & Centraliser & Maximal parabolic & Rank & Representative \\
\hline 
$2A$ & $[q^{10}] {:} \Sz(q)$ & $[q^{10}] {:} (\Sz(q) \times \Cent_{q - 1})$ & $10$ & $\varrho$ \\
$2B$ & $[q^9] {:} \PSL(2, q)$ & $[q^{11}] {:} (\PSL(2, q) \times \Cent_{q - 1})$ & $12$ & $\kappa$  
\end{tabular}

Moreover, in the $2A$ case the centre of the centraliser has order $q$. The cyclic group $\Cent_{q - 1}$ acts fixed-point freely on $[q^{10}]$.

Let $i, j \in G$ be involutions. If $i$ and $j$ are conjugate, then $\abs{ij}$ is odd with probability $1 - \OR{1/q}$. If $i$ and $j$ are not conjugate, then $\abs{ij}$ is even.
\end{pr}
\begin{proof}
The structure of the centralisers can be found in \cite{malle_bigree}. The statement about $i$ and $j$ follows immediately from \cite[Theorem $13$]{parkerwilson06}.
\end{proof}

\begin{cl}[The dihedral trick] \label{bigree_dihedral_trick} 
There exists a Las Vegas algorithm that,
  given $\gen{X} \leqslant \GL(26, q)$ such that $\gen{X} = \LargeRee(q)$
  and given conjugate involutions $a, b \in \gen{X}$, finds $c \in \gen{X}$ such
  that $a^c = b$. If $a, b$ are given as $\SLP$s of length $\OR{n}$, then $c$ will be
  found as an $\SLP$ of length $\OR{n}$. The algorithm has expected time complexity $\OR{\xi}$ field operations.
\end{cl}
\begin{proof}
Follows from Proposition \ref{pr_invol_facts} and
Theorem \ref{dihedral_trick}.

\end{proof}

\begin{pr} \label{bigree_parabolic_gens}
  Let $\lambda$ be a primitive element of $\F_q$. A maximal parabolic
  of type $2A$ is conjugate to $\gen{\zeta, z^{\kappa \varrho
      \kappa}, \varrho, \nu, \varsigma(1, \lambda), \varsigma(\lambda,
    1)}$, which consists of lower block-triangular matrices. A maximal parabolic of type $2B$ is conjugate to
  \[ \gen{\varsigma(1, \lambda), \varsigma(\lambda, 1), \kappa, \nu, \zeta, \zeta^\kappa, \zeta^{\kappa \varrho}, \zeta^{\kappa \varrho \kappa}, \nu^{\varrho}, \nu^{\varrho \kappa}, \nu^{\varrho \kappa \varrho}}.\]
\end{pr}
\begin{proof}
Follows immediately from \cite{elementary_ree}.
\end{proof}

\begin{conj} \label{pr_cent_invol} Let $j \in G = \LargeRee(q)$ be an
  involution of class $2A$ and let $H \leqslant \Cent_G(j)$ satisfy $H
  \geqslant \Zent(\Cent_G(j))$ and $H \geqslant S$ where $S \cong
  \Sz(q)$.

\begin{enumerate}
\item Let $g \in H$ be uniformly random such that $\abs{g} = 2l$. Then $g^l
\in \Zent(\Cent_G(j))$ with probability $1 - \OR{1/q}$.
\item If $H = \Cent_G(j)$ and $g \in H$ is uniformly random such that $\abs{g} = 4l$, then with high probability $g^l \in \O2(H)$ and $g^{2l} \in \Zent(H)$.
\end{enumerate}
\end{conj}

\begin{conj} \label{pr_cent_dihedral_trick}
Let $j_1, j_2 \in G = \LargeRee(q)$ be involutions of class $2A$ such that $j_2 \in \Zent(\Cent_G(j_1))$. Then the proportion of $h \in G$ such that $j_1^h = j_2$, $\abs{h} = q - 1$ and $\gen{\Cent_G(j_1), h} < G$ is high.
\end{conj}

\begin{conj} \label{pr_centraliser_facts} Let $j \in G = \LargeRee(q)$
  be an involution. Let $H \leqslant \Cent_G(j)$ satisfy $H \geqslant
  \Zent(\Cent_G(j))$ and $H \geqslant S$ where $S \cong \Sz(q)$. Let
  $M$ be the natural module of $G$.

\begin{tabular}{l|l}
Class of $j$ & Constituents and multiplicities of $M\vert_H$ \\
\hline
$2A$ & $(S_4, 4)$, $(1, 6)$, $(S_4^{\psi^t}, 1)$ \\
$2B$ & $(R_2 \otimes R_2^{\psi^i}, 2)$, $(R_2^{\psi^j}, 2)$, $(R_2, 1)$, $(R_2^{\psi^k}, 4)$, $(1, 4)$
\end{tabular}

Where $S_4$ is the natural module for $\Sz(q)$, $R_2$ is the
natural module for $\SL(2, q)$ and $0 \leqslant i,j,k \leqslant 2m$.

In the $2A$ case, $M\vert_H$ has submodules
of dimensions 
\[ 26, 25, 21, 20, 17, 16, 15, 11, 10, 9, 6, 5, 1, 0. \] 
It has $q + 1$ submodules of dimensions $10$ and $16$, and the others are
unique.

In the $2B$ case, $M\vert_H$ has submodules of every dimension $0,
\dotsc, 26$. It has $2$ submodules of dimensions $7,9,10,13,16,17,19$,
$q + 1$ submodules of dimensions $11$ and $15$ and $q + 2$ submodules of
dimensions $12$ and $14$. All the others are unique.
\end{conj}

\begin{conj} \label{pr_invol_o2_facts} Let $j \in G = \LargeRee(q)$ be
  an involution of class $2A$, let $P = \O2(\Cent_G(j))$ and let $H$
  be the corresponding maximal subgroup.
\begin{enumerate}
\item $P$ is nilpotent of class $2$ and has exponent $4$.
\item $P$ has a subnormal series 
\[ P = P_0 \triangleright P_1 \triangleright P_2 \triangleright P_3 \triangleright P_4 = 1\]
where $\abs{P_0 / P_1} = \abs{P_2 / P_3} = q^4$, $\abs{P_1 / P_2} =
\abs{P_3 / P_4} = q$. 
\item The series induces a filtration of $H$, so $P_i / P_{i + 1}$ are
  $\F_q H$-modules, for $i = 0, \dotsc, 3$.
\item $\Phi(P_0) = P_0^{\prime} = P_2$, $P_1^{\prime} = P_3$, $P_2$ is elementary abelian and $\Zent(\Cent_G(j)) = \Zent(P) = P_3$.
\item $P$ has exponent-$2$ central series
\[ P = P_0 \triangleright P_2 \triangleright P_4 = 1\]
\item Pre-images of non-identity elements of $P_0 / P_1$ have order $4$ and their squares, which lie in $P_2$ and have non-trivial images in $P_2 / P_3$, are involutions of class $2B$.
\item Pre-images of non-identity elements of $P_1 / P_2$ have order $4$ and their squares, which lie in $P_3$ and have non-trivial images in $P_3 / P_4$, are involutions of class $2A$.
\item As $H$-modules, $P / P_2$ is not isomorphic to $P_2$.
\end{enumerate}
\end{conj}

By \cite{malle_bigree}, $G$ has a
maximal subgroup $S \cong \Sz(q) \wr \Cent_2$, so from Section \ref{section:suzuki_theory} we know that $S$ contains elements of order $4(q - 1)$. 

\begin{conj} \label{pr_sz26_facts} Let $H \leqslant G = \LargeRee(q)$
  be such that $H \cong \Sz(q) \times \Sz(q)$ and let $\Sz(q) \cong S \leqslant H$
  be one of its direct factors. Let $M$ be the module of $S$ and let
  $S_4$ be the natural module of $S$.
\begin{enumerate}
\item $M \cong \left(\oplus_{i = 1}^4 1 + S_4 \right) \oplus (1 . S_4^{\psi^t} . 1)$.
\item $M$ has composition factors with multiplicities: $(S_4, 4), (1, 6), (S_4^{\psi^t}, 1)$.
\end{enumerate}
\end{conj}

\begin{conj} \label{pr_sz_sz_facts1}
Let $S \leqslant G = \LargeRee(q)$ be such that $S \cong \Sz(q) \times \Sz(q)$
and let $M$ be the module of $G$.
\begin{enumerate}
\item The elements of $S$ of order $4(q - 1)$ have $1$ as an eigenvalue of
  multiplicity $6$. The proportion of these elements in $G$ (taken over every $S$) is $1 / (2q)$.
\item $M\vert_S \cong M_{16} \oplus M_{10}$ where $\dim M_i = i$. 
\item $M_{16}$ is absolutely irreducible and $M_{16} \cong S_1 \otimes S_2$. $M_{10}$ has shape $1 . (S_3 \oplus S_4) . 1$. The $S_i$ are natural $\Sz(q)$-modules.
\item $M\vert_S$ has endomorphism ring $\EndR(M\vert_S)
  \cong \F_q^3$ and automorphism group $\Aut(M\vert_S) \cong \Cent_2
  \times \Cent_{q - 1} \times \Cent_{q - 1}$.
\end{enumerate}
\end{conj}

\begin{pr} \label{pr_sz_sz_facts2}
Let $S \leqslant G = \LargeRee(q)$ such that $S \cong \Sz(q) \times \Sz(q)$
and let $M$ be the module of $G$. Then $\Aut(M\vert_S) \cap \mathrm{O}^{-}(26, q) \cong \Cent_2$.
\end{pr}
\begin{proof}
The subgroup of $\mathrm{O}^{-}(26, q)$ that preserves the
direct sum decomposition of $M\vert_S$ has shape $\mathrm{O}^{+}(16,
q) \times \mathrm{O}^{-}(10, q)$ (the $16$-dimensional part is of plus
type since $M_{16}$ is a tensor product of natural Suzuki modules, and
these preserve orthogonal forms of plus type). This implies that
\[ \Aut(M\vert_S) \cap \mathrm{O}^{-}(26, q) = (\Aut(M_{16}) \cap \mathrm{O}^{+}(16, q)) \times (\Aut(M_{10}) \cap \mathrm{O}^{-}(10, q)) \]
Since $M_{16}$ is absolutely irreducible, $\Aut(M_{16})$ consists of
scalars only, but there are no scalars in $\mathrm{O}^{+}(16,
q)$. Hence it suffices to show that $\Aut(M_{10}) \cap
\mathrm{O}^{-}(10, q) \cong \Cent_2$.

With respect to a suitable basis, an endomorphism of $M_{10}$ has the
form $e(\alpha, \beta) = \alpha I_{10} + \beta E_{10, 1}$. It
preserves the bilinear form $\beta^{*}$ restricted to $M_{10}$ if
$\alpha = 1$, and it preserves the quadratic form $Q^{*}$ restricted
to $M_{10}$ if $e(1, \beta)$ is a transvection. This implies that only
$2$ values of $\beta$ are possible.
\end{proof}

\begin{pr} \label{bigree_suzuki2_gens}
  Let $\lambda$ be a primitive element of $\F_q$. The subgroup
  \[\gen{z, \varrho, \varrho^{\kappa}, \varsigma(1, \lambda),
    \varsigma(\lambda, 1)}\] is isomorphic to $\Sz(q) \wr \Cent_2$. It
  contains $S = \gen{z^{\kappa \varrho \kappa}, \varrho, \varsigma(1,
    \lambda)}$ and $h = \kappa \varrho \kappa$, where $S \cong \Sz(q)$
  and $[S^h, S] = \gen{1}$. The subgroup $\gen{\nu, \varsigma(1, \lambda), \varsigma(\lambda, 1), \kappa, \varrho, \varrho \kappa \varrho}$ is isomorphic to $\Sp(4, q) {:} \Cent_2$.
\end{pr}
\begin{proof}
Follows immediately from \cite{elementary_ree}.
\end{proof}


\begin{conj} \label{mult_implies_even_order2} 
Let $a \in G =
  \LargeRee(q)$ be such that $\abs{a} = q - 1$ and $a$ is conjugate to
  some $h(\lambda, \mu)$ with $\lambda^t = \mu \in \F_q^{\times}$. The
  proportion of $b \in G$ such that the elements in
  $\gen{a}b$ with $1$ as an eigenvalue of multiplicity $6$ also have even order and power up to an involution of class $2A$, is bounded below by a constant
  $c_2 > 0$.
\end{conj}

\begin{conj} \label{pr_bigree_goodelts}
  The proportion of elements in $G$ that have $1$ as an eigenvalue of multiplicity $6$ is $c_3 \in \OR{1 / q}$.
\end{conj}

\begin{conj} \label{conj_bigree_badcosets} Let $a \in G =
  \LargeRee(q)$ be such that $\abs{a} = q - 1$ and $a$ is conjugate to
  some $h(\lambda, \mu)$ with $\lambda^t = \mu \in \F_q^{\times}$.
  Then $\Norm_G(\gen{a}) \cong \Dih_{2(q - 1)} \times \Sz(q)$ and
  there exists an absolute constant $c$ such that for every $b \in G \setminus
  \Norm_G(\gen{a})$, the number of $g \in \gen{a}b$ that have $1$ as
  an eigenvalue of multiplicity at least $6$ is bounded above by $c$.
\end{conj}

\begin{pr} \label{pr_reducible_maximals}
Let $G = \LargeRee(q)$ with natural module $M$ and let $H < G$ be a
maximal subgroup. Then either $M\vert_H$ is reducible or $H \cong
\LargeRee(s)$ where $q$ is a proper power of $s$.
\end{pr}
\begin{proof}
Follows from \cite{malle_bigree}.
\end{proof}

\subsection{Tensor indecomposable representations}
\label{section:bigree_indecomposables}

It follows from \cite{MR1901354} that over an algebraically closed
field in defining characteristic, up to Galois twists, there are
precisely three absolutely irreducible tensor indecomposable
representations of $\LargeRee(q)$:
\begin{enumerate}
\item the natural representation $V$ of dimension $26$,
\item a submodule of $S$ of $V \otimes V$ of dimension $246$,
\item a submodule of $V \otimes S$ of dimension $4096$
\end{enumerate}

\chapter{Constructive recognition and membership testing}
\label{chapter:recognition}

Here we will present the algorithms for constructive recognition and
constructive membership testing. The methods we use are
specialised to each family of exceptional groups, so we treat each
family separately. When the methods are similar between the families,
we present a complete account for each family, in order to make each
section self-contained. In the cases where we have a non-constructive
recognition algorithm that improves on \cite{general_recognition}, we
will also present it here.

Recall the various cases of constructive recognition of matrix groups,
given in Section \ref{section:algorithm_overview}. For each group, we
will deal with some, but not always all, of the cases that arise. 

We first give an overview of the various methods. From Chapter
\ref{chapter:group_def_theory} we know that both the Suzuki groups and
the small Ree groups act doubly transitively on $q^2 + 1$ and $q^3 +
1$ projective points, respectively (the fields of size $q$ have different
characteristics), and the idea of how to deal with these groups is to think
of them as permutation groups. In fact we proceed similarly as in
\cite{psl_recognition} for $\PSL(2, q)$, which acts doubly transitively
on $q + 1$ projective points. The essential problem in all these cases
is to find an efficient algorithm that finds an element of the group
that maps one projective point to another.

In $\PSL(2, q)$ the algorithm proceeds by finding a random element of
order $q - 1$ and considering a random coset of the subgroup generated
by this element. Since the coset is exponentially large, we cannot
process every element, and the idea is instead to construct the
required element by solving equations. We therefore consider a matrix
whose entries are indeterminates. In this way we reduce the coset
search problem to two other problems from computational algebra:
finding roots of quadratic equations over a finite field, and the
famous \emph{discrete logarithm problem}.

In the Suzuki groups, the number of points is $q^2 + 1$ instead of $q +
1$, and this requires us to consider double cosets of elements of order
$q - 1$, instead of cosets. The problem is again reduced to finding
roots of univariate polynomials, in this case of
degree at most $60$, as well as to the discrete logarithm problem.

The small Ree groups are dealt with slightly differently, since we can
easily find involutions by random search and then use the Bray
algorithm to find the centraliser of an involution.  Then the module
of the group restricted to the centraliser decomposes, and this is
used to find an element that maps one projective point to another.

The Big Ree groups cannot be considered as permutation groups in the
same way, so the essential problem in this case is to find an
involution expressed as a product of the given generators. Again the
idea is to find a cyclic subgroup of order $q - 1$ and search for elements of
even order in random cosets of this subgroup. The underlying
observation is that the elements of even order are characterised by
having $1$ as an eigenvalue with a certain multiplicity. Therefore we
can again consider matrices whose entries are indeterminates and
construct the required element of even order.

\section{Suzuki groups} \label{section:suzuki_main_results} Here we
will use the notation from Section \ref{section:suzuki_theory}. We
will refer to Conjectures \ref{conjecture_correctness},
\ref{lem1_basevalues}, \ref{lem2_basevalues}, \ref{cl_flatdim} and \ref{conj:sz_perm_iso} simultaneously as the \emph{Suzuki
Conjectures}. We now give an overview of the algorithm for
constructive recognition and constructive membership testing. It will
be formally proved as Theorem \ref{cl_sz_constructive_recognition}.

\begin{enumerate}
\item Given a group $G \cong \Sz(q)$, satisfying the assumptions in
  Section \ref{section:algorithm_overview}, we know from Section
  \ref{section:sz_indecomposables} that the module of $G$ is
  isomorphic to a tensor product of twisted copies of the natural
  module of $G$. Hence we first tensor decompose this module. This is
  described in Section \ref{section:sz_tensor_decompose}.

\item The resulting groups in dimension $4$ are conjugates of the standard copy, so we find a conjugating element. This is described in Section \ref{section:conjugating_element}.

\item Finally we are in $\Sz(q)$. Now we can perform preprocessing for constructive membership testing and other problems we want to solve. This is described in Section \ref{section:sz_constructive_membership}.

\end{enumerate}

Given a discrete logarithm oracle, the whole process has time
complexity slightly worse than $\OR{d^5 + \log(q)}$ field operations,
assuming that $G$ is given by a bounded number of generators.

\subsection{Recognition} \label{section:sz_recognition} We now discuss
how to non-constructively recognise $\Sz(q)$. We are given a group
$\gen{X} \leqslant \GL(4, q)$ and we want to decide whether or not
$\gen{X} = \Sz(q)$, the group defined in \eqref{suzuki_def}.

To do this, it suffices to determine if $X \subseteq \Sz(q)$ and if
$X$ does not generate a proper subgroup, \emph{i.e.} if $X$ is not
contained in a maximal subgroup. To determine if $g \in X$ is in
$\Sz(q)$, first determine if $g$ preserves the symplectic form of
$\Sp(4, q)$ and then determine if $g$ is a fixed point of the
automorphism $\Psi$ of $\Sp(4, q)$, mentioned in Section
\ref{section:suzuki_theory}.

The recognition algorithm relies on the following result.
\begin{lem} \label{lemma_sz_standard_recognition}
Let $H = \gen{X} \leqslant \Sz(q) = G$, where $X = \set{x_1, \dotsc, x_n}$,
let $C = \set{[x_i, x_j] \mid 1 \leqslant i < j \leqslant n}$ and let $M$ be the natural module of $H$. Then $H = G$ if and only if the
following hold:
\begin{enumerate}
\item $M$ is an absolutely irreducible $H$-module.
\item $H$ cannot be written over a proper subfield.
\item $C \neq \set{1}$ and for every $c \in C \setminus \set{1}$ there exists $x \in X$ such that $[c, c^x] \neq 1$.
\end{enumerate}
\end{lem}
\begin{proof}
  By Theorem \ref{sz_maximal_subgroups}, the maximal subgroups of $G$
  that do satisfy the first two conditions are
  $\Norm_G(\mathcal{H})$, $\mathcal{B}_1$ and $\mathcal{B}_2$. For
  each, the derived group is contained in the normalised
  cyclic group, so all these maximal subgroups are metabelian. If $H$
  is contained in one of them and $H$ is not abelian, then $C \neq
  \set{1}$, but $[c, c^x] = 1$ for every $c \in C$ and $x \in X$ since
  the second derived group of $H$ is trivial. Hence the last
  condition is not satisfied.

Conversely, assume that $H = G$. Then clearly, the first two
conditions are satisfied, and $C \neq \set{1}$. Assume that the last
condition is false, so for some $c \in C \setminus \set{1}$ we have that
$[c, c^x] = 1$ for every $x \in X$. This implies that $c^x \in \Cent_G(c) \cap
\Cent_G(c)^{x^{-1}}$, and it follows from Proposition \ref{lemma_double_coset}
that $\Cent_G(c) = \Cent_G(c)^{x^{-1}}$. Thus $\Cent_G(c) = \Cent_G(c)^g$ for all $g
\in G$, so $\Cent_G(c) \triangleleft G$, but $G$ is simple and we have a
contradiction.
\end{proof}

\begin{thm} \label{thm_sz_standard_recognition}
  There exists a Las Vegas algorithm that, given $\gen{X} \leqslant \GL(4, q)$, decides whether or not $\gen{X} = \Sz(q)$. It has expected time complexity $\OR{\abs{X}^2 + \sigma_0(\log(q))(\abs{X} + \log(q))}$
  field operations.
\end{thm}
\begin{proof}
The algorithm proceeds as follows.
\begin{enumerate}
\item Determine if every $x \in X$ is in $\Sz(q)$, and return \texttt{false} if not.

\item Determine if $\gen{X}$ is absolutely irreducible, and return \texttt{false} if not. 

\item Determine if $\gen{X}$ can be written over a smaller field. If so, return \texttt{false}.

\item Using the notation of Lemma \ref{lemma_sz_standard_recognition}, try to find $c \in C$ such that $c \neq 1$. Return \texttt{false} if it cannot be found.

\item If such $c$ can be found, and if $[c, c^x] \neq 1$ for some $x \in X$, then return \texttt{true}, else return \texttt{false}.

\end{enumerate}

From the discussion at the beginning of this section, the first step
is easily done using $\OR{\abs{X}}$ field operations. The MeatAxe can be used to determine if the
natural module is absolutely irreducible; the algorithm described in Section \ref{section:smallerfield} can be used to determine if $\gen{X}$ can be written over a smaller field.

  The rest of the algorithm is a straightforward application of the last
  condition in Lemma \ref{lemma_sz_standard_recognition}, except that it
  is sufficient to use the condition for one non-trivial commutator
  $c$. By Lemma
  \ref{lemma_sz_standard_recognition}, if $[c, c^x] \neq 1$ then $\gen{X}
  = \Sz(q)$; but if $[c, c^x] = 1$, then $\Cent_{\gen{X}}(c) \triangleleft
  \gen{X}$ and we cannot have $\Sz(q)$.

  It follows from Section \ref{section:aschbacher_algorithms} that
  the expected time complexity of the algorithm is as stated. Since
  the MeatAxe is Las Vegas, this algorithm is also Las Vegas.
\end{proof}

We are also interested in determining if a given group is a
\emph{conjugate} of $\Sz(q)$, without necessarily finding a
conjugating element. We consider the subgroups of $\Sp(4,
q)$ and rule out all except those isomorphic to $\Sz(q)$. This relies
on the fact that, up to Galois automorphisms, $\Sz(q)$ has only one
equivalence class of faithful representations in $\GL(4,q)$, so if we
can show that $G \cong \Sz(q)$ then $G$ is a conjugate of $\Sz(q)$.

\begin{thm} \label{thm_sz_conj_recognition}
  There exists a Las Vegas algorithm that, given $\gen{X} \leqslant \GL(4,
  q)$, decides whether or not there exists $h \in \GL(4, q)$ such that $\gen{X}^h = \Sz(q)$. The algorithm has expected time complexity $\OR{\abs{X}^2 + \sigma_0(\log(q)) (\abs{X} + \log(q))}$ field
  operations.
\end{thm}
\begin{proof}
  Let $G = \gen{X}$. The algorithm proceeds as follows.

\begin{enumerate}
\item Determine if $G$ is absolutely irreducible, and return \texttt{false} if not.

\item Determine if $G$ preserves a non-degenerate symplectic form $M$. If so we
  conclude that $G$ is a subgroup of a conjugate of $\Sp(4, q)$, and if not then return \texttt{false}. Since $G$ is
  absolutely irreducible, the form is unique up to a scalar multiple.

\item Conjugate $G$ so that it preserves the form $J$. This amounts to finding a symplectic basis, \emph{i.e.} finding an
invertible matrix $X$ such that $X J X^{T} = M$, which is easily done.
Then $G^X$ preserves the form $J$ and thus $G^X \leqslant \Sp(4, q)$ so
that we can apply $\Psi$.

\item Determine if $V \cong V^{\Psi}$, where $V$ is the natural module for $G$ and $\Psi$ is the automorphism from Lemma \ref{lem_steinberg_lang}. If so we conclude that $G$ is a subgroup of some conjugate of $\Sz(q)$, and if not then return \texttt{false}.

\item Determine if $G$ is a proper subgroup of $\Sz(q)$, \emph{i.e.} if it is contained in a maximal subgroup. This can
be done using Lemma \ref{lemma_sz_standard_recognition}. If so, then return \texttt{false}, else return \texttt{true}.
\end{enumerate}

From the descriptions in Section \ref{section:meataxe},
the algorithms for finding a preserved form and for module isomorphism
testing are Las Vegas, with the same expected time complexity as the
MeatAxe. Hence we obtain a Las Vegas algorithm, with the same expected
time complexity as the algorithm from Theorem
\ref{thm_sz_standard_recognition}.
\end{proof}

\subsection{Finding an element of a stabiliser} \label{section:sz_stabiliser_elements} 
In this section the matrix degree is constant, so we set $\xi = \xi(4)$. 
In constructive membership testing for $\Sz(q)$ the essential problem
is to find an element of the stabiliser of a given point $P \in \OV$,
expressed as an $\SLP$ in our given generators $X$ of $G = \Sz(q)$.
The idea is to map $P$ to $Q \neq P$ by a random $g_1 \in G$, and
then compute $g_2 \in G$ such that $Pg_2 = Q$, so that $g_1
g_2^{-1} \in G_P$.

Thus the problem is to find an element that maps $P$ to $Q$, and the
idea is to search for it in double cosets of cyclic subgroups of order
$q - 1$. We first give an overview of the method.

Begin by selecting random $a, h \in G$ such that
$a$ has pseudo-order $q - 1$, and consider the equation
\begin{equation} \label{stab_eqn1}
P a^j h a^i = Q
\end{equation}
in the two indeterminates $i,j$. If we can solve this equation for $i$ and
$j$, thus obtaining integers $k, l$ such that $1 \leqslant k, l \leqslant q - 1$ and $P a^l h a^k = Q$, then we have an
element that maps $P$ to $Q$.

Since $a$ has order dividing $q - 1$, by Proposition \ref{pr_sz_subgroups_conj}, $a$ is conjugate to a matrix $M^{\prime}(\lambda)$ for some
$\lambda \in \F_q^{\times}$. This implies that we can diagonalise $a$
and obtain a matrix $x \in \GL(4, q)$ such that $M^{\prime}(\lambda)^x
= a$. It follows that if we define $P^{\prime} = P x^{-1}$,
$Q^{\prime} = Q x^{-1}$ and $g = h^{x^{-1}}$ then \eqref{stab_eqn1} is
equivalent to
\begin{equation} \label{stab_eqn2}
P^{\prime} M^{\prime}(\lambda)^j g M^{\prime}(\lambda)^i = Q^{\prime}.
\end{equation}

Now change indeterminates to $\alpha$ and $\beta$ by letting $\alpha = \lambda^j$ and $\beta = \lambda^i$, so that we obtain the following equation:
\begin{equation} \label{stab_eqn3}
P^{\prime} M^{\prime}(\alpha) g M^{\prime}(\beta) = Q^{\prime}.
\end{equation}
This determines four equations in $\alpha$ and $\beta$, and in Section
\ref{solve_magic_poly} we will describe how to find solutions for
them. A solution $(\gamma, \delta) \in \F_q^{\times} \times \F_q^{\times}$ determines 
$M^{\prime}(\gamma), M^{\prime}(\delta) \in \mathcal{H}$, and hence
also $c, d \in H = \mathcal{H}^x$. 

If $\abs{a} = q - 1$ then $\gen{a} = H$, so that there exist integers $k$ and $l$ as above with $a^l = c$ and $a^k = d$. These
integers can be found by computing discrete logarithms, since we also
have $\lambda^l = \gamma$ and $\lambda^k = \delta$. Hence we obtain a
solution to \eqref{stab_eqn1} from the solution to \eqref{stab_eqn3}.
If $\abs{a}$ is a proper divisor of $q-1$, then it might happen that
$c \notin \gen{a}$ or $d \notin \gen{a}$, but by Proposition
\ref{sz_totient_prop} we know that this is unlikely.

Thus the overall algorithm is as in Algorithm \ref{alg:stab_element}.
We prove that the algorithm is correct in Section
\ref{section:trick_correctness}.

\begin{figure}[hb]
\begin{codebox} 
\refstepcounter{algorithm}
\label{alg:stab_element}
\Procname{\kw{Algorithm} \ref{alg:stab_element}: $\proc{FindMappingElement}(X, P, Q)$}
\li \kw{Input:} Generating set $X$ for $G = \Sz(q)$ and points $P \neq Q \in \OV$.
\li \kw{Output:} A random element $g$ of $G$, as an $\SLP$ in $X$, such that $Pg = Q$.
\zi \Comment{Assumes the existence of a function \textsc{SolveEquation} that solves \eqref{stab_eqn3}.}
\zi \Comment{Also assumes that \textsc{DiscreteLog} returns a positive integer if a}
\zi \Comment{discrete logarithm exists, and $0$ otherwise.}
\zi \Repeat
\li \Comment{Find random element $a$ of pseudo-order $q - 1$}
\zi \Repeat
\li  $a := \proc{Random}(G)$
\li \Until{$\abs{a} \mid q - 1$}
\li $(M^{\prime}(\lambda), x) := \proc{Diagonalise}(a)$ 
\li \Comment{Now $M^{\prime}(\lambda)^x = a$}
\zi \Repeat
\li $h := \proc{Random}(G)$ 
\li $\id{flag} := \proc{SolveEquation}(h^{x^{-1}}, P x^{-1}, Q x^{-1})$ \label{alg:stab_element_solve_poly}
\li \Until{$\id{flag}$} \label{alg:stab_element_solve_poly_test}
\li Let $(\gamma, \delta)$ be a solution to \eqref{stab_eqn3}.
\li $l := \proc{DiscreteLog}(\lambda, \gamma)$ 
\li $k := \proc{DiscreteLog}(\lambda, \delta)$ 
\li \Until{$k > 0 \kw{ and } l > 0$} \label{alg:stab_element_discrete_log}
\li \Return{$a^l  h  a^k$} \label{alg:stab_element_final_element} 
\end{codebox}
\end{figure}

\subsubsection{Solving equation \eqref{stab_eqn3}} \label{solve_magic_poly} We
will now show how to obtain the solutions of \eqref{stab_eqn3}. It
might happen that there are no solutions, in which case the method
described here will detect this and return with failure.

By letting $P^{\prime} = (q_1 : q_2 :
q_3 : q_4)$, $Q^{\prime} = (r_1 : r_2 : r_3 : r_4)$ and $g = [g_{i,j}]$, we can write out \eqref{stab_eqn3}
and obtain
\begin{equation} \label{alpha_beta_eqns}
\begin{split}
(q_1 g_{1,1} \alpha^{t + 1} + q_2 g_{2,1} \alpha + q_3 g_{3,1} \alpha^{-1} + q_4 g_{4,1} \alpha^{-t-1}) \beta^{t + 1} &= C r_1 \\
(q_1 g_{1,2} \alpha^{t + 1} + q_2 g_{2,2} \alpha + q_3 g_{3,2} \alpha^{-1} + q_4 g_{4,2} \alpha^{-t-1}) \beta &= C r_2 \\
(q_1 g_{1,3} \alpha^{t + 1} + q_2 g_{2,3} \alpha + q_3 g_{3,3} \alpha^{-1} + q_4 g_{4,3} \alpha^{-t-1}) \beta^{-1} &= C r_3 \\
(q_1 g_{1,4} \alpha^{t + 1} + q_2 g_{2,4} \alpha + q_3 g_{3,4} \alpha^{-1} + q_4 g_{4,4} \alpha^{-t-1}) \beta^{-t - 1} &= C r_4
\end{split}
\end{equation}
for some constant $C \in \F_q$. Henceforth, we assume that $r_i \neq
0$ for $i = 1, \dotsc, 4$, since this is the difficult case, and also
very likely when $q$ is large, as can be seen from Proposition
\ref{prop_trick_general_case}. A method similar to the one described
in this section will solve \eqref{stab_eqn3} when some $r_i = 0$ and
Algorithm \ref{alg:stab_element} does not assume that all $r_i
\neq 0$.

For convenience, we denote the expressions in the parentheses at
the left hand sides of \eqref{alpha_beta_eqns} as $K, L, M$ and $N$
respectively. Since $C = L \beta r_2^{-1}$ we obtain three equations
\begin{equation} \label{alpha_beta_eqn1}
\begin{split} 
K \beta^t &= r_1 r_2^{-1} L \\
M \beta^{-2} &= r_3 r_2^{-1} L \\
N \beta^{-t-2} &= r_4 r_2^{-1} L 
\end{split}
\end{equation}
and in particular $\beta$ is a function of $\alpha$, since
\begin{equation} \label{beta_from_alpha}
\beta = \sqrt{L^{-1} M r_3^{-1} r_2}.
\end{equation}
If we eliminate $\beta$ and $\beta^t$ by using the first two equations into the third in \eqref{alpha_beta_eqn1}, we obtain
\begin{equation} \label{alpha_eqn1}
NK r_2 r_3 = r_1 r_4 ML
\end{equation}
and by raising the first equation to the $t$-th power and substituting into the second, we obtain
\begin{equation} \label{alpha_eqn2}
r_1 r_3^{t / 2} L^{1 + t/2} = r_2^{1 + t/2} M^{t/2} K.
\end{equation}
Also, $C = M \beta^{-1} r_3^{-1}$ and if we proceed similarly, we obtain two more equations
\begin{align}
N^t L r_3^{t + 1} &= M^{t + 1} r_2 r_4^t \label{alpha_eqn3} \\
N L^{t/2} r_3^{1 + t/2} &= M^{1 + t/2} r_4 r_2^{t/2}. \label{alpha_eqn4}
\end{align}
Now \eqref{alpha_eqn1}, \eqref{alpha_eqn2},
\eqref{alpha_eqn3} and \eqref{alpha_eqn4} are equations in $\alpha$
only, and by multiplying them by suitable powers of $\alpha$, they can
be turned into polynomial equations such that $\alpha$ only occurs to
the powers $ti$ for $i = 1, \dotsc, 4$ and to lower powers that are
independent of $t$. The suitable powers of $\alpha$ are $2t + 2$, $t +
t/2 + 2$, $2t + 3$ and $2t + t/2 + 2$, respectively.

Thus we obtain the following four equations.
\begin{equation} \label{alpha_poly_eqns}
\begin{split}
\alpha^{4t} c_1 + \alpha^{3t} c_2 + \alpha^{2t} c_3 + \alpha^t c_4 & = d_1 \\
\alpha^{4t} c_5 + \alpha^{3t} c_6 + \alpha^{2t} c_7 + \alpha^t c_8 & = d_2 \\
\alpha^{4t} c_9 + \alpha^{3t} c_{10} + \alpha^{2t} c_{11} + \alpha^t c_{12} & = d_3 \\
\alpha^{4t} c_{13} + \alpha^{3t} c_{14} + \alpha^{2t} c_{15} + \alpha^t c_{16} & = d_4 \\
\end{split}
\end{equation}
The $c_i$ and $d_j$ are polynomials in $\alpha$ with degree
independent of $t$, for $i = 1, \dotsc, 16$ and $j = 1, \dotsc, 4$
respectively, so \eqref{alpha_poly_eqns} can be considered a linear
system in the variables $\alpha^{nt}$ for $n = 1,\dotsc,4$, with
coefficients $c_i$ and $d_j$.  Now the aim is to obtain a single
polynomial in $\alpha$ of bounded degree. For this we need the
following conjecture.

\begin{conj} \label{conjecture_correctness} For every $P^{\prime} = P
  x^{-1}, Q^{\prime} = Q x^{-1}, g = h^{x^{-1}}$ where $P, Q \in \OV$,
  $h \in G$ and $x \in \GL(4, q)$, if we regard
  \eqref{alpha_poly_eqns} as simultaneous linear equations in the
  variables $\alpha^{nt}$ for $n = 1,\dotsc,4$, over the polynomial
  ring $\F_q[\alpha]$, then it has non-zero determinant.
\end{conj}

In other words, the determinant of the coefficients $c_i$ is not the
zero polynomial. 

\begin{lem} \label{lemma_magic_poly} Assume Conjecture
  \ref{conjecture_correctness}. Given $P^{\prime}, Q^{\prime}$ and $g$ as in Conjecture
  \ref{conjecture_correctness}, there exists a univariate polynomial
  $f(\alpha) \in \F_q[\alpha]$ of degree at most $60$, such that for every $(\gamma,
  \delta) \in \F_q^{\times} \times \F_q^{\times}$ that is a solution for $(\alpha, \beta)$ in \eqref{stab_eqn3}, we have $f(\gamma) = 0$.
\end{lem}
\begin{proof}
  So far in this section we have shown that if we can solve \eqref{alpha_poly_eqns} we can also solve \eqref{stab_eqn3}. From the four equations of \eqref{alpha_poly_eqns} we can
  eliminate $\alpha^t$. We can solve for $\alpha^{4t}$ from the fourth
  equation, and substitute into the third,
  thus obtaining a rational expression with no occurrence of
  $\alpha^{4t}$. Continuing this way and substituting into the other
  equations, we obtain an expression for $\alpha^t$ in terms of the
  $c_i$ and the $d_i$ only. This can be substituted into any of the
  equations of \eqref{alpha_poly_eqns}, where $\alpha^{nt}$ for $n =
  1, \dotsc, 4$ is obtained by powering up the expression for
  $\alpha^t$. Thus we obtain a rational expression $f_1(\alpha)$ of degree
  independent of $t$. We now take $f(\alpha)$ to be the numerator of
  $f_1$.

  In other words, we think of the $\alpha^{nt}$ as independent
  variables and of \eqref{alpha_poly_eqns} as a linear system in
  these variables, with coefficients in $\F_q[\alpha]$. By Conjecture
  \ref{conjecture_correctness} we can solve this linear system.
  
  Two possible problems can occur: $f$ is identically zero or some of
  the denominators of the expressions for $\alpha^{nt}$, $n = 1,
  \dotsc, 4$ turn out to be $0$. However, Conjecture
  \ref{conjecture_correctness} rules out these possibilities. By
  Cramer's rule, the expression for $\alpha^t$ is a rational
  expression where the numerator is a determinant, so it consists of
  sums of products of $c_i$ and $d_j$. Each product consists of three
  $c_i$ and one $d_j$. By considering the calculations leading up to
  \eqref{alpha_poly_eqns}, it is clear that each of the products has
  degree at most $15$. Therefore the expression for $\alpha^{4t}$ and
  hence also $f(\alpha)$ has degree at most $60$.

We have only done elementary algebra to obtain $f(\alpha)$
from \eqref{alpha_poly_eqns}, and it is clear that
\eqref{alpha_poly_eqns} was obtained from \eqref{alpha_beta_eqns} by
elementary means only. Hence all solutions $(\gamma, \delta)$ to
\eqref{alpha_beta_eqns} must also satisfy $f(\gamma) = 0$, although
there may not be any such solutions, and $f(\alpha)$ may also have
other zeros.
\end{proof}

\begin{cl} \label{lemma_trick_time} Assume Conjecture
  \ref{conjecture_correctness}. There exists a Las Vegas algorithm
  that, given $P^{\prime}, Q^{\prime}$ and $g$ as in Conjecture
  \ref{conjecture_correctness}, finds all $(\gamma, \delta) \in
  \F_q^{\times} \times \F_q^{\times}$ that are solutions of
  \eqref{stab_eqn3}. The algorithm has expected time complexity $\OR{\log{q}}$
  field operations.
\end{cl}
\begin{proof}
Let $f(\alpha)$ be the polynomial constructed in Lemma
\ref{lemma_magic_poly}.  To find all solutions to \eqref{stab_eqn3},
we find the zeros $\gamma$ of $f(\alpha)$, compute the
corresponding $\delta$ for each zero $\gamma$ using
\eqref{beta_from_alpha}, and check which pairs $(\gamma, \delta)$
satisfy \eqref{alpha_beta_eqns}. These pairs must be all solutions of
\eqref{stab_eqn3}.

The only work needed is simple matrix arithmetic and finding the roots of
a polynomial of bounded degree over $\F_q$. Hence the expected time complexity is
$\OR{\log{q}}$ field operations. The algorithm is Las Vegas, since
by Theorem \ref{thm_solve_univariate_polys} the algorithm for finding
the roots of $f(\alpha)$ is Las Vegas, with this expected time complexity.
\end{proof}

By following the procedure outlined in Lemma \ref{lemma_magic_poly}, it is straightforward to obtain an expression for $f(\alpha)$, where
the coefficients are expressions in the entries of $g$, $P^{\prime}$
and $Q^{\prime}$, but we will not display it here, since it would take up too much space.

\subsubsection{Correctness and complexity}
\label{section:trick_correctness} There are
two issues when considering the correctness of Algorithm
\ref{alg:stab_element}. Using the notation in the algorithm, we have
to show that \eqref{stab_eqn3} has a solution with high probability,
and that the integers $k$ and $l$ are positive with high probability.

The algorithm in Corollary \ref{lemma_trick_time} tries to find an
element in the double coset $\mathcal{H} g \mathcal{H}$, where $g =
h^{x^{-1}}$, and we will see that this succeeds with high probability
when $g \notin \Norm_G(\mathcal{H})$, which is very likely.

If the element $a$ has order precisely $q - 1$, then from the
discussion at the beginning of Section
\ref{section:sz_stabiliser_elements}, we know that the integers $k$ and
$l$ will be positive. By Proposition \ref{sz_totient_prop} we know that it is likely that $a$ has order precisely $q - 1$ rather than just a divisor
of $q - 1$.



\begin{lem} \label{lemma_trick_correctness}
Assume Conjecture \ref{conjecture_correctness}. Let $G = \Sz(q)$ and let $P \in \OV$ and $a, h \in G$ be given, such that $\abs{a} = q - 1$. Let $Q \in \OV$ be uniformly random. If $h \notin \Norm_G(\gen{a})$, then
\begin{equation}
\frac{(q - 1)^2}{(q^2 + 1) \deg{f}}  \leqslant \Pr{Q \in P\gen{a} h \gen{a}} \leqslant \frac{(q - 1)^2}{q^2 + 1}
\end{equation}
where $f(\alpha)$ is the polynomial constructed in Lemma \ref{lemma_magic_poly}. If instead $h \in \Norm_G(\gen{a})$ then
\begin{equation}
\Pr{Q \in P\gen{a} h \gen{a}} = \frac{(q - 1) (q^2 - 1) + 2}{(q^2 + 1)^2}.
\end{equation}
\end{lem}
\begin{proof}
If $h \notin \Norm_G(\gen{a})$ then by Lemma \ref{lemma_double_coset}, $\abs{\gen{a} h \gen{a}} = (q -
1)^2$, and hence $\abs{ P\gen{a} h
\gen{a}} \leqslant (q - 1)^2$. 

On the other hand, for every $Q \in \OV$ we have 
\begin{equation} 
\abs{\set{(k_1, k_2) \mid k_1, k_2 \in \gen{a}, \, P k_1 h k_2 = Q}} \leqslant \deg{f}
\end{equation}
since this is the equation we consider in Section
\ref{solve_magic_poly}, and from Lemma \ref{lemma_magic_poly} we know
that all solutions must be roots of $f$. Thus $\abs{ P\gen{a} h
  \gen{a}} \geqslant \abs{\gen{a} h \gen{a}} / \deg{f}$. Since $Q$ is
uniformly random from $\OV$, and $\abs{\OV} = q^2 + 1$, the
result follows.

If $h \in \Norm_G(\gen{a})$ then $\gen{a} h
\gen{a} = h\gen{a}$, and $\abs{P h\gen{a}} = \abs{\gen{a}}$ if $\gen{a}$
does not fix $Ph$. By Proposition \ref{pr_sz_subgroups_conj}, the number
of cyclic subgroups of order $q - 1$ is $\binom{\abs{\OV}}{2}$ and
$\abs{\OV} - 1$ such subgroups fix $Ph$. Moreover, if $\gen{a}$ fixes $Ph$
then $Ph \gen{a} = \set{Ph}$. Thus
\begin{multline}
\Pr{Q \in P \gen{a} h \gen{a}} = \Pr{Q \in P h \gen{a}} \Pr{P ha \neq Ph} + \\
+ \Pr{Q = Ph} \Pr{Ph a = Ph} = \frac{\abs{P h\gen{a}}}{\abs{\OV}} \left(1 - \frac{\abs{\OV} - 1}{\binom{\abs{\OV}}{2}}\right) + \frac{1}{\abs{\OV}} \frac{\abs{\OV} - 1}{\binom{\abs{\OV}}{2}} 
\end{multline}
and the result follows.
\end{proof}

\begin{thm} \label{trick_alg} Assume Conjecture
  \ref{conjecture_correctness} and an
  oracle for the discrete logarithm problem in $\F_q$. Algorithm
  \ref{alg:stab_element} is a Las Vegas algorithm with expected time
  complexity $\OR{(\xi + \chi_D(q) + \log(q) \log\log(q))\log\log(q)}$ field operations. The length of the returned $\SLP$ is $\OR{\log \log (q)}$.
\end{thm}
\begin{proof}
  We use the notation from the algorithm. Let $g = h^{x^{-1}}$, $H =
  \mathcal{H}^x$, $P^{\prime} = P x^{-1}$ and $Q^{\prime} = Q x^{-1}$.
  Corollary \ref{lemma_trick_time} implies that line
  \ref{alg:stab_element_solve_poly_test} will succeed if $Q^{\prime} \in
  P^{\prime} \mathcal{H} g \mathcal{H}$. If $\abs{a} = q - 1$, then $H
  = \gen{a}$, and the previous condition is equivalent to $Q \in P
  \gen{a} h \gen{a}$. Moreover, if $\abs{a} = q - 1$ then line
  \ref{alg:stab_element_discrete_log} will always succeed. 

Let $s$ be the probability that the return statement is reached. Then $s$ satisfies the following inequality.
\begin{equation}
\begin{split}
s &\geqslant \Pr{\abs{a} = q - 1} (\Pr{h \in \Norm_G(\gen{a})} \Pr{Q \in P \gen{a} h \gen{a} \mid h \in \Norm_G(\gen{a})} + \\
&+\Pr{h \notin \Norm_G(\gen{a})} \Pr{Q \in P \gen{a} h \gen{a} \mid h \notin \Norm_G(\gen{a})}).
\end{split}
\end{equation}
Since $h$ is uniformly random, using Theorem \ref{sz_maximal_subgroups} we obtain
\begin{equation}
\Pr{h \in \Norm_G(\gen{a})} = \frac{2 (q - 1)}{\abs{G}} = \frac{2}{q^2 (q^2 + 1)}.
\end{equation}
From Proposition \ref{sz_totient_prop} and Lemma \ref{lemma_trick_correctness} we obtain
\begin{multline}
s \geqslant \frac{\phi(q - 1)}{2 (q - 1)} \\
\left[ \frac{(q - 1)^2}{(q^2 + 1) \deg{f}} - \frac{2}{q^2 (q^2 + 1)}\frac{(q - 1)^2}{(q^2 + 1)} + \frac{2}{q^2 (q^2 + 1)} \frac{2 + (q - 1) (q^2 - 1)}{(q^2 + 1)^2} \right] = \\
= \frac{\phi(q - 1)}{2 (q - 1) \deg{f}} + \OR{1/q}.
\end{multline}

By Proposition \ref{sz_totient_prop}, the expected number of
iterations of the outer repeat statement is $\OR{\log \log(q)}$. The expected number of random selections to find $h$ is $\OR{1}$. By
Theorem \ref{thm_solve_univariate_polys}, diagonalising a matrix uses
expected $\OR{\log q}$ field operations, since it involves finding the
eigenvalues, \emph{i.e.} finding the roots of a polynomial of constant
degree over $\F_q$. Clearly, the expected time complexity for finding
$a$ is $\OR{\xi + \log (q)\log{\log{(q)}}}$ field operations.

From Corollary \ref{lemma_trick_time}, it follows that line \ref{alg:stab_element_solve_poly} uses $\OR{\log q}$ field
operations. We conclude that Algorithm
\ref{alg:stab_element} uses expected $\OR{(\xi + \chi_D(q) + \log (q)\log{\log{(q)}}) \log{\log{(q)}}}$ field operations.

Each call to Algorithm \ref{alg:stab_element} uses independent random
elements, so the double cosets
under consideration are uniformly random and independent. Therefore
the elements returned by Algorithm \ref{alg:stab_element} must be
uniformly random. The returned $\SLP$ is composed of $\SLP$s of $a$ and $h$, both of which have length $\OR{\log\log(q)}$ because of the number of iterations of the loops. 
\end{proof}

\begin{cl} \label{thm_stab_elt} Assume Conjecture
  \ref{conjecture_correctness} and an oracle for the discrete logarithm problem in $\F_q$. There exists a Las Vegas algorithm that, given $\gen{X} \leqslant \GL(4, q)$ such that $G = \gen{X} = \Sz(q)$
  and $P \in \OV$, finds a uniformly random $g \in G_P$, expressed as an $\SLP$ in $X$. The algorithm has expected time complexity $\OR{(\xi + \chi_D(q) + \log(q) \log\log(q))\log\log(q)}$ field
  operations. The length of the returned $\SLP$ is $\OR{\log\log(q)}$.
\end{cl}
\begin{proof}
We compute $g$ as follows.
\begin{enumerate}
\item Find random $x \in G$. Let $Q = Px$ and repeat until $P \neq Q$.
\item Use Algorithm \ref{alg:stab_element} to find $y \in G$ such that $Qy = P$.
\item Now $g = x y \in G_P$.
\end{enumerate}
We see from Algorithm \ref{alg:stab_element} that the choice of $y$
does not depend on $x$. Hence $g$ is uniformly random, since $x$ is
uniformly random. Therefore this is a Las Vegas algorithm. The
probability that $P = Q$ is $1/\abs{\OV}$, so the dominating term in
the complexity is the call to Algorithm \ref{alg:stab_element}, with
expected time complexity given by Theorem \ref{trick_alg}.

The element $g$ will be expressed as an $\SLP$ in $X$, since $x$ is
random and elements from Algorithm \ref{alg:stab_element} are
expressed as $\SLP$s. Clearly the length of the $\SLP$ is the same as the length of the $\SLP$s from Algorithm \ref{alg:stab_element}.
\end{proof}

\begin{rem} \label{sz_rem_stab_elt}
In fact, the algorithm of Corollary \ref{thm_stab_elt} works in any conjugate of $\Sz(q)$, since in Algorithm \ref{alg:stab_element} the diagonalisation always moves into the standard copy.
\end{rem}

\subsection{Constructive membership testing}
\label{section:sz_constructive_membership}

We will now give an algorithm for constructive membership testing in
$\Sz(q)$. Given a set of generators $X$, such that $G = \gen{X} =
\Sz(q)$, and given $g \in G$, we want to express $g$ as an $\SLP$ in
$X$. The matrix degree is constant here, so we set $\xi = \xi(4)$.
Membership testing is straightforward, using the first steps from the
algorithm in Theorem \ref{thm_sz_standard_recognition}, and will not
be considered here.

\subsubsection{Preprocessing}
The algorithm for constructive membership testing has a preprocessing
step and a main step. The preprocessing step consists of finding
\lq \lq standard generators'' for $\O2(G_{P_{\infty}}) = \mathcal{F}$ and
$\O2(G_{P_0})$. In the case of $\O2(G_{P_{\infty}})$ the standard generators are defined as matrices $\set{S(a_i, x_i)}_{i = 1}^{n} \cup \set{S(0, b_i)}_{i =
1}^{n}$ for some unspecified $x_i \in \F_q$, such that $\set{a_1,
\dotsc, a_n}$ and $\set{b_1, \dotsc, b_n}$ form vector space bases of
$\F_q$ over $\F_2$ (so $n = \log_2{q} = 2m + 1$).

\begin{lem} \label{sz_row_operations}
There exist algorithms for the following row reductions.
\begin{enumerate}
\item Given $g = M^{\prime}(\lambda) S(a, b) \in G_{P_{\infty}}$, find $h \in \O2(G_{P_{\infty}})$ expressed in the standard generators, such that $gh = M^{\prime}(\lambda)$.
\item Given $g = S(a, b) M^{\prime}(\lambda) \in G_{P_{\infty}}$, find $h \in \O2(G_{P_{\infty}})$ expressed in the standard generators, such that $hg = M^{\prime}(\lambda)$.
\item Given $P_{\infty} \neq P \in \OV$, find $g \in \O2(G_{P_{\infty}})$ expressed in the standard generators, such that $Pg = P_0$.
\end{enumerate}
Analogous algorithms exist for $G_{P_0}$. If the standard generators are expressed as $\SLP$s of length $\OR{n}$, the elements returned will have length $\OR{n \log(q)}$. The time complexity of the algorithms is $\OR{\log(q)^3}$ field operations. 
\end{lem}
\begin{proof} 
The algorithms are as follows.

\begin{enumerate}
\item \begin{enumerate}
\item Solve a linear system of size $\log(q)$ to construct the linear
  combination $a = g_{2, 1} / g_{2, 2} = \sum_{i = 1}^{2m + 1}
  \alpha_i a_i$ with $\alpha_i \in \F_2$. Let $h^{\prime} = \prod_{i =
    1}^{2m + 1} S(a_i, x_i)^{\alpha_i}$ and $g^{\prime} = g
  h^{\prime}$, so that $g^{\prime} = M^{\prime}(\lambda) S(0,
  b^{\prime})$ for some $b^{\prime} \in \F_q$.
\item Solve a linear system of size $\log(q)$ to construct the linear
  combination $b^{\prime} = g^{\prime}_{3, 1} / g^{\prime}_{3, 3} =
  \sum_{i = 1}^{2m + 1} \beta_i b_i$ with $\beta_i \in \F_2$.  Let
  $h^{\prime \prime} = \prod_{i = 1}^{2m + 1} S(0, b_i)^{\beta_i}$ and
  $g^{\prime \prime} = g^{\prime} h^{\prime\prime}$, so that
  $g^{\prime\prime} = M^{\prime}(\lambda)$.
\item Now $h = h^{\prime} h^{\prime\prime}$.
\end{enumerate}
\item Analogous to the previous case.
\item 
\begin{enumerate}
\item Normalise $P$ so that $P = (ab + a^{t + 2} + b^t : b : a : 1)$ for some $a,b \in \F_q$. Solve a linear
  system of size $\log(q)$ to construct the linear combination $a =
  \sum_{i = 1}^{2m + 1} \alpha_i a_i$ with $\alpha_i \in \F_2$.
  Let $h^{\prime} = \prod_{i = 1}^{2m + 1} S(a_i, x_i)^{\alpha_i}$.
\item Solve a linear system of size $\log(q)$ to construct the linear
  combination $b + h^{\prime}_{3, 1} = 
  \sum_{i = 1}^{2m + 1} \beta_i b_i$ with $\beta_i \in \F_2$. Let
  $h^{\prime \prime} = \prod_{i = 1}^{2m + 1} S(0, b_i)^{\beta_i}$
\item Now $g = h^{\prime} h^{\prime\prime}$ maps $P$ to $P_0$.
\end{enumerate}
\end{enumerate}
Clearly the dominating term in the time complexity is the solving of
the linear systems, which requires $\OR{\log(q)^3}$ field operations.
The elements returned are constructed using $\OR{\log(q)}$
multiplications, hence the length of the $\SLP$ follows.
\end{proof}

\begin{thm} \label{sz_thm_pre_step}
Assume Conjecture \ref{conjecture_correctness} and an oracle for the discrete logarithm problem in
  $\F_q$. The preprocessing step is a Las Vegas algorithm that finds
  standard generators for $\O2(G_{P_{\infty}})$ and $\O2(G_{P_0})$, as $\SLP$s in $X$ of length $\OR{\log\log(q)^2}$. It
  has expected time complexity $\OR{(\xi + \chi_D(q) + \log(q)
  \log\log(q))(\log\log(q))^2}$ field operations.

\end{thm}
\begin{proof}
The preprocessing step proceeds as follows.
\begin{enumerate}
\item Find random elements $a_1 \in G_{P_{\infty}}$ and $b_1
  \in G_{P_0}$ using the algorithm from Corollary \ref{thm_stab_elt}. Repeat until $a_1$ can be diagonalised to $M^{\prime}(\lambda) \in
G$, where $\lambda \in \F_q^{\times}$ and $\lambda$ 
does not lie in a proper subfield of $\F_q$. Do similarly for $b_1$. 
\item Find random elements $a_2 \in G_{P_{\infty}}$ and $b_2
  \in G_{P_0}$ using the algorithm from Corollary \ref{thm_stab_elt}.
  Let $c_1 = [a_1, a_2]$, $c_2 = [b_1,
  b_2]$. Repeat until $\abs{c_1} = \abs{c_2} = 4$.

\item Let $Y_{\infty} = \set{c_1, a_1}$ and $Y_0 = \set{c_2, b_1}$. As standard generators for $\O2(G_{P_{\infty}})$ we now take 
\begin{equation} \label{standard_gens}
L = \bigcup_{i = 1}^{2m + 1} \set{c_1^{d_1^i}, (c_1^2)^{d_1^i}}
\end{equation}
and similarly we obtain $U$ for $\O2(G_{P_0})$. 
\end{enumerate}

It follows from \eqref{sz_matrix_id1} and \eqref{sz_matrix_id2} that
\eqref{standard_gens} provides the standard generators for
$G_{P_{\infty}}$. These are expressed as $\SLP$s in $X$, since this is
true for the elements returned from the algorithm described in Corollary
\ref{thm_stab_elt}. Hence the algorithm is Las Vegas.

By Corollary \ref{thm_stab_elt}, the expected time to find $a_1$ and
$b_1$ is $\OR{(\xi + \chi_D(q) + \log(q) \log\log(q))\log\log(q)}$, and these are uniformly
distributed independent random elements. The elements of order
dividing $q - 1$ can be diagonalised as required. By Theorem \ref{thm_suzuki_props}, the proportion of elements of order $q - 1$ in
$G_{P_{\infty}}$ and $G_{P_0}$ is $\phi(q - 1) / (q - 1)$. Hence the expected time for the first step is $\OR{(\xi + \chi_D(q) + \log(q) \log\log(q))(\log\log(q))^2}$ field operations.

Similarly, by Proposition \ref{sz_prop_frobenius} the expected time for
the second step is $\OR{(\xi + \chi_D(q) + \log(q) \log\log(q))(\log\log(q))^2}$ field
operations.



By the remark preceding the theorem, $L$ determines two sets of
field elements $\set{a_1, \dotsc, a_{2m+1}}$ and $\set{b_1, \dotsc,
b_{2m+1}}$. In this case each $a_i = a \lambda^i$ and $b_i = b
\lambda^{i(t + 1)}$, for some fixed $a,b \in \F_q^{\times}$, where $\lambda$
is as in the algorithm. Since $\lambda$ does not lie in a proper subfield, these sets form vector space bases of $\F_q$ over $\F_2$.

To determine if $a_1$ or $b_1$ diagonalise to some $M^{\prime}(\lambda)$ it is sufficient to consider the eigenvalues on the diagonal, since both $a_1$ and $b_1$ are triangular. To determine if $\lambda$ lies
in a proper subfield, it is sufficient to determine if $\abs{\lambda} \mid
2^n - 1$, for some proper divisor $n$ of $2m + 1$. Hence the
dominating term in the complexity is the first step.
\end{proof}

\subsubsection{Main algorithm}
Now we
consider the algorithm that expresses $g$ as an $\SLP$ in
$X$. It is given formally as Algorithm \ref{alg:sz_main_alg}.

\begin{figure}[ht]
\begin{codebox} 
\refstepcounter{algorithm}
\label{alg:sz_main_alg}
\Procname{\kw{Algorithm} \ref{alg:sz_main_alg}: $\proc{ElementToSLP}(L, U, g)$}
\li \kw{Input:} Standard generators $L$ for $G_{P_{\infty}}$ and $U$ for $G_{P_0}$. Matrix $g \in \gen{X} = G$.
\li \kw{Output:} An $\SLP$ for $g$ in $X$.
\zi \Repeat
\li  $r := \kw{Random}(G)$ 
\li \Until $gr$ has an eigenspace $Q \in \OV$ \label{main_alg_find_point}
\li Find $z_1 \in G_{P_{\infty}}$ using $L$ such that $Qz_1 = P_0$. \label{main_alg_row_op1} 
\li \Comment{ Now $(gr)^{z_1} \in G_{P_0}$.}
\li Find $z_2 \in G_{P_0}$ using $U$ such that $(gr)^{z_1} z_2 = M^{\prime}(\lambda)$ for some $\lambda \in \F_q^{\times}$. \label{main_alg_row_op2} 
\li \Comment{ Express diagonal matrix as $\SLP$}
\li $x := \Tr(M^{\prime}(\lambda))$ 
\li Find $h = [S(0, (x^t)^{1 / 4}), S(0, 1)^T]$ using $L \cup U$. \label{main_alg_row_op3} 
\li \Comment{ Now $\Tr(h) = x$.}
\li Let $P_1, P_2 \in \OV$ be the fixed points of $h$. 
\li Find $a \in G_{P_{\infty}}$ using $L$ such that $P_1 a = P_0$. \label{main_alg_row_op4} 
\li Find $b \in G_{P_0}$ using $U$ such that $(P_2 a)b = P_{\infty}$. \label{main_alg_row_op5} 
\li \Comment{ Now $h^{ab} \in G_{P_{\infty}} \cap G_{P_0} = \mathcal{H}$, so $h^{ab} \in \set{M^{\prime}(\lambda)^{\pm 1}}$.}
\li \If{$h^{ab} = M^{\prime}(\lambda)$}
\zi \Then
\li     Let $w$ be an $\SLP$ for $(h^{ab} z_2^{-1})^{z_1^{-1}} r^{-1}$. \label{main_alg_get_slp1} 
\li      \Return{$w$} 
\zi \Else
\li     Let $w$ be an $\SLP$ for $((h^{ab})^{-1} z_2^{-1})^{z_1^{-1}} r^{-1}$. \label{main_alg_get_slp2} 
\li     \Return{$w$} 
    \End
\zi \kw{end}
\end{codebox}
\end{figure}

\begin{thm} \label{thm_element_to_slp}
Algorithm \ref{alg:sz_main_alg} is a Las Vegas algorithm with expected
time complexity $\OR{\xi + \log(q)^3}$ field operations. The length of
the $\SLP$ is $\OR{\log(q)(\log\log(q))^2}$.
\end{thm}
\begin{proof}
First observe that since $r$ is randomly chosen we obtain it as an $\SLP$. On line \ref{main_alg_find_point}
we check if $gr$ fixes a point, and from Proposition
\ref{sz_totient_prop} we see that the probability that the test succeeds is at least $1/2$.

The elements found at lines \ref{main_alg_row_op1} and
\ref{main_alg_row_op2} are constructed using Lemma \ref{sz_row_operations}, so we can
obtain them as $\SLP$s.

The element $h$ found at line \ref{main_alg_row_op3} clearly has trace
$x$, and it is constructed using Lemma \ref{sz_row_operations}, so we obtain
it as an $\SLP$. From Lemma
\ref{sz_conjugacy_classes} we know that $h$ is conjugate to
$M^{\prime}(\lambda)$ and therefore must fix $2$ points of
$\OV$. Hence lines
\ref{main_alg_row_op4} and \ref{main_alg_row_op5} make
sense, and the elements found are constructed using Lemma \ref{sz_row_operations} and therefore we obtain them as $\SLP$s. 

The only elements in $\mathcal{H}$ that are conjugate to $h$ are
$M^{\prime}(\lambda)^{\pm 1}$, so clearly $h^{ab}$ must be one of them.

Finally, the elements that make up $w$ were found as $\SLP$s, and it is clear that if we evaluate $w$ we obtain
$g$. Hence the algorithm is Las Vegas.

From Lemma \ref{sz_row_operations} it follows that lines
\ref{main_alg_row_op1}, \ref{main_alg_row_op2},
\ref{main_alg_row_op3}, \ref{main_alg_row_op4} and
\ref{main_alg_row_op5} use $\OR{\log(q)^3}$ field operations.

Finding the fixed points of $h$, and performing the check at line
\ref{main_alg_find_point} only amounts to considering eigenspaces,
which uses $\OR{\log{q}}$ field operations. Thus the expected time complexity of
the algorithm is $\OR{\xi + \log{q}^3}$ field operations.

The $\SLP$s of the standard generators have length $\OR{\log\log(q)^2}$.
Because of the row operations, $w$ will have
length $\OR{\log(q)(\log\log(q))^2}$.

\end{proof}

\subsection{Conjugates of the standard copy}
\label{section:conjugating_element} 

Now we assume that we are given $G \leqslant \GL(4, q)$ such that $G$
is a conjugate of $\Sz(q)$, and we turn to the problem of finding some
$g \in \GL(4, q)$ such that $G^g = \Sz(q)$, thus obtaining an
isomorphism to the standard copy. The matrix degree is constant here,
so we set $\xi = \xi(4)$

\begin{lem} \label{lem_find_ovoid_point} There exists a Las Vegas
  algorithm that, given $\gen{X} \leqslant \GL(4, q)$ such that $\gen{X}^h =
  \Sz(q)$ for some $h \in \GL(4, q)$, finds a point $P \in
  \OV^{h^{-1}} = \set{Q h^{-1} \mid Q \in \OV}$. The algorithm has expected time complexity 
\[ \OR{(\xi + \log(q)\log\log(q))\log\log(q)}\] field operations.
\end{lem}
\begin{proof}
  Clearly $\OV^{h^{-1}}$ is the set on which $\gen{X}$ acts doubly
  transitively. For a matrix $M^{\prime}(\lambda) \in \Sz(q)$ we see
  that the eigenspaces corresponding to the eigenvalues $\lambda^{\pm
    (t + 1)}$ will be in $\OV$. Moreover, every element of order
  dividing $q-1$, in every conjugate $G$ of $\Sz(q)$, will have
  eigenvalues of the form $\mu^{\pm (t + 1)}$, $\mu^{\pm 1}$ for some
  $\mu \in \F_q^{\times}$, and the eigenspaces corresponding to $\mu^{\pm (t +
    1)}$ will lie in the set on which $G$ acts doubly transitively.

  Hence to find a point $P \in \OV^{h^{-1}}$ it is sufficient to find a
  random $g \in \gen{X}$, of order dividing $q - 1$. We compute the pseudo-order using expected $\OR{\log(q) \log\log(q)}$ field operations, and by 
  Proposition \ref{sz_totient_prop}, the expected time to find the element is $\OR{(\xi  + \log(q) \log\log(q)) \log\log(q)}$ field operations. We then find the eigenspaces of $g$.

  Clearly this is a Las Vegas algorithm with the stated time complexity.
\end{proof}

\begin{lem} \label{lem_diagonal_conj}
  There exists a Las Vegas algorithm that, given $\gen{X} \leqslant \GL(4, q)$ such that $\gen{X}^d = \Sz(q)$ where
  $d = \diag(d_1, d_2, d_3, d_4) \in \GL(4, q)$, finds a diagonal matrix $e
  \in \GL(4, q)$ such that $\gen{X}^e = \Sz(q)$, using expected
\[ \OR{(\xi + \log(q)\log\log(q))\log\log(q) + \abs{X}}\]
 field operations.
\end{lem}
\begin{proof}
Let $G = \gen{X}$. Since $G^d = \Sz(q)$, $G$ must preserve the symplectic form
\begin{equation}
K = d J d = \begin{bmatrix}
0 & 0 & 0 & d_1 d_4 \\
0 & 0 & d_2 d_3 & 0 \\
0 & d_2 d_3 & 0 & 0 \\
d_1 d_4 & 0 & 0 & 0 
\end{bmatrix}
\end{equation}
where $J$ is given by \eqref{standard_symplectic_form}. Using the
MeatAxe, we can find this form, which is determined up to a scalar
multiple. Hence the diagonal matrix $e = \diag(e_1, e_2, e_3, e_4)$,
that we want to find, is also determined up to a scalar multiple (and
up to multiplication by a diagonal matrix in $\Sz(q)$).

Since $e$ must take $J$ to $K$, we must have
$K_{1, 4}= d_1 d_4 = e_1 e_4$ and $K_{2, 4} = d_2 d_3 = e_2 e_3$. Because $e$ is determined up to a scalar multiple, we can choose
$e_4 = 1$ and $e_1 = K_{1, 4}$. Hence it only remains to determine
$e_2$ and $e_3$.

To conjugate $G$ into $\Sz(q)$, we must have $Pe \in \OV$ for every
$P \in \OV^{d^{-1}}$, which is the set on which $G$ acts doubly
transitively. By Lemma \ref{lem_find_ovoid_point}, we can find $P = (p_1 : p_2 : p_3 : 1) \in \OV^{d^{-1}}$, and the
condition $P e = (p_1 K_{1, 4} : p_2 e_2 : p_3 e_3 : 1) \in \OV$ is
given by \eqref{sz_ovoid_def} and amounts to
\begin{equation} \label{conjugate_final_eq1}
p_2 p_3 K_{2, 3} + (p_2 e_2)^t + (p_3 e_3)^{t + 2} - p_1 K_{1, 4} = 0
\end{equation}
which is a polynomial equation in the two variables $e_2$ and
$e_3$. 

Notice that we can consider $e_2^{t}$ to be the
variable, instead of $e_2$, since if $x = e_2^{t}$, then $e_2 =
\sqrt{x^t}$. Similarly, we can let $e_3^{t + 2}$ be the variable
instead of $e_3$, since if $y = e_3^{t + 2}$ then $e_3 = y^{1 - t/2}$.
Thus instead of \eqref{conjugate_final_eq1} we obtain a linear equation
\begin{equation} \label{conjugate_final_eq2}
p_2^t x + p_3^{t + 2} y = p_1 K_{1, 4} - p_2 p_3 K_{2, 3}
\end{equation}
in the variables $x, y$. Thus the complete algorithm for finding $e$
proceeds as follows.
\begin{enumerate}
\item Find the form $K$ that is preserved by $G$, using the MeatAxe.
\item Find $P \in \OV^{d^{-1}}$ using Lemma \ref{lem_find_ovoid_point}. 
\item Let $P = (p_1 : p_2 : p_3 : p_4)$ and find $Q = (q_1 : q_2 : q_3
  : q_4)$ using Lemma \ref{lem_find_ovoid_point} until the following
  linear system in the variables $x$ and $y$ is non-singular. 
\begin{equation}
\begin{split}
p_2^t x + p_3^{t + 2} y &= p_1 K_{1, 4} - p_2 p_3 K_{2, 3} \\
q_2^t x + q_3^{t + 2} y &= q_1 K_{1, 4} - q_2 q_3 K_{2, 3}
\end{split}
\end{equation}
By Proposition
  \ref{prop_trick_general_case}, the probability of finding such a $Q$ is $1 - \OR{1/\sqrt{q}}$.
\item Let $(\alpha, \beta)$ be a solution to the linear system. The diagonal matrix $e = \diag(K_{1, 4}, \sqrt{\alpha^t}, \beta^{1 - t/2}, 1)$ now satisfies $G^e = \Sz(q)$.
\end{enumerate}
By Lemma \ref{lem_find_ovoid_point} and Section \ref{section:meataxe}, this
is a Las Vegas algorithm that uses expected $\OR{(\xi + \log(q)\log\log(q))\log\log(q) + \abs{X}}$ field operations.
\end{proof}

\begin{lem} \label{lem_conjugate_to_digaonal}
  There exists a Las Vegas algorithm that, given subsets $X$, $Y_P$
  and $Y_Q$ of $\GL(4, q)$ such that $\O2(G_P) < \gen{Y_P} \leqslant
  G_P$ and $\O2(G_Q) < \gen{Y_Q} \leqslant G_Q$, respectively, where
  $\gen{X} = G$, $G^h = \Sz(q)$ for some $h \in \GL(4, q)$ and $P \neq Q
  \in \OV^{h^{-1}}$, finds $k \in \GL(4, q)$ such that $(G^k)^d =
  \Sz(q)$ for some diagonal matrix $d \in \GL(4, q)$.  The algorithm
  has expected time complexity $\OR{\abs{X}}$ field operations.
\end{lem}
\begin{proof}

  Notice that the natural module $V = \F_q^4$ of $\mathcal{F}
  \mathcal{H}$ is uniserial with four non-zero submodules, namely $V_i
  = \set{(v_1, v_2, v_3, v_4) \in \F_q^4 \mid v_j = 0, j > i}$ for $i
  = 1, \dotsc, 4$. Hence the same is true for $\gen{Y_P}$ and
  $\gen{Y_Q}$ (but the submodules will be different) since they lie in
  conjugates of $\mathcal{F} \mathcal{H}$.

  Now the algorithm proceeds as follows.
  
\begin{enumerate}
\item Let $V = \F_q^4$ be the natural module for $\gen{Y_P}$ and
  $\gen{Y_Q}$. Find composition series $V = V^P_4 > V^P_3
  > V^P_2 > V^P_1$ and $V = V^Q_4 > V^Q_3 >
  V^Q_2 > V^Q_1$ using the MeatAxe.

\item Let $U_1 = V_1^P$, $U_2 = V_2^P \cap V_3^Q$, $U_3 = V_3^P \cap V_2^Q$ and $U_4 = V_1^Q$. For each $i = 1, \dotsc, 4$, choose $u_i \in U_i$.

\item Now let $k$ be the matrix such that $k^{-1}$ has $u_i$ as row $i$, for $i = 1, \dotsc, 4$.
\end{enumerate}

We now motivate the second step of the algorithm. 


One can choose a basis that exhibits the series $\set{V_i^P}$, in
other words, such that the matrices acting on the module are lower
triangular with respect to this basis. Similarly one can choose a basis
that exhibits the series $\set{V_i^Q}$.

On the other hand, since $P \neq Q$, there exists $g^{\prime} \in
\Sz(q)$ such that $P h g^{\prime} = P_{\infty}$ and $Q h g^{\prime} =
P_0$. If we let $z = hg^{\prime}$, then $\gen{Y_P}^z$ and
$\gen{Y_Q}^z$ consist of lower and upper triangular matrices,
respectively. Thus, the rows of $z^{-1}$ form a basis of $V$ that
exhibits the series $\set{V_i^P}$ and the series $\set{V_i^Q}$ in
reversed order.

With respect to this basis, it is clear that $\dim V^P_2 \cap V^Q_3 =
1$, $\dim V^P_3 \cap V^Q_2 = 1$ and that all the $U_i$ are distinct.
Hence the basis chosen in the algorithm exhibits the series
$\set{V_i^P}$, and it exhibits the series $\set{V_i^Q}$ in reverse
order. Therefore the chosen $k$ satisfies that $\gen{Y_P}^k$ is lower
triangular and $\gen{Y_Q}^k$ is upper triangular. The former implies
that $k z^{-1}$ is a lower triangular matrix, and the latter
that it is an upper triangular matrix, and hence it must diagonal.

Thus the matrix $k$ found in the algorithm satisfies $z = kd$ for
some diagonal matrix $d \in \GL(4, q)$. Since $\Sz(q) = G^h = G^z =
(G^k)^d$, the algorithm returns a correct result, and it is Las Vegas
because the MeatAxe is Las Vegas. Clearly the expected time complexity is the
same as the MeatAxe, so the algorithm uses $\OR{\abs{X}}$ field
operations.
\end{proof}

\begin{thm} \label{thm_conj_problem}
  Assume Conjecture \ref{conjecture_correctness}. There exists a Las
  Vegas algorithm that, given a conjugate $\gen{X}$ of $\Sz(q)$, finds $g \in \GL(4,
  q)$ such that $\gen{X}^g = \Sz(q)$. The algorithm has expected time complexity
  $\OR{(\xi + \log(q)\log\log(q))(\log\log(q))^2 + \abs{X}}$ field operations.
\end{thm}
\begin{proof}
  Let $G = \gen{X}$. By Remark \ref{sz_rem_stab_elt}, we can use
  Corollary \ref{thm_stab_elt} in $G$, and hence we can find
  generators for a stabiliser of a point in $G$, using the algorithm
  described in Theorem \ref{sz_thm_pre_step}. In this case we do not
  need the elements as $\SLP$s, so a discrete log oracle is not
  necessary.

\begin{enumerate}
\item Find points $P, Q \in \OV^{h^{-1}}$ using Lemma
  \ref{lem_find_ovoid_point}. Repeat until $P \neq Q$.
\item Find generating sets $Y_P$ and $Y_Q$ such that $\O2(G_P) <
  \gen{Y_P} \leqslant G_P$ and $\O2(G_Q) < \gen{Y_Q} \leqslant G_Q$
  using the first three steps of the algorithm from the proof of Theorem
  \ref{sz_thm_pre_step}.
\item Find $k \in \GL(4, q)$ such that $(G^k)^d = \Sz(q)$ for some diagonal matrix $d \in \GL(4, q)$, using Lemma \ref{lem_conjugate_to_digaonal}.
\item Find a diagonal matrix $e$ using Lemma \ref{lem_diagonal_conj}.
\item Now $g = ke$ satisfies that $G^g = \Sz(q)$.
\end{enumerate}

Be Lemma \ref{lem_find_ovoid_point}, \ref{lem_conjugate_to_digaonal} and \ref{lem_diagonal_conj}, and the proof of Theorem \ref{sz_thm_pre_step}, this is a Las Vegas algorithm with expected time complexity as stated.
\end{proof}

\subsection{Tensor decomposition}
\label{section:sz_tensor_decompose}

Now assume that
$G \leqslant \GL(d, q)$ where $G \cong \Sz(q)$, $d > 4$ and $q = 2^{2m +
  1}$ for some $m > 0$. Then $\Aut{\F_q} = \gen{\psi}$, where $\psi$
is the Frobenius automorphism. Let $W$ be the given module of $G$ of dimension $d$ and let $V$ be the natural
module of $\Sz(q)$ of dimension $4$. From Section \ref{section:algorithm_overview} and Section \ref{section:sz_indecomposables} we know that
\begin{equation} \label{irreducible_module_form}
W \cong V^{\psi^{i_0}} \otimes V^{\psi^{i_1}} \otimes \dotsm \otimes
V^{\psi^{i_{n - 1}}}
\end{equation}
for some integers $0 \leqslant i_0 < i_1 < \dotsb < i_{n - 1}
\leqslant 2m$. In fact, we may assume that $i_0 = 0$ and clearly
$d = \dim W = (\dim V)^n = 4^n$. 
As described in Section \ref{section:algorithm_overview}, we now want to tensor
decompose $W$ to obtain an effective isomorphism from $W$ to $V$.

\subsubsection{The main algorithm} \label{section:sz_tensor_main_alg}

We now describe our main algorithm that finds a tensor decomposition
of $W$ when $q$ is large. It is sufficient to find a flat in $W$. For
$k = 0, \dotsc, n -1$, let $H_k \leqslant \GL(4, q)$ be the image of
the representation corresponding to $V^{\psi^{i_k}}$, and let $\rho_k
: G \to H_k$ be an isomorphism. Our goal is then to find $\rho_k$
effectively for some $k$.

We begin with an overview of the method. Our approach for finding a
flat in $W$ is to consider eigenspaces of an element of $G$ of
order dividing $q - 1$. By Proposition \ref{sz_totient_prop} we know that such
elements are easy to find by random search.

Let $g \in G$ where $\abs{g} = q - 1$, and let $t = 2^{m + 1}$.
By Proposition \ref{pr_sz_subgroups_conj} we know that for $k = 0, \dotsc,
n - 1$, $\rho_k(g)$ has four distinct eigenvalues $\lambda_k^{\pm 1}$
and $\lambda_k^{\pm (t + 1)}$ for some $\lambda_k \in \F_q^{\times}$.
Also, the eigenspaces of $\rho_k(g)$ have dimension $1$. Our method
for finding a flat in $W$ is to construct a line as a suitable sum of
eigenspaces of $g$.

Let $E$ be the multiset of eigenvalues of $g$, so that $\abs{E} = d$ and every element of $E$ has the form
\begin{equation} \label{eigenvalue_form}
\lambda_0^{j_0} \lambda_1^{j_1} \dotsm \lambda_{n - 1}^{j_{n - 1}}
\end{equation}
where each $\lambda_k \in \F_q^{\times}$ and each $j_k \in \set{\pm 1,
  \pm (t + 1)}$. A set $E^{\prime} \subseteq \F_q^{\times}$ that satisfies
\begin{itemize}
\item $\abs{E^{\prime}} = n$
\item $\lambda_k \in E^{\prime}$ or $\lambda_k^{-1} \in E^{\prime}$ for each $k = 0, \dotsc, n - 1$
\end{itemize}
is a set of \emph{base values} for $E$. Clearly $E$ is easily calculated from $E^{\prime}$.

Moreover $\lambda_k = \lambda_0^{2^{i_k}}$, and since $\abs{g} = q -
1$ we must have $\abs{\lambda_k} = q - 1$ for $k = 0, \dotsc, n - 1$.
For every $0 \leqslant k < l \leqslant n - 1$ we have $0 <
2^{i_k} \pm 2^{i_l} < q - 1$ and therefore $\lambda_k \neq
\lambda_l^{\mp 1}$.

First we try to find a set of base values for $E$. 
\begin{conj} \label{lem1_basevalues}
Assume $m > n$. Let $S = \set{\sqrt{e / f} : e, f \in E \mid e \neq f}$ and 
\begin{multline}
P = \{x \in S \mid \forall \, e \in E \; \exists \, y \in \set{x, x^{-1}, x^{t + 1}, x^{-t-1}} : \\
\set{e y x, e y x^{-1}, e y x^{t + 1}, e y x^{-t-1}} \subseteq E\}.
\end{multline}
Then the following hold:
\begin{itemize}
\item $P$ contains a set of base values for $E$. 
\item If the twists do not have a subsequence $i_r < i_r + 1 < i_r + 2$, then $\abs{P} = 2n$ and hence $P$ consists of the base values and their inverses.
\end{itemize}
\end{conj}


If the twists have a subsequence of the form in the Conjecture, or more generally a subsequence of length $l$, then
$\prod_{i = 0}^j \lambda_{r + i} \in P$ for every $j = 0, \dotsc, l -
1$. Hence we need more conditions to extract the base values.

\begin{conj} \label{lem2_basevalues}
  Let $P$ be as in Lemma \ref{lem1_basevalues}. Define $P^{\prime}
  \subseteq P$ to be those $x \in P$ for which there exists $0 =
  \alpha_0 < \alpha_1 < \dotsb < \alpha_n \leqslant 2m$ such that
  $x^{\alpha_i} \in P$ for every $i = 0, \dotsc, n$. Then
  $\abs{P^{\prime}} = 2n$ and hence $P^{\prime}$ consists of the base values and their inverses.
\end{conj}


Let $S_i$ denote the sum of eigenspaces of $g$
corresponding to the eigenvalues $\lambda_0^{j_0} \dotsm \lambda_i
\dotsm \lambda_{n - 1}^{j_{n - 1}}$ and $\lambda_0^{j_0} \dotsm
\lambda_i^{-1} \dotsm \lambda_{n - 1}^{j_{n - 1}}$, where each $j_k$
ranges over $\set{\pm 1, \pm (t + 1)}$. 

\begin{lem} \label{lem_line_subspace}
If $\dim S_k = 2 \cdot 4^{n - 1}$ for some $0 \leqslant k \leqslant n - 1$, then $S_k$ is a line in $W$.
\end{lem}
\begin{proof}
  For each $i = 0, \dotsc, n - 1$, let $e_{j_i}$ be an eigenvector of
  $\rho_i(g)$ for the eigenvalue $\lambda_i^{j_i}$, where $j_i \in
  \set{\pm 1, \pm (t + 1)}$. Observe that $W$ contains the
  following subspace.
\[ L = \gen{e_{j_0} \otimes \dotsm \otimes e_{j_{n-1}} \mid j_i \in 
\set{\pm 1, \pm (t + 1)}, i \neq k, j_k \in \set{\pm 1}} \]
Clearly, $L$ is of the form $V^{\psi^{i_0}} \otimes \dotsm \otimes V^{\psi^{i_{k - 1}}} \otimes A \otimes V^{\psi^{i_{k + 1}}} \otimes \dotsm \otimes V^{\psi^{i_{n - 1}}}$ where $\dim A = 2$, so $L$ is a line in $W$ of dimension $2 \cdot 4^{n - 1}$.

If $v = e_{j_0} \otimes \dotsm \otimes e_{j_{n - 1}} \in L$, then
\[
\begin{split}
vg & = e_{j_0} \rho_0(g) \otimes \dotsm \otimes e_{j_{n - 1}} \rho_{n - 1}(g) = \lambda_0^{j_0} e_{j_0} \otimes \dotsm \otimes \lambda_{n - 1}^{j_{n - 1}} e_{j_{n - 1}} = \\
 & = \lambda_0^{j_0} \dotsm \lambda_{n - 1}^{j_{n - 1}} v
\end{split} 
\]
and hence $v \in S_k$. Therefore $L \leqslant S_k$, so if $\dim S_k = 2 \cdot 4^{n - 1}$ then $L = S_k$ and thus $S_k$ is a line in $W$.

\end{proof}

The success probability of the algorithm for finding a flat relies on the following Conjecture.

\begin{conj} \label{cl_flatdim} Let $d = 4^n$ with $n > 1$ be fixed. If $q = 2^{2m +
    1}$ and $m \geqslant n + 1$, then for every absolutely irreducible $G \leqslant \GL(d, q)$ with $G \cong \Sz(q)$ and every $g \in G$ with
  $\abs{g} = q - 1$, we have $\dim S_i = 2 \cdot 4^{n - 1}$ for some $0
  \leqslant i \leqslant n - 1$.
\end{conj}



The algorithm for finding a
flat is shown as Algorithm \ref{alg:tensor_decompose}.

\begin{figure}[ht]
\begin{codebox} 
\refstepcounter{algorithm}
\label{alg:tensor_decompose}
\Procname{\kw{Algorithm} \ref{alg:tensor_decompose}: $\proc{TensorDecomposeSz}(X)$}
\li \kw{Input}: Generating set $X$ for $G \cong \Sz(q)$ with natural module
  $W$, where
\zi $\dim W = 4^n$, $n > 1$, $q = 2^{2m+1}$ with $m > n$ and $W$ is absolutely
\zi  irreducible and over $\F_q$. 
\li \kw{Output}: A line $S$ in $W$.
\zi \Comment{\proc{FindBaseValues} is given by Conjectures \ref{lem1_basevalues} and \ref{lem2_basevalues}.}

\zi \Repeat
\zi \Comment{Find random element $g$ of pseudo-order $q - 1$}
\zi \Repeat
\li    $g := \proc{Random}(G)$
\li \Until{$\abs{g} \mid q - 1$ \label{main_alg_g_test}}
\li $E := \proc{Eigenvalues}(g)$ 
\li $N := \proc{FindBaseValues}(E)$ 
\li \Until{$\abs{N} = n$ \label{main_alg_N_test1}}
\li    \For{$i := 0$ \To $n - 1$}
\zi    \Do
\zi       \Comment{Let $N = \set{\lambda_0, \dotsc, \lambda_{n - 1}}$}
\li       $E_i := \set{ \lambda_0^{j_0} \dotsm \lambda_i \dotsm
  \lambda_{n - 1}^{j_{n - 1}} \mid j_k \in \set{\pm 1, \pm (t + 1)}}$ 
\zi $\cup \set{ \lambda_0^{j_0} \dotsm \lambda_i^{-1} \dotsm \lambda_{n - 1}^{j_{n - 1}} \mid j_k \in \set{\pm 1, \pm (t + 1)}}$ 
\li       $S_i := \sum_{e \in E_i} \proc{Eigenspace}(g, e)$ 
\li       \If{$\dim S_i = 2 \cdot 4^{n - 1} $ \label{main_alg_S_test}}
\zi       \Then
\li           \Return{$S_i$ \label{main_alg_success_return}}

          \End
\zi       \kw{end}
    \End
\zi \kw{end}
\end{codebox}
\end{figure}

\begin{thm} \label{thm_main_alg}
Assume Conjectures \ref{lem1_basevalues}, \ref{lem2_basevalues} and \ref{cl_flatdim}. Algorithm \ref{alg:tensor_decompose} is a Las Vegas algorithm. 
The algorithm has expected time complexity 
\begin{equation*}
\OR{(\xi(d) + d^3 \log(q) \log{\log(q^d)})\log{\log(q)} }
\end{equation*}
field operations.
\end{thm}
\begin{proof}
  The expected number of iterations in the initial loop is $\OR{1}$.
  Hence the expected time for the loop is $\OR{\xi(d) + d^3 \log(q)
    \log{\log(q^d)}}$ field operations. If $\abs{g} = q - 1$ then line
  \ref{main_alg_N_test1} will succeed, and Conjecture \ref{cl_flatdim}
  asserts that line \ref{main_alg_S_test} will succeed for some $i$.
  If $\abs{g}$ is a proper divisor of $q - 1$ then these lines might
  still succeed, and the probability that $\abs{g} = q - 1$ is high.
  By Proposition \ref{sz_totient_prop}, the expected number of
  iterations of the outer repeat statement is $\OR{\log\log(q)}$.

  If line \ref{main_alg_success_return} is reached, then the algorithm
  returns a correct result and hence it is Las Vegas.

  To find the eigenvalues of $g$, we calculate the characteristic
  polynomial of $g$ using $\OR{d^3}$ field operations, and find its
  roots using Theorem \ref{thm_solve_univariate_polys}. By Conjectures \ref{lem1_basevalues} and
  \ref{lem2_basevalues}, the rest of the algorithm uses $\OR{d^2 n
    \log(q)}$ field operations. Thus the theorem follows.
\end{proof}

\subsubsection{Small field approach} \label{section:sz_tensor_small_field}

When $q$ is small, the feasibility of Algorithm
\ref{alg:tensor_decompose} is not guaranteed. In that case the
approach is to find standard generators of $G$ using permutation group
techniques, then enumerate all tensor products of the form
\eqref{irreducible_module_form} and for each one we determine if it is
isomorphic to $W$.

Since $q$ is polynomial in $d$, this will turn out to be an efficient
algorithm which is given as Algorithm \ref{alg:small_field_approach}.
It finds a permutation representation of $G \cong \Sz(q)$, which is
done using the following result.

\begin{lem} \label{thm_suzuki_perm_rep}
  There exists a Las Vegas algorithm that, given $\gen{X} \leqslant \GL(d,
  q)$ such that $q = 2^{2m + 1}$ with $m > 0$ and $\gen{X} \cong
  \Sz(q)$, finds an effective injective homomorphism $\Pi : \gen{X} \to
  \Sym{O}$ where $\abs{O} = q^2 + 1$. The algorithm has expected time complexity $\OR{q^2(\xi(d) + \abs{X} d^2 + d^3) + d^4}$ field operations.
  \end{lem}
\begin{proof}
  By Theorem \ref{thm_suzuki_props}, $\Sz(q)$ acts doubly
  transitively on a set of size $q^2 + 1$. Hence $G = \gen{X}$ also
  acts doubly transitively on $O$, where $\abs{O} = q^2 + 1$, and we
  can find the permutation representation of $G$ if we can find a
  point $P \in O$. The set $O$ is a set of projective points of
  $\F_q^d$, and the algorithm proceeds as follows.
\begin{enumerate}
\item Choose random $g \in G$. Repeat until $\abs{g} \mid q - 1$.
\item Choose random $x \in G$ and let $h = g^x$. Repeat until $[g, h]^4 = 1$ and $[g, h] \neq 1$.
\item Find a
  composition series for the module $M$ of $\gen{g, h}$ and let $P \subseteq M$ be the submodule of
  dimension $1$ in the series.
\item Find the orbit $O = P^G$ and compute the permutation group $S \leqslant \Sym{O}$ of $G$ on $O$, together with an effective isomorphism $\Pi : G \to S$.
\end{enumerate}

By Proposition \ref{pr_sz_subgroups_conj}, elements in $G$ of order
dividing $q - 1$ fix two points of $O$, and hence $\gen{g, h}
\leqslant G_P$ for some $P \in O$ if and only if $g$ and $h$ have a
common fixed point. All composition factors of $M$ have dimension $1$, so a composition series
of $M$ must contain a submodule $P$ of dimension $1$. This submodule
is a fixed point for $\gen{g, h}$ and its orbit must have size $q^2 +
1$, since $\abs{G} =
q^2 (q^2 + 1) (q - 1)$ and $\abs{G_P} = q^2 (q - 1)$. It follows that
$P \in O$.

All elements of $G$ of even order lie in the derived group of a
stabiliser of some point, which is also a Sylow $2$-subgroup of $G$,
and the exponent of this subgroup is $4$. Hence $[g, h]^4 = 1$ if and
only if $\gen{g, h}$ lie in a stabiliser of some point, if and only
if $g$ and $h$ have a common fixed point.

To find the orbit $O = P^G$ we can compute a Schreier tree on the
generators $X$, with
$P$ as root, using $\OR{\abs{X} \abs{O} d^2}$ field operations. Then
$\Pi(g)$ can be computed for any $g \in \gen{X}$ using $\OR{\abs{O}
  d^2}$ field operations, by computing the permutation on $O$
induced by $g$. Hence $\Pi$ is effective, and its image $S$ is found by computing
the image of each element of $X$. Therefore the algorithm is correct
and it is clearly Las Vegas.

We find $g$ using expected $\OR{(\xi(d) + d^3 \log(q) \log{\log(q^d)})
\log {\log(q)}}$ field operations and we find $h$ using expected
$\OR{(\xi(d) + d^3)q^2}$ field operations. Then $P$ is found using the MeatAxe, in expected
$\OR{d^4}$ field operations. Thus the result follows.
\end{proof}

\begin{pr} \label{conj:sz_perm_iso}
Let $G = \gen{X} \leqslant \Sym{\OV}$ such that $G \cong \Sz(q) = H$. There exists a Las Vegas algorithm that finds $x, h, z \in G$ as $\SLP$s in $X$ such that the map
\begin{equation}
\begin{split}
x &\mapsto S(1, 0) \\
h &\mapsto M^{\prime}(\lambda) \\
z &\mapsto T
\end{split}
\end{equation}
is an isomorphism. Its time complexity is $\OR{q^3 \log(q)^5}$. The length of the returned $\SLP$s are $\OR{q}$.
\end{pr}
\begin{proof}
Follows from \cite[Theorem 1]{sz_blackbox_gens}.
\end{proof}

\begin{figure}
\begin{codebox} 
\refstepcounter{algorithm}
\label{alg:small_field_approach}
\Procname{\kw{Algorithm} \ref{alg:small_field_approach}: $\proc{SmallFieldTensorDecompose}(X)$}
\li \kw{Input}: Generating set $X$ for $G \cong \Sz(q)$ with natural
module $W$, where 
\zi $\dim W = 4^n$, $n > 1$, $q = 2^{2m + 1}$, $m > 0$ and $W$ is absolutely
\zi  irreducible and over $\F_q$
\li \kw{Output}: A change of basis matrix $c$ which exhibits $W$ as \eqref{irreducible_module_form}. 
\zi \Comment{Find permutation representation, \emph{i.e.} permutation group and}
\zi  \Comment{corresponding isomorphism}
\li $(\pi, P_G) := \proc{SuzukiPermRep}(G)$ 
\li $x,h,z := \proc{StandardGens}(P_G)$ 
\li Evaluate the $\SLP$s of $x,h,z$ on $G$ to obtain the set $Y$.
\li $H := \gen{Y}$
\li $T := \set{(i_1, \dotsc, i_n) \mid 0 \leqslant i_1 < \dotsb < i_n \leqslant 2m}$ 
\zi \Comment{Let $V$ be the natural module of $H$}
\li \For{$(i_1, \dotsc, i_n) \in T$}
\zi \Do
\li   $U := V^{\psi^{i_1}} \otimes \dotsm \otimes V^{\psi^{i_n}}$ 
\zi   \Comment{Find isomorphism between modules}
\li   $(\id{flag}, c) := \proc{ModuleIsomorphism}(U, W)$ 
\li   \If{$\id{flag} = \const{true}$}
\zi   \Then
\li      \Return{$c$}
      \End 
\zi   \kw{end}
    \End 
\zi \kw{end}
\end{codebox}
\end{figure}

\begin{thm} \label{sz_small_field_tensor}
Assume Conjecture \ref{conj:sz_perm_iso}. Algorithm \ref{alg:small_field_approach} is a Las Vegas algorithm with expected time complexity 
\[ \OR{q^2(\xi(d) + \abs{X} d^2 + d^3 \log\log(q) + q^2 \log(q)^3) + d^3 (\abs{X} \binom{2m}{n - 1} + d)}\] field operations.
\end{thm}
\begin{proof}
  The permutation representation $\pi$ can be found using Lemma
  \ref{thm_suzuki_perm_rep}, and the elements $x,h,z$ are found using
  Conjecture \ref{conj:sz_perm_iso}. Testing if modules are
  isomorphic can be done using the MeatAxe.
  
  If the algorithm returns an element $c$ then the change of basis
  determined by $c$ exhibits $W$ as a tensor product, so the algorithm
  is Las Vegas.
  
  The lengths of the $\SLP$s of $x,h,z$ is $\OR{q^2 \log\log(q)}$, so
  we need $\OR{d^3 q^2 \log\log(q)}$ field operations to obtain $Y$.
  The set $T$ has size $\binom{2m}{n - 1}$. Module isomorphism testing
  uses $\OR{\abs{X} d^3}$ field operations. Hence by Conjecture
  \ref{conj:sz_perm_iso} and Theorem \ref{thm_suzuki_perm_rep} the
  time complexity of the algorithm is as stated.
\end{proof}

\subsection{Constructive recognition}

Finally, we can now state and prove our main theorem.

\begin{thm} \label{cl_sz_constructive_recognition} Assume the Suzuki Conjectures and an
  oracle for the discrete logarithm problem in $\F_q$. There exists a
  Las Vegas algorithm that, given $\gen{X} \leqslant \GL(d, q)$ satisfying the assumptions in Section \ref{section:algorithm_overview}, with
  $q = 2^{2m + 1}$, $m > 0$ and $\gen{X} \cong \Sz(q)$, finds an effective
  isomorphism $\varphi : \gen{X} \to \Sz(q)$ and performs preprocessing for constructive membership testing. The algorithm has expected
  time complexity $\OR{\xi(d) (d^2 + (\log\log(q))^2) + d^5 \log\log(d) + 
    d^4 \abs{X} + d^3 \log(q) \log \log(q) ( \log(d) + \log\log(q)) + \log(q) (\log \log(q))^3 + \chi_D(q) (\log \log(q))^2}$ field operations.

The inverse of $\varphi$ is also effective. Each image of $\varphi$ can be computed using $\OR{d^3}$ field operations,
and each pre-image using expected 
\[ \OR{\xi(d) + \log(q)^3 + d^3 \log(q)(\log\log(q))^2}\] field
operations.
\end{thm}
\begin{proof}
Let $V$ be the module of $G = \gen{X}$. From Section \ref{section:sz_indecomposables} we know that $d = 4^n$ where $n \geqslant 1$. The algorithm proceeds as follows:
\begin{enumerate}
\item If $d = 4$ then use Theorem \ref{thm_conj_problem} to obtain $y \in \GL(4, q)$ such that $G^y = \Sz(q)$, and hence an effective isomorphism $\varphi : G \to \Sz(q)$ defined by $g \mapsto g^y$.
\item If $m > n$, use Algorithm \ref{alg:tensor_decompose} to find a flat $L \leqslant V$. Then use the tensor decomposition algorithm described in Section \ref{section:tensor_decomposition} with $L$, to obtain $x \in \GL(d, q)$ such
  the change of basis determined by $x$ exhibits $V$ as a tensor
  product $U \otimes W$, with $\dim U = 4$. Let $G_U$ and $G_W$ be the
  images of the corresponding representations.
\item If instead $m \leqslant n$ then use Algorithm \ref{alg:small_field_approach} to find $x$.
\item Define $\rho_U : G_{U \otimes W} \to G_U$ as $g_u \otimes g_w \mapsto g_u$ and let $Y = \set{\rho_U(g^x) \mid g \in X}$. Then $\gen{Y}$ is conjugate to $\Sz(q)$.
\item Use Theorem \ref{thm_conj_problem} to get $y \in \GL(4, q)$ such that $\gen{Y}^y = \Sz(q)$. 
\item An effective isomorphism $\varphi : G \to \Sz(q)$ is given by $g \mapsto \rho_U(g^x)^y$.
\end{enumerate}
The map $\rho_U$ is straightforward to compute, since given $g \in
\GL(d, q)$ it only involves dividing $g$ into submatrices of degree
$4^{n - 1}$, checking that they are scalar multiples of each other and
returning the $4 \times 4$ matrix consisting of these scalars. Since
$x$ might not lie in $G$, but only in $\Norm_{\GL(d, q)}(G) \cong G
{:} \F_q$, the result of $\rho_U$ might not have determinant $1$.
However, since every element of $\F_q$ has a unique $4$th root, we can
easily scale the matrix to have determinant $1$. Hence
by Theorems \ref{thm_main_alg}, \ref{sz_small_field_tensor} and
\ref{thm_conj_problem}, the algorithm is Las Vegas and any image of
$\varphi$ can be computed using $\OR{d^3}$ field operations.

In the case where we use Algorithm \ref{alg:small_field_approach} we
have $m \leqslant n$, hence $\binom{2m}{n-1} < d$ and $q \leqslant
d$. The expected time complexity to find $x$ in this case is
$\OR{\xi(d) d^2 + d^4 (\abs{X} + d \log \log(d))}$ field operations.

By Theorem \ref{thm_main_alg}, finding $L$ uses $\OR{(\xi(d) + d^3 \log(q) \log{\log(q^d)})\log{\log(q)} }$ field operations. From Section \ref{section:tensor_decomposition},
finding $x$ uses $\OR{d^3 \log(q)}$ field operations when a flat $L$
is given. By Theorem \ref{thm_conj_problem}, finding $y$ uses
expected $\OR{(\xi + \chi_D(q) + \log(q)\log\log(q))(\log\log(q))^2 + \abs{Y} }$ field operations, given a random
element oracle for $\gen{Y}$ that finds a random element using
$\OR{\xi}$ field operations. In this case we can construct random
elements for $\gen{Y}$ using the random element oracle for $\gen{X}$,
and then we will find them in $\OR{\xi(d)}$ field operations.

Hence the expected time complexity is as stated. Finally,
$\varphi^{-1}(g)$ is computed by first using Algorithm
\ref{alg:sz_main_alg} to obtain an $\SLP$ of $g$ and then evaluating
it on $X$. The necessary precomputations in Theorem
\ref{sz_thm_pre_step} have already been made during the application of
Theorem \ref{thm_conj_problem}, and hence it follows from Theorem
\ref{thm_element_to_slp} that the time complexity for computing the
pre-image of $g$ is as stated.
\end{proof}

\section{Small Ree groups}
Here we will use the notation from Section \ref{section:ree_theory}.
We will refer to Conjectures \ref{conj:ree_conjugacy}, \ref{conj:tensor_decomposition_dim7},
\ref{conj:tensor_decomposition_mixed_dim},
\ref{conj:tensor_decomposition_dim_27}, \ref{conj:ree_perm_iso} and \ref{lem:symmetric_square} simultaneously as the \emph{small Ree Conjectures}. We now
give an overview of the algorithm for constructive recognition and
constructive membership testing. It will be formally proved as Theorem
\ref{cl_ree_constructive_recognition}.

\begin{enumerate}
\item Given a group $G \cong \Ree(q)$, satisfying the assumptions in
  Section \ref{section:algorithm_overview}, we know from Section
  \ref{section:ree_indecomposables} that the module of $G$ is
  isomorphic to a tensor product of twisted copies of either the natural
  module of $G$ or its $27$-dimensional module. Hence we first tensor decompose this module. This is
  described in Section \ref{section:ree_tensor_decompose}.

\item The resulting group has degree $7$ or $27$. In the latter case we need to decompose it into degree $7$. This is described in Section \ref{section:decompose_symsquare}.

\item Now we have a group of degree $7$, so it is a conjugate of
  the standard copy. We therefore find a conjugating element. This is
  described in Section \ref{section:ree_conjugacy}.

\item Finally we are in $\Ree(q)$. Now we can perform preprocessing for constructive membership testing and other problems we want to solve. This is described in Section \ref{section:ree_constructive_membership}.

\end{enumerate}

Given a discrete logarithm oracle, the whole process has time
complexity slightly worse than $\OR{d^6 + \log(q)^3}$ field operations,
assuming that $G$ is given by a bounded number of generators.

\subsection{Recognition} \label{section:small_ree_recognition}

We now consider the question of non-constructive recognition of
$\Ree(q)$, so we want to find an algorithm that, given $\gen{X} \leqslant
\GL(d, q)$, decides whether or not $\gen{X} \cong \Ree(q)$. We will only
consider this problem for the standard copy, \emph{i.e.} we will only answer the question whether or not $\gen{X} = \Ree(q)$.

\begin{thm} \label{thm_ree_standard_recognition}
  There exists a Las Vegas algorithm that, given $\gen{X} \leqslant \GL(7,
  q)$, decides whether or not $\gen{X} = \Ree(q)$. The algorithm has expected time complexity $\OR{\sigma_0(\log(q))(\abs{X} + \log(q))}$ field
  operations.
\end{thm}
\begin{proof}
Let $G = \Ree(q)$, with natural module $M$. The algorithm proceeds as
follows:
\begin{enumerate}
\item Determine if $X \subseteq G$ and return \texttt{false} if
  not. All the following steps must succeed in order to conclude that a given $g
  \in X$ also lies in $G$.
\begin{enumerate}
\item Determine if $g \in \SO(7, q)$, which is true if $\det g = 1$
  and if $g J g^T = J$, where $J$ is given by
  \eqref{standard_bilinear_form} and where $g^T$ denotes the transpose
  of $g$.
\item Determine if $g \in G_2(q)$, which is true if $g$ preserves the
  algebra multiplication $\cdot$ of $M$. The multiplication table can easily be precomputed using the fact that if $v, w \in M$ then $v \cdot w = f(v \otimes w)$, where $f$ is a generator of
  $\Hom_G(M \otimes M, M)$ (which has dimension $1$).
\item Determine if $g$ is a fixed point of the exceptional outer
  automorphism of $G_2(q)$, mentioned in Section
  \ref{section:alt_definition}. Computing the automorphism amounts to
  extracting a submatrix of the exterior square of $g$ and then
  replacing each matrix entry $x$ by $x^{3^m}$.
\end{enumerate}

\item If $\gen{X}$ is not a proper subgroup of $G$, or
  equivalently if $\gen{X}$ is not contained in a maximal subgroup,
  return \texttt{true}. Otherwise return \texttt{false}. By
  Proposition \ref{ree_maximal_subgroup_list}, it is
sufficient to determine if $\gen{X}$ cannot be written over a smaller
field and if $\gen{X}$ is irreducible. This can be done using the
  algorithms described in Sections \ref{section:meataxe} and \ref{section:smallerfield}.
\end{enumerate}

Since the matrix degree is constant, the complexity of the first step
of the algorithm is $\OR{1}$ field operations. For the same reason,
the expected time of the algorithms in Sections \ref{section:meataxe} and \ref{section:smallerfield} is
$\OR{\sigma_0(\log(q))(\abs{X} + \log(q))}$ field operations. Hence our recognition
algorithm has expected time as stated,
and it is Las Vegas since the MeatAxe is Las Vegas.
\end{proof}

\subsection{Finding an element of a stabiliser} \label{section:stabiliser_elements}

Let $G = \Ree(q) = \gen{X}$. In this section the matrix degree is constant, so we set $\xi = \xi(7)$. The algorithm for the constructive
membership problem needs to find independent random elements of $G_P$
for a given point $P$. This is straightforward if, for any pair of points $P,Q
\in \OV$, we can find $g \in G$ as an $\SLP$ in $X$ such that $Pg = Q$.


The general idea is to find an involution $j \in G$
by random search, and then compute $\Cent_G(j) \cong \gen{j} \times
\PSL(2, q)$ using the \emph{Bray algorithm} described in Section \ref{section:inv_centraliser}. The given module restricted to the centraliser splits
up as in Proposition \ref{pr_inv_centraliser_split}, and the points
$P,Q \in \OV$ restrict to points in the $3$-dimensional submodule. If the restrictions satisfy certain conditions, we can
then find an element $g \in \Cent_G(j)$ that maps these
restricted points to each other, and we obtain $g$ as an $\SLP$
in the generators of $\Cent_G(j)$ using Theorem \ref{thm_psl_recognition}. It
turns out that with high probability, we can then multiply $g$ by
an element that fixes the restriction of $P$ so that $g$
also maps $P$ to $Q$. A discrete logarithm oracle is
needed in that step. Since the Bray algorithm produces generators for the
centraliser as $\SLP$s in $X$, we obtain $g$ as an $\SLP$ in $X$.

If any of the steps fail, we can try again with another involution
$j$, so using this method we can map $P$ to $Q$ for any pair of points
$P,Q \in \OV$.

It should be noted that it is easy to find involutions using the
method described in Section
\ref{section:matrix_orders}, since by Corollary
\ref{cl_random_selections} it is easy to find elements of even order
by random search.

\subsubsection{The involution centraliser}
To use the Bray algorithm we need to provide an algorithm that
determines if the whole centraliser has been generated. Since we know
what the structure of the centraliser should be, this poses no
problem. If we have the whole centraliser, the derived group should be
$\PSL(2, q)$, and by Proposition \ref{psl_generation}, with high
probability it is sufficient to compute two random elements of the
derived group. Random elements of the derived group can be found as
described in Section \ref{section:random_elements}.



We can therefore find the involution centraliser $\Cent_G(j) \leqslant G$ and
$\Cent_G(j)^{\prime} \cong \PSL(2, q)$.

\begin{lem} \label{lem_inv_centraliser_maps}
  There exists a Las Vegas algorithm that, given $\gen{Y} \leqslant G$ such
  that $\gen{Y} = \Cent_G(j)$ for some involution $j \in G$, finds
\begin{itemize}
\item the submodule $S_j \leqslant V_j$ described in Proposition
  \ref{pr_inv_centraliser_split},
\item an effective $\gen{Y}$-module homomorphism $\varphi_V :
  V_j \to S_j$, 
\item the induced map $\varphi_{\OV} : \PS(V_j) \to \PS(S_j)$,
\item the corresponding map $\varphi_G$ from the $7$-dimensional
  representation of $\Cent_G(j)$ to the $3$-dimensional
  representation. 
\end{itemize}
The maps can be computed using $\OR{1}$ field
  operations. The algorithm has expected time complexity $\OR{\abs{Y}}$ field operations.
\end{lem}
\begin{proof}
This is a straightforward application of the MeatAxe, so the fact that
the algorithm is Las Vegas and has the stated time complexity follows
from Section \ref{section:meataxe}. The maps consist of a change of
basis followed by a projection to a subspace, and so the Lemma follows.
\end{proof}

\begin{lem} \label{lem_psl_maps} Use the notation of Lemma \ref{lem_inv_centraliser_maps}. There exists a Las Vegas algorithm that, given $H = \gen{Y} =
  \varphi_G(\Cent_G(j)^{\prime})$ for an involution $j \in G$, finds
  effective isomorphisms $\rho_G : \gen{Y} \to \PSL(2, q)$, $\pi_3 : \PSL(2, q)
  \to \gen{Y}$ and $\pi_7 : \gen{Y} \to \Cent_G(j)^{\prime}$.

  The map $\pi_3$ is the symmetric square map of $\PSL(2, q)$; both
  $\varphi_G \circ \pi_7$ and $\pi_3 \circ \rho_G$ are the identity on
  $\gen{Y}$. The maps $\rho_G$ and $\pi_3$ can be computed using $\OR{1}$ field
  operations and $\pi_7$ can be computed using $\OR{\log(q)^3}$ field operations. The algorithm has expected time complexity $\OR{(\xi + \log(q) \log\log(q)) \log\log(q) + \abs{Y} + \chi_D(q)}$ field operations.
\end{lem}
\begin{proof}
  By Proposition \ref{pr_inv_centraliser_split}, the group $\gen{Y}$
  is an irreducible $3$-dimensional copy of $\PSL(2, q)$, so it must
  be a conjugate of the symmetric square of the natural
  representation. By using a change of basis from the algorithm in
  Theorem \ref{thm_psl_recognition}, we may assume that it is the
  symmetric square. Moreover, we can use Theorem
  \ref{thm_psl_recognition} to constructively recognise $\gen{Y}$ and
  obtain the map $\rho_G$. We can also solve the constructive
  membership problem in the standard copy, and by evaluating straight
  line programs we obtain the maps $\pi_3$ and $\pi_7$.

  It follows from Theorem \ref{thm_psl_recognition} that the expected time
  complexity is as stated.
\end{proof}

\subsubsection{Finding a mapping element}
We now consider the algorithm for finding elements that map one point
of $\OV$ to another. The notation from Lemma
\ref{lem_inv_centraliser_maps} and \ref{lem_psl_maps} will be used.

If we let $M = \gen{x} \oplus \gen{y}$ then we can identify $\PS(S_j)$ with
the space of quadratic forms in $x$ and $y$ modulo scalars, so that $S_j = \gen{x^2}
\oplus \gen{xy} \oplus \gen{y^2}$. Then $\varphi_G(\Cent_G(j)^{\prime})$ acts projectively on $\PS(S_j)$ and $\abs{\PS(S_j)} = \abs{\PS^2(\F_q)} = (q^3 - 1) / (q - 1) = q^2 + q + 1$. 
\begin{pr} \label{pr_ree_ovoid_split}
Use the notation from Lemma \ref{lem_inv_centraliser_maps} and \ref{lem_psl_maps}.
\begin{enumerate}
\item The number of points in $\OV$ that are contained in $\Ker(\varphi_V)$ is $q + 1$.
\item The map $\varphi_{\OV}$ restricted to $\OV$ is not
injective, and $\abs{\varphi_{\OV}(\OV)} \geqslant q^2$.
\end{enumerate}
\end{pr}
\begin{proof}
\begin{enumerate}
\item The map $\varphi_V$ is the projection onto $S_j$, so the kernel are
  those vectors that lie in $T_j$. From the proof of Proposition
  \ref{pr_inv_centraliser_split}, with respect to a suitable basis,
  $T_j$ is the $-1$-eigenspace of $h(-1)$. Hence by Proposition
  \ref{pr_involution_props}, $\abs{\OV \cap \PS(T_j)} = q + 1$.
\item Since $\abs{\OV} = q^3 + 1$ and $\abs{\PS(S_j)} = q^2 + q + 1$,
  it is clear that the map is not injective. In the above basis, the map
  $\varphi_V$ is defined by $(p_1, \dotsc, p_7) \mapsto (p_2,
  p_4, p_6)$. Hence if $P_{\infty} \neq P \in \OV$ then $\varphi_{\OV}(P) = (a^t : (ab)^t - c^t : -c - (bc)^t - a^{3t+2} - a^t b^{2t})$.

  If $a = c = 0$ we do not get a point in $\PS^2(\F_q)$ and if $a
  = 0$ and $c \neq 0$ we obtain $q$ points. Now let $a \neq 0$ and let
  $(x, y) \in \F_q^2$ such that $x^2 + y \neq 0$. Then $-x^2 - y$ is a
  square in $\F_q$ if and only if $(-x^2-y)^t$ is a square, so
  $(-x^2-y)^{1-t}$ is always a square. Hence, if $c = 0$, $b = x^{3t}$
  and $a = (-x^2-y)^{(1-t)/4}$, the image of $P$ is $(1 : x : y) \in
  \PS^2(\F_q)$. This gives $q^2 - q$ points.
\end{enumerate}
\end{proof}

\begin{pr} \label{pr_psl_ovoid_action}
Under the action of $H = \gen{Y} = \varphi_G(\Cent_G(j)^{\prime})$, the set $\PS(S_j)$ splits up into $3$ orbits.
\begin{enumerate}
\item The orbit containing $xy$, \emph{i.e.} the non-degenerate quadratic forms that represent $0$, which has size $q(q + 1) / 2$.
\item The orbit containing $x^2 + y^2$, \emph{i.e.} the non-degenerate quadratic forms that do not represent $0$, which has size $q(q - 1) / 2$.
\item The orbit containing $x^2$ (and $y^2$), \emph{i.e.} the degenerate quadratic forms, which has size $q + 1$.
\end{enumerate}
The pre-image in $\SL(2, q)$ of $\rho_G(\varphi_G(\Cent_G(j)^{\prime})_{xy})$ is dihedral of order $2(q - 1)$, generated by the matrices
\begin{flalign} \label{sl2_stab_gens}
\begin{bmatrix}
\alpha & 0 \\
0 & \alpha^{-1} 
\end{bmatrix} &
\begin{bmatrix}
0 & 1 \\
-1 & 0 
\end{bmatrix} 
\end{flalign}
where $\alpha$ is a primitive element of $\F_q$.
\end{pr}
\begin{proof}
Let $g = \begin{bmatrix}
a & b \\
c & d
\end{bmatrix}$ be any element of $\PSL(2, q)$, so that the symmetric square $h = \mathcal{S}^2(g) = \pi_3(g) \in \varphi_G(\Cent_G(j)^{\prime})$. Notice that
\begin{equation}
h = \begin{bmatrix}
a^2 & -ab & b^2 \\
ac & ad + bc & bd \\
c^2 & -ad & d^2
\end{bmatrix}
\end{equation}
Let $P = (xy)h$, $Q = (x^2)h$ and $R = (x^2 + y^2)h$ be points in $\PS(S_j)$. Then $P = (ac) x^2 + (ad + bc) xy + (bd) y^2$, and the equation $P =
x^2$ implies that $b = 0$ or $d = 0$. If $b = 0$ then $d = 0$ or $a =
0$ which is impossible since $\det{g} = 1$. Similarly, we cannot have
$d = 0$, and hence $xy$ and $x^2$ are not in the same
orbit.

Similarly, let $Q = a^2 x^2 - (ab) xy + b^2 y^2$, and then the equation $Q = x^2 + y^2$ implies that $a = 0$ or $b = 0$, which is impossible. Hence
$x^2$ and $x^2 + y^2$ are not in the same orbit.

Finally, let $R = (a^2 + c^2) x^2 - (ab + cd) xy + (b^2 + d^2) y^2$, and the equation $R = xy$ implies that $a^2 + c^2 = 0$. Since $-1$ is not
a square in $\F_q$, we must have $a = c = 0$. But this is impossible
since $\det{g} = 1$. Hence $x^2 + y^2$ and $xy$ are not in the same
orbit.

To verify the orbit sizes, consider the stabilisers of the three
points. Clearly the equation $Q = x^2$ implies that $b = 0$, so the
stabiliser of $x^2$ consists of the (projections of the) lower
triangular matrices. There are $q - 1$ choices for $a$ and $q$
choices for $c$, so the stabiliser has size $q(q - 1) / 2$ modulo
scalars, and the index in $H$ is therefore $q + 1$.

Similarly, the equation $P = xy$ implies that $ac = bd = 0$. If $a = d
= 0$ we obtain a matrix of order $2$ and if $b = c = 0$ we obtain a
diagonal matrix. It follows that the stabiliser is dihedral of order
$q - 1$, and that the pre-image in $\SL(2, q)$ is as in \eqref{sl2_stab_gens}.

Finally, in a similar way we obtain that the stabiliser of $x^2 + y^2$
has order $q + 1$, and hence the three orbits make up the whole of
$\PS(S_j)$.
\end{proof}

The algorithm that maps one point to another is given as Algorithm \ref{alg:find_mapping_element}.

\begin{figure}
\begin{codebox}
\refstepcounter{algorithm}
\label{alg:find_mapping_element}
\Procname{\kw{Algorithm} \ref{alg:find_mapping_element}: $\proc{FindMappingElement}(X, \Cent_G(j), P, Q)$}
\li \kw{Input}: Generating set $X$ for $G = \Ree(q)$.
\zi Points $P \neq Q \in \OV$ such that both $\varphi_{\OV}(P)$ and $\varphi_{\OV}(Q)$ are non-degenerate and 
\zi represent $0$. Involution centraliser $\Cent_G(j)$ with the maps from 
\zi Lemma \ref{lem_inv_centraliser_maps} and \ref{lem_psl_maps}.
\li \kw{Output}: An element $h \in G$, written as an $\SLP$ in $X$, such that $Ph = Q$.
\li $P_3 := \varphi_{\OV}(P)$ 
\li $Q_3 := \varphi_{\OV}(Q)$ 
\li \If $\exists \, \text{upper triangular} \ g \in \PSL(2, q)$ such that $P_3 \pi_3(g) = Q_3$ \label{find_mapping_element_tria_map}
\zi \Then
\li     $R_3 := \varphi_{\OV}(P \pi_7(\pi_3(g)))$ 
\li     \Comment Now $R_3 = Q_3$
\li     Find $c \in \GL(3, q)$ such that $(xy) c = R_3$ \label{find_mapping_element_zero_map}
\li     Let $D$ be the image in $\PSL(2, q)$ of the diagonal matrix in \eqref{sl2_stab_gens} 
\li     $s := \pi_7(\pi_3(D)^c)$
\li     \Comment Now $\gen{s} \leqslant \varphi_G^{-1}(H_{R_3})$
\li     $\delta, z := \proc{Diagonalise}(s)$ 
\li     \Comment Now $\delta = s^z$
\li     \If $\exists \, \lambda \in \F_q^{\times}$ such that $(P \pi_7(\pi_3(g)) z) h(\lambda) = Q z$ \label{find_mapping_element_diag_map}
\zi     \Then
\li          $i := \proc{DiscreteLog}(\delta, h(\lambda))$ 
\li          \Comment Now $\delta^i = h(\lambda)$
\li          \Return $\pi_7(\pi_3(g)) s^i$ \label{find_mapping_element_ret}
        \End
\zi     \kw{end}
    \End
\zi \kw{end}
\li \Return \const{fail}
\end{codebox}
\end{figure}

\subsubsection{Finding a stabilising element} \label{section:find_stab_element}
The complete algorithm for finding a uniformly random element of $G_P$ is then as follows, given a generating set $X$ for $G$ and $P \in \OV$.
\begin{enumerate}
\item Find an involution $j \in G$.
\item Compute probable generators for $\Cent_G(j)$ using the Bray algorithm, and probable generators for $\Cent_G(j)^{\prime}$ by taking commutators of the generators of $\Cent_G(j)$.
\item Use the MeatAxe to verify that the module for
  $\Cent_G(j)^{\prime}$ splits up only as in Proposition
  \ref{pr_inv_centraliser_split}. Use Theorem \ref{thm_psl_naming} to verify that we have the whole of $\Cent_G(j)^{\prime}$. Return to the previous step if not.
\item Compute the maps $\varphi_{\OV}$ and $\varphi_G$ using Lemma \ref{lem_inv_centraliser_maps}. Return to the first step if $P$ lies in the kernel of $\varphi_V$, if $\varphi_{\OV}(P)$ is degenerate, or if it does not represent $0$.
\item Compute the maps from Lemma \ref{lem_psl_maps}.
\item Take random $g_1 \in \Cent_G(j)^{\prime}$ and let $Q = Pg_1$. Then $\varphi_{\OV}(Q) = \varphi_{\OV}(P) \varphi_G(g_1)$, so $Q$ is not in the kernel of $\varphi_V$, and $\varphi_{\OV}(Q)$ is non-degenerate and represents $0$. Repeat until $P \neq Q$.
\item Use Algorithm \ref{alg:find_mapping_element} to find $g_2 \in \Cent_G(j)^{\prime}$ such that $Q = P g_2$. Return to the previous step if it fails, and otherwise return $g_1 g_2^{-1}$.
\end{enumerate}

\subsubsection{Correctness and complexity}

\begin{lem} \label{lem_zero_forms} If $P \in \OV$ is uniformly random,
  such that $P \nsubseteq \Ker(\varphi_V)$, then $\varphi_{\OV}(P)$ is non-degenerate and represents $0$ with probability at least $1/2 + \OR{1/q}$.
\end{lem}
\begin{proof}
  Since $P$ is uniformly random and $\varphi_{\OV}$ was chosen
  independently of $P$, it follows that $\varphi_{\OV}(P)$ is
  uniformly random from $\varphi_{\OV}(\OV)$. From the proof of Proposition
  \ref{pr_ree_ovoid_split}, with probability $1 - \OR{1/q}$, $\varphi_{\OV}(P) = x^2 + b xy + c y^2$ where $(1 : b : c)$ is uniformly distributed in $\PS^2(\F_q)$ such that $b^2 + c \neq 0$.
  
  This represents $0$ if the
  discriminant $b^2 - c$ is a non-zero square in $\F_q$. This is not
  the case if $b^2 = c$, but since $b^2 + c \neq 0$, this implies $b = c = 0$. If $b^2 - c \neq 0$ then it is a square with probability
  $1/2$, so
\[\Pr{b^2 - c \in (\F_q^{\times})^2} = \frac{1}{2} (1 - \frac{1}{q^2 - q})\]
and the Lemma follows.
\end{proof}

\begin{lem} \label{lem_tri_map} If $P, Q \in \varphi_{\OV}(\OV)$ are
  uniformly random, such that $\varphi_{\OV}(P)$ and
  $\varphi_{\OV}(Q)$ represent $0$, then the
  probability that there exists an element $g \in \PSL(2, q)$, such
  that the pre-image of $g$ in $\SL(2, q)$ is upper triangular and $P
  \pi_3(g) = Q$, is at least $1/2 + \OR{1/q}$.
\end{lem}
\begin{proof}
Let $P = x^2 + a xy + b y^2$, $Q = x^2 + l xy + n y^2$ and
\begin{equation}
g = \begin{bmatrix}
u & v \\
0 & 1/u
\end{bmatrix}
\end{equation}
where $(1 : a : b)$ and $(1 : l : n)$ are uniformly distributed in $\PS^2(\F_q)$, $u, v \in \F_q$ and $u \neq 0$.

We want to determine $u, v$ such that $P \pi_3(g) = Q$. Note that
$g$ is the pre-image in $\SL(2, q)$ of an element in $\PSL(2, q)$ and
therefore $\pm u$ determine the same element of $\PSL(2, q)$. The map $\pi_3$ is the symmetric square map, so
\begin{equation}
\pi_3(g) = \mathcal{S}^2(g) = \begin{bmatrix}
u^2 & -uv & v^2 \\
0 & 1 & v/u \\
0 & 0 & 1/u^2
\end{bmatrix}
\end{equation}

This leads to the following equations:
\begin{align}
u^2 &= C \label{tri_map_eq1} \\
-uv + a &= C l \label{tri_map_eq2} \\
v^2 + a v u^{-1} + b u^{-2} &= C n \label{tri_map_eq3}
\end{align}
for some $C \in \F_q^{\times}$. We can solve for $u$ in
\eqref{tri_map_eq1} and for $v$ in \eqref{tri_map_eq2}, so that
\eqref{tri_map_eq3} becomes
\begin{equation}
C^2 (n - m^2) + a^2 - b = 0
\end{equation}
This quadratic equation has a solution if the discriminant $-(n - m^2) (a^2 - b) \in (\F_q^{\times})^2$. This does not happen if $n = m^2$ or $b = a^2$, which each happens with probability $q/(q^2 + q + 1)$. If the discriminant is non-zero then it is a square with probability $1/2$. Therefore, the
probability that we can find $g$ is
\[ \Pr{-(n - m^2) (a^2 - b) \in (\F_q^{\times})^2} = \frac{1}{2}(1 - \frac{q}{q^2 + q + 1})^2\]
This is $1/2 + \OR{1/q}$ and the Lemma follows.
\end{proof}

\begin{thm}  \label{thm_find_mapping_element_pr}
If Algorithm \ref{alg:find_mapping_element} returns an element $g$, then $P g = Q$. If $P$ and $Q$ are uniformly random, such that $\varphi_{\OV}(P)$ and $\varphi_{\OV}(Q)$ represent $0$, then the probability that Algorithm \ref{alg:find_mapping_element} finds such an element is at least $1/4 + \OR{1/q}$.
\end{thm}
\begin{proof}
  By Proposition \ref{pr_psl_ovoid_action}, the point $R_3$ is in the
  same orbit as $xy$, so the element $c$ at line
  \ref{find_mapping_element_zero_map} can easily be found by
  diagonalising the form corresponding to $R_3$. Then $\pi_3(D)^c \in H_{R_3}$ is of
  order $(q - 1) / 2$. Hence $s$ also has order $(q - 1)/2$, and $s
  \in \varphi_G^{-1}(H_{R_3})$.
  
  By definition of $Q$, there exists $h
  \in \Cent_G(j)^{\prime}$ such that $Ph = Q$, and if we let $R = P \pi_7(g)$ then $R
  \pi_7(g)^{-1} h = Q$ and $\varphi_{\OV}(R) = R_3 = Q_3$. Hence
  $\varphi_G(\pi_7(g)^{-1} h) \in H_{Q_3}$, and therefore $\pi_7(g)^{-1} h
  \in \varphi_G^{-1}(H_{R_3})$.
  
  By Proposition \ref{pr_psl_ovoid_action}, $\varphi_G^{-1}(H_{R_3})$ is
  dihedral of order $q - 1$, and $s$ generates a subgroup of index
  $2$. Therefore $\Pr{\pi_7(g)^{-1} h \in \gen{s}} = 1/2$, which is
  the success probability of line \ref{find_mapping_element_diag_map}.
  
  It is straightforward to determine if $\lambda$ exists, since
  $h(\lambda)$ is diagonal. The success probability of line
  \ref{find_mapping_element_tria_map} is given by Lemma
  \ref{lem_tri_map}. Hence the success probability of the algorithm is
  as stated.
\end{proof}

\begin{thm} \label{thm_find_mapping_element_comp}
  Assume an oracle for the discrete logarithm problem in $\F_q$.
The time complexity of Algorithm \ref{alg:find_mapping_element} is $\OR{\log(q)^3 + \chi_D(q)}$ field operations. The length of the returned $\SLP$ is $\OR{\log(q) \log\log(q)}$.
\end{thm}
\begin{proof}
  By Lemma \ref{lem_tri_map}, line \ref{find_mapping_element_tria_map}
  involves solving a quadratic equation in $\F_q$, and hence uses
  $\OR{1}$ field operations. Evaluating the maps $\pi_3$ and $\pi_7$
  uses $\OR{\log(q)^3}$ field operations, and it is clear that the rest
  of the algorithm can be done using $\OR{\chi_D(q)}$ field operations.

  By Theorem \ref{thm_psl_recognition}, the length of the $\SLP$ from the
  constructive membership testing in $\PSL(2, q)$ is $\OR{\log(q) \log\log(q)}$,
  which is therefore also the length of the returned $\SLP$.
\end{proof}

\begin{cl} \label{cl_find_stab_element} Assume an oracle for the discrete logarithm problem in $\F_q$. There exists a Las Vegas algorithm
  that, given $\gen{X} \leqslant \GL(7, q)$ such that $G = \gen{X} =
  \Ree(q)$ and $P \in \OV$, computes a random element of $G_P$ as
  an $\SLP$ in $X$. 
The expected time complexity of the algorithm is
  $\OR{\xi  \log\log(q)  + \log(q)^3 + \chi_D(q)}$ field operations. The length of the returned
  $\SLP$ is $\OR{\log(q) \log\log(q)}$.
\end{cl}
\begin{proof}
  The algorithm is given in Section \ref{section:find_stab_element}.
  
  An involution is found by finding a random element and then use
  Proposition \ref{pr_element_powering}. Hence by Corollary
  \ref{cl_random_selections}, the expected time to find an involution
  is $\OR{\xi + \log(q) \log\log(q)}$ field operations.

  As described in Section \ref{section:inv_centraliser}, the Bray algorithm will
  produce uniformly random elements of the centraliser. Hence as
  described in Section \ref{section:random_elements}, we can
  also obtain uniformly random elements of its derived group. By
  Proposition \ref{psl_generation}, two random
  elements will generate $\PSL(2, q)$ with high probability. This
  implies that the expected time to obtain probable generators for
  $\PSL(2, q)$ is $\OR{1}$ field operations.
  
  By Proposition \ref{pr_psl_ovoid_action}, the point $Q$ is equal to $P$ with probability $2/(q(q+1))$ and by
  Lemma \ref{lem_zero_forms} the point $P$ do not represent zero
  with probability $1/2$, so the expected time of the penultimate step
  is $\OR{1}$ field operations.

  Since the points $P, Q$ can be considered uniformly random and
  independent in Algorithm \ref{alg:find_mapping_element}, the element
  returned by that algorithm is uniformly random. Hence the element
  returned by the algorithm in Section \ref{section:find_stab_element}
  is uniformly random.
  
  The expected time complexity of the last step is given by Theorem
  \ref{thm_find_mapping_element_comp} and
  \ref{thm_find_mapping_element_pr}. It follows by the above and from
  Lemma \ref{lem_inv_centraliser_maps} and \ref{lem_psl_maps} and
  Corollary \ref{ree_totient_prop} that the expected time complexity
  of the algorithm in Section \ref{section:find_stab_element} is as
  stated.

  The algorithm is clearly Las Vegas, since it
  is straightforward to check that the element we compute really fixes
  the point $P$.
\end{proof}

\begin{rem}
  The elements returned by the algorithm in Corollary
  \ref{cl_find_stab_element} are not uniformly random from the
  \emph{whole} of $G_P$, but from $G_P \cap \Cent_G(j)$. Hence, to obtain generators for the whole stabiliser, it is necessary to execute the algorithm at least twice, with different choices of the involution $j$.
\end{rem}

\begin{rem} \label{ree_rem_stab_elt} The algorithm in Corollary
  \ref{cl_find_stab_element} works in any conjugate of $\Ree(q)$,
  since it does not assume that the matrices lie in the standard copy.
\end{rem}

\subsection{Constructive membership testing}
\label{section:ree_constructive_membership}

We now describe the constructive membership algorithm for our standard
copy $\Ree(q)$. The matrix degree is constant here, so we set $\xi =
\xi(7)$. Given a set of generators $X$, such that $G = \gen{X} =
\Ree(q)$, and given an element $g \in G$, we want to express $g$ as an
$\SLP$ in $X$. Membership testing is straightforward, using the first
step from the algorithm in Theorem \ref{thm_ree_standard_recognition},
and will not be considered here.

The general structure of the algorithm is the same as the algorithm for the
same problem in the Suzuki groups. It consists of a
preprocessing step and a main step.

\subsubsection{Preprocessing} \label{section:preprocessing_step}
The preprocessing step consists of finding
\lq \lq standard generators'' for $\mathrm{O}_3(G_{P_{\infty}}) = U(q)$ and
$\mathrm{O}_3(G_{P_0})$. In the case of
$\mathrm{O}_3(G_{P_{\infty}})$ the standard generators are defined as matrices 
\begin{equation}
\set{S(a_i, x_i, y_i)}_{i = 1}^{n} \cup \set{S(0, b_i, z_i)}_{i = 1}^{n} \cup \set{S(0, 0, c_i)}_{i = 1}^n
\end{equation}
for some unspecified $x_i, y_i, z_i \in \F_q$, such that $\set{a_1, \dotsc, a_n}$,
$\set{b_1, \dotsc, b_n}$, $\set{c_1, \dotsc, c_n}$ form vector space bases of $\F_q$ over $\F_3$
(so $n = \log_3{q} = 2m + 1$). 

\begin{lem} \label{ree_row_operations}
There exist algorithms for the following row reductions.
\begin{enumerate}
\item Given $g = h(\lambda) S(a, b, c) \in G_{P_{\infty}}$, find $h \in \mathrm{O}_3(G_{P_{\infty}})$ expressed in the standard generators, such that $gh = h(\lambda)$.
\item Given $g = S(a, b, c) h(\lambda) \in G_{P_{\infty}}$, find $h \in \mathrm{O}_3(G_{P_{\infty}})$ expressed in the standard generators, such that $hg = h(\lambda)$.
\item Given $P_{\infty} \neq P \in \OV$, find $g \in \mathrm{O}_3(G_{P_{\infty}})$ expressed in the standard generators, such that $Pg = P_0$.
\end{enumerate}
Analogous algorithms exist for $G_{P_0}$. If the standard generators are expressed as $\SLP$s of length $\OR{n}$, the elements returned will have length $\OR{n \log(q)}$. The time complexity of the algorithms is $\OR{\log(q)^3}$ field operations. 
\end{lem}
\begin{proof} 
The algorithms are as follows.

\begin{enumerate}
\item \begin{enumerate}
\item Solve a linear system of size $\log(q)$ to construct the linear
  combination $-a = -g_{2, 1} / g_{2, 2} = \sum_{i = 1}^{2m + 1}
  \alpha_i a_i$ with $\alpha_i \in \F_3$. Let $h^{\prime} = \prod_{i =
    1}^{2m + 1} S(a_i, x_i, y_i)^{\alpha_i}$ and $g^{\prime} = g
  h^{\prime}$, so that $g^{\prime} = h(\lambda) S(0,
  b^{\prime}, c^{\prime})$ for some $b^{\prime}, c^{\prime} \in \F_q$.
\item Solve a linear system of size $\log(q)$ to construct the linear
  combination $-b^{\prime} = -g^{\prime}_{3, 1} / g^{\prime}_{3, 3} =
  \sum_{i = 1}^{2m + 1} \beta_i b_i$ with $\beta_i \in \F_3$.  Let
  $h^{\prime \prime} = \prod_{i = 1}^{2m + 1} S(0, b_i, y_i)^{\beta_i}$ and
  $g^{\prime \prime} = g^{\prime} h^{\prime\prime}$, so that
  $g^{\prime\prime} = h^{\prime}(\lambda) S(0, 0, c^{\prime\prime})$ for some $c^{\prime\prime} \in \F_q$.
\item Solve a linear system of size $\log(q)$ to construct the linear
  combination $-c^{\prime\prime} = -g^{\prime\prime}_{4, 1} / g^{\prime\prime}_{4, 4} =
  \sum_{i = 1}^{2m + 1} \gamma_i c_i$ with $\gamma_i \in \F_3$.  Let
  $h^{\prime \prime\prime} = \prod_{i = 1}^{2m + 1} S(0, 0, z_i)^{\gamma_i}$ and
  $g^{\prime \prime \prime} = g^{\prime \prime} h^{\prime\prime \prime}$, so that
  $g^{\prime\prime \prime} = h^{\prime}(\lambda)$.

\item Now $h = h^{\prime} h^{\prime\prime} h^{\prime\prime\prime}$.
\end{enumerate}
\item Analogous to the previous case.
\item 
\begin{enumerate}
\item Normalise $P$ so that $P = (1 : p_1 : \dotsb : p_6)$. 
\item Let $\alpha = -p_1^{3t}$, $\beta = (p_1 \alpha + p_2)^{3t}$ and $\gamma = ((\alpha \beta)^t + p_1 (\alpha^{t + 1} + \beta^t) + p_2 \alpha^t + p_3)^{3t}$. Then $S(\alpha, \beta, \gamma)$ maps $P$ to $P_{\infty}$.
\item Use the algorithm above to find $h \in \mathrm{O}_3(G_{P_{\infty}})$ such that $S(\alpha, \beta, \gamma) h = 1$, and hence express $S(\alpha, \beta, \gamma)$ in the standard generators.
\end{enumerate}
\end{enumerate}
Clearly the dominating term in the time complexity is the solving of the linear systems, which requires $\OR{\log(q)^3}$ field operations. The elements returned are constructed using $\OR{\log(q)}$
multiplications, hence the length of the $\SLP$ follows.
\end{proof}


\begin{thm} \label{thm_pre_step} Given an oracle for the discrete logarithm problem in $\F_q$, the
  preprocessing step is a Las Vegas algorithm that finds standard
  generators for $\mathrm{O}_3(G_{P_{\infty}})$ and
  $\mathrm{O}_3(G_{P_0})$ as $\SLP$s in $X$ of length $\OR{\log(q) (\log\log(q))^2}$. It has expected time complexity $\OR{(\xi \log\log(q) + \log(q)^3 + \chi_D(q)) \log\log(q)}$ field operations. 
\end{thm}
\begin{proof}
The preprocessing step proceeds as follows.
\begin{enumerate}
\item Find random elements $a_1 \in G_{P_{\infty}}$ and $b_1
  \in G_{P_0}$ using the algorithm from Corollary
  \ref{cl_find_stab_element}. Repeat until $a_1$ can be diagonalised to $h(\lambda) \in
G$, where $\lambda \in \F_q^{\times}$ and $\lambda$ 
does not lie in a proper subfield of $\F_q$. Do similarly for $b_1$. 
\item Find random elements $a_2 \in G_{P_{\infty}}$ and $b_2
  \in G_{P_0}$ using the algorithm from Corollary
  \ref{cl_find_stab_element}. Let $c_1 = [a_1, a_2]$, $c_2 = [b_1,
  b_2]$. Repeat until $\abs{c_1} = \abs{c_2} = 9$.

\item Let $Y_{\infty} = \set{c_1, a_1}$ and $Y_0 = \set{c_2, b_1}$. As standard generators for $\mathrm{O}_3(G_{P_{\infty}})$ we now take $U = U_1 \cup U_2$ where 
\begin{equation} \label{standard_gens1}
U_1 = \bigcup_{i = 1}^{2m + 1} \set{c_1^{a_1^i}, (c_1^3)^{a_1^i}}
\end{equation}
and
\begin{equation} \label{standard_gens2}
U_2 = \bigcup_{1 \leqslant i < j \leqslant 2m + 1} \set{[c_1^{a_1^i}, c_1^{a_1^j}]}
\end{equation}
Similarly we obtain $L$ for $\mathrm{O}_3(G_{P_0})$.
\end{enumerate}

It follows from \eqref{ree_matrix_id1} and \eqref{ree_matrix_id2} that
\eqref{standard_gens1} and \eqref{standard_gens2} provides the standard generators for
$G_{P_{\infty}}$. These are expressed as $\SLP$s in $X$, since this is
true for the elements returned from the algorithm described in Corollary
\ref{cl_find_stab_element}. Hence the algorithm is Las Vegas.

By Corollary \ref{cl_find_stab_element}, the expected time to find
$a_1$ and $b_1$ is $\OR{\xi \log\log(q) + \log(q)^3 + \chi_D(q)}$, and
these are uniformly distributed independent random elements. The
elements of order dividing $q - 1$ can be diagonalised as required.
By Proposition \ref{sylow3_props}, the proportion of elements of
order $q - 1$ in $G_{P_{\infty}}$ and $G_{P_0}$ is $\phi(q - 1) / (q - 1)$. Hence the expected time for the first step is
$\OR{(\xi \log\log(q) + \log(q)^3 + \chi_D(q)) \log\log(q)}$ field operations.

Similarly, by Proposition \ref{ree_prop_frobenius} the expected time for
the second step is \[\OR{(\xi \log\log(q) + \log(q)^3 + \chi_D(q)) \log\log(q)}\] field
operations.



By the remark preceding the Theorem, $L$ determines three sets of
field elements $\set{a_1, \dotsc, a_{2m+1}}$, $\set{b_1, \dotsc,
  b_{2m+1}}$ and $\set{c_1, \dotsc, c_{2m+1}}$. By \eqref{ree_matrix_id2},
in this case each $a_i = a \lambda^i$, $b_i = b \lambda^{i(t + 2)}$
and $c_i = c \lambda^{i(t + 3)}$, for some fixed $a,b,c \in
\F_q^{\times}$, where $\lambda$ is as in the algorithm. Since
$\lambda$ does not lie in a proper subfield, these sets form vector
space bases of $\F_q$ over $\F_3$.

To determine if $a_1$ or $b_1$ diagonalise to some $h(\lambda)$ it is sufficient to consider the eigenvalues on the diagonal, since both $a_1$ and $b_1$ are triangular. To determine if $\lambda$ lies
in a proper subfield, it is sufficient to determine if $\abs{\lambda} \mid
3^n - 1$, for some proper divisor $n$ of $2m + 1$. Hence the
dominating term in the complexity is the first step.
\end{proof}

\subsubsection{Main algorithm}
Given $g \in G$ we now show the procedure for expressing $g$ as an $\SLP$. It is given as Algorithm \ref{alg:ree_element_to_slp}.

\begin{figure}[ht]
\begin{codebox}
\refstepcounter{algorithm}
\label{alg:ree_element_to_slp}
\Procname{\kw{Algorithm} \ref{alg:ree_element_to_slp}: $\proc{ElementToSLP}(U, L, g)$}
\li \kw{Input}: Standard generators $U$ for $G_{P_{\infty}}$ and $L$ for $G_{P_0}$. Matrix $g \in \gen{X} = G$.
\li \kw{Output}: $\SLP$ for $g$ in $X$
\li \Repeat
\li     \Repeat
\li         $r := \proc{Random}(G)$ 
\li     \Until{$gr$ has an eigenspace $Q \in \OV$ and $P \neq Q$ \label{main_alg_find_point_ree}} 
\li     Find $z_1 \in G_{P_{\infty}}$ using $U$ such that $Qz_1 = P_0$. \label{main_alg_row_op1_ree} 
\li     \Comment Now $(gr)^{z_1} \in G_{P_0}$
\li     Find $z_2 \in G_{P_0}$ using $L$ such that $(gr)^{z_1} z_2 = h(\lambda)$ for some $\lambda \in \F_q^{\times}$ \label{main_alg_row_op2_ree} 
\li     $x := \Tr(h(\lambda))$ 
\li \Until $x - 1$ is a square in $\F_q^{\times}$ \label{main_alg_square_test}
\li     \Comment{Express diagonal matrix as $\SLP$}
\li Find $u = S(0, 0, \sqrt{(x - 1)^{3t}})S(0, 1, 0)^{\Upsilon}$ using $U \cup L$ \label{main_alg_row_op3_ree} 
\li \Comment{Now $\Tr(u) = x$}
\li Let $P_1, P_2 \in \OV$ be the fixed points of $u$
\li Find $a \in G_{P_{\infty}}$ using $U$ such that $P_1 a = P_0$ \label{main_alg_row_op4_ree} 
\li Find $b \in G_{P_0}$ using $L$ such that $(P_2 a)b = P_{\infty}$ \label{main_alg_row_op5_ree} 
\li \Comment{Now $u^{ab} \in G_{P_{\infty}} \cap G_{P_0} = H(q)$, so $u^{ab} \in \set{h(\lambda)^{\pm 1}}$}
\li \If $u^{ab} = h(\lambda)$ 
\zi \Then
\li     Let $w$ be the $\SLP$ for $(u^{ab} z_2^{-1})^{z_1^{-1}} r^{-1}$ \label{main_alg_get_slp1_ree} 
\li     \Return $W$
\zi \Else
\li     Let $w$ be the $\SLP$ for $((u^{ab})^{-1} z_2^{-1})^{z_1^{-1}} r^{-1}$ \label{main_alg_get_slp2_ree} 
\li     \Return $w$
    \End 
\zi \kw{end}
\end{codebox}
\end{figure}

\subsubsection{Correctness and complexity}

\begin{thm} \label{thm_element_to_slp_ree}
Algorithm \ref{alg:ree_element_to_slp} is correct, and is a Las Vegas algorithm.
\end{thm}
\begin{proof}
  First observe that since $r$ is randomly chosen, we obtain it
  as an $\SLP$. 

The elements found at lines \ref{main_alg_row_op1_ree} and
\ref{main_alg_row_op2_ree} can be computed using Lemma \ref{ree_row_operations}, so we can obtain them as $\SLP$s.

The element $u$ found at line \ref{main_alg_row_op3_ree} clearly has trace
$x$.
Because $u$ can be computed using Lemma \ref{ree_row_operations}, we
obtain it as an $\SLP$. From Proposition \ref{ree_conjugacy_classes}
we know that $u$ is conjugate to $h(\lambda)^{\pm 1}$ and therefore
must fix two points of $\OV$. Hence lines \ref{main_alg_row_op4_ree}
and \ref{main_alg_row_op5_ree} make sense, and the elements found can
again be computed using Lemma \ref{ree_row_operations}, so we obtain
them as $\SLP$s.

Finally, the elements that make up $w$ have
been found as $\SLP$s, and it is clear that if we evaluate $w$ we obtain
$g$. Hence the algorithm is Las Vegas and the Theorem follows.
\end{proof}

\begin{thm} \label{thm_element_to_slp_complexity}
Algorithm \ref{alg:ree_element_to_slp} has expected time complexity $\OR{\xi + \log(q)^3}$ field
operations and the length of the
returned $\SLP$ is $\OR{(\log(q) \log\log(q))^2}$.
\end{thm}
\begin{proof}
It follows immediately from Lemma \ref{ree_row_operations} that the lines
\ref{main_alg_row_op1_ree}, \ref{main_alg_row_op2_ree},
\ref{main_alg_row_op3_ree}, \ref{main_alg_row_op4_ree} and
\ref{main_alg_row_op5_ree} use $\OR{\log(q)^3}$ field operations.

From Corollary \ref{cl_random_selections}, the expected time to find $r$ is $\OR{\xi}$ field operations. Half of the elements of $\F_q^{\times}$ are squares, and $x$ is uniformly random, hence the expected time of the outer repeat statement is $\OR{\xi + \log(q)^3}$ field operations.

Finding the fixed points of $u$, and performing the check at line
\ref{main_alg_find_point_ree} only amounts to considering eigenvectors,
which is $\OR{\log{q}}$ field operations. Thus the expected time complexity of
the algorithm is $\OR{\xi + \log(q)^3}$ field operations.

From Theorem \ref{thm_pre_step} each standard generator $\SLP$ has length $\OR{\log(q)(\log\log(q))^2}$ and hence $w$ will have length $\OR{(\log(q)\log\log(q))^2}$.
\end{proof}

\subsection{Conjugates of the standard copy}
\label{section:ree_conjugacy}

Now assume that we are given a conjugate $G$ of $\Ree(q)$, and we turn to the problem of finding some $g \in\GL(7, q)$ such
that $G^g = \Ree(q)$, thus obtaining an algorithm that finds effective
isomorphisms from any conjugate of $\Ree(q)$ to the standard copy. The matrix degree is constant here, so we set $\xi = \xi(7)$.

\begin{lem} \label{lem_find_ovoid_point_ree} 
There exists a Las Vegas
  algorithm that, given $\gen{X} \leqslant \GL(7, q)$ such that $\gen{X}^h =
  \Ree(q)$ for some $h \in \GL(7, q)$, finds a point $P \in
  \OV^{h^{-1}} = \set{Q h^{-1} \mid Q \in \OV}$. The algorithm has expected time complexity 
\[\OR{(\xi + \log(q) \log \log(q)) \log \log(q)}\] field operations.
\end{lem}
\begin{proof}
  Clearly $\OV^{h^{-1}}$ is the set on which $\gen{X}$ acts doubly
  transitively. For a matrix $h(\lambda) \in \Ree(q)$ we see
  that the eigenspaces corresponding to the eigenvalues $\lambda^{\pm
    t}$ will be in $\OV$. Moreover, every element of order
  dividing $q-1$, in every conjugate $G$ of $\Ree(q)$, will have
  eigenvalues of the form $\mu^{\pm t}, \mu^{\pm (1 - t)}, \mu^{\pm (2t - 1)}$,  for some
  $\mu \in \F_q^{\times}$, and the eigenspaces corresponding to $\mu^{\pm t}$ will lie in the set on which $G$ acts doubly transitively.

  Hence to find a point $P \in \OV^{h^{-1}}$ it is sufficient to find a
  random $g \in \gen{X}$ of order dividing $q - 1$. We compute the order using expected $\OR{\log(q) \log\log(q)}$ field operations, and by 
  Corollary \ref{cl_random_selections}, the expected number of iterations to find the element is $\OR{\log\log(q)}$ field operations. We then find the eigenspaces of $g$.

  Clearly this is a Las Vegas algorithm with the stated time complexity.
\end{proof}

\begin{conj} \label{conj:ree_conjugacy}
Let $P = (1 : p_2 : \dotsb : p_7) \in \OV^x$ and $Q = (1 : q_2 : \dotsb : q_7) \in \OV^x$ for some $x \in \GL(7, q)$. Then for all $a, b, c, s \in \F_q^{\times}$, the ideal in $\F_q[\alpha, \beta, \gamma, \delta]$ generated by
\begin{align}
& p_3^{3t} \beta \delta - p_2^{3t + 3} \gamma \alpha^3 \beta + p_2 s \alpha \beta + p_2^2 p_3 \alpha^2 \beta - p_5 c \; ; \\
& (p_4 s)^{3t} + (p_2 p_3)^{3t} \gamma \delta \alpha - p_3 p_4 s \alpha \beta + p_2 p_3^2 \alpha^2 \beta^2 - p_2^{6t} p_2^3 \gamma^2 \alpha^4 - p_6 b \; ;\\
\begin{split}
& -(p_4 s)^{3t} p_2 \alpha - p_2^{3t + 1} p_3^{3t} \gamma \delta \alpha - p_3^{3t + 1} \gamma \beta + p_2^{6t + 4} \gamma^2 \alpha^4 + \\
& (p_2 p_3)^2 \alpha^2 \beta^2 + p_2^{3t + 3} p_3 \gamma \alpha^3 \beta - p_7 a - (p_4 s)^2 
\end{split}
\end{align}
as well as the corresponding $3$ polynomials from $Q$, is zero-dimensional.
\end{conj}
  

\begin{lem} \label{lem_diagonal_conj_ree}
  Assume Conjecture \ref{conj:ree_conjugacy}. There exists a Las Vegas algorithm that, given $\gen{X} \leqslant \GL(7, q)$ such that $\gen{X}^d = \Ree(q)$ where
  $d = \diag(d_1, d_2, d_3, d_4, d_5, d_6, d_7) \in \GL(7, q)$, finds a diagonal matrix $e
  \in \GL(7, q)$ such that $\gen{X}^e = \Ree(q)$. The expected time complexity is $\OR{\abs{X} + (\xi + \log (q) \log \log(q)) \log \log(q)}$ field operations.
\end{lem}
\begin{proof}
Let $G = \gen{X}$. Since $G^d = \Ree(q)$, $G$ must preserve the symmetric bilinear form
\begin{equation}
K = d J d = \begin{bmatrix}
0 & 0 & 0 & 0 & 0 & 0 & d_1 d_7 \\
0 & 0 & 0 & 0 & 0 & d_2 d_6 & 0 \\
0 & 0 & 0 & 0 & d_3 d_5 & 0 & 0 \\
0 & 0 & 0 & -d_4^2 & 0 & 0 & 0 \\
0 & 0 & d_3 d_5 & 0 & 0 & 0 & 0 \\
0 & d_2 d_6 & 0 & 0 & 0 & 0 & 0 \\
d_1 d_7 & 0 & 0 & 0 & 0 & 0 & 0
\end{bmatrix}
\end{equation}
where $J$ is given by \eqref{standard_bilinear_form}. Using the MeatAxe, we can find this form, which is determined up
to a scalar multiple. In the case where $-K_{4,4}$ turns out to be a non-square in $\F_q$ we can therefore multiply $K$ with a non-square scalar matrix. The diagonal matrix $e = \diag(e_1, e_2, e_3, e_4, e_5, e_6, e_7)$
that we want to find is also determined up to a scalar multiple (and up to multiplication by a diagonal matrix in $\Ree(q)$).

Since $e$ must take $J$ to $K$, we must have
$K_{1, 7}= d_1 d_7 = e_1 e_7$, $K_{2, 6} = d_2 d_6 = e_2 e_6$, $K_{3, 5} = d_3 d_5 = e_3 e_5$ and $K_{4,4} = -d_4^2 = -e_4^2$. Because $e$ is determined up to a scalar multiple, we can choose
$e_1 = 1$, $e_7 = K_{1, 7}$ and $e_4 = s = \sqrt{-K_{4,4}}$. Furthermore, $e_5 = K_{3,5} / e_3$ and $e_6 = K_{2,6} / e_2$ so it only remains to determine
$e_2$ and $e_3$.

To conjugate $G$ into $\Ree(q)$, we must have $Pe \in \OV$ for every $P \in \OV^{d^{-1}}$, which is the set on which $G$ acts doubly
transitively. By Lemma \ref{lem_find_ovoid_point_ree}, we can find $P
= (1 : p_2 : p_3 : p_4 : p_5 : p_6 : p_7) \in \OV^{d^{-1}}$, and the
condition $P e = (1 : p_2 e_2 : p_3 e_3 : p_4 s : p_5 K_{3, 5} / e_3 :
p_6 K_{2,6} / e_2 : p_7 K_{1,7}) \in \OV$ is given by
\eqref{ree_ovoid_def} and
amounts to the polynomial equations given in Conjecture
\ref{conj:ree_conjugacy}, with $\alpha = e_2$, $\beta = e_3$, $\gamma
= e_2^t$, $\delta = e_3^t$, $K_{1, 7} = a$, $K_{2, 6} = b$, $K_{3, 5}
= c$.

By finding another random point $Q \in \OV^{d^{-1}}$ using Lemma
\ref{lem_find_ovoid_point_ree}, we obtain $6$ polynomials which we
label $1P$, $2P$, $3P$, $1Q$, $2Q$ and $3Q$. Conjecture
\ref{conj:ree_conjugacy} asserts that the resulting ideal is
zero-dimensional. 

Now it follows from Proposition \ref{ree_ovoid_points} that with high probability we have $p_3
\neq 0$ and $p_3^{3t} q_2^{3t+3} \neq q_3^{3t} p_2^{3t+3}$, so we
repeat until $P$ and $Q$ satisfy this. Then we can solve for $\delta$ from $1P$ and $\gamma$ from $1Q$, as
\begin{align}
\delta &= \frac{\alpha^3 \gamma \beta p_2^{3t + 3} - \alpha^2 \beta p_2^2 p_3 - \alpha \beta s p_2 p_4 + c p_5}{p_3^{3t} \beta} \\
\gamma &= \frac{\alpha^2 \beta p_3^{3t} q_2^2 q_3 - \alpha^2 \beta p_2^2 p_3 q_3^{3t} - \alpha \beta s p_2 p_4 q_3^{3t} + c p_5 q_3^{3t} + \alpha \beta s p_3^{3t} q_2 q_4 - c p_3^{3t} q_5}{\alpha^3 \beta (p_3^{3t} q_2^{3t + 3} - p_2^{3t + 3} q_3^{3t})}
\end{align}

and if we substitute these into the other equations we obtain $4$ polynomials in $\alpha$ and $\beta$, which generate a zero-dimensional ideal. The variety of this ideal can now be found using Theorem \ref{thm_poly_eqns_2vars}.

Hence we can find $e_2$ and $e_3$. The diagonal matrix \[e = \diag(1, e_2, e_3, s, 
K_{3, 5} e_3^{-1}, K_{2, 6} e_2^{-1}, K_{1, 7})\] now satisfies $G^e = \Ree(q)$.

By Lemma \ref{lem_find_ovoid_point_ree}, Lemma
\ref{conj:ree_conjugacy}, Theorem
\ref{thm_poly_eqns_2vars} and Section
\ref{section:meataxe}, this is a Las
Vegas algorithm with the stated time complexity.
\end{proof}

\begin{lem} \label{lem_conjugate_to_digaonal_ree}
  There exists a Las Vegas algorithm that, given subsets $X$, $Y_P$
  and $Y_Q$ of $\GL(7, q)$ such that $\mathrm{O}_3(G_P) < \gen{Y_P} \leqslant
  G_P$ and $\mathrm{O}_3(G_Q) < \gen{Y_Q} \leqslant G_Q$, respectively, where
  $\gen{X} = G$, $G^h = \Ree(q)$ for some $h \in \GL(7, q)$ and $P, Q
  \in \OV^{h^{-1}}$, finds $k \in \GL(7, q)$ such that $(G^k)^d =
  \Ree(q)$ for some diagonal matrix $d \in \GL(7, q)$.  The algorithm
  has expected time complexity $\OR{\abs{X}}$ field operations.
\end{lem}
\begin{proof}

  Notice that the natural module $V = \F_q^7$ of $U(q) H(q)$ is
  uniserial with seven non-zero submodules, namely 
\[ V_i
  = \set{(v_1, v_2, v_3, v_4, v_5, v_6, v_7) \in \F_q^7 \mid v_j = 0,
 j > i}\]
 for $i = 1, \dotsc, 7$. Hence the same is true for $\gen{Y_P}$ and
  $\gen{Y_Q}$ (but the submodules will be different) since they lie in
  conjugates of $U(q) H(q)$.

  Now the algorithm proceeds as follows.
  
\begin{enumerate}
\item Let $V = \F_q^7$ be the natural module for $\gen{Y_P}$ and
  $\gen{Y_Q}$. Find composition series $V = V^P_7 > V^P_6 > V^P_5 > V^P_4 > V^P_3 > V^P_2 > V^P_1$ and $V = V^Q_7 > V^Q_6 > V^Q_5 > V^Q_4 > V^Q_3 >
  V^Q_2 > V^Q_1$ using the MeatAxe.

\item Let $U_1 = V_1^P$, $U_2 = V_2^P \cap V_6^Q$, $U_3 = V_3^P \cap V_5^Q$, $U_4 = V_4^P \cap V_4^Q$, $U_5 = V_5^P \cap V_3^Q$, $U_6 = V_6^P \cap V_2^Q$ and $U_7 = V_1^Q$. For each $i = 1, \dotsc, 7$, choose $u_i \in U_i$.

\item Now let $k$ be the matrix such that $k^{-1}$ has $u_i$ as row $i$, for $i = 1, \dotsc, 7$.
\end{enumerate}

The motivation for the second step is analogous to the proof of Theorem \ref{lem_conjugate_to_digaonal}.




Thus the matrix $k$ found in the algorithm satisfies that $z = kd$ for
some diagonal matrix $d \in \GL(7, q)$. Since $\Ree(q) = G^h = G^z =
(G^k)^d$, the algorithm returns a correct result, and it is Las Vegas
because the MeatAxe is Las Vegas. Clearly it has the same time
complexity as the MeatAxe.
\end{proof}

\begin{thm} \label{cl_ree_conjugacy} Assume Conjecture
  \ref{conj:ree_conjugacy} and an oracle for the discrete logarithm
  problem in $\F_q$. There exists a Las Vegas algorithm that, given a
  conjugate $\gen{X}$ of $\Ree(q)$, finds $g \in \GL(7, q)$ such that
  $\gen{X}^g = \Ree(q)$. The algorithm has expected time complexity
  $\OR{(\xi \log \log(q) + \log(q)^3 + \chi_D(q))\log\log(q) +
    \abs{X}}$ field operations.
\end{thm}
\begin{proof}
  Let $G = \gen{X}$. By Remark \ref{ree_rem_stab_elt}, we can use
  Corollary \ref{cl_find_stab_element} in $G$, so we can find
  generators for a stabiliser of a point in $G$, using the algorithm
  described in Theorem \ref{thm_pre_step}.

\begin{enumerate}
\item Find points $P, Q \in \OV^{h^{-1}}$ using Lemma
  \ref{lem_find_ovoid_point_ree}. Repeat until $P \neq Q$.
\item Find generating sets $Y_P$ and $Y_Q$ such that $\mathrm{O}_3(G_P) <
  \gen{Y_P} \leqslant G_P$ and $\mathrm{O}_3(G_Q) < \gen{Y_Q} \leqslant G_Q$
  using the first two steps of the algorithm from the proof of Theorem
  \ref{thm_pre_step}.
\item Find $k \in \GL(7, q)$ such that $(G^k)^d = \Ree(q)$ for some diagonal matrix $d \in \GL(7, q)$, using Lemma \ref{lem_conjugate_to_digaonal_ree}.
\item Find a diagonal matrix $e$ using Lemma \ref{lem_diagonal_conj_ree}.
\item Now $g = ke$ satisfies that $G^g = \Ree(q)$.
\end{enumerate}

Be Lemma \ref{lem_find_ovoid_point}, \ref{lem_conjugate_to_digaonal_ree} and \ref{lem_diagonal_conj_ree}, and the proof of Theorem \ref{thm_pre_step}, this is a Las Vegas algorithm with time complexity as stated.
\end{proof}

\subsection{Tensor decomposition}
\label{section:ree_tensor_decompose}

Now assume that
$G \leqslant \GL(d, q)$ where $G \cong \Ree(q)$, $d > 7$ and $q = 3^{2m +
  1}$ for some $m > 0$. Then $\Aut{\F_q} = \gen{\psi}$, where $\psi$
is the Frobenius automorphism. Let $W$ be the given module of $G$ and let $V$ be the natural
module of $\Ree(q)$, so that $\dim W = d$ and
$\dim V = 7$. From Section \ref{section:algorithm_overview} and Section \ref{section:ree_indecomposables} we know that
\begin{equation} \label{ree_irreducible_module_form}
W \cong M^{\psi^{i_0}} \otimes M^{\psi^{i_1}} \otimes \dotsm \otimes
M^{\psi^{i_{n - 1}}}
\end{equation}
for some integers $0 \leqslant i_0 < i_1 < \dotsb < i_{n - 1}
\leqslant 2m$, and where $M$ is either $V$ or the absolutely
irreducible $27$-dimensional submodule $S$ of the symmetric square
$\mathcal{S}^2(V)$. In fact, we may assume that $i_0 = 0$. As described
in Section
\ref{section:algorithm_overview}, we now want to tensor decompose $W$
to obtain an effective isomorphism from $W$ to $V$ or to $S$. In the
latter case we also have to decompose $S$ into $V$ to obtain an
isomorphism between $W$ and $V$. We consider this problem in Section
\ref{section:decompose_symsquare}.

\begin{pr} \label{pr_27_dim_split} Let $G \leqslant \GL(27, q)$ such
  that $G \cong \Ree(q)$, let $j \in G$ be an involution and let $H =
  \Cent_G(j)^{\prime} \cong \PSL(2, q)$. Then $S\vert_H \cong V_6
  \oplus V_9 \oplus V_{12}$ as an $H$-module, where $\dim V_i = i$.
  Moreover, $V_9$ is absolutely irreducible, $V_{12} \cong
  V_4.V_4^{\psi^l} . V_4$ and $V_6 \cong 1.V_4^{\psi^k}.1$, where $k \neq l$.
\end{pr}
\begin{proof}
  By Proposition \ref{pr_inv_centraliser_split}, $V\vert_H = V_3 \oplus V_4$ and hence
  $V_3 = \mathcal{S}^2(V_2)$ and $V_4 = V_2 \otimes V_2^{\psi^n}$
  where $V_2$ is the natural module of $\PSL(2, q)$ and $n > 0$.
\begin{equation}
1 \oplus S\vert_H = \mathcal{S}^2(V\vert_H) = \mathcal{S}^2(\mathcal{S}^2(V_2)) \oplus \mathcal{S}^2(V_2 \otimes V_2^{\psi^n}) \oplus (\mathcal{S}^2(V_2) \otimes V_2 \otimes V_2^{\psi^n})
\end{equation}

Now consider 
\begin{equation}
\begin{split}
&\mathcal{S}^2(V_2 \otimes V_2^{\psi^n}) \oplus \wedge^2(V_2 \otimes V_2^{\psi^n}) = (V_2 \otimes V_2^{\psi^n}) \otimes (V_2 \otimes V_2^{\psi^n}) = \\
&= (V_2 \otimes V_2) \otimes (V_2^{\psi^n} \otimes V_2^{\psi^n}) = (\mathcal{S}^2(V_2) \oplus 1) \otimes (\mathcal{S}^2(V_2^{\psi^n}) \oplus 1) = \\
&=1 \oplus \mathcal{S}^2(V_2) \oplus \mathcal{S}^2(V_2^{\psi^n}) \oplus (\mathcal{S}^2(V_2) \otimes \mathcal{S}^2(V_2^{\psi^n}))
\end{split}
\end{equation}
The final summand has dimension $9$ and is absolutely irreducible because it is of the form \eqref{ree_irreducible_module_form}. Hence
$\wedge^2(V_2 \otimes V_2^{\psi^n}) = \mathcal{S}^2(V_2) \oplus
\mathcal{S}^2(V_2^{\psi^n})$ and $\mathcal{S}^2(V_2 \otimes
V_2^{\psi^n}) = 1 \oplus V_9$ as required.

Furthermore, a direct calculation shows that
$\mathcal{S}^2(\mathcal{S}^2(V_2))$ has shape $1.1.1 \oplus \mathcal{S}^2(V_2)$ when restricted to $\F_3$, and over $\F_q$ the $\mathcal{S}^2(V_2)$ must fuse with the middle composition factor, otherwise the module would not be self-dual
($S$ is self-dual since $G$ preserves a bilinear form).

Similarly, $\mathcal{S}^2(V_2) \otimes V_2 \otimes V_2^{\psi^n}$ has
shape $1 . 1 . 1 \oplus (\mathcal{S}^2(V_2) .
\mathcal{S}^2(V_2^{\psi^n}) . \mathcal{S}^2(V_2))$ over $\F_3$. Over
$\F_q$ each $1$ fuses with a corresponding $\mathcal{S}^2(V_2)$ and
we obtain the structure of $V_{12}$. This is also proves that the
$4$-dimensional factor of $V_6$ is not isomorphic to any of the
factors of $V_{12}$.
\end{proof}

\begin{cl} \label{cl_homspace_dim}
Let $G \leqslant \GL(27, q)$ such that $G \cong \Ree(q)$, let $j \in
G$ be an involution and let $H = \Cent_G(j)^{\prime} \cong \PSL(2,
q)$. Then $\dim \Hom_H(S\vert_H, S\vert_H) = 5$.
\end{cl}
\begin{proof}
  By Proposition \ref{pr_27_dim_split}, \[ \Hom_H(S\vert_H, S\vert_H) =
  \Hom_H(V_6, S\vert_H) \oplus \Hom_H(V_9, S\vert_H) \oplus
  \Hom_H(V_{12}, S\vert_H).\] The middle summand has dimension $1$, by
  Schur's Lemma.
  
  A homomorphism $V_6 \to S\vert_H$ must map the $5$-dimensional
  submodule to itself or to $0$. The composition factor of dimension
  $1$ at the top can either be mapped to itself, or to the factor at
  the bottom. Hence $\dim \Hom_H(V_6, S\vert_H) = 2$.

  Similarly, $\dim \Hom_H(V_{12}, S\vert_H) = 2$ since the top factor can
  either be mapped to itself, or to the bottom factor. Thus the result
  follows.
\end{proof}

Given a module $W$ of the form \eqref{ree_irreducible_module_form}, we
now consider the problem of finding a flat. For $k = 0, \dotsc, n -1$,
let $H_k$ be the image of the representation corresponding to
$M^{\psi^{i_k}}$, so $H_k \leqslant \GL(7, q)$ or $H_k \leqslant
\GL(27, q)$, and let $\rho_k : G \to H_k$ be an isomorphism. Our goal
is then to find $\rho_k$ effectively for some $k$.

For $\lambda \in \F_q^{\times}$ denote $E_{\lambda} = \set{1,
  \lambda^{\pm t}, \lambda^{\pm (t - 1)}, \lambda^{\pm (2t - 1)}}$.
We need the following conjectures.

\begin{conj} \label{conj:tensor_decomposition_dim7} Let $\Ree(q) \cong G
  \leqslant \GL(d, q)$ have module $W$ of the form
  \eqref{ree_irreducible_module_form}, with $\dim W = d = 7^n$ for some $n >
  1$.

  Let $g \in G$ have order $q - 1$ and let $E$ be its multiset of eigenvalues. If $2m > n$
  then there exists $\lambda \in \F_q^{\times}$ such that $E_{\lambda}
  \subset E$, and the sum of the eigenspaces of $g$ corresponding to
  $E_{\lambda}$ has dimension $\dim V$.
\end{conj}

\begin{conj} \label{conj:tensor_decomposition_mixed_dim} Let $\Ree(q)
  \cong G \leqslant \GL(d, q)$ have module $W$ of the form
  \eqref{ree_irreducible_module_form} and $\dim W = d > 7$. Let $j \in G$
  be an involution.

  If $W$ has tensor factors both of dimension $7$ and $27$, then
  $W\vert_{\Cent_G(j)}$ has unique submodules $W_3$ and $W_4$ of dimensions $3$
  and $4$, respectively, such that $W_3 + W_4$ is a point of $W$ of
  dimension $7$.
\end{conj}

\begin{conj} \label{conj:tensor_decomposition_dim_27} Let $\Ree(q)
  \cong G \leqslant \GL(d, q)$ have module $W$ of the form
  \eqref{ree_irreducible_module_form} and $\dim W = 27^n$ for some $n > 1$. Let $j \in G$ be an involution and let $H = \Cent_G(j)^{\prime}$.
  
  If $2m > n$ then $\dim \Hom_H(S^{\psi^{i_k}}\vert_H, W\vert_H) = 5$,
  for some $0 \leqslant k \leqslant n - 1$.
\end{conj}

\begin{thm} \label{thm_tensor_decompose_dim7}
Assume Conjecture \ref{conj:tensor_decomposition_dim7}. There exists a Las Vegas algorithm that, given $\gen{X} \leqslant \GL(d, q)$, where $q =
  3^{2m + 1}$, $d = 7^n$, $n > 1$, $2m > n$ and $\gen{X} \cong \Ree(q)$, with module $W$ of the form
  \eqref{ree_irreducible_module_form}, finds a point of $W$. The algorithm has expected time complexity
\[\OR{(\xi(d) + d^3 \log(q) \log{\log(q^d)})\log\log(q) + d^4}\] field operations.
\end{thm}
\begin{proof}
  Let $G = \gen{X}$. By Corollary \ref{cl_random_selections}, we can easily find $g \in G$
  such that $\abs{g} \mid q - 1$. Our approach is to construct a point
  as a suitable sum of eigenspaces of $g$. We know that for $k = 0, \dotsc, n - 1$, $\rho_k(g)$ has
  $7$ eigenvalues $\lambda_k^{\pm t}$, $\lambda_k^{\pm (t - 1)}$,
  $\lambda_k^{\pm (2t - 1)}$ and $1$ for some $\lambda_k \in
  \F_q^{\times}$. Let $E$ be the multiset of eigenvalues of $g$. Each
  eigenvalue has the form
\begin{equation} \label{ree_eigenvalue_form}
\lambda_0^{j_0} \lambda_1^{j_1} \dotsm \lambda_{n - 1}^{j_{n - 1}}
\end{equation}
where each $\lambda_k \in \F_q^{\times}$ and each $j_k \in \set{\pm t,
  \pm (t - 1), \pm (2t - 1), 1}$. We can easily compute $E$.

Because each $\lambda_k^{j_k}$ may be $1$, for each $k = 0, \dotsc, n - 1$ we have
$E_{\lambda_k} \subset E$. We can determine which $\lambda \in E$ can be
one of the $\lambda_k$, since if $\lambda = \lambda_k$ for some $k$,
then $E_{\lambda^{3t}} \subset E$.

Thus we can obtain a list, with length between $n$ and $d$, of subsets
$E_{\lambda}$ of $E$. Now Conjecture
\ref{conj:tensor_decomposition_dim7} asserts that there is some $\mu
\in \F_q^{\times}$ such that $E_{\mu} \subset E$, and such that the sum
of the eigenspaces corresponding to $E_{\mu}$ has dimension $7$, and
by its construction it must therefore be a point of $W$. Since $\mu^t
\in E$, the set $E_{\mu}$ will be on our list, and we can easily find
the point.

The algorithm is Las Vegas, since we can easily calculate the
dimensions of the subspaces. The expected number of random selections
for finding $g$ is $\OR{\log\log(q)}$, and we can find its order using
expected $\OR{d^3 \log(q) \log{\log(q^d)}}$ field operations. We find
the characteristic polynomial using $\OR{d^3}$ field operations and
then find the eigenvalues using expected $\OR{d(\log{d})^2
  \log{\log(d)}\log{(dq)}}$ field operations. Finally we use the
Algorithm from Section \ref{section:tensor_decomposition} to verify
that we have a point, using $\OR{d^3 \log(q)}$ field operations. The rest of the algorithm is linear algebra, and hence
the expected time complexity is as stated.
\end{proof}

\begin{thm} \label{thm_tensor_decompose_general}
Assume Conjecture \ref{conj:tensor_decomposition_mixed_dim}. There exists a Las Vegas algorithm that, given $\gen{X} \leqslant \GL(d, q)$, where $q = 3^{2m + 1}$, $7 \mid d$, $27 \mid d$ and $\gen{X} \cong \Ree(q)$, with module $W$ of the form
  \eqref{ree_irreducible_module_form}, finds a point of $W$. The algorithm has expected time complexity
\[\OR{\xi(d) + d^3 \log(q) (\log{\log(q^d)} + \log(d)^2)}\] field operations.
\end{thm}
\begin{proof}
  Let $G = \gen{X}$. Similarly as in Corollary
  \ref{cl_find_stab_element}, we find an involution $j \in G$ and
  probable generators for $\Cent_G(j)^{\prime}$. We can then use
  Theorem \ref{thm_psl_naming} to verify that we have got the whole
  centraliser.

  Using the MeatAxe, we find the
  composition factors of $W\vert_{\Cent_G(j)}$. For each pair of factors
  of dimensions $3$ and $4$ we compute module homomorphisms into
  $W\vert_{\Cent_G(j)}$ and then find the sum of their images.
  
  Conjecture \ref{conj:tensor_decomposition_mixed_dim} asserts that
  this will produce a point of $W$. We can easily calculate the
  dimensions of the submodules and use the tensor decomposition
  algorithm to verify that we do obtain a point, so the algorithm is
  Las Vegas.

  The expected time complexity for finding $j$ is $\OR{\xi(d) + d^3
    \log(q) \log{\log(q^d)}}$ field operations. From the proof of
  Corollary \ref{cl_find_stab_element} we see that we can find
  probable generators for $\Cent_G(j)$ and verify that we have the
  whole centraliser using expected $\OR{\xi(d) + d^3 \log(q)
    \log{\log(q^d)}}$ field operations, if we let $\varepsilon =
  \log\log(q)$ in Theorem \ref{thm_psl_naming}. The MeatAxe uses
  expected $\OR{d^3}$ field operations in this case, since the number
  of generators for the centraliser is constant. Then we consider
  $\OR{\log(d)^2}$ pairs of submodules, and for each one we use the tensor
  decomposition algorithm to determine if we have a point, using
  $\OR{d^3 \log(q)}$ field operations. Hence the expected time
  complexity is as stated.
\end{proof}

\begin{thm} \label{thm_tensor_decompose_dim27}
Assume Conjecture \ref{conj:tensor_decomposition_dim_27} and an oracle for the discrete logarithm problem. There exists a Las Vegas algorithm that, given $\gen{X} \leqslant \GL(d, q)$, where $q = 3^{2m + 1}$, $d = 27^n$, $n > 1$, $2m > n$ and $\gen{X} \cong \Ree(q)$, with module $W$ of the form
  \eqref{ree_irreducible_module_form}, finds a point of $W$. The algorithm has expected time complexity
\[ \OR{(\xi(d) + d^3 \log(q) \log \log(q^d)) \log \log(q) + \log(q)^3 + d^5 \sigma_0(d) \abs{X} + d \chi_D(q) + \xi(d) d} \]
field operations.
\end{thm}
\begin{proof}
  Let $G = \gen{X}$. Similarly as in Corollary
  \ref{cl_find_stab_element}, we find an involution $j \in G$ and
  probable generators for $H = \Cent_G(j)^{\prime} \cong \PSL(2, q)$.
  We can then use Theorem \ref{thm_psl_naming} to verify that we have
  got the whole centraliser.  

  Using the MeatAxe, we find the
  composition factors of $W\vert_H$. Let $S_1$ be the group
  corresponding to a non-trivial composition factor. Using
  Theorem \ref{thm_psl_recognition} we constructively recognise $S_1$ as $\PSL(2,
  q)$ and obtain an effective isomorphism $\pi_1 : \PSL(2, q) \to H$.
  
  Now let $R$ be the image of the representation corresponding to $S$,
  so $R \leqslant \GL(27, q)$. Again we find an involution $j^{\prime}
  \in R$ and probable generators for $K = \Cent_R(j^{\prime})^{\prime}
  \cong \PSL(2, q)$. As above, we chop the module $S\vert_K$ with the
  MeatAxe, constructively recognise one of its non-trivial factors and
  obtain an effective isomorphism $\pi_2 : \PSL(2, q) \to K$.

  Note that both $\pi_1$ and $\pi_2$ have effective inverses. Hence we
  can obtain standard generators for $H$ and $K$. For each $i = 0,
  \dotsc, 2m$, do the following:
\begin{enumerate}
\item Find $M = \Hom_{\PSL(2, q)}(S^{\psi^i}\vert_K, W\vert_H)$ using the standard generators.
\item If $\dim M = 5$, then find random $f \in M$ such that $\dim
  \Ker{f} = 0$. Use the algorithm in Section
  \ref{section:tensor_decomposition} to determine if $U = \Imm(f)$ is
  a point.
\end{enumerate}

From Proposition \ref{pr_27_dim_split} we know that $S\vert_K$ has a
submodule of dimension $1$. This implies that that $W\vert_H$ has a
submodule $U^{\prime}_k = \gen{v_1} \otimes \dotsm \otimes \gen{v_{k - 1}} \otimes S^{\psi^{i_k}} \otimes \gen{v_{k + 1}} \otimes \dotsm
\otimes \gen{v_{n-1}} \cong S^{\psi^{i_k}}$, for any $k = 0, \dotsc, n - 1$, and some $v_1, \dotsc, v_{n-1}$ (depending on $k$).
Moreover, $\Hom_{\PSL(2, q)}(S^{\psi^{i_k}}\vert_K, W\vert_H) \geqslant
\Hom_{\PSL(2, q)}(S^{\psi^{i_k}}\vert_K, U^{\prime}_k)$, and by Corollary
\ref{cl_homspace_dim}, the latter has dimension $5$.

But by Conjecture \ref{conj:tensor_decomposition_dim_27}, for some $i$
the former also has dimension $5$, and hence these vector spaces are
equal. Therefore, for some $i = i_k$, the subspace $U$ found in the algorithm must be equal to $U^{\prime}_k$, and hence it is a point.


The expected time complexity for finding the involutions is
\[ \OR{\xi(d) + d^3 \log(q) \log{\log(q^d)}}\] field operations.  From
the proof of Corollary \ref{cl_find_stab_element} we see that we can
find probable generators for $\Cent_G(j)$ and verify that we have the
whole centraliser using expected $\OR{\xi(d) + d^3 \log(q)
  \log{\log(q^d)}}$ field operations, if we let $\varepsilon =
\log\log(q)$ in Theorem \ref{thm_psl_naming}. The MeatAxe uses
expected $\OR{d^3}$ field operations in this case since the number of
generators are constant. In the loop, we require $\OR{d^3 \log(q)}$
field operations to verify that $U$ is a point. Hence the expected
time complexity follows from Theorem \ref{thm_psl_recognition}.
\end{proof}

Conjectures \ref{conj:tensor_decomposition_dim7} and
\ref{conj:tensor_decomposition_dim_27} do not apply when $2m \leqslant
n$, so in this case we need another algorithm. Then $q \in \OR{d}$ so
we are content with an algorithm that has time complexity polynomial
in $q$. The approach is not to use tensor decomposition, since in this
case we have no efficient method of finding a flat. Instead we find standard generators of $G$ using permutation group
techniques, then enumerate all tensor products of the form
\eqref{ree_irreducible_module_form}, and for each one we determine if
it is isomorphic to $W$.

\begin{lem} \label{lem_ree_perm_rep}
  There exists a Las Vegas algorithm that, given $\gen{X} \leqslant \GL(d,
  q)$ such that $q = 3^{2m + 1}$ with $m > 0$ and $\gen{X} \cong
  \Ree(q)$, finds an effective injective homomorphism $\Pi : \gen{X} \to
  \Sym{O}$ where $\abs{O} = q^3 + 1$. The algorithm has expected time complexity $\OR{q^3(\xi(d) + \abs{X} d^2 + d^3) + d^4}$ field operations.
  \end{lem}
\begin{proof}
  By Proposition \ref{ree_doubly_transitive_action}, $\Ree(q)$ acts doubly
  transitively on a set of size $q^3 + 1$. Hence $G = \gen{X}$ also
  acts doubly transitively on $O$, where $\abs{O} = q^3 + 1$, and we
  can find the permutation representation of $G$ if we can find a
  point $P \in O$. The set $O$ is a set of projective points of
  $\F_q^d$, and the algorithm proceeds as follows.
\begin{enumerate}
\item Choose random $g \in G$. Repeat until $\abs{g} \mid q - 1$.
\item Choose random $x \in G$ and let $h = g^x$. Repeat until $[g, h]^9 = 1$ and $[g, h] \neq 1$.
\item Find a
  composition series for the module $M$ of $\gen{g, h}$ and let $P \subseteq M$ be the submodule of
  dimension $1$ in the series.
\item Find the orbit $O = P^G$ and compute the permutation group $S \leqslant \Sym{O}$ of $G$ on $O$, together with an effective isomorphism $\Pi : G \to S$.
\end{enumerate}

By Proposition \ref{cyclic_subgroups_conjugate}, elements in $G$ of order
dividing $q - 1$ fix two points of $O$, and hence $\gen{g, h}
\leqslant G_P$ for some $P \in O$ if and only if $g$ and $h$ have a
common fixed point. All composition factors of $M$ have dimension $1$, so a composition series
of $M$ must contain a submodule $P$ of dimension $1$. This submodule
is a fixed point for $\gen{g, h}$, and its orbit must have size $q^3 +
1$, since $\abs{G} =
q^3 (q^3 + 1) (q - 1)$ and $\abs{G_P} = q^3 (q - 1)$. It follows that
$P \in O$.

All elements of $G$ of order a power of $3$ lie in the derived group of a
stabiliser of some point, which is also a Sylow $3$-subgroup of $G$,
and the exponent of this subgroup is $9$. Hence $[g, h]^9 = 1$ if and
only if $\gen{g, h}$ lie in a stabiliser of some point, if and only
if $g$ and $h$ have a common fixed point.

To find the orbit $O = P^G$ we can compute a Schreier tree on the
generators in $X$ with $P$ as root, using $\OR{\abs{X} \abs{O} d^2}$
field operations. Then $\Pi(g)$ can be computed for any $g \in
\gen{X}$ using $\OR{\abs{O} d^2}$ field operations, by computing the
permutation on $O$ induced by $g$. Hence $\Pi$ is effective, and its
image $S$ is found by computing the image of each element of $X$.
Therefore the algorithm is correct and it is clearly Las Vegas.

We find $g$ using expected $\OR{(\xi(d) + d^3 \log(q) \log{\log(q^d)})
\log {\log(q)}}$ field operations and we find $h$ using expected
$\OR{(\xi(d) + d^3)q^2}$ field operations. Then $P$ is found using the MeatAxe, in expected
$\OR{d^4}$ field operations. Thus the result follows.
\end{proof}

\begin{conj} \label{conj:ree_perm_iso}
Let $G = \gen{X} \leqslant \Sym{\OV}$ such that $G \cong \Ree(q) = H$. There exists a Las Vegas algorithm that finds $x, h, z \in G$ as $\SLP$s in $X$ such that the map
\begin{align}
x &\mapsto S(1, 0, 0) \\
h &\mapsto h(\lambda) \\
z &\mapsto \Upsilon
\end{align}
is an isomorphism. Its time complexity is $\OR{q^5 (\log(q))^4}$ field operations. The length of the returned $\SLP$s are $\OR{q^3 \log \log(q)}$.
\end{conj}

\begin{rem}
  There exists an implementation of the above mentioned algorithm, and
  the Conjecture is then it always produces a correct result and has
  the stated complexity.
\end{rem}

\begin{thm} \label{small_field_tensor_decompose}
  Assume Conjecture \ref{conj:ree_perm_iso}. There exists a Las Vegas algorithm that, given $\gen{X}
  \leqslant \GL(d, q)$, where $q = 3^{2m + 1}$, $n > 1$ and $d = 7^n$ or $d = 27^n$ and $\gen{X} \cong \Ree(q)$, with module $W$ of the form
  \eqref{ree_irreducible_module_form}, finds a tensor decomposition of $W$. The algorithm has time complexity
$\OR{q^3(\xi(d) + \abs{X} d^2 + d^3 \log\log(q) + q^2 (\log(q))^4) + d^3 (\abs{X} \binom{2m}{n-1} + d)}$ field operations.
\end{thm}
\begin{proof}
Let $G = \gen{X}$. The algorithm proceeds as follows:
\begin{enumerate}
\item Find permutation representation $\pi : G \to G_S \leqslant \Sym{q^3 + 1}$ using Lemma \ref{lem_ree_perm_rep}.
\item Find standard generators $x,h,z \in G$ using Conjecture \ref{conj:ree_perm_iso}. Evaluate them on $G$ to obtain a generating set $Y$.
\item Let $H = \gen{Y}$ and let $V$ be the module of $H$. If $3 \mid d$ then replace $V$ with $S$.
\item Construct each module of dimension $d$ of the form \eqref{ree_irreducible_module_form} using $V$ as base. For each one test if it is isomorphic to $W$, using the MeatAxe.
\item Return the change of basis from the successful isomorphism test.
\end{enumerate}
  
The returned change of basis exhibits $W$ as a tensor product, so by
  Lemma \ref{lem_ree_perm_rep} the algorithm is Las Vegas.
  
The lengths of the $\SLP$s of $x,h,z$ is $\OR{q^3 \log\log(q)}$, so
we need $\OR{d^3 q^3 \log\log(q)}$ field operations to obtain $Y$.
The number of modules of dimension $d$ of the form
  \eqref{ree_irreducible_module_form} using $V$ as base is $\binom{2m}{n-1}$.
  Module isomorphism testing uses $\OR{\abs{X} d^3}$ field
  operations. Hence by Conjecture \ref{conj:ree_perm_iso} and Lemma \ref{lem_ree_perm_rep}
  the time complexity of the algorithm is as stated.
\end{proof}

\subsection{Symmetric square decomposition}
\label{section:decompose_symsquare}
The two basic irreducible modules of $\Ree(q)$ are the natural module
$V$ of dimension $7$, and an irreducible submodule $S$ of the symmetric
square $\mathcal{S}^2(V)$. The symmetric square itself is not irreducible,
since $\Ree(q)$ preserves a quadratic form, and $\mathcal{S}^2(V)$ therefore has
a submodule of dimension $1$. The complement of this has dimension
$27$ and is the irreducible module $S$.

\begin{conj} \label{lem:symmetric_square}
The exterior square of $S$ has a submodule isomorphic to a twisted version of $V$.
\end{conj}

\begin{thm} \label{thm_symsquare_decompose}
  Assume Conjecture \ref{lem:symmetric_square}. There exists a Las
  Vegas algorithm that, given $\gen{X} \leqslant \GL(27, q)$ with module $W$
  such that $W$ is isomorphic to a twisted version of $S$, finds an
  effective isomorphism from $\gen{X}$ to $\Ree(q)^g$ for some $g \in \GL(7, q)$. The algorithm has expected
  time complexity $\OR{\abs{X}}$ field operations.
\end{thm}
\begin{proof}
  Using Conjecture \ref{lem:symmetric_square}, this is just an application of the
  MeatAxe. We construct the
  exterior square $\wedge^2(W)$ of $W$, which has dimension $351$, and
  find a composition series of this module using the MeatAxe. By the Conjecture, the
  natural module of dimension $7$ will be one of the composition
  factors and the MeatAxe will provide an effective isomorphism to this
  factor, in the form of a change of basis $A \in GL(27, q)$ of $W$ that exhibits the
  action on the composition factors.

  This induces an isomorphism $\varphi : \gen{X} \to H$, where $H$ is
  conjugate to $\Ree(q)$. For $g \in \gen{X}$, $\varphi(g)$ is computed
  by taking a submatrix of $g^A$ of degree $7$. Clearly $\varphi$ can be
  computed using $\OR{1}$ field operations.

  Since the MeatAxe is Las Vegas and has expected time complexity
  $\OR{\abs{X}}$, the result follows.
\end{proof}

\subsection{Constructive recognition}

Finally, we can now state and prove our main theorem.

\begin{thm} \label{cl_ree_constructive_recognition} 
Assume the small Ree Conjectures, and an oracle for the discrete
  logarithm problem in $\F_q$. There exists a Las Vegas algorithm
  that, given $\gen{X} \leqslant \GL(d, q)$ satisfying the assumptions in Section \ref{section:algorithm_overview}, with $q = 3^{2m + 1}$, $m > 0$ and $\gen{X} \cong \Ree(q)$,
  finds an effective isomorphism $\varphi : \gen{X} \to \Ree(q)$ and performs preprocessing for constructive membership testing. The
  algorithm has expected time complexity 
$\OR{\xi(d) (d^3 + (\log\log(q))^2) + d^6 \log\log(d) + d^5 \sigma_0(d) \abs{X} + 
d^3 \log(q) \log\log(q) \log\log(q^d) + \log(q)^3 \log\log(q) + \chi_D(q) (d + \log\log(q))}$ field operations.

Each image of $\varphi$ can be computed in $\OR{d^3}$ field
operations, and each pre-image in expected $\OR{\xi(d) + \log(q)^3 + d^3 (\log(q)
  \log\log(q))^2}$ field operations.

\end{thm}
\begin{proof}

Let $W$ be the module of $G = \gen{X}$. The algorithm proceeds as follows:
\begin{enumerate}
\item If $d = 7$ then use Theorem \ref{cl_ree_conjugacy} to obtain $y
  \in \GL(7, q)$ such that $G^y = \Ree(q)$, and hence an
  effective isomorphism $\varphi : G \to \Ree(q)$ defined by $g \mapsto g^y$.
\item If $3 \nmid d$ then $d = 7^n$ for some $n > 1$. If $2m > n$ then use Theorem \ref{thm_tensor_decompose_dim7} to find a flat $L \leqslant M$. If $3 \mid d$ but $d$ is not a proper power of $27$ then use Theorem \ref{thm_tensor_decompose_general} to find such an $L$. Otherwise $d = 27^n$ for some $n > 1$. If $2m > n$, then use Theorem \ref{thm_tensor_decompose_dim27} to find a flat $L \leqslant M$.
\item Use the tensor decomposition algorithm described in Section \ref{section:tensor_decomposition} with $L$, to obtain $x \in \GL(d, q)$ such
  the change of basis determined by $x$ exhibits $W$ as a tensor
  product $A \otimes B$, with $\dim A = 7$ or $\dim A = 27$. If $d = 7^n$ or $d = 27^n$ and $2m \leqslant n$ then use Theorem
  \ref{small_field_tensor_decompose} to find $x$. Let $G_A$ and $G_B$
  be the images of the corresponding representations.
\item Define $\rho_A : G_{A \otimes B} \to G_A$ as $g_a \otimes g_b
  \mapsto g_a$ and let $Y = \set{\rho_A(g^x) \mid g \in X}$. If $\dim
  A = 27$ then 
let $\theta$ be the effective isomorphism from Theorem \ref{thm_symsquare_decompose}, otherwise let $\theta$ be the identity map.

\item Let $Z = \set{\theta(x) \mid x \in Y}$. Then $\gen{Z}$ is conjugate to $\Ree(q)$. Use Theorem \ref{cl_ree_conjugacy} to obtain $y \in \GL(7, q)$ such that $\gen{Z}^y = \Ree(q)$. 
\item An effective isomorphism $\varphi : G \to \Ree(q)$ is given by $g \mapsto \theta(\rho_A(g^x))^y$.
\end{enumerate}
The map $\rho_A$ is straightforward to compute, since given $g \in
\GL(d, q)$ it only involves dividing $g$ into submatrices of degree $d
/ 7$ or $d / 27$, checking that they are scalar multiples of each
other and returning the $7 \times 7$ or $27 \times 27$ matrix consisting
of these scalars. Since $x$ might not lie in $G$, but only in
$\Norm_{\GL(d, q)}(G) \cong G {:} \F_q$, the result of $\rho_A$ might
not have determinant $1$.  However, since every element of $\F_q$ has
a unique $7$th root, we can easily scale the matrix to have
determinant $1$. Hence by Theorem \ref{thm_tensor_decompose_dim7},
Theorem \ref{thm_tensor_decompose_general}, Theorem
\ref{thm_tensor_decompose_dim27}, Section
\ref{section:tensor_decomposition}, Theorem
\ref{thm_symsquare_decompose} and Theorem \ref{cl_ree_conjugacy}, the
algorithm is Las Vegas, and $\varphi$ can be computed using $\OR{d^3}$
field operations.

In the case where we use Theorem \ref{small_field_tensor_decompose} we
have $2m \leqslant n$ and hence $q < 3d$. We see that $\binom{2m}{n-1}
\leq n$, and the time complexity of the algorithm to find $x$, in
Theorem \ref{small_field_tensor_decompose}, simplifies to $\OR{d^3
  (\xi(d) + \abs{X} d^2 + d^3 \log\log(d))}$.

In the other cases, finding $L$ uses $\OR{(\xi(d) + d^3 \log(q) \log
  \log(q^d)) \log \log(q) + d^5 \sigma_0(d) \abs{X} + d
  \chi_D(q) + \xi(d) d}$ field operations. From Section
\ref{section:tensor_decomposition}, finding $x$ uses $\OR{d^3
  \log(q)}$ field operations when a flat $L$ is given.

From Theorem \ref{thm_symsquare_decompose} finding $\theta$ uses
$\OR{\abs{Y}}$ field operations, and from Theorem
\ref{cl_ree_conjugacy}, finding $y$ uses $\OR{(\xi(d) \log{\log(q)} + \log(q)^3 + \chi_D(q))\log \log(q) + \abs{Z}}$ field
operations. Hence the expected time complexity is as stated. Finally,
$\varphi^{-1}(g)$ is computed by first using Algorithm
\ref{alg:ree_element_to_slp} to obtain an $\SLP$ of $g$ and then
evaluating it on $X$. The necessary precomputations in Theorem
\ref{thm_pre_step} have already been made during the application of
Theorem \ref{cl_ree_conjugacy}, and hence it follows from Theorem
\ref{thm_element_to_slp_complexity} that the time complexity for
computing the pre-image of $g$ is as stated.
\end{proof}

\section{Big Ree groups}

Here we will use the notation from Section
\ref{section:big_ree_theory}. We
will refer to Conjectures \ref{pr_cent_invol}, \ref{pr_cent_dihedral_trick},
\ref{pr_centraliser_facts}, \ref{pr_invol_o2_facts}, 
\ref{pr_sz26_facts}, \ref{pr_sz_sz_facts1}, \ref{mult_implies_even_order2},
\ref{pr_bigree_goodelts} and \ref{conj:bigree_trick} simultaneously as the \emph{Big Ree Conjectures}.

With the Big Ree groups, we will only deal with the natural
representation, the last case in Section
\ref{section:algorithm_overview}. We will not attempt any tensor
decomposition, or decomposition of the tensor indecomposables listed
in Section \ref{section:bigree_indecomposables}. This can be partly
justified by the fact that at the present time the only
representation, other than the one of dimension $26$, that it is
within practical limits in the MGRP, is the one of dimension $246$. We
have some evidence that it is feasible to decompose this representation
into the natural representation using the technique known as
\emph{condensation}, see \cite[Section $7.4.5$]{hcgt} and
\cite{lux_condensation}.

The main constructive recognition theorem is Theorem \ref{thm_bigree_constructive_recognition}.

\subsection{Recognition} \label{section:bigree_recognition}

We now consider the question of non-constructive recognition of
$\LargeRee(q)$, so we want to find an algorithm that, given a set $\gen{X} \leqslant \GL(26, q)$, decides whether or not $\gen{X} = \LargeRee(q)$. 

\begin{thm} \label{thm_nonexplicit_recognition}
  There exists a Las Vegas algorithm that, given $\gen{X} \leqslant \GL(26,
  q)$, decides whether or not $\gen{X} = \LargeRee(q)$. The algorithm has expected time complexity $\OR{\sigma_0(\log(q))(\abs{X} + \log(q))}$ field
  operations.
\end{thm}
\begin{proof}
Let $G = \LargeRee(q)$, with natural module $M$. The algorithm proceeds as
follows:
\begin{enumerate}
\item Determine if $X \subseteq G$ and return \texttt{false} if
  not. All the following steps must succeed in order to conclude that a given $g
  \in X$ also lies in $G$.
\begin{enumerate}
\item Determine if $g \in \mathrm{O}^{-}(26, q)$, which is true if $\det g = 1$
  and if $g Q g^T = Q$, where $Q$ is the matrix corresponding to the quadratic form $Q^{*}$ and where $g^T$ denotes the transpose of $g$.
\item Determine if $g \in \mathrm{F}_4(q)$, which is true if $g$ preserves the
  exceptional Jordan algebra multiplication. This is easy using the multiplication table given in \cite{elementary_ree}.
\item Determine if $g$ is a fixed point of the automorphism of $\mathrm{F}_4(q)$ which defines $\LargeRee(q)$. By \cite{elementary_ree}, computing the automorphism amounts to taking a
  submatrix of the exterior square of $g$ and then replacing each matrix
  entry $x$ by $x^{2^m}$.
\end{enumerate}

\item If $\gen{X}$ is not a proper subgroup of $G$, or
  equivalently if $\gen{X}$ is not contained in a maximal subgroup,
  return \texttt{true}. Otherwise return \texttt{false}. By Proposition \ref{pr_reducible_maximals}, it is
sufficient to determine if $\gen{X}$ cannot be written over a smaller
field and if $\gen{X}$ is irreducible. This can be done using the Las Vegas algorithms from Sections \ref{section:meataxe} and \ref{section:smallerfield}.
\end{enumerate}

Since the matrix degree is constant, the complexity of the first step of the algorithm is $\OR{1}$
field operations. For the same reason, the complexity of the algorithms from Sections \ref{section:meataxe} and \ref{section:smallerfield} is $\OR{\sigma_0(\log(q))( \abs{X} + \log(q))}$ field operations.
Hence the expected time complexity is as stated.
\end{proof}

\subsection{Finding elements of even order}
\label{section:even_order_elements} In constructive recognition and membership testing of $\LargeRee(q)$, the essential
problem is to find elements of even order, as $\SLP$s in the given
generators. Let $G = \LargeRee(q) = \gen{X}$. We begin with an overview of the method. The matrix degree is constant here, so we set $\xi = \xi(26)$.

Choose random $a \in G$ of order $q - 1$, by choosing a
random element of order $(q - 1)(q + t + 1)$ and powering up. By
Proposition \ref{pr_random_torus_elt} it is easy to find such
elements, and by Proposition \ref{pr_diagonal_elt}, we can diagonalise
$a$ and obtain $c \in \GL(26, q)$ such that $a^c = \delta = h(\lambda, \mu)$
for some $\lambda, \mu \in \F_q^{\times}$.

Now choose random $b \in G$. Let $B = b^c$ and let $A(u, v)$ be a
diagonal matrix of the same form as $h(\lambda, \mu)$, where $\lambda$
and $\mu$ are replaced by indeterminates $u$ and $v$, so $A(u, v)$ is
a matrix over the function field $\F_q(u, v)$. 

For any $r, s \in \F_q^{\times}$, such that $r^t = s$, the matrix $(A(r, s) B)^{c^{-1}} \in
\gen{a} b$. Hence by Conjecture \ref{mult_implies_even_order2} if we can
find $r, s$ such that $r^t = s$ and $A(r, s) B$ has the eigenvalue $1$ with
multiplicity $6$, then with high probability $A(r, s) B$ will have even order.

\begin{pr} \label{conj:bigree_trick}
Assume Conjecture \ref{conj_bigree_badcosets}. For every $b \in G \setminus \Norm_G(\gen{a})$ the $6$ lowest coefficients $f_1, \dotsc, f_6 \in \F_q[u, v]$ of the characteristic polynomial of $A(u, v) B - I_{26}$ generate a zero-dimensional ideal.
\end{pr}
\begin{proof}
  Since $b$ does not normalise $\gen{a}$, by Conjecture
  \ref{conj_bigree_badcosets}, there must be a bounded number of
  solutions. Hence the system must be zero-dimensional.
\end{proof}

Finally we can solve the discrete logarithm problem and find an integer
$k$ such that $\delta^k = A(r, s)$. Then $\delta^k B$ has even order, and
therefore also $a^k b$ has even order. Since $a$ and $b$ are
random, we obtain an element of even order as an $\SLP$ in $X$. The algorithm for finding elements of even order is given formally as Algorithm \ref{alg:find_even_order_elt}.

\begin{lem} \label{lem_interpolation} Assume Conjecture \ref{conj_bigree_badcosets}. There exists a Las Vegas
  algorithm that, given the matrices $A(u, v)$ and $B$, finds $r, s
  \in \F_q^{\times}$ such that $r^t = s$ and $A(r, s) B$ has $1$ as an
  eigenvalue of multiplicity at least $6$. The algorithm has
  expected time complexity $\OR{\log q}$ field operations.
\end{lem}
\begin{proof}
  If we can find the characteristic polynomial $f(x) \in \F_q(u,
  v)[x]$ of $A(u, v) B$, then the condition we want to impose is that
  $1$ should be a root of multiplicity $6$, or equivalently that $y^5$
  should divide $g(y) = f(y + 1)$.

  Hence we obtain $6$ polynomial equations in $u$ and $v$ of bounded
  degree. By Proposition \ref{conj:bigree_trick}, we can use Theorem \ref{thm_poly_eqns_2vars} to find the possible values for $r$ and $s$. 

  Thus it only remains to find $f(x)$, which has the form
\begin{equation}
f(x) = a_n x^n + \dotsb + a_1 x + a_0
\end{equation}
where $a_i \in \F_q(u, v)$ and $0 \leqslant i \leqslant n \leqslant 26$. Recall that $A(u,
v)$ is diagonal of the same form as $h(\lambda, \mu)$. This implies
that in an echelon form of $A(u, v) B$, the diagonal has the form
$A(u, v)D$ for some diagonal matrix $D \in \GL(26, q)$. We obtain
$f(x)$ by multiplying these diagonal elements, and since the sum of
the positive powers of $u$ on the diagonal is $10$, and the sum of the
positive powers of $v$ on the diagonal is $6$, each $a_i$ has the form
\begin{equation}
a_i = \sum_j c_{ij} u^{z_{ij}} v^{y_{ij}}
\end{equation}
where each $c_{ij} \in \F_q$, $-10 \leqslant z_{ij} \leqslant 10$ and
$-6 \leqslant w_{ij} \leqslant 6$.

Because of these bounds on the exponents $z_{ij}$ and $w_{ij}$, we can
find the coefficients $c_{ij}$, and hence the coefficients $a_i$ and
$f(x)$, using interpolation. Each $c_{ij}$ is uniquely determined by
at most $(2 \cdot 10 + 1)(2 \cdot 6 + 1) = 273$ values of $u,v$ and
the corresponding value of $a_i$.

Therefore, choose $273$ random pairs $(e_k, f_k) \in \F_q \times
\F_q$. For each pair, calculate the characteristic polynomial of
$A(e_k, f_k) B$, thus obtaining the corresponding values of the
coefficients $a_i$. Finally perform the interpolation by solving $n$
linear systems with $273$ equations and variables.

It is clear that the algorithm is Las Vegas, and the dominating term
in the time complexity is the root finding of univariate polynomials
 of bounded degree over $\F_q$. 
\end{proof}

\begin{figure}[hb]
\begin{codebox}
\refstepcounter{algorithm}
\label{alg:find_even_order_elt}
\Procname{\kw{Algorithm} \ref{alg:find_even_order_elt}: $\proc{FindEvenOrderElement}(X)$}
\li \kw{Input}: $\gen{X} \leqslant \GL(26, q)$ such that $\gen{X} \cong \LargeRee(q)$.
\li \kw{Output}: An element of $\gen{X}$ of even order, expressed as an $\SLP$ in $X$.
\li \Comment{\proc{FindElementInCoset} is given by Lemma \ref{lem_interpolation}}
\li \Repeat
\li         \Repeat
\li             $h := \proc{Random}(\gen{X})$
\li         \Until $\abs{h} \mid (q - 1)(q + t + 1)$
\li         $a := h^{q + t + 1}$
\li         $\delta, c := \proc{Diagonalise}(a)$
\li         \Comment{Now $a^c = \delta = h(\lambda, \mu)$ where $\lambda, \mu \in \F_q^{\times}$ and $\mu = \lambda^t$}
\li     \Repeat
\li         \Repeat
\li         \Repeat
\li         $b := \proc{Random}(\gen{X})$
\li         \Until $b \notin \Norm_G(\gen{a})$ \label{alg:test_bad_coset}
\li         $(\id{flag}, r, s) := \proc{FindElementInCoset}(b^c)$
\li         \Until $\id{flag}$ \label{alg:find_elt_in_coset}
\li         \Comment{Now $r^t = s$ and $A(r, s) b^c$ has $1$ as a
  $6$-fold eigenvalue}
\li     \Until $\abs{A(r, s) b^c}$ is even \label{alg:even_order_test}
\li     $k := \proc{DiscreteLog}(\delta_{2, 2}, r)$
\li \Until $k > 0$ \label{alg:even_order_discrete_log}
\li \Return $a^k b$
\end{codebox}
\end{figure}

\begin{lem} \label{lem_even_order_cosets} Assume Conjectures \ref{pr_bigree_goodelts} and \ref{conj_bigree_badcosets}. Let $a \in G$ be such that
  $\abs{a} = q - 1$ and $a$ is conjugate to some $h(\lambda, \mu)$
  with $\lambda^t = \mu \in \F_q^{\times}$. The proportion of $b \in
  G \setminus \Norm_G(\gen{a})$, such that $\gen{a} b$ contains an element with $1$ as an
  eigenvalue of multiplicity $6$, is bounded below by a constant $c_4
  > 0$.
\end{lem}
\begin{proof}
  Given such a coset $\gen{a} b$, from the proof of Lemma
  \ref{lem_interpolation} we see that our algorithm constructs a
  bounded number $d_1$ of candidates of elements of the required type
  in the coset.

  Let $c$ be the number of cosets containing an element of the
  required type. By Conjecture \ref{pr_bigree_goodelts}, the total
  number of elements of the required type is $c_3 \abs{G}$. Hence $c$
  is minimised if all the $c$ cosets contain $d_1$ such elements, in
  which case $c d_1 = c_3 \abs{G}$. Thus $c \geqslant c_3 \abs{G} /
  d_1$ and the proportion of cosets is $c (q - 1) / \abs{G} \geqslant
  c_3 (q - 1) / d_1$, which is bounded below by a constant $c_4 > 0$
  since $c_3 \in \OR{1/q}$.
\end{proof}

\begin{thm} \label{thm_find_even_order_elt}
Assume Conjectures
  \ref{conj_bigree_badcosets}, \ref{pr_bigree_goodelts} and \ref{mult_implies_even_order2}, and an oracle for the discrete logarithm
  problem in $\F_q$. Algorithm \ref{alg:find_even_order_elt} is a Las
  Vegas algorithm with expected time complexity $\OR{(\xi + \log(q)\log\log(q)  + \chi_D(q))
    \log\log(q)}$ field operations. The length of the returned $\SLP$ is $\OR{\log \log(q)}$.
\end{thm}
\begin{proof}
  By Proposition \ref{pr_diagonal_elt}, $a$ is conjugate to some
  $h(\lambda, \mu)$ and by Proposition \ref{pr_random_torus_elt}, we
  can find $h$ using expected $\OR{\xi \log\log(q)}$ field operations.
  The test at line \ref{alg:test_bad_coset} is easy since $b$ either
  centralises or inverts $a$. Furthermore, by Lemma
  \ref{lem_even_order_cosets}, the test at line
  \ref{alg:find_elt_in_coset} will succeed with high probability and
  by Conjecture \ref{mult_implies_even_order2}, the test at line
  \ref{alg:even_order_test} will succeed with high probability. The
  test at line \ref{alg:even_order_discrete_log} can only fail if
  $\abs{a}$ is a proper divisor of $q - 1$, which happens with low
  probability.

  Hence by Lemma \ref{lem_interpolation}, the algorithm is Las Vegas
  and the time complexity is as stated. Clearly, the length of the $\SLP$ of
  the returned element is the same as the length of the $\SLP$ of $h$.
\end{proof}

\begin{rem} \label{rem_find_even_order_elt}
  If we are given $g \in \gen{X} \cong \LargeRee(q)$, then a trivial
  modification of Algorithm \ref{alg:find_even_order_elt} finds an
  element of $\gen{X}$, of even order, of the form $hg$ for some $h \in
  \gen{X}$. If we also have an $\SLP$ of $g$ in $X$, then we will
  obtain $hg$ as $\SLP$, otherwise we will only obtain an $\SLP$ for $h$.
\end{rem}

\begin{pr} \label{pr_trick_involution_class}
Assume Conjecture \ref{mult_implies_even_order2}. With probability $1 - \OR{1/q}$, the element returned by Algorithm \ref{alg:find_even_order_elt} powers up to an involution of class $2A$.
\end{pr}
\begin{proof}
Follows immediately from Conjecture \ref{mult_implies_even_order2}.
\end{proof}

\subsection{Constructive membership testing} 
\label{section:bigree_constructive_membership}

The overall method we use for constructive membership testing in $\LargeRee(q)$ is the \emph{Ryba algorithm} described in Section \ref{section:inv_centraliser}.

Since we know that there are only two conjugacy classes of involutions
in $\LargeRee(q)$, and since we know the structure of their centralisers, we can
improve upon the basic Ryba algorithm. When solving constructive
membership testing in the centralisers, instead of applying the Ryba
algorithm recursively, we can do it in a more direct way using Theorem
\ref{thm_psl_recognition} and the algorithms for constructive membership testing in the Suzuki group, described in Section \ref{section:suzuki_main_results}. Involutions of class
$2A$ can be found using Algorithm \ref{alg:find_even_order_elt}, and
Conjecture \ref{pr_invol_o2_facts} give us a method for finding
involutions of class $2B$ using random search. The Ryba algorithm needs to find
involutions of both classes, since it needs to find two involutions
whose product has even order.

As a preprocessing step to Ryba, we can therefore find an involution
of each class and compute their centralisers. In each call to Ryba we
then conjugate the involutions that we find to one of these two
involutions, which removes the necessity of computing involution
centralisers at each call.

\subsubsection{The involution centralisers}

We use the Bray algorithm to find generating sets for the involution centralisers. This algorithm is described in Section \ref{section:inv_centraliser}.

The following results show how to precompute generators and how to
solve the constructive membership problem for a centraliser of an
involution of class $2A$, using our Suzuki group algorithms to constructively
recognise $\Sz(q)$. Analogous results hold for the centraliser of an
involution of class $2B$, using Theorem \ref{thm_psl_recognition} to
constructively recognise $\SL(2, q)$.

\begin{lem} \label{lem_centraliser_comp_series}
  Assume Conjecture \ref{pr_centraliser_facts}, and use its notation. There exists a Las Vegas algorithm that, given $\gen{Y} \leqslant G \leqslant \GL(26, q)$ such that $G \cong \LargeRee(q)$, $S \leqslant \gen{Y} \leqslant \Cent_G(j)$ for an involution $j \in G$ of class $2A$ and $S \cong \Sz(q)$,
  finds a composition series for the natural module $M$ of $\gen{Y}$
  such that the composition factors are ordered as $1, S_4, 1, S_4, 1, S_4^{\psi^t},
  1, S_4, 1, S_4, 1$, and finds the corresponding filtration of $\O2(\gen{Y})$. The algorithm has time complexity $\OR{\abs{Y}}$ field operations.
\end{lem}
\begin{proof}
By Proposition \ref{pr_centraliser_facts} the composition factors are as stated, and we just have to order them correctly. 
\begin{enumerate}
\item Find a composition factor of $F_1$ of $M$, such that $\dim F_1 = 1$, and find $H_1 = \Hom_{\gen{Y}}(M, F_1)$. By Conjecture \ref{pr_centraliser_facts}, $M$ has a unique $1$-dimensional submodule, so $\dim H_1 = 1$.
\item Let $M_{25} = \Ker \alpha_1$ where $\gen{\alpha_1} = H_1$. Then $\dim M_{25} = 25$.
\item Let $M_1 = M_{25}^{\perp} \cap M_{25}$ be the orthogonal complement under the bilinear form preserved by $G$, so that $\dim M_1 = 1$.
\item Find a composition factor of $F_4$ of $M_{25}$, such that $\dim F_4 = 4$, and find $H_4 = \Hom_{\gen{Y}}(M_{25}, F_4)$. By Conjecture \ref{pr_centraliser_facts}, $M_{25}$ has a unique $4$-dimensional submodule, so $\dim H_4 = 1$ for one of its $5$ composition factors of dimension $4$.
\item Let $M_{21} = \Ker \alpha_4$ where $\gen{\alpha_4} = H_4$. Then $\dim M_{21} = 21$.
\item Let $M_5 = M_{21}^{\perp} \cap M_{21}$ be the orthogonal complement under the bilinear form preserved by $G$, so that $\dim M_5 = 5$.
\item Now we have four proper submodules $M_1, M_5, M_{21}, M_{25}$ in the composition series that we want to find, and we can obtain the other submodules by continuing in the same way inside $M_{21}$.
\end{enumerate}
The filtration is determined by the composition factors, and
immediately found. Clearly the time complexity is the same as the
MeatAxe, which is $\OR{\abs{Y} }$ field operations.
\end{proof}

\begin{lem} \label{lem_centraliser_generators}
  Assume the Suzuki Conjectures, Conjectures \ref{pr_cent_invol}, \ref{pr_cent_dihedral_trick}, \ref{pr_centraliser_facts} and \ref{pr_invol_o2_facts}, and an oracle for the discrete
  logarithm problem in $\F_q$. There exists a Monte Carlo algorithm with no false positives that,
  given $\gen{Y} \leqslant G \leqslant \GL(26, q)$ such that $G \cong \LargeRee(q)$, $S \leqslant \gen{Y} \leqslant
  \Cent_G(j)$ and $\Zent(\Cent_G(j)) \leqslant \gen{Y}$, where $j \in G$ is an involution of class $2A$ and $S \cong
  \Sz(q)$:
\begin{itemize}
\item decides whether or not $\gen{Y} = \Cent_G(j)$,
\item finds effective homomorphisms $\varphi : \gen{Y} \to \Sz(q)$ and $\pi : \Sz(q) \to \gen{Y}$,
\item finds $u \in G$ such that $\abs{u} \mid q - 1$, $\Cent_G(j) < \gen{Y, u}$ and
  $\gen{Y, u}$ is contained in a maximal parabolic in $G$,
\item finds $\gen{Z} \leqslant G$ such that $\gen{Z} = \O2(\Cent_G(j))$, and $\abs{Z} \in \OR{\log(q)}$.
\end{itemize}
The algorithm has expected time complexity 
\[\OR{\abs{Y} + \log(q)^3 + (\xi + \chi_D(q))
  (\log\log(q))^2}\]
field operations.
\end{lem}
\begin{proof}
The algorithm proceeds as follows:
\begin{enumerate}
\item Find a composition series of the natural module of $\gen{Y}$, as
  in Lemma \ref{lem_centraliser_comp_series}. By projecting to the
  middle composition factor we obtain an effective surjective
  homomorphism $\varphi_1 : \gen{Y} \to S_4$ where $S_4 \leqslant
  \GL(4, q)$ and $S_4 \cong \Sz(q)$. Also obtain an effective
  surjective homomorphism $\rho : \O2(\gen{Y}) \to N$, where $N$ is the
  first non-zero block in the filtration of $\O2(\gen{Y})$
  (\emph{i.e.} $N$ is a vector space). By Conjecture
  \ref{pr_invol_o2_facts}, $\dim N = 4$.

\item Use Theorem \ref{cl_sz_constructive_recognition} to
  constructively recognise $\varphi_1(\gen{Y})$ and obtain an effective
  injective homomorphism $\pi : \Sz(q) \to \gen{Y}$, and an effective
  isomorphism $\varphi_2 : S_4 \to \Sz(q)$. Now $\varphi = \varphi_2 \circ
  \varphi_1$.

\item Find random $g \in \gen{Y}$ such that $\abs{g} = 2l$. By
  Proposition \ref{pr_invol_facts}, the proportion of such elements is
  high. Repeat until $g^l = j^{\prime}$ is of class $2A$ and
  $j^{\prime} \neq j$, which by Conjecture \ref{pr_cent_invol}
  happens with high probability.

\item Using the dihedral trick, find $h \in G$ such that $j^h = j^{\prime}$, $\abs{h} \mid q - 1$ and $\gen{Y, h}$ is reducible. By Conjecture \ref{pr_cent_dihedral_trick}, these elements are easy to find. 

\item Let $u = h ((\pi \circ \varphi)(h))^{-1}$, so that $u$ commutes
  with $S$.  Now $\gen{\Cent_G(j), u}$ is contained in a maximal
  parabolic, and $\Cent_G(j)$ is a proper subgroup, since $u \notin
  \Cent_G(j)$, but $\gen{Y, u}$ is reducible and hence a proper
  subgroup of $G$.

\item Diagonalise $u$ to obtain $\varsigma(a, b)$ for some $a, b \in \F_q^{\times}$. Repeat the two previous steps (find another $h$) if $a$ or $b$ lie in a proper subfield of $\F_q$. The probability that this happens is low, since $\abs{h} = q - 1$ with high probability.

\item Find random $y_1, \dotsc, y_4 \in \gen{Y}$, and let $x_i = y_i
  ((\pi \circ \varphi)(y_i))^{-1}$ for $i = 1, \dotsc, 4$. Then $x_1,
  \dotsc, x_4$ are random elements of $\O2(\gen{Y})$. Return
  \texttt{false} if $\rho(x_1), \dotsc, \rho(x_4)$ are not linearly
  independent elements of $N$, since then with high probability
  $\gen{Y} < \Cent_G(j)$. Clearly, if the elements are linearly
  independent, then $\gen{Y} = \Cent_G(j)$ so the algorithm has no
  false positives.

\item Find random $x \in \gen{Y}$ such that $\abs{x} = 4k$. By Conjecture \ref{pr_cent_invol}, with high probability $x_5 = x^k \in \O2(\gen{Y})$ and $x_5^2 \in \Zent(\gen{Y})$. Repeat until this is true.

\item Finally let
\[ Z = \bigcup_{i = 0}^{2m + 1} \bigcup_{j = 1}^5 \set{x_j^{u^i}, (x_j^2)^{u^i}} \]
\end{enumerate}

Clearly, the dominating term in the running time is Theorem \ref{cl_sz_constructive_recognition} and the computation of $\pi$, so the expected time complexity is as stated.
\end{proof}

\begin{lem} \label{lem_bigree_find_centraliser}
  Assume the Suzuki Conjectures, Conjectures \ref{pr_cent_dihedral_trick}, \ref{pr_centraliser_facts} and \ref{pr_invol_o2_facts} and an oracle for the discrete logarithm problem in $\F_q$. There
  exists a Las Vegas algorithm that, given $G = \gen{X} = \LargeRee(q)$
  and an involution $j \in \gen{X}$ of class $2A$, as an $\SLP$ in $X$ of length $\OR{n}$,
\begin{itemize}
\item finds $\gen{Y} \leqslant G$ such that $\gen{Y} = \Cent_G(j)$,
\item finds effective inverse isomorphisms $\varphi : \gen{Y} \to \Sz(q)$ and $\pi : \Sz(q) \to \gen{Y}$,
\item finds $u \in G$ such that $\abs{u} \mid q - 1$, $\Cent_G(j) < \gen{Y, u}$ and $\gen{Y, u}$ is contained in a maximal parabolic in $G$,
\item finds $\gen{Z} \leqslant G$ such that $\gen{Z} = \O2(\Cent_G(j))$, and $\abs{Z} \in \OR{\log(q)}$.
\end{itemize}
The elements $Y, Z, u$ are found as $\SLP$s in $X$ of length $\OR{n}$. The algorithm has expected time complexity $\OR{\abs{Y} + \log(q)^3 + (\xi + \chi_D(q)) (\log\log(q))^2}$ field operations.
\end{lem}
\begin{proof}
The algorithm proceeds as follows:
\begin{enumerate}
\item Use the Bray algorithm to find probable generators $Y$ for $\Cent_G(j)$. 
\item Use the MeatAxe to split up the module of $\gen{Y}$ and verify that it splits up as in Conjecture \ref{pr_centraliser_facts}. Use Theorem \ref{thm_sz_conj_recognition} to verify that the groups acting on the $4$-dimensional submodules are Suzuki groups. Return to the first step if not. It then follows from Proposition \ref{pr_invol_facts} that $\Zent(\Cent_G(j)) \leqslant \gen{Y}$.
\item Use Lemma \ref{lem_centraliser_generators} to determine if $\gen{Y} = \Cent_G(j)$. Return to the first step if not. Since the algorithm of Lemma \ref{lem_centraliser_generators} has no false positives, this is Las Vegas. 
\end{enumerate}
By Proposition \ref{pr_invol_facts}, $\OR{1}$ elements is sufficient
to generate $\Cent_G(j)$ with high probability, so the expected time
complexity is as stated, and the elements of $Y$ will be found as
$\SLP$s of the same length as $j$.

From Lemma \ref{lem_centraliser_generators} we also obtain $u$, $Z$,
$\varphi$ and $\pi$ as needed. We see from its proof that $u$ and $Z$ will be found as $\SLP$s of the same length as $j$. 
\end{proof}

\begin{lem} \label{lem_centraliser_membership}
  There exists a Las Vegas algorithm that, given 
\begin{itemize}
\item $\gen{Y}, \gen{Z} \leqslant G = \gen{X} = \LargeRee(q)$ such that $\gen{Y} = \Cent_G(j)$ and $\gen{Z} = \O2(\gen{Y})$ where $j \in G$ is an involution of class $2A$ and $Y, Z$ are given as $\SLP$s in $X$ of length $\OR{n}$,
\item an effective surjective homomorphism $\varphi : \gen{Y} \to \Sz(q)$,
\item an effective injective homomorphism $\pi : \Sz(q) \to \gen{Y}$,
\item $g \in \GL(26, q)$, 
\end{itemize}
decides whether or not $g \in \gen{Y}$ and if so
  returns an $\SLP$ of $g$ in $X$ of length $\OR{n(\log(q)(\log\log(q))^2 + \abs{Z})}$. The algorithm has expected time complexity
  $\OR{\xi + \log(q)^3 + \abs{Z}}$ field operations.
\end{lem}
\begin{proof}
Note that $\varphi$ consists of a change of basis followed by a projection to a submatrix, and hence can be applied to any element of $\GL(26, q)$ using $\OR{1}$ field operations.
\begin{enumerate}
\item Use Algorithm \ref{alg:sz_main_alg} to express $\varphi(g)$ in the generators
  of $\Sz(q)$, or return \texttt{false} if $\varphi(g) \notin \Sz(q)$.
  Hence we obtain an $\SLP$ for $(\pi \circ \varphi)(g)$ in $Y$, of length $\OR{\log(q)(\log\log(q))^2}$.
\item Now $h = g((\pi \circ \varphi)(g))^{-1} \in \gen{Z}$. Using the
  elements of $Z$, we can apply row reduction to $h$, and hence obtain
  an $\SLP$ for $h$ in $Z$ of length $\OR{\abs{Z}}$. Return
  \texttt{false} if $h$ is not reduced to the identity matrix using
  $Z$.
\item Since $Y$ and $Z$ are $\SLP$s in $X$, in time $\OR{\log(q) (\log\log(q))^2}$ we obtain an $\SLP$ for $g$ in $X$, of the specified length.
\end{enumerate}
The expected time complexity then follows from Theorem \ref{thm_element_to_slp}.
\end{proof}

We are now ready to state our modified Ryba algorithm, which assumes that the precomputations given by the above results have been done.

\begin{figure}[ht]
\begin{codebox}
\refstepcounter{algorithm}
\label{alg:ryba_alg}
\Procname{\kw{Algorithm} \ref{alg:ryba_alg}: $\proc{Ryba}(X, g, X_A, X_B)$}
\li \kw{Input}: $\gen{X} \leqslant \GL(26, q)$ such that $\gen{X} = \LargeRee(q)$, $g \in \GL(26, q)$. 
\zi Involution centralisers $\gen{X_A}, \gen{X_B} \leqslant \gen{X}$ for involutions $j_A, j_B \in \gen{X}$ 
\zi of class $2A$ and $2B$, respectively. Effective surjective homomorphisms 
\zi $\varphi_A : \gen{X_A} \to \Sz(q)$, $\varphi_B : \gen{X_B} \to \SL(2, q)$. Effective injective homomorphisms 
\zi $\pi_A : \Sz(q) \to \gen{X_A}$, $\pi_B : \SL(2, q) \to \gen{X_B}$ and $Z_A \subseteq \gen{X_A}$, $Z_B \subseteq \gen{X_B}$ 
\zi such that $\gen{Z_A} = \O2(\gen{X_A})$ and $\gen{Z_B} = \O2(\gen{X_B})$.
\li \kw{Output}: If $g \in \gen{X}$, \const{true} and an $\SLP$ of $g$ in $X$, otherwise \const{false}.
\zi \Comment{\proc{FindEvenOrderElement} is given by Remark \ref{rem_find_even_order_elt}}
\li Use Theorem \ref{thm_nonexplicit_recognition} to determine if $g \in \gen{X}$ and return \const{false} if not.
\li \Repeat
\li     $h := \proc{FindEvenOrderElement}(X, g)$
\li     Let $w_{h}$ be the $\SLP$ returned for $h$.
\li     Let $z$ be an involution obtained from $h$ by powering up.
\li \Until $z$ is of class $2A$.
\li Find random involution $x \in \gen{Z_A}$ of class $2B$.
\li Let $y$ be an involution obtained from $xz$ by powering up.
\li Find $c \in \gen{X}$ as $\SLP$ in $X$, such that $x^c = j_B$.
\li Let $w_y$ be an $\SLP$ for $y^c$ in $X$ 
\li \If $y$ is of class $2A$
\zi \Then
\li     Find $c \in \gen{X}$ as $\SLP$ in $X$, such that $y^c = j_A$. Let $Y := X_A$.
\zi \Else
\li     Find $c \in \gen{X}$ as $\SLP$ in $X$, such that $y^c = j_B$. Let $Y := X_B$.
    \End
\zi \kw{end}
\li Let $w_z$ be an $\SLP$ for $z^c$ in $X$. 
\li Find $c \in \gen{X}$ as $\SLP$ in $X$, such that $z^c = j_A$.
\li Let $w_{hg}$ be an $\SLP$ for $h^c$ in $X$. 
\li Let $w_g := w_h^{-1} w_{hg}$ be an $\SLP$ for $g$ in $X$.
\li \Return \const{true}, $w_g$
\end{codebox}
\end{figure}

\begin{thm}
Assume Conjectures \ref{pr_invol_o2_facts}, \ref{mult_implies_even_order2}, \ref{pr_bigree_goodelts} and \ref{conj:bigree_trick}, and an oracle for the discrete
logarithm problem in $\F_q$. Algorithm \ref{alg:ryba_alg} is a Las
Vegas algorithm with expected time complexity $\OR{(\xi + \chi_D(q)) \log\log(q) + \log(q)^3 + \abs{Z_A} + \abs{Z_B}}$ field
operations. The length of the returned $\SLP$ is $\OR{n(\log(q)(\log\log(q))^2 + \abs{Z_A} + \abs{Z_B})}$ where $n$ is the length of the $\SLP$s for $X_A, X_B, Z_A, Z_B$ in $X$.
\end{thm}
\begin{proof}
By Theorem \ref{thm_find_even_order_elt}, the length of $w_{h}$ is $\OR{\log \log(q)}$. By Remark \ref{rem_find_even_order_elt}, $h = h_1 g$ and $w_h$ is an $\SLP$ for $h_1$.

Then $z$ is found using Proposition \ref{pr_element_powering}, and by
Proposition \ref{pr_trick_involution_class} it is of class $2A$ with
high probability. By Proposition \ref{pr_invol_facts}, the class can be determined by computing the Jordan form.

By Conjecture \ref{pr_invol_o2_facts}, $x$ will
have class $2B$ with high probability, and then $xz$ has even order by
Proposition \ref{pr_invol_facts}. Again we use Proposition \ref{pr_element_powering} to find $y$. 

Using the dihedral trick, we find $c$. Note that $c$ will be found as an $\SLP$ of length $\OR{n}$, since we have an $\SLP$ for $x$, and we can assume that the $\SLP$ for $j_B$ has length $\OR{n}$. Now $\gen{x, z}$ is
dihedral with central involution $y$, so $y^c \in \gen{X_B}$. Using the $\SL(2,q)$ version of Lemma \ref{lem_centraliser_membership}, we find $w_y$ using $\OR{\xi +
  \log(q)^3 + \abs{Z_B}}$ field operations, and $w_y$ has length $\OR{n
  (\log(q)(\log\log(q))^2 + \abs{Z_B})}$.

The next $c$ is again found using the dihedral trick, and comes as
an $\SLP$ of the same length as $w_y$. Then $z^c \in \gen{Y}$
since $y$ is central in $\gen{x, z}$. Hence we again use Lemma
\ref{lem_centraliser_membership} (or its $\SL(2,q)$ version) to obtain $w_z$, with the same length as $w_y$ (or with $Z_B$ replaced by $Z_A$).

Finally, $h$ clearly centralises $z$, and we now have an $\SLP$ for $z$,
so we obtain another $c$ as $\SLP$ in $X$, and use Lemma
\ref{lem_centraliser_membership} to obtain an $\SLP$ for $h^c$. Hence
we obtain $w_{gh}$, which is an $\SLP$ for $h$, and finally an $\SLP$ $w_g$
for $g$. Since $w_h$ has length $\OR{\log\log(q)}$, the length of $w_g$ is as specified.

The expected time complexity follows from Theorem \ref{thm_find_even_order_elt}, Lemma \ref{lem_centraliser_membership} and Proposition \ref{pr_element_powering}.
\end{proof}

\subsection{Conjugates of the standard copy}

We now consider the situation where we are given $\gen{X} \leqslant \GL(26,
q)$, such that $\gen{X} \cong \LargeRee(q)$, so that $\gen{X}$ is a
conjugate of $\LargeRee(q)$, and the problem is to find $g \in \GL(26,
q)$, such that $\gen{X}^g = \LargeRee(q)$.

\begin{lem} \label{lem_sz_compement}
  Assume Conjectures \ref{pr_centraliser_facts}, \ref{pr_invol_o2_facts} and \ref{pr_sz26_facts}. There exists a Las Vegas
  algorithm that, given
\begin{itemize}
\item $G = \gen{X} \leqslant \GL(26, q)$ such that $\gen{X} \cong \LargeRee(q)$, 
\item $\gen{Y} \leqslant \gen{X}$ such that $\gen{Y} = \Cent_G(j)$ for some involution $j \in G$ of class $2A$,
\item $g \in \gen{X}$ such that $\abs{g} \mid q - 1$, $\gen{g} \cap \gen{Y} = \gen{1}$ and $P =
  \gen{Y, g}$ is contained in a maximal parabolic in $\gen{X}$, 
\end{itemize}
finds $\gen{Z} \leqslant \gen{Y, g}$ such that $\gen{Z} = \Cent_P(g) \cong \Sz(q) \times \Cent_{q - 1}$ and $\gen{W} \leqslant \gen{Z}$ such that $\gen{W} \cong \Sz(q)$.

The expected time complexity is $\OR{\sigma_0(\log(q)) \log(q)}$ field operations. If $Y$ and $g$ are given as $\SLP$s in $X$ of length $\OR{n}$, then $Z$ and $W$ will be returned as $\SLP$s in $X$, also of length $\OR{n}$.
\end{lem}
\begin{proof}
  By Conjecture \ref{pr_centraliser_facts} we have $\gen{Y} \cong
  [q^{10}] {:} \Sz(q)$. Since $\gen{g} \cap \gen{Y} = \gen{1}$, it
  follows from Proposition \ref{pr_invol_facts} that $\gen{g}$ lies in
  the cyclic group $\Cent_{q - 1}$ on top of $P$. Then $g$ acts fixed-point freely on $\O2(P)$ and hence $\Cent_P(g) \cong
  \Sz(q) \times \Cent_{q - 1}$ and $\Cent_P(g)^{\prime} \cong \Sz(q)$.

The algorithm proceeds as follows:
\begin{enumerate}
\item Choose random $a_0 \in P$ and use Corollary \ref{cl_the_formula} to find $b_1$, such that $a_1 = b_1^{-1} a_0$ centralises $g$ modulo $\Phi(\O2(P))$. 
\item Use Corollary \ref{cl_the_formula} to find $b_2$, such that $c_1
  = b_2^{-1} a_1$ centralises $g$ modulo $\Phi(\Phi(\O2(P)))$. By
  Conjecture \ref{pr_invol_o2_facts}, $\Phi(\Phi(\O2(P))) =
  \gen{1}$, so $c_1 \in \Cent_P(g)$. Similarly find $c_2 \in
  \Cent_P(g)$.
\item Now $Z = \set{g, c_1, c_2}$ satisfies $\gen{Z} \leqslant
  \Cent_P(g)$, so find probable generators $W$ for $\gen{Z}^{\prime}$. Clearly, $\gen{Z} = \Cent_P(g)$ if and only if $\gen{W} \cong \Sz(q)$. Use the MeatAxe to split up the module for $\gen{W}$ and verify
  that it has the structure given by Conjecture \ref{pr_sz26_facts}.
  Return to the first step if not. 

\item From the $4$-dimensional
  submodules, we obtain an image $W_4$ of $W$ in $\GL(4, q)$. Use
  Theorem \ref{thm_sz_conj_recognition} to determine if $\gen{W_4}
  \cong \Sz(q)$. Return to the first step if not.
\end{enumerate}

By Proposition \ref{sz_2_generation}, two random elements generate
$\Sz(q)$ with high probability, so the probability that
$\gen{Z}^{\prime} \cong \Sz(q)$ is also high. Hence by Theorem \ref{thm_sz_conj_recognition}, the expected time complexity is as stated.
\end{proof}

\begin{lem} \label{lem_find_sz_wr_2} 
Assume the Suzuki Conjectures, the Big Ree Conjectures and an oracle for the discrete logarithm
  problem in $\F_q$. There exists a Las Vegas algorithm that, given
  $\gen{X} \leqslant \GL(26, q)$ such that $\gen{X} \cong
  \LargeRee(q)$, finds $\gen{Y} \leqslant \gen{X}$ and $g \in \gen{X}$
  such that $\gen{Y} \cong \Sz(q)$ and $\gen{Y, g} \cong \Sz(q) \wr
  \Cent_2$. The elements of $Y$ are expressed as $\SLP$s in $X$ of
  length $\OR{\log\log(q)}$; $g$ is expressed as an $\SLP$ in $X$ of length
  $\OR{(\log\log(q))^2}$. The algorithm has expected time complexity
  $\OR{\abs{X} + \log(q)^3 + (\xi + \chi_D(q)) (\log\log(q))^2}$ field operations.
\end{lem}
\begin{proof}
  The idea is to first find one copy of $\Sz(q)$ by finding a
  centraliser of an involution of class $2A$, which has structure
  $[q^{10}] {:} \Sz(q)$, then use the Formula to find the $\Sz(q)$
  inside this. Next we find a maximal parabolic inside this $\Sz(q)$
  and use the dihedral trick and the Formula to conjugate it back to
  our involution centraliser. The conjugating element together with
  the $\Sz(q)$ will generate a copy of $\Sz(q) \wr \Cent_2$.

The algorithm proceeds as follows:
\begin{enumerate}
\item Use Algorithm \ref{alg:find_even_order_elt} to find an element
  of even order and then use Proposition \ref{pr_element_powering} to
  find an involution $j \in \gen{X}$. Repeat until $j$ is of class
  $2A$, which by Proposition \ref{pr_trick_involution_class} happens
  with high probability.
\item Use Lemma \ref{lem_bigree_find_centraliser} to find $\gen{C}
  \leqslant \gen{X}$ and $y_1 \in \gen{X}$ such that $\gen{C} =
  \Cent_G(j)$, $\abs{y_1} \mid q - 1$, $y_1 \notin \Cent_G(j)$ and
  $\gen{C, y_1}$ is contained in a maximal parabolic in $\gen{X}$.
\item Use Lemma \ref{lem_sz_compement} to find $\gen{Y} \leqslant \gen{C, y_1}$ such that $\gen{Y} \cong \Sz(q)$ and $\gen{Y}$ commutes with $y_1$.

\item Use the MeatAxe to split up the module of $\gen{Y}$. By
  Conjecture \ref{pr_sz26_facts} we obtain $4$-dimensional
  submodules, and hence a homomorphism $\rho : \gen{Y}
  \to \GL(4, q)$.

\item Use Theorem \ref{section:suzuki_main_results} to find an effective
  isomorphism $\varphi : \rho(\gen{Y}) \to \Sz(q)$.

\item Use the first steps of the proof of Theorem
  \ref{sz_thm_pre_step} to find $a^{\prime}, y_2^{\prime} \in
  \rho(\gen{Y})$, as $\SLP$s in the generators of $\rho(\gen{Y})$, such that $\abs{a^{\prime}} = 4$, $\abs{y_2^{\prime}}
  \mid q - 1$ and $\gen{a^{\prime}, y_2^{\prime}}$ is contained in a
  maximal parabolic in $\rho(\gen{Y})$. Evaluate the $\SLP$s on $Y$ to obtain $a, y_2 \in \gen{Y}$ with similar properties.

\item Using the dihedral trick, find $h_1 \in \gen{X}$ such that
  $j^{h_1} = a^2$ and let $y_3 = y_1^{h_1}$. Since $\gen{a, y_2}$ is a
  proper subgroup of $\gen{Y}$, it follows that $\gen{\Cent_G(a^2),
    y_2}$ is a proper subgroup of $G$, and hence it is contained in a maximal
  parabolic. Clearly $\gen{\Cent_G(a^2), y_3}$ is also contained in the same maximal
  parabolic. We know from the structure of $\gen{a, y_2}$ that $y_2 \notin \Cent_G(a^2)$, and since $y_1 \notin \Cent_G(j)$, it follows that $y_3 \notin
  \Cent_G(a^2)$, so $\gen{y_2}$ and $\gen{y_3}$ both lie in a group of shape $[q^{10}] : (q - 1)$ and hence are conjugate modulo $\O2(\Cent_G(a^2))$. Now we want to conjugate $\gen{y_3}$ to $\gen{y_2}$ while fixing $a^2$.
\item Diagonalise $y_2$ and $y_3$ to obtain $\varsigma(a_2, b_2)$ and $\varsigma(a_3, b_3)$. Use the discrete logarithm
  oracle to find an integer $k \in \Z$ such that $a_2^k = a_3$. If no
  such $k$ exists, then find another pair of $a, y_2$, but this can
  only happen if $\abs{y_2}$ is a proper divisor of $q - 1$. 

\item Notice that $y_3$ also can diagonalise to $\varsigma(a_3^{-1},
  b_3)$, so now \[y_2^k \equiv y_3 \mod \O2(\Cent_G(a^2))\] or $y_2^k \equiv y_3^{-1} \mod \O2(\Cent_G(a^2))$. In the latter case, invert $y_3$.
\item Use Lemma
  \ref{lem_the_formula} to find $c_1 \in \gen{y_2^k, y_3}$
  such that \[(y_2^k)^{c_1} \equiv y_3 \mod \Phi(\O2(\Cent_G(a^2)))\]
  and then $c_2 \in \gen{(y_2^k)^{c_1}, y_3}$ such that \[(y_2^k)^{c_1 c_2} \equiv y_3 \mod \Phi(\Phi(\O2(\Cent_G(a^2)))).\] By
  Conjecture \ref{pr_invol_o2_facts}, $\Phi(\Phi(\O2(\Cent_G(a^2)))) =
  \gen{1}$, so $\gen{y_2}^{c_1 c_2} = \gen{y_3}$.
\item Now $h_2 = (c_1 c_2)^{-1} \in \Cent_G(a^2)$ since by Lemma
  \ref{lem_the_formula}, both $c_1$ and $c_2$ centralise $a^2$.
  Clearly, $h_2$ conjugates $\gen{y_3}$ to $\gen{y_2}$. Hence $g = h_1
  h_2$ conjugates $\gen{j, y_1}$ to $\gen{a^2, y_2}$, and by
  construction $\gen{a^2, y_2} \leqslant \gen{Y}$ commutes with both
  $j$ and $y_1$.  Since $\gen{j, y_1}$ is contained in a maximal
  parabolic of another copy of $\Sz(q)$, it follows that $\gen{Y}$
  commutes with this copy. Thus $\gen{Y, g} \cong \Sz(q) \wr \Cent_2$.
\end{enumerate}

By Theorem \ref{thm_find_even_order_elt}, the length of the $\SLP$ for
$j$ is $\OR{\log\log(q)}$. Then by Lemma
\ref{lem_bigree_find_centraliser}, the $\SLP$s of $C$ and $y_1$ will
also have length $\OR{\log\log(q)}$. By Lemma \ref{lem_sz_compement},
the $\SLP$s for $Y$ will have length $\OR{\log \log(q)}$. By Theorem \ref{sz_thm_pre_step}, the $\SLP$s for $a$ and $y_2$ will have length $\OR{\log\log(q)^2}$, and hence $h_1$ and $h_2$ will also have this length.

The expected time complexity follows from Lemma
\ref{lem_bigree_find_centraliser}.
\end{proof}

\begin{lem} \label{lem_final_conjugation}
  Assume Conjecture \ref{pr_sz_sz_facts1}. There exists a Las Vegas algorithm that, given $\gen{X}, \gen{Y}, \gen{Z} \leqslant \GL(26, q)$ such
  that $\gen{X} \cong \LargeRee(q) = \gen{Y}$, $\gen{Z} \cong \Sz(q)
  \times \Sz(q)$ and $\gen{Z} \leqslant \gen{X} \cap \gen{Y}$, finds
  $g \in \GL(26, q)$ such that $\gen{X}^g = \gen{Y}$. The algorithm has expected time complexity $\OR{\abs{X} + \abs{Y} + \log(q)}$ field operations.
\end{lem}
\begin{proof}
  Let $M$ be the module of $\gen{Z}$. Observe that $g$ must centralise
  $\gen{Z}$, so $g \in \Cent_{\GL(26, q)}(\gen{Z}) = \Aut(M) \subseteq
  \EndR_{\gen{Z}}(M)$. By Conjecture \ref{pr_sz_sz_facts1}, the endomorphism
  ring of $M$ has dimension $3$. The algorithm proceeds as follows:
\begin{enumerate}
\item Use the MeatAxe to find $e_1, e_2, e_3 \in \GL(26, q)$ such that $\EndR_{\gen{Z}}(M) = \oplus_{i = 1}^3 \gen{e_i}$.
\item Let $x_1, x_2, x_3$ be indeterminates and let
\[h(x_1, x_2, x_3) = \sum_{i = 1}^3 x_i e_i \in \Mat_{26}(\F_q[x_1, x_2, x_3]). \]
\item Use the MeatAxe to find matrices $Q_X, Q_Y$ corresponding to the quadratic forms preserved by $\gen{X}$ and $\gen{Y}$.
\item A necessary condition on $h(x_1, x_2, x_3)$ for it to conjugate $\gen{X}$ to $\gen{Y}$ is the following equation:
\begin{equation}
h(x_1, x_2, x_3) Q_X h(x_1, x_2, x_3) = Q_Y
\end{equation}
which determines $676$ quadratic equations in $x_1, x_2, x_3$.
\item Hence we obtain $P \subseteq \F_q[x_1, x_2, x_3]$ where each element of $P$ has degree $2$ and $\abs{P} \leqslant 676$. Every $f \in P$ has $7$ coefficients, so we obtain an additive group homomorphism $\rho : \gen{P} \to \F_q^7$. 
\item Now $\dim \rho(P) = 3$, so let $b_1, b_2, b_3$ be a basis of $\rho(P)$ and let $f_i = \rho^{-1}(b_i) \in P$ for $i = 1, \dotsc, 3$.
\item By Proposition \ref{pr_sz_sz_facts2}, the variety of the ideal $I = \gen{f_1, f_2, f_3} \trianglelefteq \F_q[x_1, x_2, x_3]$ has size $2$. Find this variety using Theorem \ref{thm_poly_eqns_many_vars}.
\item Let $h_1, h_2$ be the corresponding elements of
  $\EndR_{\gen{Z}}(M)$. Clearly, one of them must also lie in
  $\Aut(M)$ and conjugate $\gen{X}$ to $\gen{Y}$, since our $g$ exists
  and satisfies the necessary conditions which led to $h_1$ and $h_2$.
  Use Theorem \ref{thm_nonexplicit_recognition} to determine which
  $h_i$ satisfies $\gen{X}^{h_i} = \gen{Y}$.
\end{enumerate}

Clearly, this is a Las Vegas algorithm and the expected time complexity follows from Theorem \ref{thm_poly_eqns_many_vars}.
\end{proof}

\begin{thm} \label{thm_bigree_conj_problem}
  Assume the Suzuki Conjectures, the Big Ree Conjectures, and an oracle for the discrete logarithm problem in $\F_q$. There
  exists a Las Vegas algorithm that, given $\gen{X}, \gen{Y} \leqslant
  \GL(26, q)$ such that $\gen{X} \cong \LargeRee(q) = \gen{Y}$, finds
  $g \in \GL(26, q)$ such that $\gen{X}^g = \gen{Y}$. The algorithm
  has expected time complexity $\OR{\abs{X} + \log(q)^3 + (\xi + \chi_D(q))
    (\log\log(q))^2}$ field operations.
\end{thm}
\begin{proof}
The algorithm proceeds as follows:
\begin{enumerate}
\item Use Lemma \ref{lem_find_sz_wr_2} to find $S_1 \subseteq \gen{X}$
  and $c_1 \in \gen{X}$ such that $\gen{S_1} \cong \Sz(q)$ and
  $\gen{S_1, c_1} \cong \Sz(q) \wr \Cent_2$. Let $\gen{S_2} = \gen{S_1}^{c_1}$ so that
  $\gen{S_1, S_2} \cong \Sz(q) \times \Sz(q)$.
\item Similarly find $S_3, S_4 \subseteq \gen{Y}$ such that $\gen{S_3, S_4} \cong \Sz(q) \times \Sz(q)$.
\item Use the MeatAxe to split up the modules of each $\gen{S_i}$. By
  Proposition \ref{pr_sz26_facts}, we obtain $4$-dimensional
  submodules, and hence surjective homomorphisms $\rho_i : \gen{S_i}
  \to \GL(4, q)$ for $i = 1, \dotsc, 4$.

\item Use Theorem \ref{section:suzuki_main_results} to find effective
  isomorphisms $\varphi_i : \rho_i(\gen{S_i}) \to \Sz(q)$ for $i = 1,
  \dotsc, 4$. If $S$ is the standard generating set for $\Sz(q)$, we
  then obtain standard generating sets $R_i$ for $\gen{S_i}$ by
  obtaining $\SLP$s for $S$ in the generators of $\rho_i(\gen{S_i})$
  and then evaluating these on $S_i$. 

\item Let $M_1$ be the module for $\gen{R_1, R_2}$ and let $M_2$ be
  the module for $\gen{R_3, R_4}$. Now $M_1 \cong M_2$, and all $R_i$
  are equal, so we can use the MeatAxe to find a change of basis
  matrix $h_1 \in \GL(26, q)$ between $M_1$ and $M_2$.
\item Then $\gen{S_1, S_2}^{h_1} = \gen{S_3, S_4}$ and hence $\gen{S_3,
    S_4} \leqslant \gen{X}^{h_1} \cap \gen{Y}$. Use Lemma
  \ref{lem_final_conjugation} to find $h_2 \in \GL(26, q)$ such that
  $\gen{X}^{h_1 h_2} = \gen{Y}$. Hence $g = h_1 h_2$.
\end{enumerate}

Clearly, this is a Las Vegas algorithm and the expected time complexity follows from Lemma \ref{lem_find_sz_wr_2}.
\end{proof}

\subsection{Constructive recognition}

Finally, we can now state and prove our main theorem.

\begin{thm} \label{thm_bigree_constructive_recognition} Assume the Suzuki Conjectures, the Big Ree Conjectues, and an oracle for the discrete logarithm
  problem in $\F_q$. There exists a Las Vegas algorithm that, given
  $\gen{X} \leqslant \GL(26, q)$ satisfying the assumptions in Section
  \ref{section:algorithm_overview}, with $q = 2^{2m + 1}$, $m > 0$ and
  $\gen{X} \cong \LargeRee(q)$, finds an effective isomorphism
  $\varphi : \gen{X} \to \LargeRee(q)$. The algorithm has expected time
  complexity $\OR{(\xi + \chi_D(q))(\log\log(q))^2 + \abs{X} + \log(q)^3}$ field
  operations.

  The inverse of $\varphi$ is also effective. Each image and pre-image of $\varphi$ can be computed using
  $\OR{1}$ field operations.
\end{thm}
\begin{proof}
  Use Theorem \ref{thm_bigree_conj_problem} to obtain $c \in \GL(26,
  q)$ such that $\gen{X}^c = \LargeRee(q)$. An an effective
  isomorphism $\varphi : \gen{X} \to \LargeRee(q)$ is then defined by
  $g \mapsto g^c$, which clearly can be computed in $\OR{1}$ field operations. The expected time complexity follows from Theorem
  \ref{thm_bigree_conj_problem}. 
\end{proof}

\chapter{Sylow subgroups}
\label{chapter:sylow_subgroups}

We will now describe algorithms for finding and conjugating Sylow subgroups of the exceptional groups under consideration. Hence we consider the following problems:
\begin{enumerate}
\item Given $\gen{X} \leqslant \GL(d, q)$, such that $\gen{X} \cong G$
  for one of our exceptional groups $G$, and given a prime number $p
  \mid \abs{G}$, find $\gen{Y} \leqslant \gen{X}$ such that $\gen{Y}$
  is a Sylow $p$-subgroup of $\gen{X}$.
\item Given $\gen{X} \leqslant \GL(d, q)$, such that $\gen{X} \cong G$ for some of our exceptional groups $G$, and given a prime number $p \mid \abs{G}$ and $\gen{Y}, \gen{Z} \leqslant \gen{X}$ such that both $\gen{Y}$ and $\gen{Z}$ are Sylow $p$-subgroups of $\gen{X}$, find $c \in \gen{X}$ such that $\gen{Y}^c = \gen{Z}$.
\end{enumerate}

The second problem is the difficult one, and often there are some
primes that are especially difficult. We will refer to these problems
as the \lq\lq Sylow subgroup problems'' for a certain prime $p$. The
first problem is referred to as \lq\lq Sylow generation'' and the
second as \lq\lq Sylow conjugation''.

\section{Suzuki groups}
\label{section:sylow_suzuki}

We now consider the Sylow subgroup problems for the Suzuki groups. We
will use the notation from Section \ref{section:suzuki_theory}, and we
will make heavy use of the fact that we can use Theorem
\ref{cl_sz_constructive_recognition} to constructively recognise the
Suzuki groups. Hence we assume that $G$ satisfies the assumptions in
Section \ref{section:algorithm_overview}, so $\Sz(q) \cong G \leqslant
\GL(d, q)$.

By Theorem \ref{thm_suzuki_props} and Proposition \ref{pr_sz_subgroups_conj}, $\abs{G} = q^2
(q^2 + 1) (q - 1)$ and all three factors are pairwise relatively prime. Hence we
obtain three cases for a Sylow $p$-subgroup $S$ of $G$.

\begin{enumerate}
\item $p$ divides $q^2$, so $p = 2$. Then $S$ is conjugate to $\mathcal{F}$ and hence $S$ fixes a unique
  point $P_S$ of $\OV$, which is easily found using the MeatAxe.

\item $p$ divides $q - 1$. Then $S$ is cyclic and conjugate
  to a subgroup of $\mathcal{H}$. Hence $S$ fixes two distinct points
  $P, Q \in \OV$, and these points are easily found using the MeatAxe.

\item $p$ divides $q^2 + 1$. Then $S$ is cyclic and
  conjugate to a subgroup of $U_1$ or $U_2$, and $S$ has no fixed
  points. This is the difficult case.
\end{enumerate}

\begin{thm} \label{thm:sz_sylow2}
  Assume the Suzuki Conjectures and an oracle for the discrete logarithm problem in $\F_q$. There exist Las
  Vegas algorithms that solve the Sylow subgroup problems for $p = 2$
  in $\Sz(q) \cong G \leqslant \GL(d, q)$. Once constructive recognition has been performed, the
  expected time complexity of the Sylow generation is $\OR{d^3\log(q) (\log\log(q))^2}$ field operations, and $\OR{d^3
    (\abs{Y} + \abs{Z} + \log(q)(\log\log(q))^2) + \log(q)^3}$ field operations for the Sylow
  conjugation.
\end{thm}
\begin{proof}
  Let $H = \Sz(q)$. Using the effective isomorphism, it is sufficient
  to solve the problems in the standard copy. 

  The constructive
  recognition uses Theorem \ref{sz_thm_pre_step} to find sets $L$ and
  $U$ of \lq \lq standard generators'' for $H_{P_{\infty}}$ and
  $H_{P_0}$, respectively.

  A generating set for a random Sylow $2$-subgroup $S$ of $H$ can
  therefore be computed by taking a random $h \in H$, and as
  generating set for $S$ take $L^h$. To obtain a
  Sylow $2$-subgroup $R$ of $G$, note that we already have the $\OR{\log(q)}$ generators
  $L$ and $h$ as $\SLP$s of length $\OR{(\log\log(q))^2}$, so we can
  evaluate them on $X$. Hence the expected time complexity is as stated.
  
  Given two Sylow $2$-subgroups of $G$, we use the effective
  isomorphism to map them to $H$ using $\OR{d^3 (\abs{Y} + \abs{Z})}$
  field operations. We can use the MeatAxe to find the points $P_Y,
  P_Z \in \OV$ that are fixed by the subgroups. Then use Lemma \ref{sz_row_operations} with $U$ to find $a \in H_{P_0}$ such
  that $P_Y a = P_{\infty}$ and $b \in H_{P_0}$ such that $P_Z b =
  P_{\infty}$. Then $a b^{-1}$ conjugates one Sylow subgroup to the
  other, and we already have this element as an $\SLP$ of length
  $\OR{\log(q) (\log\log(q))^2}$. 

  In the case where $P_Y = P_{\infty}$ and $P_Z = P_0$, we know that
  $T$ maps $P_Y$ to $P_Z$. Then use Algorithm \ref{alg:sz_main_alg} to
  obtain an $\SLP$ for $T$ of length $\OR{\log(q)(\log\log(q))^2}$.

  Hence we can evaluate it on $X$ and obtain $c \in G$ that conjugates
  $\gen{Y}$ to $\gen{Z}$. Thus the expected time complexity follows
  from Lemma \ref{sz_row_operations} and Theorem
  \ref{thm_element_to_slp}.
\end{proof}

\begin{thm} \label{thm:sz_sylow_easy_cyclic}
  Assume the Suzuki Conjectures and an oracle for the discrete logarithm problem in $\F_q$. There exist Las
  Vegas algorithms that solve the Sylow subgroup problems for $p \mid q - 1$
  in $\Sz(q) \cong G \leqslant \GL(d, q)$. Once constructive recognition has been performed, the
  expected time complexity of the Sylow generation is $\OR{(\xi(d) + \log(q) \log\log(q) + d^3)\log\log(q)}$ field operations, and $\OR{d^3
    (\abs{Y} + \abs{Z} + \log(q)(\log\log(q))^2) + \log(q)^3}$ field operations for the Sylow
  conjugation.
\end{thm}
\begin{proof}
  Let $H = \Sz(q)$. In this case the Sylow generation is easy, since
  we can determine the highest power $e$ of $p$ such that $p^e \mid q
  - 1$, find a random element of pseudo-order $q - 1$, use Proposition
  \ref{pr_element_powering} to obtain an element of order $p^e$, then
  evaluate its $\SLP$ on $X$. By Proposition \ref{sz_totient_prop},
  the expected number of random selections, and hence the length of
  the $\SLP$, is $\OR{\log \log q}$, and we need $\OR{\log(q)
    \log\log(q)}$ field operations to find the order. Then we evaluate
  the $\SLP$ on $X$ using $\OR{d^3 \log \log(q)}$ field operations, so
  the expected time complexity is as stated.

  For the Sylow conjugation, recall that the constructive recognition
  uses Theorem \ref{sz_thm_pre_step} to find sets $L$ and $U$ of \lq \lq
  standard generators'' for $H_{P_{\infty}}$ and $H_{P_0}$,
  respectively.

  Given two Sylow $p$-subgroups of $G$, we use the effective
  isomorphism to map them to $H$ using $\OR{d^3 (\abs{Y} + \abs{Z})}$
  field operations. Let $H_Y, H_Z \leqslant H$ be the resulting
  subgroups. Using the MeatAxe, we can find $P_Y \neq Q_Y \in \OV$
  that are fixed by $H_Y$ and $P_Z \neq Q_Z \in \OV$ that are fixed by
  $H_Z$. Order the points so that $P_Y \neq P_0$ and $P_Z \neq P_0$.

  Use Lemma \ref{sz_row_operations} with $U$ to find $a_1 \in H_{P_0}$
  such that $P_Y a_1 = P_{\infty}$. Then use $L$ to find $a_2 \in
  H_{P_{\infty}}$ such that $Q_Y a_1 a_2 = P_0$. Similarly we find
  $b_1, b_2 \in H$ such that $P_Z b_1 b_2 = P_{\infty}$ and $Q_Z b_1
  b_2 = P_0$. Then $a_1 a_2 (b_1 b_2)^{-1}$ conjugates one Sylow
  subgroup to the other, and we already have this element as an $\SLP$
  of length $\OR{\log(q) (\log\log(q))^2}$. Hence we can evaluate it on $X$ and obtain
  $c \in G$ that conjugates $\gen{Y}$ to $\gen{Z}$. Thus the expected
  time complexity is as stated.

\end{proof}

\begin{lem} \label{lem_find_irreducible_conjs}
There exists a Las Vegas algorithm that, given $g, h \in \Sz(q)$ with
$\abs{g} = \abs{h}$ both dividing $q \pm t + 1$, finds $c \in \gen{g}$
such that $c$ is conjugate to $h$ in $\Sz(q)$. The expected time
complexity is $\OR{\log q}$ field operations.
\end{lem}
\begin{proof}
The algorithm proceeds as follows:
\begin{enumerate}
\item Find the minimal polynomial $f_1$ of $g$. By Theorem
  \ref{thm_suzuki_props}, $g$ acts irreducibly on $\F_q^4$,
  so $f_1$ is irreducible. Let $F = \F_q[x] / \gen{f_1}$, so that $F$ is
  the splitting field of $f_1$.  Clearly $F \cong \F_{q^4}$ and
  $F^{\times} = \gen{\alpha}$ where $\alpha$ is a root of
  $f_1$. Moreover, $x \mapsto g$ defines an isomorphism $F \to
  \F_q(\gen{g})$, where the latter is the subfield of $\Mat_4(\F_q)$
  generated by $g$.
\item Find the minimal polynomial $f_2$ of $h$. Then $F$ is also the
  splitting field of $f_2$, and if $\beta \in F$ is a root of $f_2$,
  then $\beta$ is expressed as a polynomial $f_3$ in $\alpha$, with
  coefficients in $\F_q$. Similarly, $x \mapsto h$ defines an isomorphism
  $F \to \F_q(\gen{h})$.
\item 
  Now $f_3$ defines an isomorphism $\F_q(\gen{h}) \to \F_q(\gen{g})$
  as $h \mapsto f_3(g)$, because $h$ and $f_3(g)$ have the same
  minimal polynomial.
  Hence if we let $c = f_3(g)$, then $c$ has the same eigenvalues as
  $h$, so $c$ and $h$ are conjugate in $\GL(4, q)$. Then $\abs{c} = \abs{h}$ and $\Tr(c) =
  \Tr(h)$, so by Proposition \ref{sz_conjugacy_classes}, $c$ is also
  conjugate to $h$ in $\Sz(q)$. Moreover, both $\gen{g}$ and $\gen{c}$
  are subgroups of $\F_q(\gen{g})^{\times}$, but since $\abs{c} =
  \abs{h} = \abs{g}$, they must be the same. Thus $c \in \gen{g}$.
\end{enumerate}

By \cite{MR1350753}, the minimal polynomial is found using $\OR{1}$
field operations. Hence by Theorem \ref{thm_solve_univariate_polys}, the expected time complexity is as stated.
\end{proof}

The conjugation algorithm described in the following result is essentially due to Mark Stather and Scott Murray.

\begin{thm} \label{thm:sz_sylow_hard_cyclic}
  Assume the Suzuki Conjectures and an oracle for the discrete logarithm problem in $\F_q$. There exist Las
  Vegas algorithms that solve the Sylow subgroup problems for $p \mid q^2 + 1$
  in $\Sz(q) \cong G \leqslant \GL(d, q)$. Once constructive recognition has been performed, the
  expected time complexity of the Sylow generation is $\OR{(\xi(d) + \log(q) \log\log(q) + d^3)\log\log(q)}$ field operations, and $\OR{\xi(d) + d^3
    (\abs{Y} + \abs{Z} + \log(q) (\log\log(q))^2) + \log(q)^3}$ field operations for the Sylow
  conjugation.
\end{thm}
\begin{proof}
  Let $H = \Sz(q)$. The Sylow generation is analogous to the case in Theorem \ref{thm:sz_sylow_easy_cyclic}, since by Proposition \ref{sz_totient_prop}, we can easily find elements of pseudo-order $q \pm t + 1$.
  
  Given two Sylow $p$-subgroups
  of $G$, we use the effective isomorphism to map them to $H$, using
  $\OR{d^3 (\abs{Y} + \abs{Z})}$ field operations. The resulting
  generating sets must contain elements $h_y, h_z \in H$ of order $p$,
  since the Sylow subgroups are cyclic. Let $J$ be as in
  \eqref{standard_symplectic_form}, so that $H$ preserves the symplectic
  form $J$, and let $\Psi$ be as in Section \ref{section:sz_alt_def}:
  the automorphism of $\Sp(4, q)$ whose set of fixed points is $\Sz(q)$.

\begin{enumerate}
\item Use Lemma \ref{lem_find_irreducible_conjs} to replace $h_y$. Henceforth assume that $h_y$ and $h_z$ are conjugate in $H$.

\item Find $g_1 \in \GL(4, q)$ such that $h_y^{g_1} = h_z$. This can be
  done by a similarity test, or computation of Jordan forms, using
  \cite{steel_canonical_forms}. The next step is to find a matrix $g_2$ such that $g_2 g_1 \in \Sp(4, q)$, and $g_2 g_1$ also conjugates $h_y$ to $h_z$.

\item Let $A = \Cent_{\GL(4, q)}(h_y)$ (the automorphism group of the
  module of $\gen{h_y}$). Since $\gen{h_y}$ is irreducible, by Schur's Lemma $A \cong \F^{\times}_{q^4}$. Such an
  isomorphism $\theta : A \to \F_{q^4}$, and its inverse, can be found
  using the MeatAxe.

\item Now define an automorphism $\varphi$ of $A$ as $\varphi(a) = J a^T
  J^{-1}$. Then $\varphi$ has order $2$ and $A \cong \F^{\times}_{q^4}$. Recall that $\F_{q^4}$ has a unique automorphism of order $2$ ($k \mapsto
  k^{q^2}$), which must be $\theta \circ \varphi \circ \theta^{-1}$.

\item Let $t_1 = J g_1^{-T} J^{-1} g_1^{-1}$ and observe that $t_1 \in A$.
  We want to find $g_2 \in A$ such that $g_2 g_1 J (g_2 g_1)^T = J$,
  which is equivalent to $\varphi(g_2) g_2 = t_1$. Using $\theta$, this is
  a norm equation in $\F_{q^4}$ over $\F_{q^2}$. In other words, we consider
  $\theta(g_2)^{q^2 + 1} = \theta(t_1)$, which is solved for example
  using \cite[Lemma 2.2]{maximal_subs}.

\item Hence $g_2 g_1$ lies in $\Sp(4, q)$, and $g_2$ fixes $h_y$, so
  $g_2 g_1$ conjugates $h_y$ to $h_z$. The next step is to find a
  matrix $g_3$, such that $g_3 g_2 g_1 \in \Sz(q)$, and such that $g_3 g_2 g_1$
  also conjugates $h_y$ to $h_z$. Hence we want $g_3 \in \Sp(4, q)$
  and $\Psi(g_3 g_2 g_1) = g_3 g_2 g_1$.

\item Find $w \in \F^{\times}_{q^4}$ of order $q^2 + 1$, by taking the
  $q^2 - 1$ power of a primitive element. Then $\varphi(\theta^{-1}(w))
  \theta^{-1}(w) = 1$, which implies that $ \theta^{-1}(w) J
  \theta^{-1}(w)^T = J$, and hence $\theta^{-1}(w) \in \Sp(4, q)$.
  Similarly, every element of $\gen{w}$ gives rise to matrices in
  $\Sp(4, q)$. We therefore want to find an integer $i$, such that $w^i
  = g_3$.

\item Moreover, we want 
\begin{equation}
\begin{split} 
\Psi(\theta^{-1}(w^i) g_2 g_1) &= \theta^{-1}(w^i) g_2 g_1 \Leftrightarrow \\
\Psi(\theta^{-1}(w))^i \Psi(g_2 g_1) &= \theta^{-1}(w)^i g_2 g_1 \Leftrightarrow \\
\Psi(\theta^{-1}(w))^i \theta^{-1}(w)^{-i} &= g_2 g_1 \Psi(g_2 g_1)^{-1} \Leftrightarrow \\
\theta(\Psi(\theta^{-1}(w)))^i w^{-i} &= \theta(g_2 g_1 \Psi(g_2 g_1)^{-1})
\end{split}
\end{equation}
so if we let $t_2 = \theta(g_2 g_1 \Psi(g_2 g_1)^{-1})$, we want to find an integer $i$ such that $\theta(\Psi(\theta^{-1}(w)))^i w^{-i} = t_2$.

\item Use the discrete log oracle to find $k$ such that $\theta(\Psi(\theta^{-1}(w))) = w^k$. Since $g_2 g_1 \in \Sp(4, q)$ it follows that
  $t_2 \in \gen{w}$. Use the discrete log oracle to find $n
  \in \Z$ such that $w^n = t_2$. Our equation turns into $(k - 1) i \equiv n \pmod{q^2 + 1}$, which we solve to find $i$.
\end{enumerate}

By \cite[Lemma 2.2]{maximal_subs}, this whole process uses expected
$\OR{\log q}$ field operations. Finally we use the effective isomorphism to map the conjugating element back to $G$. Hence the time complexity is as stated.
\end{proof}

\section{Small Ree groups}
\label{section:sylow_smallree}

We now consider the Sylow subgroup problems for the small Ree groups.
We will use the notation from Section \ref{section:ree_theory}, and we will make heavy use of the
fact that we can use Theorem \ref{cl_ree_constructive_recognition} to
constructively recognise the small Ree groups. Hence we assume that $G$ satisfies the assumptions in Section \ref{section:algorithm_overview}, so $\Ree(q) \cong G \leqslant \GL(d, q)$.

By Proposition \ref{sylow3_props}, we obtain $4$ cases for a Sylow $p$-subgroup $S$ of $G$.

\begin{enumerate}
\item $p = 2$, so that by \cite[Chapter $11$, Theorem $13.2$]{huppertIII}, $S$ is elementary abelian of order $8$ and $[\Norm_G(S) : S] = 21$.
\item $p$ divides $q^3$, so $p = 3$. Then $S$ is conjugate to $U(q)$ and hence $S$ fixes a unique
  point $P_S$ of $\OV$, which is easily found using the MeatAxe.

\item $p$ divides $q - 1$ and $p > 2$. Then $S$ is cyclic and conjugate
  to a subgroup of $H(q)$. Hence $S$ fixes two distinct points
  $P, Q \in \OV$, and these points are easily found using the MeatAxe.

\item $p$ divides $q^3 + 1$ and $p > 2$. Then $S$ is cyclic and
  conjugate to a subgroup of $A_0$, $A_1$ or $A_2$ from Proposition \ref{ree_maximal_subgroup_list}. In this case, we have only solved the Sylow generation problem.
\end{enumerate}

\begin{thm} \label{thm:ree_sylow3}
  Assume the small Ree Conjectures and an oracle for the discrete logarithm problem in $\F_q$. There exist Las
  Vegas algorithms that solve the Sylow subgroup problems for $p = 3$
  in $\Ree(q) \cong G \leqslant \GL(d, q)$. Once constructive recognition has been performed, the
  expected time complexity of the Sylow generation is $\OR{d^3(\log(q) \log\log(q))^2}$ field operations, and $\OR{d^3
    (\abs{Y} + \abs{Z} + (\log(q) \log\log(q))^2) + \log(q)^3}$ field operations for the Sylow
  conjugation.
\end{thm}
\begin{proof}
  Let $H = \Ree(q)$. The constructive recognition uses Theorem
  \ref{thm_pre_step} to find
  sets $L$ and $U$ of \lq \lq standard generators'' for $H_{P_{\infty}}$ and
  $H_{P_0}$, respectively.

  A generating set for a random Sylow $3$-subgroup $S$ of $H$ can
  therefore be computed by finding a random $h \in H$, and as
  generating set for $S$ take $\set{m^g \mid m \in L}$. To obtain a
  Sylow subgroup $R$ of $G$, note that we already have the
  $\OR{\log(q)}$ generators of $L$ and $h$ as $\SLP$s of length
  $\OR{\log(q)(\log\log(q)^2)}$, so we can evaluate them on $X$. Hence
  the expected time complexity is as stated.
  
  Given two Sylow $3$-subgroups of $G$, we use the effective
  isomorphism to map them to $H$ using $\OR{d^3 (\abs{Y} + \abs{Z})}$
  field operations. We can use the MeatAxe to find the points $P_Y,
  P_Z \in \OV$ that are fixed by the subgroups. Then use Lemma \ref{ree_row_operations} with $U$ to find $a \in H$, such
  that $P_Y a = P_{\infty}$, and $b \in H$, such that $P_Z b =
  P_{\infty}$. Then $a b^{-1}$ conjugates one Sylow subgroup to the
  other, and we already have this element as an $\SLP$ of length
  $\OR{(\log(q) \log\log(q))^2}$. 

  In the case where $P_Y = P_{\infty}$ and $P_Z = P_0$, we know that
  $\Upsilon$ maps $P_Y$ to $P_Z$. Then use Algorithm \ref{alg:ree_element_to_slp} to obtain an $\SLP$ for $\Upsilon$ of length $\OR{(\log(q)\log\log(q))^2}$.

  Hence we can evaluate it on $X$ and obtain $c \in G$ that conjugates
  $\gen{Y}$ to $\gen{Z}$. Thus the expected time complexity follows
  from Lemma \ref{ree_row_operations} and Theorem \ref{thm_element_to_slp_complexity}.
\end{proof}

\begin{thm} \label{thm:ree_sylow_easy_cyclic}
  Assume the small Ree Conjectures and an oracle for the discrete logarithm problem in $\F_q$. There exist Las
  Vegas algorithms that solve the Sylow subgroup problems for $p \mid q - 1$, $p > 2$,
  in $\Ree(q) \cong G \leqslant \GL(d, q)$. Once constructive recognition has been performed, the
  expected time complexity of the Sylow generation is $\OR{(\xi(d) + \log(q) \log\log(q) + d^3)\log\log(q)}$ field operations, and $\OR{d^3
    (\abs{Y} + \abs{Z} + (\log(q) \log\log(q))^2) + \log(q)^3}$ field operations for the Sylow
  conjugation.
\end{thm}
\begin{proof}
  Let $H = \Ree(q)$. In this case the Sylow generation is easy, since we can determine the highest power $e$ of $p$ such that $p^e \mid q - 1$, find a random element of pseudo-order $q - 1$, use Proposition \ref{pr_element_powering} to obtain an element of order $p$, then
  evaluate its $\SLP$ on $X$. By Proposition \ref{ree_totient_prop}, the expected number of random
  selections, and hence the length of the $\SLP$ is $\OR{\log \log
    q}$, and we need $\OR{\log(q) \log\log(q)}$ field operations to
  find the order. Then we evaluate the $\SLP$ on $X$ using $\OR{d^3 \log
    \log(q)}$ field operations, so the expected time complexity is as
  stated.

  For the Sylow conjugation, recall that the constructive recognition
  uses Theorem \ref{thm_pre_step} to find sets $L$ and $U$ of \lq \lq
  standard generators'' for $H_{P_{\infty}}$ and $H_{P_0}$,
  respectively.

  Given two Sylow $p$-subgroups of $G$, we use the effective
  isomorphism to map them to $H$ using $\OR{d^3 (\abs{Y} + \abs{Z})}$
  field operations. Let $H_Y, H_Z \leqslant H$ be the resulting subgroups. Using the MeatAxe, we can
  find $P_Y, Q_Y \in \OV$ that are fixed by $H_Y$ and $P_Z, Q_Z
  \in \OV$ that are fixed by $H_Z$. Order the points so that $P_Y \neq P_0$ and $P_Z \neq P_0$.

  Use Lemma \ref{ree_row_operations} with $U$ to find $a_1 \in H_{P_0}$,
  such that $P_Y a_1 = P_{\infty}$. Then use $L$ to find $a_2 \in
  H_{P_{\infty}}$, such that $Q_Y a_1 a_2 = P_0$. Similarly we find
  $b_1, b_2 \in H$, such that $P_Z b_1 b_2 = P_{\infty}$ and $Q_Z b_1
  b_2 = P_0$. Then $a_1 a_2 (b_1 b_2)^{-1}$ conjugates one Sylow
  subgroup to the other, and we already have this element as an $\SLP$
  of length $\OR{(\log(q) \log\log(q))^2}$. Hence we can evaluate it on $X$ and obtain
  $c \in G$ that conjugates $\gen{Y}$ to $\gen{Z}$. Thus the expected
  time complexity is as stated.
\end{proof}

\begin{thm} \label{thm:ree_sylow_hard_cyclic}
  Assume the small Ree Conjectures and an oracle for the discrete logarithm problem in $\F_q$. There exists a Las
  Vegas algorithm that solves the Sylow generation problem for $p \mid q^3 + 1$, $p > 2$,
  in $\Ree(q) \cong G \leqslant \GL(d, q)$. Once constructive recognition has been performed, the
  expected time complexity of the Sylow generation is $\OR{(\xi(d) + \log(q) \log\log(q) + d^3)\log\log(q)}$ field operations.
\end{thm}
\begin{proof}
  Let $H = \Ree(q)$. The Sylow generation is easy, since
  we can find an element of pseudo-order $q \pm 3t + 1$ or $(q + 1) / 2$, use Proposition \ref{pr_element_powering} to obtain an element of order $p$, then
  evaluate its $\SLP$ on $X$. By Proposition \ref{ree_totient_prop}, the expected number of random
  selections, and hence the length of the $\SLP$ is $\OR{\log \log
    q}$, and we need $\OR{\log(q) \log\log(q)}$ field operations to
  find the order. Then we evaluate the $\SLP$ on $X$ using $\OR{d^3 \log
    \log(q)}$ field operations, so the expected time complexity is as
  stated.
\end{proof}

This result is due to Mark Stather and is the same as \cite[Lemma $4.35$]{statherthesis}.

\begin{lem}\label{index2sylow}
Let $G$ be a group and let $k \in \Z$ be such that $\abs{G} = 2^k n$ with $n$ odd. Let $P \leqslant G$ have order $2^{k-1}$. Let
$\gen{P,x}$ and $\gen{P,y}$ be Sylow $2$-subgroups of
$G$. Then $\abs{xy} = 2^t$ for some $t \in \Z$ if and only if $\gen{P,x} = \gen{P,y}$. Moreover if $\abs{xy}= 2^t (2s + 1)$, then $\gen{P,x}^{(yx)^s}=\gen{P,y}$
\end{lem}
\begin{proof}
Let $H = \gen{P,xy}$. Then $H$ is a subgroup of $\gen{P, x, y}$ of index $2$, that contains $P$, but does not contain $x$ or $y$.
Since $P$ has index $2$ in both $\gen{P,x}$ and $\gen{P, y}$ it follows that $xy \in \Norm_H(P)$. But $P$ is a Sylow $2$-subgroup of $H$ so,
\[ \abs{xy}=2^k \Leftrightarrow xy \in P \Leftrightarrow \gen{P,x}= \gen{P,y}\]
The second statement is an application of Proposition \ref{dihedral_trick_basis}, modulo $H$.
\end{proof}

\begin{thm} \label{thm:ree_sylow2_gen}
  Assume the small Ree Conjectures and an oracle for the discrete logarithm problem in $\F_q$. There exists a Las
  Vegas algorithm that solves the Sylow generation problem for $p = 2$
  in $\Ree(q) \cong G \leqslant \GL(d, q)$. Once constructive recognition has been performed, the
  expected time complexity is $\OR{(\xi(d) + d^3 \log(q))\log\log(q) + \log(q)^3 + \chi_D(q)}$ field operations.
\end{thm}
\begin{proof}
  Let $H = \Ree(q)$. We want to find three
  commuting involutions in $H$. Using the first three steps of the
  algorithm in Section
  \ref{section:find_stab_element}, we find an involution $j_1 \in H$
  and $\Cent_H(j_1)^{\prime} \cong \PSL(2, q)$. Using the notation of
  that algorithm, we can let the second involution $j_2 \in
  \Cent_H(j_1)^{\prime}$ be $\pi_7(\pi_3(j))$ where $j$ is the second
  matrix in \eqref{sl2_stab_gens}.

  We then want to find the third involution in the centraliser of
  $j_2$ in $\Cent_H(j_1)^{\prime}$. In our case this centraliser has structure
  $(\Cent_2 \times (\Cent_2 {:} A_0)$. Hence its proportion of elements of
  even order is $3/4$, and $1/2$ of its elements are involutions other
  than $j_2$. Using the Bray algorithm we can therefore compute random
  elements of this centraliser until we find such an involution $j_3$.

  Clearly $j_1, j_2, j_3$ will all commute. As in the proof of
  Corollary \ref{cl_find_stab_element}, the expected time to find
  $j_1$, constructively recognise $\Cent_H(j_1)^{\prime}$ and find
  $j_2$ is $\OR{\xi \log\log(q) + + \log(q)^3 + \chi_D(q)}$ field
  operations. By the above, the expected time to find $j_3$ is
  $\OR{1}$ field operations.  The involutions will be found as
  $\SLP$s, where $j_1$ and $j_3$ have length $\OR{1}$, because the
  generators of $\Cent_H(j_1)^{\prime}$ are $\SLP$s of length
  $\OR{1}$. By Lemma \ref{lem_psl_maps} and Theorem
  \ref{thm_psl_recognition}, $j_2$ has length $\OR{\log(q)
    \log\log(q)}$. Thus we can evaluate them on $X$ using $\OR{d^3
    \log(q) \log\log(q)}$ field operations, and the expected time
  complexity is as stated.
\end{proof}

\begin{thm} \label{thm:ree_sylow2_conj} Assume the small Ree Conjectures and an
  oracle for the discrete logarithm problem in $\F_q$. There exists a
  Las Vegas algorithm that solves the Sylow conjugation problem for $p
  = 2$ in $\Ree(q) \cong G \leqslant \GL(d, q)$. Once constructive
  recognition has been performed, the expected time complexity is
  $\OR{\xi(d) + d^3 ((\log(q) \log\log(q))^2 + \abs{Y} +
    \abs{Z}) + \log(q)^3}$ field operations.
\end{thm}
\begin{proof}
  Let $H = \Ree(q)$. Given two Sylow $2$-subgroups of $G$, we use the
  effective isomorphism to map them to $H$, using $\OR{d^3 (\abs{Y} +
    \abs{Z})}$ field operations. The resulting generating sets are $P =
  \set{y_1, y_2, y_3}$ and $S = \set{z_1, z_2, z_3}$,
  where both the $y_i$ and $z_i$ are commuting involutions. We may assume that $\gen{Y} \neq \gen{Z}$ so that $P \neq S$.

The algorithm proceeds as follows:
\begin{enumerate}
\item By Proposition \ref{ree_dihedral_trick}, we can use the dihedral trick in $H$ and hence find $c_1 \in H$ such that $y_1^{c_1} = z_1$. We then want to conjugate $y_2$ to $z_2$ while fixing $z_1$.

\item Using the first steps of the algorithm in Section \ref{section:find_stab_element},
  we find $C_1 = \Cent_H(z_1) \cong \gen{z_1} \times \PSL(2, q)$, and use Theorem \ref{thm_psl_naming} to determine when we have the whole of $C_1$. Observe that $y_i^{c_1}, z_i \in C_1$ for all $i$.

\item Choose random $g \in C_1^{\prime}$. If $z_2 g$ has odd order,
  then $z_2 \in C_1^{\prime}$. Conversely, if $z_2 \in
  C_1^{\prime}$ then $z_2 g$ has odd order with probability
  $\OR{1}$. Similarly, if $z_1 z_2 g$ has odd order, then $z_1
  z_2 \in C_1^{\prime}$, and the probability is the same.
  Repeat until either $z_2$ or $z_1 z_2$ has been proved to lie in
  $C_1^{\prime}$ and replace $z_2$ with this element. Do the
  same procedure with $y_2^{c_1}, y_3^{c_1}, z_3$.

\item Now $y_2^{c_1}, z_2, y_3^{c_1}, z_3 \in C_1^{\prime} \cong \PSL(2,
  q)$, and by \cite[Theorem $13$]{parkerwilson06}, the dihedral trick works in $\PSL(2, q)$. Hence find $c_2 \in C_1^{\prime}$ such that $y_2^{c_1 c_2} = z_2$.

\item Let $\abs{y_3^{c_1 c_2} z_3} = 2^t s$, where $s$ is odd. If $s = 1$ then let $c_3 = 1$ and otherwise let $c_3 = (y_3^{c_1 c_2} z_3)^{(s - 1) / 2}$. By Lemma \ref{index2sylow}, $P^{c_1 c_2 c_3} = S$.
\end{enumerate}

Finally, we use the effective isomorphism to map $c_1 c_2 c_3$ back to
$G$. As in the proof of Corollary \ref{cl_find_stab_element},
$C_1$ is found using expected $\OR{\xi(d) + \log(q) \log\log(q)}$
field operations, if we let $\varepsilon = \log\log(q)$ in Theorem \ref{thm_psl_naming}. The expected time complexity of the effective isomorphism follows from Theorem \ref{cl_ree_constructive_recognition}.
\end{proof}

\section{Big Ree groups}

We now consider the Sylow subgroup problems for the Big Ree groups. We
will use the notation from Section \ref{section:big_ree_theory}, and
we will make heavy use of the fact that we can use Theorem
\ref{thm_bigree_constructive_recognition} to constructively recognise
the Big Ree groups. However, we can only do this in the natural
representation, and hence we will only consider the Sylow subgroup
problems in the natural representation. Hence we assume that $\LargeRee(q) \cong G \leqslant \GL(26, q)$.

It follows from \cite{MR1700483} that if $g \in G$ then $\abs{g}$ is even, or divides any of the numbers $\set{q + 1, q - 1, q \pm t + 1, q^2 \pm 1, q^2 - q + 1, q^2 \pm tq + q \pm t + 1}$. Hence we
obtain several cases for a Sylow $p$-subgroup $S$ of $G$.

\begin{enumerate}
\item $p = 2$. Then $S$ has order $q^{12}$ and lies in $\Cent_G(j)$
  for some involution $j$ of class $2A$. It consists of
  $\O2(\Cent_G(j))$ extended by a Sylow $2$-subgroup in a Suzuki
  group contained in the centraliser.

\item $p$ divides $o \in \set{q - 1, q \pm t + 1}$. Then $S$ has
  structure $\Cent_p \times \Cent_p$. If $G \geqslant H \cong \Sz(q) \times
  \Sz(q)$, then $S$ is contained in $H$ and consists of Sylow
  $p$-subgroups from each Suzuki factor.

\item $p$ divides $q^2 - q + 1$ or $q^2 \pm tq + q \pm t + 1$. Then
  $S$ is cyclic of order $p$, and hence these Sylow subgroups are trivial to find. We do not consider this case.

\item $p$ divides $q + 1$. We do not consider this case.
\end{enumerate}

\begin{thm} \label{thm:bigree_sylow2} Assume the Suzuki Conjectures, the Big Ree Conjectures and an
  oracle for the discrete logarithm problem in $\F_q$. There exists a Las
  Vegas algorithm thats solve the Sylow generation problem for $p = 2$
  in $\LargeRee(q) \cong \gen{X} \leqslant \GL(26, q)$. Once constructive
  recognition has been performed, the expected time complexity is $\OR{\log(q) (\log\log(q))^2}$ field operations.
\end{thm}
\begin{proof}
  Let $G = \gen{X}$. We see from the proof of Theorem
  \ref{thm_bigree_conj_problem} that during the constructive
  recognition, we find $\Cent_G(j)$ for some involution $j \in G$ of
  class $2A$. We also find $\gen{Y}, \gen{Z} \leqslant \Cent_G(j)$ such
  that $\gen{Y} = \O2(\Cent_G(j))$ and $\gen{Z} \cong \Sz(q)$.
  Moreover, $\gen{Z}$ is constructively recognised, and $Y, Z$ are
  expressed as $\SLP$s in $X$ of length $\OR{\log\log(q)}$.

  Hence we can apply Theorem \ref{thm:sz_sylow2} and obtain a Sylow
  $2$-subgroup $\gen{W}$ of $\gen{Z}$ using $\OR{\log(q) (\log\log(q))^2}$ field
  operations. Now $\gen{Y, W}$ is a Sylow $2$-subgroup of $G$.
\end{proof}

\begin{thm} \label{thm:bigree_sylow_sz} Assume the Suzuki Conjectures, the Big Ree Conjectures and an
  oracle for the discrete logarithm problem in $\F_q$. There exist Las
  Vegas algorithms that solve the Sylow generation problems for $p
  \mid q - 1$ or $p \mid q \pm t + 1$ in $\LargeRee(q) \cong G
  \leqslant \GL(26, q)$. Once constructive recognition has been
  performed, the expected time complexity is $\OR{(\xi +
    \log(q)\log\log(q))\log\log(q)}$ field operations.
\end{thm}
\begin{proof}
  Let $G = \gen{X}$. We see from the proof of Theorem
  \ref{thm_bigree_conj_problem} that during the constructive
  recognition, we find $\gen{Y_1}, \gen{Y_2} \cong \Sz(q)$, and they commute, so $\gen{Y_1, Y_2} \cong \Sz(q) \times \Sz(q)$. Moreover,
  $\gen{Y_1}$ and $\gen{Y_2}$ are constructively recognised, and $Y_1, Y_2$ are expressed as $\SLP$s in $X$ of length
  $\OR{\log\log(q)^2}$. 

  Hence we can apply Theorem \ref{thm:sz_sylow_easy_cyclic} or
  \ref{thm:sz_sylow_hard_cyclic} and obtain Sylow $p$-subgroups of
  $\gen{Y_1}$ and $\gen{Y_2}$, using $\OR{(\xi +
    \log(q)\log\log(q))\log\log(q)}$ field operations. From the proof
  of the Theorems, we see that there will be a constant number of
  generators, which will be expressed as $\SLP$s in $\gen{Y_i}$ of
  length $\OR{\log\log(q)}$. Hence we can evaluate them on $X$ using $\OR{(\log\log(q))^3}$ field operations.
\end{proof}

%

\chapter{Maximal subgroups}

We will now describe algorithms for finding and conjugating maximal
subgroups of the exceptional groups under consideration. Hence we
consider the following problems:
\begin{enumerate}
\item Given $\gen{X} \leqslant \GL(d, q)$, such that $\gen{X} \cong G$
  for some of our exceptional groups $G$, find representatives
  $\gen{Y_1}, \dotsc, \gen{Y_n}$ of the conjugacy classes of (some or
  all of) the maximal subgroups of $G$.
\item Given $\gen{X} \leqslant \GL(d, q)$, such that $\gen{X} \cong G$
  for some of our exceptional groups $G$, and given $\gen{Y}, \gen{Z}
  \leqslant \gen{X}$ such that $\gen{Y}$ and $\gen{Z}$ are conjugate
  to a specified maximal subgroup of $\gen{X}$, find $c \in \gen{X}$ such that
  $\gen{Y}^c = \gen{Z}$.
\end{enumerate}

It will turn out that because of the results about Sylow subgroup
conjugation in Chapter \ref{chapter:sylow_subgroups}, the second
problem will most often be easy. The first problem is therefore the
difficult one. We will refer to these problems as the \lq\lq maximal
subgroup problems''. The first problem is referred to as \lq\lq
maximal subgroup generation'' and the second as \lq\lq maximal
subgroup conjugation''.

\section{Suzuki groups}
\label{section:maximal_suzuki}

We now consider the maximal subgroup problems for the Suzuki groups.
We will use the notation from Section \ref{section:suzuki_theory}, and we will make heavy use of the
fact that we can use Theorem \ref{cl_sz_constructive_recognition} to
constructively recognise the Suzuki groups. Hence we assume that $G$ satisfies the assumptions in
Section \ref{section:algorithm_overview}, so $\Sz(q) \cong G \leqslant
\GL(d, q)$.

The maximal subgroups are given by Theorem \ref{sz_maximal_subgroups}.

\begin{thm}
  Assume the Suzuki Conjectures, an oracle for the
  discrete logarithm problem in $\F_q$ and an oracle for the integer factorisation problem. There exist Las Vegas
  algorithms that solve the maximal subgroup conjugation in $\Sz(q)
  \cong G \leqslant \GL(d, q)$.
  Once constructive recognition has been performed, the expected time
  complexity is $\OR{\xi(d) \log(q) + \log(q)^3 + d^3 (\log(q) (\log\log(q))^2 + (\abs{Y} + \abs{Z}) \sigma_0(\log(q))) + d^2 \log(q) \sigma_0(\log(q)) + \chi_F(4, q)}$
  field operations.
\end{thm}
\begin{proof}
  Let $H = \Sz(q)$. In each case, we first use the effective
  isomorphism to map the subgroups to $H$ using $\OR{d^3 (\abs{Y} +
    \abs{Z})}$ field operations. Therefore we henceforth assume that
  $\gen{Y}, \gen{Z} \leqslant H$.
  
  Observe that all maximal subgroups, except the Suzuki groups over
  subfields, are the normalisers of corresponding cyclic subgroups or
  Sylow $2$-subgroups. The Sylow conjugation algorithms can conjugate
  these cyclic subgroups around, not only the Sylow subgroups that they
  contain. Moreover, the cyclic subgroups and Sylow $2$-subgroups are
  the derived groups of the corresponding maximal subgroups.

  Hence we can obtain probable generators for the cyclic subgroups or
  Sylow $2$-subgroups using $\OR{\xi(d) \log(q)}$ field operations, and we can verify that we have the whole of these subgroups as follows:
\begin{enumerate}
\item For $\mathcal{B}_i$, the generators of the derived group $U_i$
  must contain an element of order $q \pm t + 1$.
\item For $\Norm_H(\mathcal{H})$, the generators of the derived group
  $\mathcal{H}$ must contain an element of order $q - 1$.
\item For $\mathcal{F} \mathcal{H}$, we need not obtain the whole
  derived group $\mathcal{F}$. It is enough that we obtain a subgroup of
  the derived group that fixes a unique point of $\OV$, but this might require $\OR{\log(q)}$ generators.
\end{enumerate}

Note that in these cases we need the integer factorisation oracle to
find the precise order. When we have generators for the cyclic
subgroups, we use Theorem \ref{thm:sz_sylow2},
\ref{thm:sz_sylow_easy_cyclic} and \ref{thm:sz_sylow_hard_cyclic} to
find conjugating elements for the cyclic subgroups. These elements will
also conjugate the maximal subgroups, because they normalise the cyclic
subgroups.

Finally, consider the case when $\gen{Y}$ and $\gen{Z}$ are isomorphic
to a Suzuki group over $\F_s < \F_q$. In this case we first use the
algorithm in Section \ref{section:smallerfield}, to obtain $c_1, c_2
\in \GL(4, q)$ that conjugates $\gen{Y}$ and $\gen{Z}$ into $\GL(4,
s)$. Then we use Theorem \ref{thm_conj_problem} to find $c_3 \in
\GL(4, s)$ that conjugates the Suzuki groups to each other. Hence $c =
c_1 c_3 c_2^{-1}$ conjugates $\gen{Y}$ to $\gen{Z}$, and therefore it
normalises $H$. However, $c$ does not necessarily lie in $H$, but only
in $\Norm_{\GL(4, q)}(H) \cong H {:} \F_q$, since neither $c_1$ nor
$c_2$ are guaranteed to lie in $H$. Therefore $c = (\gamma I_4) g$
where $g \in H$ and $\gamma \in \F_q$. We can find $\gamma$ by
calculating the determinant and taking its (unique) $4$th root, so we
can divide by the scalar matrix, and we then end up with $g$, which also
conjugates $\gen{Y}$ to $\gen{Z}$.

Finally we use the effective isomorphism to map $g$ to $G$. The expected time complexity follows from Theorem
\ref{thm:sz_sylow2}, \ref{thm:sz_sylow_easy_cyclic}, \ref{thm:sz_sylow_hard_cyclic} and \ref{thm_conj_problem}.
\end{proof}

\begin{lem} \label{lem_sz_hard_maximals}
If $g, h \in G = \Sz(q)$ satisfy that $\abs{g} = 2$, $\abs{h} = 4$, $\abs{gh} = 4$ and $\abs{gh^2} = q \pm t + 1$, then $\gen{g, h} \cong \Cent_{q \pm t + 1} {:} \Cent_4$ and hence is a maximal subgroup of $G$.
\end{lem}
\begin{proof}
  Clearly, $\gen{g, h}$ is an image in $G$ of the group $H = \gen{x, y
    \mid x^2, y^4, (xy)^4} \cong \Z^2 {:} \Cent_4$. Since $H$ is
  soluble and $gh^2$ has the specified order, $\gen{g, h}$ must
  be one of the $\mathcal{B}_i$ from Theorem \ref{sz_maximal_subgroups}.
\end{proof}


\begin{thm}
  Assume the Suzuki Conjectures, an oracle for the discrete logarithm
  problem in $\F_q$ and an oracle for the integer factorisation
  problem. There exist Las Vegas algorithms that solve the maximal
  subgroup generation in $\Sz(q) \cong G \leqslant \GL(d, q)$.  Once
  constructive recognition has been performed, the expected time
  complexity is $\OR{\xi(d) (\sigma_0(\log(q)) + \log\log(q)) + \chi_F(4, q) \log\log(q) + \sigma_0(\log(q)) (\log(q)^3 + d^3 \log(q) (\log\log(q))^2)}$
  field operations.
\end{thm}
\begin{proof}
  Let $H = \Sz(q)$. Using the effective isomorphism, it is sufficient
  to obtain generators for the maximal subgroups in $H$.
  
  Let $\alpha \in \F_q$ be a primitive element. Clearly $\mathcal{F}
  \mathcal{H} = \gen{M^{\prime}(\alpha), S(1, 0)}$, and
  $\Norm_H(\mathcal{H}) = \gen{M^{\prime}(\alpha), T}$. For each $e >
  0$ such that $2e + 1 \mid 2m + 1$, we have $\F_s < \F_q$ where $s = 2^{2e
    + 1}$. Hence $s - 1 \mid q - 1$ and $\Sz(s) = \gen{T, S(1, 0),
    M^{\prime}(\alpha)^{(q - 1) / (s - 1)}}$.

  The difficult case is therefore $\mathcal{B}_1$ and $\mathcal{B}_2$.
  We want to use Lemma \ref{lem_sz_hard_maximals}, with $T$ and
  $S(a,b)$ playing the roles of $x$ and $y$. Hence we proceed as follows:
\begin{enumerate}
\item Choose random $g \in H$ of order $q \pm t + 1$. Note that
  we need the integer factorisation oracle, since we need the precise
  order of $g$. Let $\lambda = \Tr(g)$. 
\item Let $a,b$ be indeterminates and consider the equations $\Tr(T
    S(a, b)) = 0$ and $\Tr(T S(a,b)^2) = \lambda$. If we can find
  solutions for $a,b$, then by Proposition \ref{sz_trace0_order4},
  $\abs{T S(a,b)} = 4$ with high probability, and Proposition
  \ref{sz_conjugacy_classes} implies that $\abs{T S(a,b)^2} = q \pm t
  + 1$.
\item The second equation implies $a = \lambda^{t + 2}$, and
\begin{equation}
\begin{split}
\Tr(T S(a, b)) = &\; a^t + a^{t + 2} + ab + b^t = 0 \Leftrightarrow \\
& a^2 + a^{2t + 2} + a^t b^t + b^2 = 0 \Rightarrow \\
& b^2 + a^{t + 1} b + a^2 + a^{2t} = 0
\end{split}
\end{equation}
where the third equation is $a^t$ times the first added to the second. 
\item The quadratic equation has solutions $b_1 = a^{t+1}
  \sum_{i=1}^{m+1} a^{-2^i}$ and $b_2 = b_1 + a^{t + 1}$, which both
  give the value $a^{s+2}(1 + \sum_{i = 0}^{2m} a^{-2^i})$ of $\Tr(T S(a, b))$.
  Hence repeat with another $g$ if $\sum_{i = 0}^{2m} a^{-2^i} \neq 1$, which happens with probability $1/2$.
\item Lemma \ref{lem_sz_hard_maximals} now implies that
  $\gen{T, S(a, b)}$ is $\mathcal{B}_1$ or $\mathcal{B}_2$. 
\end{enumerate}

Finally, we see that we have $\OR{\sigma_0(\log(q))}$ generators, which we map
back to $G$ using the effective isomorphism. Hence the expected time
complexity is as stated.
\end{proof}

\section{Small Ree groups}
\label{section:maximal_ree}

We now consider the maximal subgroup problems for the small Ree groups.
We will use the notation from Section \ref{section:suzuki_theory}. We will make heavy use of the
fact that we can use Theorem \ref{cl_ree_constructive_recognition} to
constructively recognise the small Ree groups. Hence we assume that $G$ satisfies the assumptions in
Section \ref{section:algorithm_overview}, so $\Ree(q) \cong G \leqslant
\GL(d, q)$.

The maximal subgroups are given by Proposition \ref{ree_maximal_subgroup_list}.

\begin{thm}
  Assume the small Ree Conjectures and an oracle for the discrete
  logarithm problem in $\F_q$. There exist Las Vegas algorithms that
  solve the maximal subgroup conjugation in $\Ree(q) \cong G \leqslant
  \GL(d, q)$ for the point stabiliser, the involution centraliser and
  Ree groups over subfields. Once constructive recognition has been
  performed, the expected time complexity is $\OR{\xi(d) \log(q) +
    \log(q)^3 + \chi_D(q) \log\log(q) + d^3 ((\abs{Y} + \abs{Z}) \sigma_0(\log(q)) +
    (\log(q) \log\log(q))^2) + d^2 \log(q) \sigma_0(\log(q))}$ field operations.
\end{thm}
\begin{proof}
  Let $H = \Ree(q)$. In each case, we first use the effective
  isomorphism to map the subgroups to $H$ using $\OR{d^3 (\abs{Y} +
    \abs{Z})}$ field operations. Therefore we henceforth assume that
  $\gen{Y}, \gen{Z} \leqslant H$.
  
  Observe that the point stabiliser is the normaliser of a Sylow
  $3$-subgroup, which is the derived group of the point stabiliser. We
  can therefore obtain probable generators for the Sylow $3$-subgroup
  using $\OR{\xi(d) \log(q)}$ field operations. We only need enough
  generators so that they generate a subgroup of the derived group
  that fixes a unique point of $\OV$, and this we can easily verify
  using the MeatAxe. When we have generators for this subgroup, we use
  Theorem \ref{thm:ree_sylow3} to find a conjugating element. This
  element will also conjugate the maximal subgroup, because it
  normalises the Sylow subgroup.

For the involution centraliser, we choose random elements of
$\gen{Y}$. Since $\gen{Y} \cong \gen{y} \times \PSL(2, q)$ for some
involution $y$, with probability $\OR{1}$ we will obtain an element of
even order that powers up to $y$. We can check that we obtain $y$
since it is the unique involution that is centralised by $\gen{Y}$ (and therefore by $Y$). Hence we
can find the involutions $y$ and $z$ that are centralised by $\gen{Y}$
and $\gen{Z}$. By Proposition \ref{ree_dihedral_trick} we can use the dihedral trick to find
$c \in H$ that conjugates $y$ to $z$, using $\OR{\log(q) \log\log(q)}$ field operations. Since $\gen{Y}$ and $\gen{Z}$
centralise these, it follows that $\gen{Y}^c = \gen{Z}$.

Finally, consider the case when $\gen{Y}$ and $\gen{Z}$ are isomorphic
to a Ree group over $\F_s < \F_q$. In this case we first use the
algorithm in Section
\ref{section:smallerfield}, to obtain $c_1, c_2 \in \GL(7, q)$ that
conjugates $\gen{Y}$ and $\gen{Z}$ into $\GL(7, s)$. Then we use
Theorem \ref{cl_ree_conjugacy} to
find $c \in \GL(7, s)$ that conjugates the resulting Ree groups to each
other. Hence $c_1 c c_2^{-1}$ conjugates $\gen{Y}$ to $\gen{Z}$, and hence normalises $H$. However, it does not necessarily lie in $H$, but only in $\Norm_{\GL(7,
  q)}(H) \cong H {:} \F_q$, since neither $c_1$ nor $c_2$ has to lie in $H$. Therefore it is of the form $(\gamma I_7) g$, where $g \in
H$ and $\gamma \in \F_q$. We can find $\gamma$ by calculating the determinant and taking the (unique) $7$th root, so we can divide by the scalar matrix, and we then end up with $g$, that also conjugates $\gen{Y}$ to $\gen{Z}$.

Finally we use the effective isomorphism to map the conjugating element back to $G$. The expected time complexity follows from Theorem
\ref{thm:ree_sylow3}, \ref{cl_ree_conjugacy} and \ref{cl_ree_constructive_recognition}.
\end{proof}

\begin{lem} \label{lem_ree_hard_maximals}
If $g, h \in G = \Ree(q)$ satisfy that $\abs{g} = 2$, $\abs{h} = 3$, $\abs{gh} = 6$ and $\abs{[g,h]} = q \pm 3t + 1$ or $\abs{[g,h]} = (q + 1) / 2$, then $\gen{g, h} \cong \Cent_{q \pm 3t + 1} {:} \Cent_6$ or $\gen{g, h} \cong (\Cent_2 \times \Cent_2 \times \Cent_{(q + 1) / 4}) ){:} \Cent_6$ and hence is a maximal subgroup of $G$.
\end{lem}
\begin{proof}
  Clearly, $\gen{g, h}$ is an image in $G$ of the group $H = \gen{x, y
    \mid x^2, y^3, (xy)^6} \cong \Z^2 {:} \Cent_6$. Since $H$ is
  soluble and $[g,h]$ has the specified order, $\gen{g, h}$ must
  be one of the $\Norm_G(A_i)$ from Proposition \ref{ree_maximal_subgroup_list}.
\end{proof}

\begin{lem} \label{ree_hard_maximal_conj2} Let $G = \Ree(q)$. For
    each $k \in \set{q \pm 3t + 1, (q + 1) / 2}$, there exist $x, y \in G =
    \Ree(q)$ such that $\abs{x} = 2$, $\abs{y} = 3$, $\abs{xy} = 6$
    and $\abs{[x, y]} = k$.
\end{lem}
\begin{proof}
  Consider the case $k = q \pm 3t + 1$. There exists $H \leqslant G$
  such that $H = \gen{a}\gen{b} \cong \Cent_k {:} \Cent_6$ with $\abs{a} = k$
  and $\abs{b} = 6$. Observe that $H^{\prime} = \gen{a}$. If $x = a^{-i}
  b^3$ and $y = a^i [b^3, a^i]^{-1} b^{-2}$ then $xy = b$, and $[x, y] = a^j$ for some $j$, depending on $i$. Clearly we can
  choose $i$ such that $\gcd(j, k) = 1$, and hence $\abs{[x, y]} = k$.

  The other case is analogous.
\end{proof}

\begin{lem} \label{lem_ree_hard_maximal}
Let $G = \Ree(q)$ and $\upsilon = h(-1) \Upsilon$. For each $k \in \set{q \pm 3t + 1, (q
  + 1) / 2}$, there exist $a, b \in \F_q$ such that $\abs{\Upsilon
  S(0, a, b)} = 6$ and $\abs{[\Upsilon, S(0, a, b)]} = k$ or
  $\abs{\upsilon S(0, a, b)} = 6$ and $\abs{[\upsilon, S(0, a, b)]} =
  k$.
\end{lem}
\begin{proof}
Let $(x, y)$ be as in Lemma \ref{ree_hard_maximal_conj2}. It is
sufficient to prove that there exists $a, b \in \F_q$ such that $(x, y)$
is conjugate to $(\Upsilon, S(0, a, b))$ or $(\upsilon, S(0, a, b))$.

Since $y$ has order $3$, it fixes a point $P$. Also, $G$ is doubly
transitive so there exists $c_1 \in G$ such that $P_{\infty} c_1 =
P$. Then $y^{c_1} = S(0, a, b)$ for some $a, b \in
\F_q$. Observe that $x^{c_1}$ does not fix $P_{\infty}$, since otherwise
$\abs{[x, y]} = 3$.

Now $P_{\infty} x^{c_1} = R$ for some point $R$, and $G_{P_{\infty}}$
acts transitively on the points other than $P_{\infty}$, so there
exists $c_2 \in G_{P_{\infty}}$ such that $R c_2 = P_0$. Then $x^{c_1
  c_2}$ interchanges $P_0$ and $P_{\infty}$, and so does $\Upsilon$.
Hence $x^{c_1 c_2} \Upsilon^{-1} \in \gen{h(\lambda)}$, so $x^{c_1
  c_2} = h(\lambda)^i \Upsilon$, for some $0 \leqslant i < q - 1$.

Let $k \equiv i / 2 \pmod{(q - 1)/2}$ such that $0 \leqslant k < q -
1$, and let $c_3 = h(\lambda)^k$. There are two possible
values for $k$, either $2k = i$ or $2k = i + (q - 1)/2$. In the former
case $x^{c_1 c_2 c_3} = \Upsilon$, and in the latter case $x^{c_1 c_2
  c_3} = h(\lambda)^{(q - 1) / 2} \Upsilon = \upsilon$.
\end{proof}

\begin{conj} \label{ree_hard_maximal_conj1}
Let $q = 3^{2m + 1}$ for some $m > 0$ and let $t = 3^m$. For every $a \in \F_q^{\times}$, the ideals in $\F_q[b_1, c_1, b_2, c_2]$ generated by the following systems are zero-dimensional:
\begin{equation}
\label{ree_hard_maximal_eqn1} \left\{ \begin{aligned}
& b_2^2 + b_1 b_2 + c_2^2 = 0 \\
& b_1^2 + b_2^3 b_1 + c_1^2 = 0 \\
& 1 - a - b_1^2 + b_2^4 + b_1^2 b_2^2 - c_1^2 - b_1 b_2 c_2^2 + c_2^4 - b_2^2 c_2^2 = 0 \\
& 1 - a^{3 t} - b_2^6 + b_1^4 + b_2^6 b_1^2 - c_2^6 - c_1^2 b_2^3 b_1 + c_1^4 - b_1^2 c_1^2 = 0
\end{aligned} \right. 
\end{equation}
\begin{equation}
\label{ree_hard_maximal_eqn2} \left\{ \begin{aligned}
& b_2^2 - b_1 b_2 - c_2^2 = 0 \\
& b_1^2 - b_2^3 b_1 - c_1^2 = 0 \\
& 1 - a - b_1^2 + b_2^4 + b_1^2 b_2^2 + c_1^2 - b_1 b_2 c_2^2 + c_2^4 + b_2^2 c_2^2 = 0 \\
& 1 - a^{3 t} - b_2^6 + b_1^4 + b_2^6 b_1^2 + c_2^6 - c_1^2 b_2^3 b_1 + c_1^4 + b_1^2 c_1^2 = 0
\end{aligned} \right. 
\end{equation}
\end{conj}
  



\begin{thm}
  Assume the small Ree Conjectures, Conjecture \ref{ree_hard_maximal_conj1}, an oracle for the
  discrete logarithm problem in $\F_q$ and an oracle for the integer factorisation problem. There exist Las Vegas
  algorithms that solve the maximal subgroup generation in $\Ree(q)
  \cong G \leqslant \GL(d, q)$.
  Once constructive recognition has been performed, the expected time
  complexity is $\OR{\xi(d) (\sigma_0(\log(q)) + \log\log(q)) + \chi_F(7, q) \log\log(q) + \sigma_0(\log(q)) (d^3 (\log(q) \log\log(q))^2 + \log(q)^3)}$ field operations.
\end{thm}
\begin{proof}
  Let $H = \Ree(q)$. Using the effective isomorphism, it is sufficient
  to obtain generators for the maximal subgroups in $H$.
  
  Let $\alpha \in \F_q$ be a primitive element. Clearly $U(q) H(q) =
  \gen{h(\alpha), S(1, 0, 0)}$, and by following a procedure similar
  to \cite{elementary_ree}, we see that $\Cent_H(h(-1)) =
  \gen{h(\alpha), \Upsilon, S(0, 1, 0)}$. For each $e > 0$ such that $2e + 1
  \mid 2m + 1$, we have $\F_s < \F_q$ where $s = 3^{2e + 1}$. Hence $s
  - 1 \mid q - 1$ and $\Ree(s) = \gen{\Upsilon, S(1, 0, 0), h(\alpha)^{(q
      - 1) / (s - 1)}}$.

  The difficult cases are therefore the $\Norm_H(A_i)$.  In the light
  of Lemma \ref{lem_ree_hard_maximal}, we can proceed as
  follows:
\begin{enumerate}
\item Choose random $g \in H$ of order $q \pm 3t + 1$ or $(q + 1) /
  2$, corresponding to the order of $A_i$. By Corollary
  \ref{cl_random_selections} this is done using expected $\OR{(\xi(d)
  + \log(q) \log\log(q) + \chi_F(7, q)) \log\log(q)}$ field operations. Note that we need the integer factorisation oracle since we need the precise order in this case.

\item Introduce indeterminates $x, y$ and consider the equations
    $\Tr(\Upsilon S(0, x, y)) = 1$ and $\Tr([\Upsilon, S(0, x,
    y)]) = \Tr(g)$, or
    similarly with $\upsilon$ instead of $\Upsilon$. We want to find
    solutions $a,b$ for $x,y$ as in the Lemma. By Proposition
    \ref{ree_conjugacy_classes}, the trace determines the order in
    these cases, which leads us to consider these equations.
\item Elements of order $6$ have trace $1$. Hence we obtain equations in $x,y$:
\begin{equation}
\begin{split}
\Tr(\Upsilon S(0, x, y)) & = 1 \\
(\Tr(\Upsilon S(0, x, y))^{3t} & = 1 \\
\Tr{[\Upsilon, S(0, x, y)]} &= \Tr{g} \\
(\Tr{[\Upsilon, S(0, x, y)]})^{3t} &= (\Tr{g})^{3t} 
\end{split}
\end{equation}
By letting $b_1 = x$, $b_2 = x^t$, $c_1 = y$, $c_2 = y^t$, this is
precisely one of the systems in Conjecture
\ref{ree_hard_maximal_conj1}, and thus we can use Theorem
\ref{thm_poly_eqns_many_vars} to find
all solutions using $\OR{\log q}$ field operations. 
\item By Lemma \ref{lem_ree_hard_maximal}, there will be
solutions $a, b$. By Lemma
\ref{lem_ree_hard_maximals}, the resulting $S(0,
a, b)$ generates $\Norm_H(A_i)$ together with $\Upsilon$ or
$\upsilon$.
\end{enumerate}

Finally, we see that we have $\OR{\sigma_0(\log(q))}$ generators, which we map
back to $G$ using the effective isomorphism. Hence the expected time
complexity is as stated.
\end{proof}

\section{Big Ree groups}
\label{section:maximal_bigree}

We now consider the maximal subgroup problems for the Big Ree groups.
We will use the notation from Section \ref{section:big_ree_theory},
and we will make heavy use of the fact that we can use Theorem
\ref{thm_bigree_constructive_recognition} to constructively recognise
the Big Ree groups. However, we can only do this in the natural
representation, and hence we will only consider the maximal subgroup
problems in the natural representation.

The maximal subgroups are listed in \cite{malle_bigree}, but we will
only generate some of them.

\begin{thm}
  Assume Conjectures \ref{conjecture_correctness} and
  \ref{conj:bigree_trick}, and an oracle for the discrete logarithm
  problem in $\F_q$. There exist Las Vegas algorithms that, given
  $\LargeRee(q) \cong \gen{X} \leqslant \GL(26, q)$ finds $\gen{Y_1},
  \gen{Y_2}, \gen{Y_3}, \gen{Y_4} \leqslant \gen{X}$ such that $\gen{Y_1}$ and
  $\gen{Y_2}$ are the two maximal parabolics, $\gen{Y_3} \cong
  \Sz(q) \wr \Cent_2$ and $\gen{Y_4} \cong \Sp(4, q) {:} \Cent_2$. Once constructive recognition has
  been performed, the expected time complexity is $\OR{1}$ field
  operations.
\end{thm}
\begin{proof}
  Let $H = \LargeRee(q)$, and let $\varphi : \gen{X} \to H$ be the
  effective isomorphism. Generators for the subgroups are given in
  Proposition \ref{bigree_parabolic_gens} and Proposition
  \ref{bigree_suzuki2_gens}. Since they all have constant size, and
  pre-images of $\varphi$ can be computed in $\OR{1}$ field
  operations, we obtain $Y_1, Y_2, Y_3, Y_4$ in $\OR{1}$ field
  operations.
\end{proof}

\chapter{Implementation and performance} \label{chapter:implementation}

All the algorithms that have been described have been implemented in
the computer algebra system $\MAGMA$. As we remarked in Section
\ref{section:intro_conj}, the implementation has been a major part of
the work and has heavily influenced the nature of the theoretical
results. The algorithms have been developed with the implementation in
mind from the start, and hence only algorithms that can be implemented
and executed on current hardware have been developed.

This chapter is concerned with the implementation, and we will provide
experimental evidence of the fact that the algorithms indeed are
efficient in practice. The evidence will be in the form of benchmark
results, tables and diagrams. This chapter is therefore not so much
about mathematics, but rather about software engineering or computer
science.

The implementations were developed during a time span of 2-3 years,
using $\MAGMA$ versions 2.11-5 and above. The benchmark results have
all been produced using version 2.13-12, Intel64 flavour, statically
linked.

The hardware used during the benchmark was a PC, with an Intel Xeon
CPU, clocked at $2.80$ GHz, and with $1$ GB of RAM. The operating
system was Debian GNU/Linux Sarge, with kernel version
2.6.8-12-em64t-p4-smp.

The implementations used the existing $\MAGMA$ implementations of the
algorithms described in Chapter \ref{chapter:intro}. These include implementations of the following:
\begin{itemize}
\item A discrete log algorithm, in particular Coppersmith's algorithm. The implementation is described in \cite{MR1934518}.
\item The product replacement algorithm.
\item The algorithm from Theorem \ref{thm_solve_univariate_polys}.
\item The Order algorithm, for calculating the
  order (or pseudo-order) of a matrix.
\item The black box naming algorithm from \cite{general_recognition}.
\item The algorithms from Theorems \ref{thm_psl_recognition} and \ref{thm_psl_naming}.
\item The three algorithms from Section \ref{section:aschbacher_algorithms}.
\end{itemize}

We used MATLAB and \cite{r_man} to produce the figures. In every case,
the benchmark of an algorithm was performed by running the algorithm a
number of times for each field with size $q$ lying in some specified
range. We then recorded the time $t$ (in seconds) taken by the
algorithm. However, to be able to compare the benchmark results with
our stated time complexities, we want to display not the time in
seconds, but the number of field operations. Moreover, the input size
is polynomial in $\log(q)$ and not $q$. Hence we first recorded the
time $t_k$ for $k$ multiplications in $\F_q$ and display $t / t_k$
against $\log(q)$. Of course, $k = 1$ is in principle enough, but we chose
$k = 10^6$ to achieve a scaling of the graph.

In $\MAGMA$, Zech logarithms are used for the finite field
arithmetic in $\F_q$ if $q \leqslant q_Z$ (for some $q_Z$, at present $q_Z \approx 2^{20}$), and for
larger fields $\MAGMA$ represents the field elements as polynomials
over the largest subfield that is smaller than $q_Z$, rather than over
the prime field. The reason is that the polynomials then have fewer terms,
and hence the field arithmetic is faster, than if the polynomials have
coefficients in the prime field. Since the subfield is small enough to
use Zech logarithms, arithmetic in the subfield is not much slower
than in the prime field.

Now consider a field of size $p^n$. If $n$ is prime, there are no
subfields except the prime field, and the field arithmetic will be
slow, but if $n$ has a divisor only slightly smaller than $q_Z$, then
the field arithmetic will be fast. Since it might happen that $n$ is
prime but $n+1$ is divisible by $q_Z$, we will get jumps in our
benchmark figures, unless we turn off all these optimisations in
$\MAGMA$. Therefore, this is what we do, and hence any jumps are the
result of group theoretical properties, the discrete log and
factorisation oracles, and the probabilistic nature of the algorithms.

All the non-constructive recognition algorithms that we have
presented, in Sections \ref{section:sz_recognition},
\ref{section:small_ree_recognition} and
\ref{section:bigree_recognition} are extremely fast, and in practice
constant time for the field sizes under consideration. Hence we do not
display any benchmarks of them.

\section{Suzuki groups}

In the cases of the Suzuki groups, the field size is always $q = 2^{2m
  + 1}$ for some $m > 0$. Hence we display the time against $m$. 



In Figure \ref{fig:sz_stab_benchmark} we show the benchmark of the
first two steps of the algorithm in Theorem \ref{sz_thm_pre_step},
where a stabiliser in $\Sz(q)$ of a point of $\OV$ is computed. For
each field size, we made $100$ runs of the algorithms, using random
generating sets and random points.

Notice that the time is very much dominated by the discrete logarithm
computations. The oscillations in the discrete log timings have number
theoretic reasons. When $m = 52$, the factorisation of $q - 1$
contains no prime with more than $6$ decimal digits, hence discrete
log is very fast. On the other hand, when $m = 64$, the factorisation
of $q - 1$ contains a prime with $26$ decimal digits.

\begin{figure}[ht]
\includegraphics[scale=0.7]{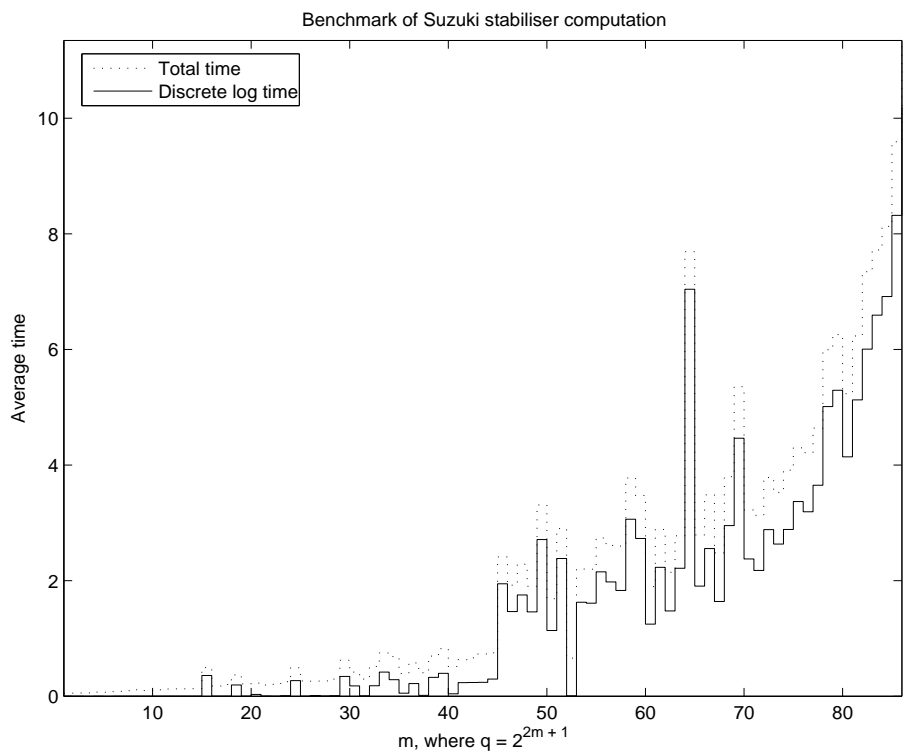}
\caption{Benchmark of Suzuki stabiliser computation}
\label{fig:sz_stab_benchmark}
\end{figure}

In Figure \ref{fig:sz_conj_benchmark} we show the benchmark of the
algorithm in Theorem \ref{thm_conj_problem}. For
each field size, we made $100$ runs of the algorithms, using random
generating sets of random conjugates of $\Sz(q)$.

The time complexity stated in the theorem suggests that the graph
should be slightly worse than linear. Figure
\ref{fig:sz_conj_benchmark} clearly supports this. The minor
oscillations can have at least two reasons. The algorithm is
randomised, and the core of the algorithm is to find an element of
order $q - 1$ by random search. The proportion of such elements is
$\phi(q - 1) / (2 (q - 1))$ which oscillates slightly when $q$
increases.

\begin{figure}[ht]
\includegraphics[scale=0.7]{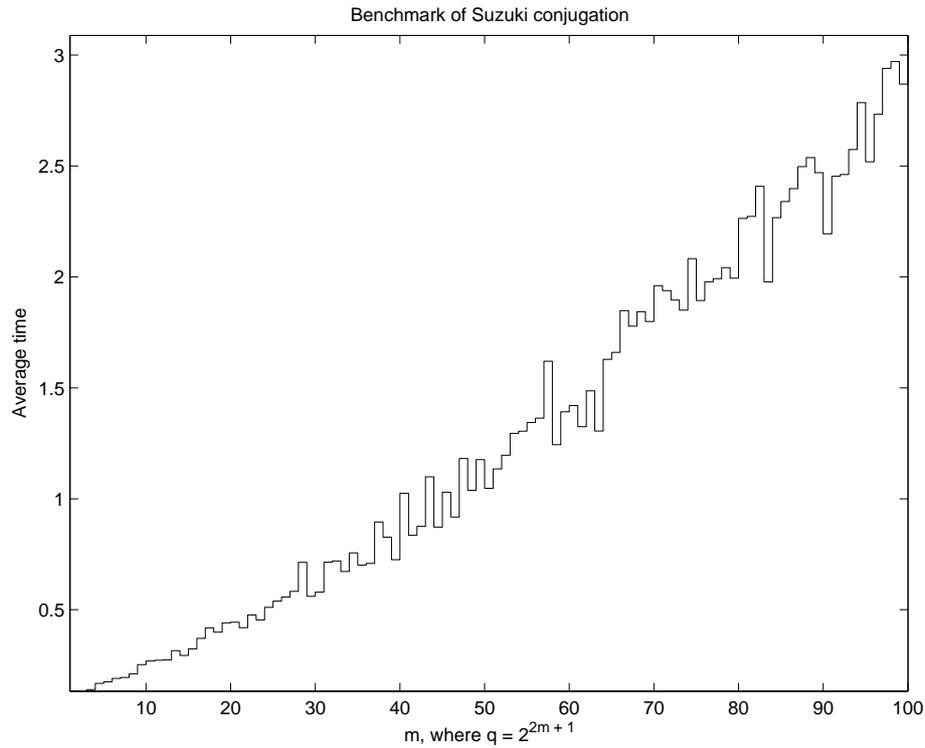}
\caption{Benchmark of Suzuki conjugation}
\label{fig:sz_conj_benchmark}
\end{figure}

We do not include graphs of the tensor decomposition algorithms for
the Suzuki groups. The reason is that at the present time, they can
only be executed on a small number of inputs (certainly not more than
$d \in \set{16, 64, 256}$ and $q \in \set{8, 32, 128}$) before running
out of memory. Hence there is not much of a graph to display.

\section{Small Ree groups}

In the cases of the small Ree groups, the field size is always $q = 3^{2m
  + 1}$ for some $m > 0$. Hence we display the time against $m$. 

\begin{figure}[ht]
\includegraphics[scale=0.7]{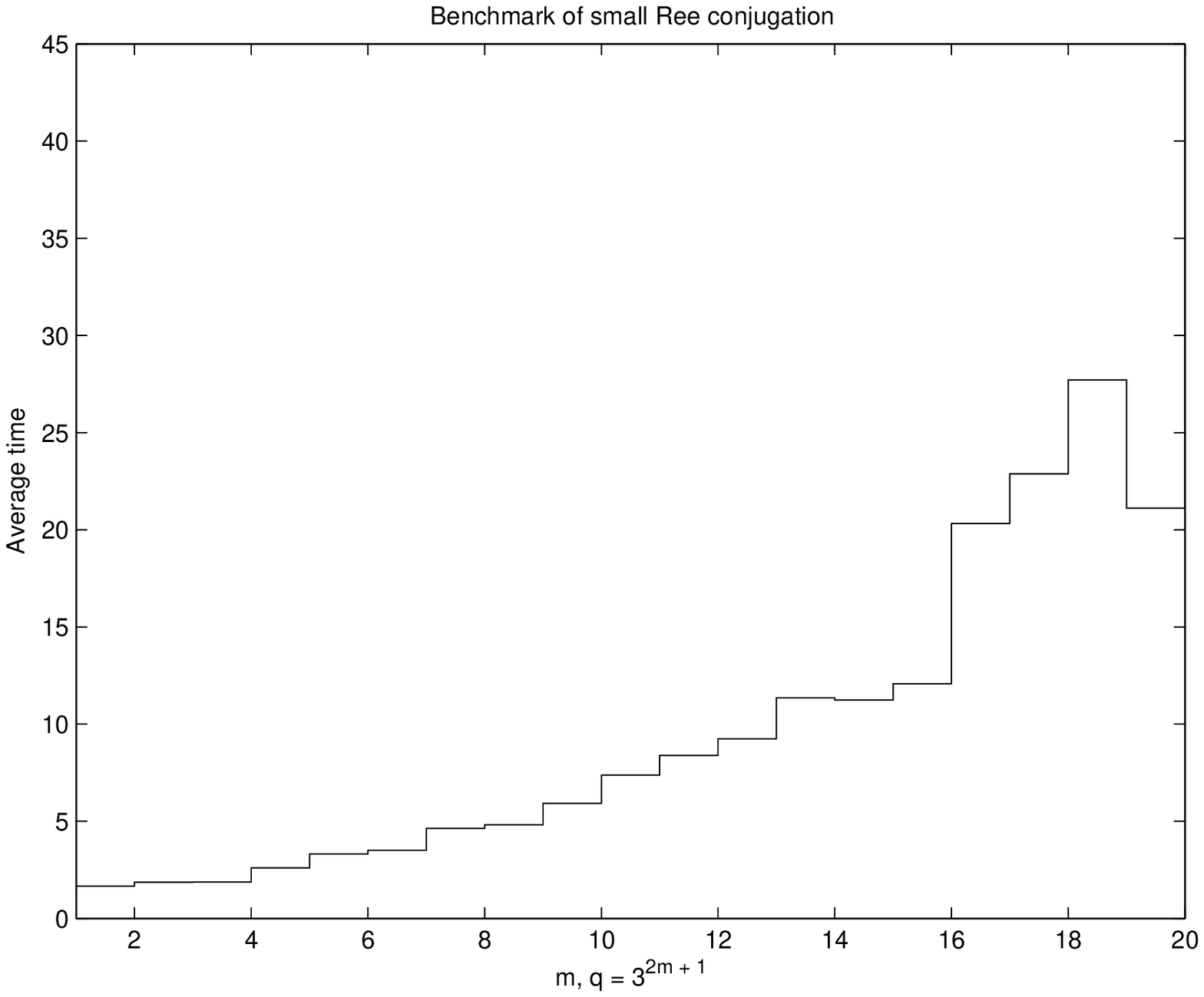}
\caption{Benchmark of small Ree conjugation}
\label{fig:ree_conj_benchmark}
\end{figure}

In Figure \ref{fig:ree_conj_benchmark} we show the benchmark of the
algorithm in Theorem \ref{cl_ree_conjugacy}. For each field size, we
made $100$ runs of the algorithms, using random generating sets of
random conjugates of $\Ree(q)$. As can be seen from the proof of the
Theorem, the algorithm involves many ingredients: discrete logarithms,
$\SLP$ evaluations, $\SL(2, q)$ recognition. To avoid making the graph
unreadable, we avoid displaying the timings for these various steps,
and only display the total time. The graph still has jumps, for
reasons similar as with the Suzuki groups.

We do not include graphs of the tensor decomposition algorithms for
the small Ree groups. The reason is that at the present time, they can
only be executed on a small number of inputs (certainly not more than
$d \in \set{49, 189, 729}$ and $q \in \set{27, 243}$) before running
out of memory. Hence there is not much of a graph to display.

\section{Big Ree groups}

In the cases of the Big Ree groups, the field size is always $q = 2^{2m
  + 1}$ for some $m > 0$. Hence we display the time against $m$. 

\begin{figure}[ht]
\includegraphics[scale=0.7]{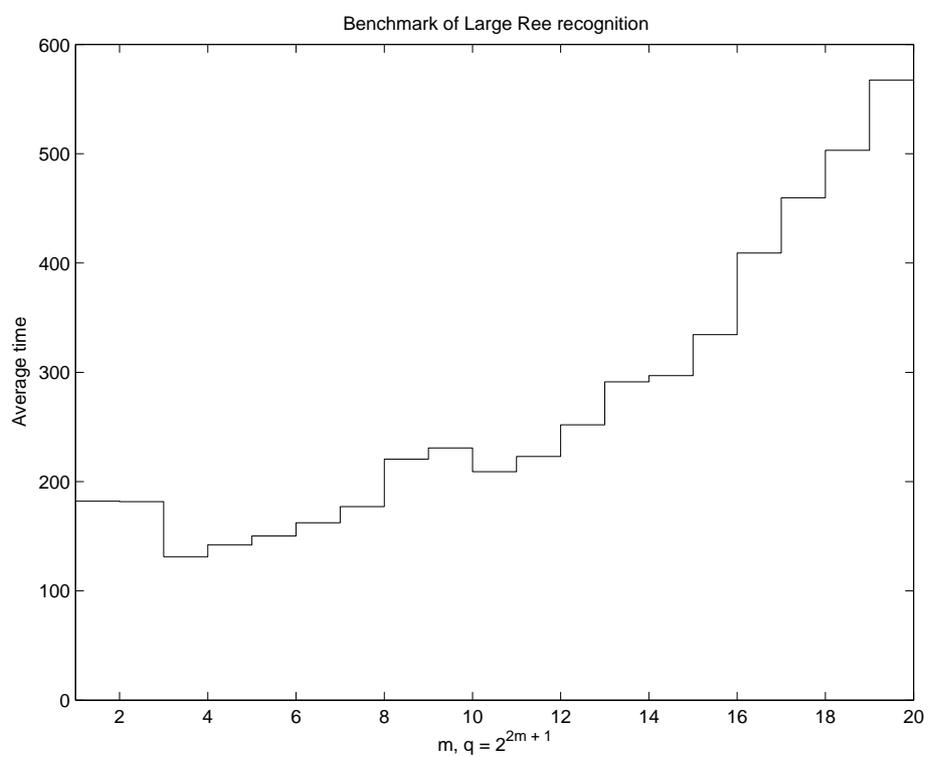}
\caption{Benchmark of Large Ree conjugation}
\label{fig:bigree_recognition_benchmark}
\end{figure}

In Figure \ref{fig:bigree_recognition_benchmark} we show the benchmark of the
algorithm in Theorem \ref{thm_bigree_constructive_recognition}. This
involves all the results presented for the Big Ree groups.

For each field size, we made $100$ runs of the algorithms, using
random generating sets of random conjugates of $\LargeRee(q)$. As can
be seen from the proof of the Theorem, the algorithm involves many
ingredients: discrete logarithms, $\SLP$ evaluations, $\Sz(q)$
recognition, $\SL(2, q)$ recognition. To avoid making the graph
unreadable, we avoid displaying the timings for these various steps,
and only display the total time.

\appendix

\backmatter

\bibliographystyle{amsalpha}
\bibliography{alg_twisted}

\end{document}